\documentclass[12pt]{amsart}

\usepackage[margin=1.0in]{geometry}
\usepackage{caption}
\usepackage{graphicx,longtable,wrapfig,xfrac,enumitem,rotating,extpfeil,amsmath,tikz,tikz-cd,amssymb,mathrsfs,float,array,mathtools,bm,amsthm,hyperref,booktabs,makecell}

\newcolumntype{L}[1]{>{\raggedright\arraybackslash}p{#1}}
\newcolumntype{M}[1]{>{\centering\arraybackslash}m{#1}}
\usetikzlibrary{calc,decorations.pathmorphing,shapes}
\def\R{0.6cm}
\def\S{0.4cm}
\def\T{0.3cm}
\def\P{0.1cm}
\tikzset{
  pics/mynode1/.style args={#1,#2,#3,#4,#5}{
    code={
      \filldraw[thick, color=violet,fill=violet!5] (90:0.8) circle(0.6);
      \filldraw[thick, color=blue,fill=blue!5] (-30:0.8) circle(0.6);
      \draw[thick, color=cyan,fill=cyan!5] (210:0.8) circle(0.6);
      \node (#1) at (0,0) {};
      \node[#5] at (90:0.8) {#2};
      \node[#5] at (-30:0.8) {#3};
      \node[#5] at (210:0.8) {#4};
      \node at (90:1.7) {$k$};
      \node at (-30:1.7) {$j$};
      \node at (210:1.7) {$i$};
     }
  }
}

\newcommand{\hextwist}[3]{\begin{tikzpicture}
   \coordinate (c) at (1,1);
   \draw (0:\S) \foreach \x in {60,120,...,360} {  -- (\x:\S) };
   \draw [blue, thick, -{Stealth}] {[rounded corners] (-150:\S*0.8) -- (90:\S*0.8) -- (-30:\S*0.8)};
   \foreach \x/\l/\p in
    { 60/{(2,3,1)}/above,
      120/{(2,1,3)}/above,
      180/{(1,2,3)}/left,
      240/{(1,3,2)}/below,
      300/{(3,1,2)}/below,
      360/{(3,2,1)}/right
    }
    \node at (\x:\S) {};
    \node[label={#1}] at (90:\S*0.6) {};
    \node[label={#2}] at (-50:\S*1.9) {};
    \node[label={#3}] at (-130:\S*1.9) {};
  \end{tikzpicture}}

\newcommand{\hexchange}[3]{\begin{tikzpicture}
   \coordinate (c) at (1,1);
   \draw (0:\R) \foreach \x in {60,120,...,360} {  -- (\x:\R) };
   \draw [color = blue, thick,decorate,decoration={coil,aspect=0}] (-155:\R*0.8) -- (95:\R*0.8);
   \draw [color = blue, thick,decorate,decoration={coil,aspect=0}] (-145:\R*0.8) -- (-35:\R*0.8);
   \draw [color = violet, arrows = {-Stealth[inset=4pt, angle=80:9pt]}] [line width=1pt, double distance=1pt] (120:\R*0.4) to [bend left] (-60:\R*0.4);
   \foreach \x/\l/\p in
    { 60/{(2,3,1)}/above,
      120/{(2,1,3)}/above,
      180/{(1,2,3)}/left,
      240/{(1,3,2)}/below,
      300/{(3,1,2)}/below,
      360/{(3,2,1)}/right
    }
    \node at (\x:\R) {};
    \node[label={#1}] at (90:\R*0.7) {};
    \node[label={#2}] at (-50:\R*1.7) {};
    \node[label={#3}] at (-130:\R*1.7) {};
  \end{tikzpicture}}
\newcommand{\hexchangeDots}[3]{\begin{tikzpicture}
   \coordinate (c) at (1,1);
   \draw (0:\R) \foreach \x in {60,120,...,360} {  -- (\x:\R) };
   \draw [color = blue, thick,decorate,decoration={coil,aspect=0}] (-155:\R*0.8) -- (95:\R*0.8);
   \draw [color = blue, dashed, thick,decorate,decoration={coil,aspect=0}] (-145:\R*0.8) -- (-35:\R*0.8);
   \foreach \x/\l/\p in
    { 60/{(2,3,1)}/above,
      120/{(2,1,3)}/above,
      180/{(1,2,3)}/left,
      240/{(1,3,2)}/below,
      300/{(3,1,2)}/below,
      360/{(3,2,1)}/right
    }
    \node at (\x:\R) {};
    \node[label={#1}] at (90:\R*0.7) {};
    \node[label={#2}] at (-50:\R*1.7) {};
    \node[label={#3}] at (-130:\R*1.7) {};
  \end{tikzpicture}}
\newcommand{\Dchange}[4]{\begin{tikzpicture}
    \coordinate (a) at (-1,0);
    \coordinate (b) at (0,0);
    \coordinate (c) at (0.8,0.5);
    \coordinate (d) at (0.8,-0.5);
    \coordinate (m) at ({$(b)!.7!(c)$} -| {$(b)!.7!(c)$});
    \coordinate (m2) at ({$(b)!.7!(d)$} -| {$(b)!.7!(d)$});
    
    \draw [shorten <=0.2cm,shorten >=0.2cm] (a) -- (b);
    \draw [shorten <=0.2cm,shorten >=0.2cm] (b) -- (c);
    \draw [shorten <=0.2cm,shorten >=0.2cm] (b) -- (d);
    \draw [color = violet, shorten <=0.1cm,shorten >=0.05cm, out=-70,in=70, arrows = {-Stealth[inset=4pt, angle=70:10pt]},line width=1pt, double distance=1pt] (m) to (m2);
    
    \node at (a) {#1};
    \node at (b) {#2};
    \node at (c) {#3};
    \node at (d) {#4};
  \end{tikzpicture}}
\newcommand{\DchangeDash}[4]{\begin{tikzpicture}
    \coordinate (a) at (-1,0);
    \coordinate (b) at (0,0);
    \coordinate (c) at (0.8,0.5);
    \coordinate (d) at (0.8,-0.5);
    \coordinate (m) at ({$(b)!.7!(c)$} -| {$(b)!.7!(c)$});
    \coordinate (m2) at ({$(b)!.7!(d)$} -| {$(b)!.7!(d)$});
    
    \draw [shorten <=0.2cm,shorten >=0.2cm] (a) -- (b);
    \draw [shorten <=0.2cm,shorten >=0.2cm] (b) -- (c);
    \draw [dashed, shorten <=0.2cm,shorten >=0.2cm] (b) -- (d);
    \draw [color = violet, shorten <=0.1cm,shorten >=0.05cm, out=-70,in=70, arrows = {-Stealth[inset=4pt, angle=70:10pt]},line width=1pt, double distance=1pt] (m) to (m2);
    
    \node at (a) {#1};
    \node at (b) {#2};
    \node at (c) {#3};
    \node at (d) {#4};
  \end{tikzpicture}}

\newcommand\longline{%
  \mathrel{\begin{tikzpicture}[baseline=-0.5ex]
    \draw (0,0) -- (0.4cm,0);
  \end{tikzpicture}}%
}

\newcommand\xline[1]{%
  \mathrel{\begin{tikzpicture}[baseline=-0.5ex]
    \draw (0,0) -- (0.4cm,0) node [midway, label={[label distance=-0.2cm]90:$\scriptstyle #1$}] {};
  \end{tikzpicture}}%
}

\newcommand\xsquigline[1]{%
  \mathrel{\begin{tikzpicture}[baseline=-0.5ex]
    \draw[decorate, decoration={snake, amplitude=0.7pt, segment length=1.2mm}] (0,0) -- (0.4cm,0) node [midway, label={[label distance=-0.2cm]90:$\scriptstyle #1$}] {};
  \end{tikzpicture}}%
}

\DeclareMathOperator{\sgn}{sgn}
\DeclareMathOperator{\im}{im}
\DeclareMathOperator{\Id}{Id}
\DeclareMathOperator{\Mid}{mid}
\DeclareMathOperator{\Sq}{Sq}
\DeclareMathOperator{\sq}{sq}
\DeclareMathOperator{\gr}{gr}
\DeclareMathOperator{\Ob}{Ob}
\DeclareMathOperator{\hocolim}{hocolim}
\DeclareMathOperator{\conv}{conv}
\DeclareMathOperator{\rank}{rank}
\DeclareMathOperator{\cell}{cell}
\DeclareMathOperator{\Int}{int}

\hypersetup{
    colorlinks=true,
    linkcolor=blue,
    filecolor=magenta,
    urlcolor=cyan,
    pdftitle={On Steenrod squares for even and odd Khovanov homology}
  }
\newtheorem{theorem}{Theorem}[section]
\newtheorem{proposition}[theorem]{Proposition}
\newtheorem{corollary}[theorem]{Corollary}
\newtheorem{lemma}[theorem]{Lemma}
\theoremstyle{definition}
\newtheorem{procedure}[theorem]{Procedure}
\newtheorem{definition}[theorem]{Definition}
\newtheorem{notation}[theorem]{Notation}
\newtheorem{question}[theorem]{Question}
\newtheorem{construction}[theorem]{Construction}
\newtheorem*{notation*}{Notation}
\theoremstyle{remark}    
\newtheorem{example}{Example}[theorem]
\newtheorem{remark}{Remark}[theorem]
\newcommand{\bigslant}[2]{{\raisebox{.2em}{$#1$}\left/\raisebox{-.2em}{$#2$}\right.}}
\newcommand{\boldvarphi}{\boldsymbol{\varphi}}
\makeatletter
\newcommand{\pushright}[1]{\ifmeasuring@#1\else\omit\hfill$\displaystyle#1$\fi\ignorespaces}
\makeatother

\graphicspath{{images/inkscapeimages}}
\title{On Steenrod squares for even and odd Khovanov homology}
\author{Advika Rajapakse}
\address{Department of Mathematics\\University of California\\Los Angeles, CA 90095}
\email{advika@math.ucla.edu}
\thanks{AR was supported by NSF Grant DMS-2136090}
\begin{document}
\maketitle
\begin{abstract}
  For an arbitrary link $L\subset S^3$, Sarkar-Scaduto-Stoffregen constructed a family $\mathcal{X}_l(L)$, $l\geq 0$, of spaces, giving a family of spatial refinements of even and odd Khovanov homology. We give a computation of $\Sq^2$ on these spaces, determining the stable homotopy type of $\mathcal{X}_l(K)$ for all $l$ and all knots $K$ up to 11 crossings. We also prove that the Steenrod squares $\Sq_0^2$, $\Sq_1^2$ defined by Sch\"utz do arise as Steenrod squares on these spaces.
\end{abstract}
\tableofcontents
\section{Introduction}
\subsection{Khovanov homologies}
Khovanov homology, a categorification of the Jones polynomial, gives a bigraded vector space $Kh^{i,j}(L)$ for each link $L\subset S^3$ whose graded Euler characteristic is the Jones polynomial $V(L)$ \cite{MR1740682}. Furthermore, $Kh^{i,j}(L)$ is an invariant of the isotopy class of $L$. Following its discovery, we have seen several generalizations, such as tangle invariants  \cite{MR1928174, MR2174270}, and perturbations \cite{MR2174270, MR2173845}. We have also seen that Khovanov homology has a functorial property where associated to a link cobordism in $\mathbb{R}^3\times [0,1]$, there is a homomorphism of Khovanov chain complexes \cite{MR2174270}. This result has found many exciting applications, such as Rasmussen's $s$-invariant \cite{MR2729272} giving a lower bound for the slice genus of a knot and a proof \cite{MR4076631} that the Conway knot is not slice. There have also been numerous applications in Khovanov homology to bound the Thurston-Bennequin number for Legendrian knots \cite{MR2186113,MR2250492}. Finally, Khovanov homology, along with Khovanov-Rozansky homology, has been used to define an exciting new family of smooth four-manifold invariants \cite{MR4562565}.\par
In \cite{MR3071132}, Ozsv\'ath,  Rasmussen, and Szab\'o construct a modified version of $Kh$, which we call $Kh_o$, that shares the same $\mathbb{Z}/2$ reduction as $Kh$, but differ over $\mathbb{Q}$. $Kh_o$ shares similar structural properties and reduced theories, and there are even chain maps associated to cobordisms \cite{MR3363817}.
\subsection{Khovanov homotopy types}
Lipshitz-Sarkar \cite{MR3230817} have constructed a space-level link invariant $\mathcal{X}_{Kh}(L)$ that refines Khovanov homology. In particular, taking the cohomology of $\mathcal{X}_{Kh}(L)$ recovers $Kh^{i,j}(L)$, and furthermore, the stable homotopy type of $\mathcal{X}_{Kh}(L)$ is a link invariant. It has been shown \cite{MR4153651} that $\mathcal{X}_{Kh}(L)$ enjoys further structural properties, in particular, regarding split unions, connect sums, and mirrors. Furthermore, stable cohomology operations, like the Steenrod squares $\Sq^n$, on these spaces give operations $\Sq^n: Kh^{i,j}(L)\to Kh^{i+n,j}(L)$ which are not generally trivial \cite{MR4153651}.\par
Sarkar-Scaduto-Stoffregen \cite{MR4078823} have since constructed modifications of this space, defining a family of spaces $\mathcal{X}_l(L)$, $l\geq 0$ beginning with $\mathcal{X}_0(L) = \mathcal{X}_{Kh}(L)$ and $\mathcal{X}_1(L) = \mathcal{X}_o(L)$. For $l$ even, $\mathcal{X}_l(L)$ is a refinement of $Kh$, and for $l$ odd, $\mathcal{X}_l(L)$ is a refinement of $Kh_o$. It has been a question whether the stable homotopy type of $\mathcal{X}_l(L)$ only depends on $l\ \mathrm{mod}\ 2$, and whether these spectra $\mathcal{X}_l(L)$ share structural properties analogous to the original $\mathcal{X}_{Kh}(L)$, and. We answer the former question in the negative, and some of the latter questions follow immediately from our computation of the second Steenrod square $\Sq^2$ on these spaces $\mathcal{X}_l(L)$.\par
Our first theorem both answers Question (q-1) and (q-7) from \cite{MR4078823}. In particular, (q-1) is answered in the positive and (q-7) is answered in the negative.
\begin{theorem}\label{depends more than mod 2}
  The spectrum $\mathcal{X}_2(T_{3,4})$ is a wedge sum of Moore spaces, implying $\mathcal{X}_0(T_{3,4})\not\sim \mathcal{X}_2(T_{3,4})$. Furthermore, $\mathcal{X}_1(T_{3,4})$ is a wedge sum of Moore spaces while $\mathcal{X}_3(T_{3,4})$ is not, so we have $\mathcal{X}_1(T_{3,4})\not\sim \mathcal{X}_3(T_{3,4})$,
\end{theorem}\par
Our next theorem refutes a potential property regarding links and their mirrors. Indeed, note that from \cite{MR4153651} the property
\[
  \mathcal{X}_0(m(L)) \sim \mathcal{X}_0(L)^\vee,
\]
where ``$\vee$'' denotes the Spanier-Whitehead dual in this case. Question (q-5) in \cite{MR4078823} asks if this property extends to the odd Khovanov spectra. The following theorem answers this question in the negative.
\begin{theorem}\label{not dual}
  Let $\vee$ denote the Spanier-Whitehead dual operation. We have $\mathcal{X}_1(m(T_{3,4}))\not\sim \mathcal{X}_1(T_{3,4})^{\vee}$, $\mathcal{X}_3(m(T_{3,4}))\not\sim \mathcal{X}_3(T_{3,4})^\vee$.
\end{theorem}

\begin{theorem}\label{not smash product}
  We have $\mathcal{X}_k(T_{2,3}\sqcup T_{2,3})\not\sim \mathcal{X}_k(T_{2,3})\wedge \mathcal{X}_k(T_{2,3})$ for all odd $k$. 
\end{theorem}\par
Recall from \cite{MR3071132} the property $Kh^{i,j}_o(L) \cong Kh^{i,j-1}_o(L)\oplus Kh^{i,j+1}_o(L)$. The following theorem refutes a potential space-level lift of this property, and answers Question (q-2) from \cite{MR4078823} in the negative.
\begin{theorem}\label{relation with sum of reduced}
  We have $\mathcal{X}^{-11}_1(T_{3,-4}) \not\sim \widetilde{\mathcal{X}}^{-12}_1(T_{3,-4})\vee \widetilde{\mathcal{X}}^{-10}_1(T_{3,-4})$.
\end{theorem}\par
We now discuss Conway mutation and the odd Khovanov spectra. Note that while Khovanov homology is not invariant under Conway mutation, Khovanov homology over $\mathbb{F}_2$ \cite{MR2657645}, and more generally, odd Khovanov homology \cite{MR2592723}, is mutation-invariant. However, these results do not extend to the odd Khovanov homotopy types, as can be seen by studying the mutant pair $T_{2,3}\sqcup T_{2,3}$, $T_{2,3}\# T_{2,3}\sqcup U$.
\begin{theorem}\label{mutant}
  We have a non-equivalence of mutant pairs $\mathcal{X}_k(T_{2,3}\sqcup T_{2,3})\not\sim \mathcal{X}_k(T_{2,3}\# T_{2,3}\sqcup U)$ for all odd $k$.
\end{theorem}

\subsection{Steenrod squares}
Lipshitz-Sarkar \cite{MR3252965} have given an explicit formula for the Steenrod square $\Sq^2$ on the even Khovanov homotopy type $\mathcal{X}_e(L)$, computing the stable homotopy type of $\mathcal{X}_e(L)$ for all prime links up to 11 crossings and giving a computable definition for $\Sq^2: Kh^{i,j}\to Kh^{i,j}$. Sch\"utz \cite{MR4970173} modified Lipshitz-Sarkar's definition of $\Sq^2$ defining operations $\Sq_0^2:Kh_o^{i,j}\to Kh_o^{i+2,j}$, $\Sq_1^2: Kh_o^{i,j}\to Kh_o^{i+2,j}$, which are themselves link invariants, and give rise to new $s$-invariants. Sch\"utz conjectured that the operations $\Sq_0^2$, $\Sq_1^2$, arise from the $\Sq^2$ operations on the odd Khovanov spectra $\mathcal{X}_{2l+1}(L)$. We confirm this conjecture:
\begin{theorem}\label{agrees with Schutz}
  When viewing the second Steenrod square $\Sq^2|_{\mathcal{X}_l(L)}$ on the space $\mathcal{X}_l(L)$ as an operation on $Kh(L;\mathbb{F}_2)$, we have $\Sq^2|_{\mathcal{X}_1(L)} = \Sq_1^2$ and $\Sq^2|_{\mathcal{X}_3(L)} = \Sq_0^2$.
\end{theorem}
We might ask whether there are other Steenrod squares arising from the rest of these spaces $\mathcal{X}_l(L)$. The answer is that there are only four total, including the $\Sq_0^2$, $\Sq_1^2$ on odd $Kh$, and the original $\Sq^2$ on even $Kh$.
\begin{theorem}\label{sums to 0}
  The second Steenrod square $\Sq^2|_{\mathcal{X}_l(L)}$, viewed as an operation on $Kh(L;\mathbb{F}_2)$, only depends on $l\ \mathrm{mod}\ 4$. Furthermore, $\Sq^2|_{\mathcal{X}_0(L)}+\Sq^2|_{\mathcal{X}_1(L)}+\Sq^2|_{\mathcal{X}_2(L)}+\Sq^2|_{\mathcal{X}_3(L)} = 0$.
\end{theorem}\par

We conclude with some questions motivated by the above theorem and our computations.
\begin{question}
  Does the stable homotopy type of the space $\mathcal{X}^j_l(L)$ only depend on $l\ \mathrm{mod}\ 4$?
\end{question}
\begin{question}
  Does there exist a non-split link $L$ such that $\mathcal{X}^j_2(L)$ is not a wedge sum of Moore spaces?
\end{question}

\begin{question}
  Is the odd Khovanov spectrum $\mathcal{X}^j_1(m(L))$ Spanier-Whitehead dual to the spectrum $\mathcal{X}^j_3(L)$? Is it dual to $\mathcal{X}^j_l(L)$ for some $l\equiv 3\mod 4$?
\end{question}
\begin{question}
  Do we have $\mathcal{X}^j_2(L\sqcup L') \sim \mathcal{X}^j_2(L)\wedge \mathcal{X}^j_2(L')$ for arbitrary links $L$, $L'$?
\end{question}
\begin{question}
  Are the odd Khovanov spectra $\mathcal{X}_1(L),\mathcal{X}_3(L),\ldots$ invariant under component-preserving Conway mutation?
\end{question}

\paragraph{Acknowledgments}
The author would like to thank Sucharit Sarkar for many helpful conversations, and for introducing the author to this problem. The author would also like to thank Robert Lipshitz for his helpful comments. The author is also grateful to Krishnendu Kar for suggesting the author compute the odd Khovanov spectrum for Conway mutations.
 \section{Outline of argument and review of $\text{Sq}^2$}
Our first goal is to compute the Steenrod square $\Sq^2:H^{*}(X;\mathbb{F}_2)\to H^{*+2}(X;\mathbb{F}_2)$ of the odd Khovanov spectrum $X:= \mathcal{X}_1(L)$ of a link $L$. In our construction of $X$, $X$ is the formal desuspension $\Sigma^{-N}Y$ of some CW complex $Y$. Our focus now turns to studying $\Sq^2:H^{*}(Y;\mathbb{F}_2)\to H^{*+2}(Y;\mathbb{F}_2)$. We simplify further by studying $\Sq^2:H^{m}(Y';\mathbb{F}_2)\to H^{m+2}(Y';\mathbb{F}_2)$ for a simpler complex $Y'$ with only cells of dimension $m$, $(m+1)$, $(m+2)$, where $m>2$. Our strategy is now as follows.
\begin{enumerate}
\item Fix a cycle $\mathbf{c}\in C^m(Y';\mathbb{F}_2)$.
\item Construct an Eilenberg MacLane space $K_m:= K(\mathbb{Z}/2,m)$ with one $m$-cell $e^m$, one $(m+1)$-cell $e^{m+1}$, one $(m+2)$-cell $e^{m+2}$, and higher-dimension cells ($e^m$ shall be the fundamental class $\iota$). We only need the $(m+2)$ skeleton $K_m^{(m+2)}$.
\item Construct a map $\mathfrak{c}: Y'\to K_m^{(m+2)}$ such that $\mathfrak{c}^*\iota = [\mathbf{c}]$.
\item Conclude that $\Sq^2([\mathbf{c}])= \mathfrak{c}^*\Sq^2(\iota) =\mathfrak{c}^*[e^{m+2}] = [\mathfrak{c}^* e^{m+2}]$.
\end{enumerate}
There are many choices of space $K_m^{(m+2)}$, so it is wise to pick one that is simple enough to work with. For $m>2$, we construct a model of the $m^{\text{th}}$ Eilenberg Maclane space $K(m,\mathbb{Z}/2)$, which we call $K_m$. We can choose the $m$-cell to be $e^m$ with the entire boundary $\partial e^m$ glued to the basepoint. To satisfy $\pi_m(K_m)=\mathbb{Z}/2$, we attach a single $(m+1)$-cell by a degree $2$ map $\partial e^{m+1}\to K_k^{(k)}=S^m$. The resulting $(m+1)$ skeleton $K_m^{(m+1)}$ has $\pi_{m+1}(K_m^{(m+1)})\cong \mathbb{Z}/2$, which is a consequence of the following lemmas:
\begin{lemma}\label{1 stem}
  $\pi_1^s(\mathbb{S}) = \mathbb{Z}/2$, with generator represented by $\eta$, where $\eta:S^3\to S^2$ is the Hopf map.
\end{lemma}
\begin{proof}
  From the fiber bundle $S^1\to S^3\xrightarrow{\eta}S^2$, we have that $\pi_3(S^2)\cong \pi_3(S^3)\cong \mathbb{Z}$, and from the Freudenthal suspension theorem, we see that the sequence $\pi_2(S^1)\to \pi_3(S^2)\to \pi_4(S^3)\to \ldots$ stabilizes at $\pi_4(S^3)$. Furthermore, the map $\pi_3(S^2)\to \pi_4(S^3)$ is surjective, so $\pi_4(S^3)$ is cyclic, generated by $\eta$. We do not give a full proof that $\pi_1^s(\mathbb{S})\cong \mathbb{Z}/2$, but we only show that $0\neq [\Sigma \eta]\in \pi_4(S^3)$. For if $\Sigma\eta = 0$, then $\Sq^2$ would act trivially on the (reduced) mapping cone $C(\Sigma \eta)$. But $C(\Sigma \eta) = \Sigma C(\eta) = \Sigma \mathbb{C}P^2$, which has nontrivial $\Sq^2$.
\end{proof}
\begin{lemma}\label{suspension RP2}
  $\pi_2^s(\mathbb{R}P^2) \cong \mathbb{Z}/2$, with generator represented by $S^3\xrightarrow{\eta} S^2\xhookrightarrow{\Sigma i} \Sigma\mathbb{R}P^2$, where $\eta$ is the Hopf map and $i:S^1\to \mathbb{R}P^2$ is the inclusion of the $1$-skeleton. In fact, we have the suspension sequence
  \begin{equation*}
    \begin{tikzpicture}[scale=2]
  \path[draw, ->, shorten <=1.1cm,shorten >=1.1cm, thick] (0,1) -- (2,1) node [midway, label={[label distance=-0.2cm]90:$\times 2$}] {};
  \path[draw, ->, shorten <=1.1cm,shorten >=1.1cm, thick] (2,1) -- (4,1) node [midway, label={[label distance=-0.2cm]90:$\mathrm{mod}\ 2$}] {};
  \path[draw, ->, shorten <=1.1cm,shorten >=1.1cm, thick] (4,1) -- (6,1);
  \draw [double equal sign distance] (2,0.9) to (2,0.5);
  \path[draw, ->, shorten <=0.3cm,shorten >=0.3cm, thick] (0,1) -- (0,0.4) node [midway, label={[label distance=-0.3cm, rotate=-90]45:$\cong$}] {};
  \path[draw, ->, shorten <=0.3cm,shorten >=0.3cm, thick] (4,1) -- (4,0.4) node [midway, label={[label distance=-0.3cm, rotate=-90]45:$\cong$}] {};
     \node at (0, 1) {$\pi_2(\mathbb{R}P^2)$};
   \node at (2, 1) {$\pi_3(\Sigma\mathbb{R}P^2)$};
   \node at (4, 1) {$\pi_4(\Sigma^2\mathbb{R}P^2)$};
   \node at (4, 1) {$\ldots$};
   \node at (0, 0.4) {$\mathbb{Z}$};
   \node at (2, 0.4) {$\langle \eta \rangle$};
   \node at (4, 0.4) {$\mathbb{Z}/2$};
   \node at (5.8, 1) {$\ldots$};
   \end{tikzpicture}
 \end{equation*}
   which stabilizes by $\pi_4(\Sigma^2\mathbb{R}P^2)$. Furthermore, the suspension map $\Sigma^i: \pi_2(\mathbb{R}P^2)\to \pi_{2+i}(\mathbb{R}P^2)$ is nullhomotopic for $i\geq 2$.
 \end{lemma}
 \begin{proof}
   The first homomorphism being multiplication by $2$ is explained from the fact $\Sigma q\cong 2\eta$, where $q$ is canonical generator of $\pi_2(\mathbb{R}P^2)$. The third and following homomorphisms are isomorphism by the Freudenthal suspension theorem. It remains to explain the second homomorphism is $\mathrm{mod}\ 2$. We use the following commutative diagram:
     \begin{equation*}
       \begin{tikzpicture}[scale=2]
         
  \path[draw, ->, shorten <=0.7cm,shorten >=0.7cm, thick] (0,1) -- (2,1) node [midway, label={[label distance=-0.2cm]90:$\mathrm{mod}\ 2$}] {};
  \path[draw, ->, shorten <=0.7cm,shorten >=0.7cm, thick] (2,1) -- (4,1) node [midway, label={[label distance=-0.2cm]90:$\cong$}] {};
  \path[draw, ->, shorten <=0.7cm,shorten >=1.1cm, thick] (4,1) -- (6,1) node [midway, label={[label distance=-0.2cm]90:$\cong$}] {};
  \path[draw, ->, shorten <=1.1cm,shorten >=1.1cm, thick] (0,0) -- (2,0) node [midway, label={[label distance=-0.2cm]90:$\mathrm{mod}\ 2$}] {};
  \path[draw, ->, shorten <=1.1cm,shorten >=1.1cm, thick] (2,0) -- (4,0) node [midway, label={[label distance=-0.2cm]90:$\cong$}] {};
  \path[draw, ->, shorten <=1.1cm,shorten >=1.1cm, thick] (4,0) -- (6,0) node [midway, label={[label distance=-0.2cm]90:$\cong$}] {};
  \path[draw, ->>, shorten <=0.3cm,shorten >=0.3cm, thick] (0,1) -- (0,0) {};
  \path[draw, ->, shorten <=0.3cm,shorten >=0.3cm, thick] (2,1) -- (2,0) node [midway, label={[label distance=-0.3cm, rotate=-90]45:$\cong$}] {};
  \path[draw, ->, shorten <=0.3cm,shorten >=0.3cm, thick] (4,1) -- (4,0) node [midway, label={[label distance=-0.3cm, rotate=-90]45:$\cong$}] {};
  \draw [double equal sign distance] (0,1.2) to (0,1.5);
  \path[draw, ->, shorten <=0.3cm,shorten >=0.3cm, thick] (2,1) -- (2,1.7) node [midway, label={[label distance=-0.3cm, rotate=-90]45:$\cong$}] {};
  \node at (0,1.7) {$\langle\eta\rangle$};
  \node at (2,1.7) {$\mathbb{Z}/2$};
     \node at (0, 1) {$\pi_3(S^2)$};
   \node at (2, 1) {$\pi_4(S^3)$};
   \node at (4, 1) {$\pi_5(S^4)$};
   \node at (4, 1) {$\ldots$};
   \node at (0, 0) {$\pi_3(\Sigma\mathbb{R}P^2)$};
   \node at (2, 0) {$\pi_4(\Sigma^2\mathbb{R}P^2)$};
   \node at (4, 0) {$\pi_5(\Sigma^3\mathbb{R}P^2)$};
   \node at (5.8, 1) {$\ldots$};
   \node at (5.8, 0) {$\ldots$};
   \end{tikzpicture}
 \end{equation*}
 We explain the arrows in this digram first. The top horizontal arrows are a surjection followed by isomorphisms by the proof of Lemma \ref{1 stem}. It now suffices to prove that the vertical isomorphisms are a surjection followed by isomorphisms. We use the cofiber long exact sequence
    \[
      \pi_{m+2} (\Sigma^{m-1}\mathbb{R}P^2,S^m)\to \pi_{m+1}(S^m)\to \pi_{m+1} (\Sigma^{m-1}\mathbb{R}P^2)\to \pi_{m+1} (\Sigma^{m-1}\mathbb{R}P^2,S^m)\to \pi_{m}(S^m),
    \]
    and note that for $m=2$, the sequence is
    \[
      \pi_4 (\Sigma\mathbb{R}P^2,S^2)\to \mathbb{Z}/2\to \pi_3(\Sigma\mathbb{R}P^2)\to \mathbb{Z}\xrightarrow{2}\mathbb{Z},
    \]
    and for $m>2$, the sequence is 
    \[
      \mathbb{Z}/2\xrightarrow{2} \mathbb{Z}/2\to \pi_{m+1}(\Sigma^{m-1}\mathbb{R}P^2)\to \mathbb{Z}\xrightarrow{2}\mathbb{Z}.\qedhere
    \]
 \end{proof}
 To zero out the $\pi_{m+1}(K_m^{(m+1)})$, we attach a $(m+2)$-cell $e^{m+2}$, with the attaching map $\partial e^{m+2}\cong S^{m+1}\xrightarrow{\Sigma^{m-2}\eta} S^m\cong K_m^{(m)}$ being a $(m-2)$-fold suspension of the Hopf map $\eta$.
 \section{Permutohedra and twists}
\subsection{Permutohedra}
We first give our definition of permutohedra:
\begin{definition}
  In $\mathbb{R}^n$, let $v_\sigma = (\sigma^{-1}(1),\ldots,\sigma^{-1}(n))\in \mathbb{R}^n$ be the $\sigma$ permutation of the tuple $(1,\ldots,n)$. The $(n-1)$-dimensional permutohedron $\Pi_{n-1}$ is the convex hull in $\mathbb{R}^n$ of the $n!$ points $v_\sigma$.
\end{definition}
\begin{remark}\label{metric discussion}
  Note that $\Pi_{n-1}$ is a $(n-1)$-dimensional polytope in the affine subspace $\mathbb{A}^{n-1}:=\{(x_1,\ldots,x_{n})\in \mathbb{R}^n\ |\ \Sigma_i x_i = n(n-1)/2\}$ in $\mathbb{R}^n$. We also remark that the affine $(n-1)$-space $\mathbb{A}^{n-1}$ inherits a smooth structure and Riemannian metric from the space $\mathbb{R}^n$ that it lies in. Therefore, $\Pi^{n-1}\subset \mathbb{A}^{n-1}$ inherits a well-defined tangent space $T(\Pi^{n-1})$ and Riemannian metric from $\mathbb{A}^{n-1}$. 
\end{remark}
\begin{notation}
  We let $e_1,\ldots,e_n$ be the canonical ordered basis of unit vectors in $\mathbb{R}^n$. $e_i-e_j$ are tangent vectors in $\mathbb{A}^{n-1}$ for $1\leq i,j\leq n$, and since $\Pi^{n-1}$ is codimension $0$, these are tangent vectors in $\Pi^{n-1}$.
\end{notation}
We can also define $\Pi^{n-1}$ as an intersection of half-spaces $H_S\subset \mathbb{A}^{n-1}$, which we shall define:
\begin{definition}
  Let $S\in \underline{2}^{\{0,\ldots, n-1\}}\backslash \{\overline{1},\overline{0}\}$, that is $S$ is a non-empty proper subset of $\{0,\ldots, n-1\}$. Let $|S| = k$. Define $H_S\subset \mathbb{A}^{n-1}\subset \mathbb{R}^n$ to be the half-space $\{(x_1,\ldots, x_{n})\}\in \mathbb{A}^{n-1}\mid \Sigma_{i\in S}x_i\geq k(k+1)/2\}$. Define the \textit{facet} $F_S$ of $\Pi^{n-1}$ to be $\Pi^{n-1}\cap \partial H_S$.
\end{definition}
There are  $2^n-2$ of these half-spaces $H_S$, and indeed $\bigcap_{S} H_S=\Pi^{n-1}$. Furthermore, the union of the facets $F_S$ form the boundary $\partial \Pi^{n-1}$. These facets can also be identified with lower-dimensional permutohedra:
\begin{lemma}[\cite{MR4153651}]\label{face identification}
  Let $a_1<a_2<\dots <a_k$ be the elements of $S$ and let $b_1<b_2<\dots b_{n-k}$ be the elements of $\{1,2,\ldots n\}$. The map
  \[
    f_S (x_1,\ldots, x_{n}) = ((x_{a_1}\ldots x_{a_{k}}),(x_{b_1}-k\ldots x_{b_{n-k}}-k))
  \]
  identifies the facet $F_S$ with $\Pi^{k-1}\times \Pi^{n-k-1}$.
\end{lemma}
\begin{proof}[Proof (\cite{MR4153651})]
  It suffices to prove that $f_S$ takes the vertices of $F_S$ to the vertices of $\Pi^{k}\times \Pi^{n-k}$. But the vertices of $F_S$ are the points $(x_1,\ldots x_{n})$ so that $\{x_{a_1}\ldots x_{a_k}\} = \{1,\ldots k\}$ and $\{x_{a_1}\ldots x_{a_k}\} = \{k+1,\ldots n\}$.
\end{proof}
\begin{lemma}\label{face intersection}
  Let $S,T$ be two non-empty, proper subsets of $\{1,\ldots, n\}$, and $F_S,F_T$ their associated facets of $\Pi^{n-1}$. $F_S\cap F_T$ is nonempty if and only if $S\subseteq T$ or $T\subseteq S$.
\end{lemma}
\begin{proof}
  Suppose that $(x_1,\ldots,x_{n})\in F_S\cap F_T$. Then
  \begin{equation}\label{F addition}
    \sum_{i\in S}x_i+\sum_{i\in T}x_i=(1+\ldots+|S|)+(1+\ldots+|T|).
  \end{equation}
  But since $F_S\cap F_T\subset \Pi^{n-1}\subset H_{S\cup T}\cap H_{S\cap T}$,
  \begin{equation}\label{H addition}
    \sum_{i\in S}x_i+\sum_{i\in T}x_i=\sum_{i\in S\cap T}x_i+\sum_{i\in S\cup T}x_i\geq (1+\ldots + |S\cap T|)+ (1 +\ldots + |S\cup T|).
  \end{equation}
  (\ref{F addition}) and (\ref{H addition}) both hold only if either $S\subseteq T$ or $T\subseteq S$.
\end{proof}
An intersection of facets $F_{S_1}\cap\ldots\cap F_{S_m}$ can produce higher codimension boundary components of $\Pi^{n-1}$. Note that from Lemma \ref{face intersection} that the intersection only if the sets $S_1,\ldots, S_m$ can be ordered by inclusion.
\begin{notation}
  Given the set $\{1,\ldots,n\}$, we denote $\overline{0} = \emptyset$, $\overline{1} = \{1,\ldots, n\}$. 
\end{notation}

\begin{definition}
  Given a chain $c=\{\overline{1}=S_m>S_{m-1}>\ldots >S_{1}>S_0=\overline{0}\}$ in $2^{\{1,\ldots, n\}}$, define the \textit{face} $F_c$ of $\Pi^{n-1}$ as the intersection of facets $\bigcap_{1\leq i \leq m-1} F_{S_i}$.
\end{definition}
Another characterization of $F_c$ is
  \[
    F_c = \Big\{(x_1,\ldots,x_n): \sum_{i\in S_{k+1}\backslash S_{k}}x_i = (|S_k|+1)+(|S_k|+2)+\ldots+|S_{k+1}|\text{ for all }0\leq k\leq m-1\Big\},
  \]
  implying that $F_c$ is an $(n-m)$-dimensional face.

\begin{lemma}\label{chain to face}
  Given a chain $c=\{\overline{1}=S_m>S_{m-1}>\ldots >S_{1}>S_0=\overline{0}\}$ in $2^{\{1,\ldots,n\}}$, the map
\[
  f_c(x_1,\ldots, x_n) = (\mathbf{t}_{S_{1},S_0},\ldots,\mathbf{t}_{S_{m},S_{m-1}})
\]
where if $S,T\in 2^{\{1,\ldots, n\}}$, $S>T$, then if $a_1<\ldots < a_{|S|-|T|}$ are the elements of $S\backslash T$, then 
\[
  \mathbf{t}_{S,T} = (x_{a_1}-|T|,\ldots x_{a_{|u|-|v|}}-|T|)
\]
identifies the face $F_c$ of $\Pi^{n-1}$ with $\Pi^{|S_1|-|S_{0}|-1}\times \ldots\times \Pi^{|S_m|-|S_{m-1}|-1}$.
\end{lemma}
Certain facets of $\Pi^{n-1}$ will be particularly useful to us:
\begin{definition}\label{f_i def}
  Fix the permutohedron $\Pi^{n-1}$ and let $0\leq i\leq n-1$ be an integer. For $0\leq a_1<\ldots <a_k \leq n-1$, we define $G_{\{a_1,\ldots,a_k\}}:= F_{\{0,\ldots,n-1\}\backslash\{a_1+1,\ldots,a_k+1\}}$. So in particular, $G_{\{i\}}\subset \Pi^{n-1}$ is the facet $F_{\{1,\ldots,\widehat{i+1},\ldots,n\}}$. Equivalently, $G_{\{i\}}$ is the subset of $\Pi^{n-1}\subset \mathbb{R}^n$ where the $(i+1)^\text{th}$ coordinate is $n$. Our upshifting of indices is done only to be compatibile with the notation for signed flow categories.
  We often write $G_i = G_{\{i\}}$, $G_{ij} = G_{\{i,j\}}$, and so on.\par
  We also define face-identifying maps $g_i$ by $g_i:= f_S$, where $S=\{1,\ldots, \widehat{i+1},\ldots, n\}$. In particular, we can define $G_i$ with $\Pi^{n-2}$ through the map
    \[
      G_i \xrightarrow{\substack{g_i\\ \sim}} \Pi^{n-2}\times \Pi^0\cong \Pi^{n-2}.
    \]
\end{definition}

\begin{remark}\label{omnitruncation}
  On any abstract polytope $P$ of dimension $d$, there exists a polytope $t_I P$, $I = \{0,1,\ldots,d-1\}$ called the \textit{omnitruncation} of $P$. See (\cite[Chapter 8]{MR370327}, or \cite[Chapter 3, Section 6]{MR3358167} for the defintion of omnitruncation). The omnitruncation $t_I P$ has facets corresponding to the faces of $P$, and a vertex for each flag $P$. We outline one way to view $t_IP$: Chop off each vertex of $P$ with a hyperplane normal to the vector pointing from the centroid of $P$ to the vertex. These hyperplanes should just barely intersect $P$ so that they do not intersect with one another \textit{inside of} $P$. We then chop off the original edges of $P$ with similar hyperplanes that are normal to the vectors from the centroid to the midpoint. If $P$ is higher than $3$-diensional, we continue this process of truncating higher and higher dimensional facets. An equivalent definition is that $t_IP$ to be the dual of the barycentric subdivision $\mathcal{B}$ of $P$ (note that the vertices of $B$ correspond to the faces of $P$ and the facets of $B$ corresponds to the flags of $P$).\par
  Observe that the $\Pi^{n-1}$ is the omnitruncation of the standard topological $(n-1)$-simplex $\Delta^{n-1}$ with vertices $0<1<\ldots<n-1$. In this regard, the facets $G_i$ correspond to the vertices $0,1,\ldots, n-1\in \Delta^{n-1}$, the facets $G_{\{i,j\}}$ correspond to the edges of $\Delta^{n-1}$ and the faces $G_{\{i,j,k\}}$ correspond to $2$-faces. The $2$-skeleton of $\Delta^{m-1}$ is simply connected for $n\geq 3$, so any loop $\mathcal{K}$ in $\Pi^{n-1}$ passing through just the facets of type $G_i$, $G_{\{i,j\}}$ should be nullhomotopic. We will use this idea in Section \ref{simplifying}, where we manipulate circular tubes lying inside polytopes.
  \begin{figure}
    \begin{tabular}{ m{5.5cm} m{1.5cm} m{6cm} }
      \def\svgscale{1.2}
\begingroup%
  \makeatletter%
  \providecommand\color[2][]{%
    \errmessage{(Inkscape) Color is used for the text in Inkscape, but the package 'color.sty' is not loaded}%
    \renewcommand\color[2][]{}%
  }%
  \providecommand\transparent[1]{%
    \errmessage{(Inkscape) Transparency is used (non-zero) for the text in Inkscape, but the package 'transparent.sty' is not loaded}%
    \renewcommand\transparent[1]{}%
  }%
  \providecommand\rotatebox[2]{#2}%
  \newcommand*\fsize{\dimexpr\f@size pt\relax}%
  \newcommand*\lineheight[1]{\fontsize{\fsize}{#1\fsize}\selectfont}%
  \ifx\svgwidth\undefined%
    \setlength{\unitlength}{143.93691296bp}%
    \ifx\svgscale\undefined%
      \relax%
    \else%
      \setlength{\unitlength}{\unitlength * \real{\svgscale}}%
    \fi%
  \else%
    \setlength{\unitlength}{\svgwidth}%
  \fi%
  \global\let\svgwidth\undefined%
  \global\let\svgscale\undefined%
  \makeatother%
  \begin{picture}(1,0.90997223)%
    \lineheight{1}%
    \setlength\tabcolsep{0pt}%
    \put(0,0){\includegraphics[width=\unitlength,page=1]{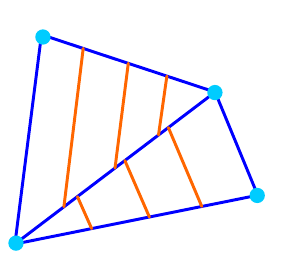}}%
    \put(0.08674917,0.86251417){\color[rgb]{0,0,0}\makebox(0,0)[lt]{\lineheight{1.25}\smash{\begin{tabular}[t]{l}$3$\end{tabular}}}}%
    \put(0.72052549,0.65568679){\color[rgb]{0,0,0}\makebox(0,0)[lt]{\lineheight{1.25}\smash{\begin{tabular}[t]{l}$0$\end{tabular}}}}%
    \put(0.89620439,0.21577451){\color[rgb]{0,0,0}\makebox(0,0)[lt]{\lineheight{1.25}\smash{\begin{tabular}[t]{l}$1$\end{tabular}}}}%
    \put(-0.00387668,0.00362675){\color[rgb]{0,0,0}\makebox(0,0)[lt]{\lineheight{1.25}\smash{\begin{tabular}[t]{l}$2$\end{tabular}}}}%
  \end{picture}%
\endgroup%
 & \begin{tikzpicture}
      [decoration=snake,
   line around/.style={decoration={pre length=#1,post length=#1}}]
      \draw[->, decorate, line around = 3pt] (0,0)--(1.2,0);
    \end{tikzpicture} & \def\svgscale{0.75}
\begingroup%
  \makeatletter%
  \providecommand\color[2][]{%
    \errmessage{(Inkscape) Color is used for the text in Inkscape, but the package 'color.sty' is not loaded}%
    \renewcommand\color[2][]{}%
  }%
  \providecommand\transparent[1]{%
    \errmessage{(Inkscape) Transparency is used (non-zero) for the text in Inkscape, but the package 'transparent.sty' is not loaded}%
    \renewcommand\transparent[1]{}%
  }%
  \providecommand\rotatebox[2]{#2}%
  \newcommand*\fsize{\dimexpr\f@size pt\relax}%
  \newcommand*\lineheight[1]{\fontsize{\fsize}{#1\fsize}\selectfont}%
  \ifx\svgwidth\undefined%
    \setlength{\unitlength}{301.1362271bp}%
    \ifx\svgscale\undefined%
      \relax%
    \else%
      \setlength{\unitlength}{\unitlength * \real{\svgscale}}%
    \fi%
  \else%
    \setlength{\unitlength}{\svgwidth}%
  \fi%
  \global\let\svgwidth\undefined%
  \global\let\svgscale\undefined%
  \makeatother%
  \begin{picture}(1,0.85949531)%
    \lineheight{1}%
    \setlength\tabcolsep{0pt}%
    \put(0,0){\includegraphics[width=\unitlength,page=1]{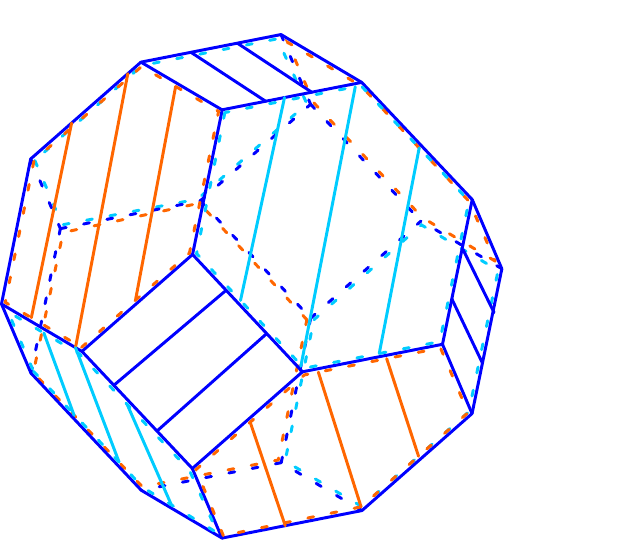}}%
    \put(0.5363881,0.49875379){\color[rgb]{0,0,0}\makebox(0,0)[lt]{\lineheight{1.25}\smash{\begin{tabular}[t]{l}$G_0$\end{tabular}}}}%
    \put(0.50781594,0.16649568){\color[rgb]{0,0,0}\makebox(0,0)[lt]{\lineheight{1.25}\smash{\begin{tabular}[t]{l}$G_{\{0,1,2\}}$\end{tabular}}}}%
    \put(0.2416672,0.27369111){\color[rgb]{0,0,0}\makebox(0,0)[lt]{\lineheight{1.25}\smash{\begin{tabular}[t]{l}$G_{\{0,2\}}$\end{tabular}}}}%
    \put(0.1281924,0.5768699){\color[rgb]{0,0,0}\makebox(0,0)[lt]{\lineheight{1.25}\smash{\begin{tabular}[t]{l}$G_{\{0,2,3\}}$\end{tabular}}}}%
    \put(0.25457038,0.83681126){\color[rgb]{0,0,0}\makebox(0,0)[lt]{\lineheight{1.25}\smash{\begin{tabular}[t]{l}$G_{\{0,3\}}$\end{tabular}}}}%
    \put(0.82259199,0.32689613){\color[rgb]{0,0,0}\makebox(0,0)[lt]{\lineheight{1.25}\smash{\begin{tabular}[t]{l}$G_{\{0,1\}}$\end{tabular}}}}%
    \put(0.11785305,0.19780678){\color[rgb]{0,0,0}\makebox(0,0)[lt]{\lineheight{1.25}\smash{\begin{tabular}[t]{l}$G_2$\end{tabular}}}}%
    \put(0,0){\includegraphics[width=\unitlength,page=2]{truncated_pyramid.pdf}}%
  \end{picture}%
\endgroup%
\\
    \end{tabular}
    \caption{Left, the solid tetrahedron $\Delta^{n-1}$ $(n=4)$. Right: The omnitruncation of $\Delta^{n-1}$, which is the $(n-1)$-dimensional permutohedron $\Pi^{n-1}$.}
    \label{permutohedron to simplicial complex}
  \end{figure}
\end{remark}
\subsection{Twists}
\begin{notation}
  Define $J:=[-1,1]\subset \mathbb{R}$. Let $e_J$ be the unit vector in this $\mathbb{R}$ direction.
\end{notation}
\begin{definition}
  Consider a connected topological group $G$, where we imagine $G$ to be a group of rotations. We define a \textit{twist} in $G$ to be a continuous path $\gamma:[0,1]\to G$ such that $\gamma(0)= \text{\text{Id}}$ and $\gamma(t)$ is constant for $t$ in a neighborhood of $\{0,1\}$. Composition of twists is performed through the diamond operation $\diamond$. If $\phi,\psi: [0,1]\to G$ are two twists, we define $\phi\diamond\psi:[0,1]\to G$ to be the twist that is $\psi(2t)$ on $[0,1/2]$ and $\phi(2t-1)\circ\psi(1)$ on $[1/2,1]$.\par
  We denote the nontwist $c_{\Id}$ by the constant map $[0,1]\to G$, $t\mapsto \Id$. $c_{\Id}$ acts as the identity (up to homotopy) under the diamond composition.
\end{definition}
Composition of twists is not associative, but it is associative up to homotopy. In fact, we observe the following property:
\begin{lemma}\label{twist homotopy}
  If $g$ and $f$ are twists, then the twist $t\mapsto g(t)f(t)$ is homotopic to $g \diamond f$ relative the endpoints $0,1\in [0,1]$.
\end{lemma}
\begin{example}
  Consider the product $J\times \Pi^k$ and the center point $P=(1/2,p)$, where $p$ is the center of $\Pi^k$. We can imagine the group of rotations of $J\times \Pi^k$ to be given by $SO(T_P(J\times \Pi^k))$. Therefore, a twist of $J\times \Pi^k$ is given by a path $[0,1]\to SO(T_P(J\times \Pi^k))$ starting at $\text{Id}$.
\end{example}
\begin{definition}
  Let $*$ denote concatenation of paths. That is, if $f,g:[0,1]\to G$ are paths with $f(1)=g(0)$, define $f*g:[0,1]\to G$ to be the path that is $f(2t)$ on the subinterval $[0,1/2]$ and $g(2t-1)$ on the subinterval $[1/2,1]$.
\end{definition}

\begin{proof}
  At time $s$, define the homotopy $h_s:=g(st)f(t)*f(s+(1-s)t)f(1)$. $h_0=g\diamond f$ and $h_1=c_{g(1)f(1)}*g(t)f(t)$, which is homotopic to $g(t)f(t)$. See Figure \ref{homotopy drawings}.
\end{proof}
\begin{lemma}\label{conjugation}
  Given twists $f: [0,1]\to G$ and $g: [0,1]\to G$, $g(0)=f(0)=\text{Id}$,
  \begin{equation*}
    g(1)f(t)g(1)^{-1}\cong g(t)f(t)g(t)^{-1}\cong  g(t)\diamond f(t)\diamond g(t)^{-1},
  \end{equation*}
  where the homotopies are relative the endpoints $\{0,1\}\subset [0,1]$.
\end{lemma}
\begin{proof}
  The first homotopy is $(\left.g(\cdot)\right|_{[s,1]})f(\cdot)(\left.g(\cdot)^{-1}\right|_{[s,1]})$ (see Figure \ref{homotopy drawings} for a picture). To understand the second homotopy, use Lemma \ref{twist homotopy} to calculate
  \begin{equation*}
    f(t) \diamond g(t) \diamond f(t)^{-1}\cong f(t)\diamond (g(t)f(t)^{-1})\cong f(t)g(t)f(t)^{-1}.\qedhere
  \end{equation*}
\end{proof}
\begin{figure}
  $\footnotesize{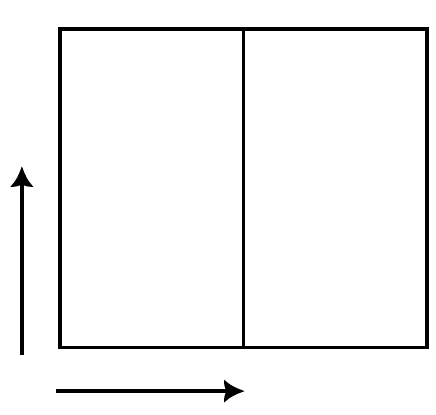}$\hspace{0.7cm}
  $\footnotesize{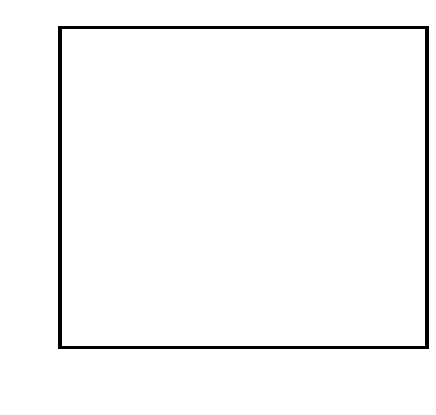}$
  \caption{Left: The homotopy in Lemma \ref{twist homotopy}. Right: The homotopy in Lemma \ref{conjugation}}
  \label{homotopy drawings}
\end{figure}
\begin{definition}\label{subtracting twists}
  Let $g$ and $f$ be twists $[0,1]\to G$ with the same endpoint $f(1)=g(1)\in G$. The concatenation $g*\overline{f} = \overline{f}\diamond g$ is a loop in $G$ starting and ending at $\text{Id}$. $g*\overline{f}$ represents an element in $\pi_1(G)$, which we call $g-f$.
\end{definition}

\subsection{Twisting manifolds with corners}\label{twists}
The twists we focus on in this paper are twists of the form $\varphi: [0,1]\to SO(T_PW)$, where $W$ is a manifold with corners equipped with a Riemannian metric, and $P\in \Int W$ (we think of $P$ as a point of rotation for $W$).
\begin{notation}
  Let $P\in W$, and $c_{\Id}: [0,1]\to SO(T_PW)$ be the nontwist. We often abuse notation and write $c_{\Id} = W$.
\end{notation}

\begin{definition}
  Let $W$, $Z$ be manifolds with corners, with $P\in W$, $P'\in Z$. If $f: [0,1]\to SO(T_PW)$, $g: [0,1]\to SO(T_{P'}Z)$ are twists, we identify $SO(T_{(P,P')}(W\times Z))\cong T_{P}W\times T_{P'}Z$ and define the twist $f\times g: [0,1]\to SO(T_{(P',P)}(W\times Z))$ to be $f$ in the $T_P W$ component and $g$ in the $T_{P'}Z$ component. So in particular, $f\times Z$ is $f$ in the $T_P W$ component and the nontwist in the $T_{P'}Z$ component.\par
  We call an element $g\in SO(T_PW)$ a \textit{symmetry} if $L$ is induced by an isometry $W\to W$ (which is therefore a diffeomorphism).\par
  We call $f$ \textit{gluable} if $f(1): T_P W\to T_P W$ is a symmetry, and $f$ is constant near the endpoints $0,1$.
\end{definition}
Note that $f\times g$ is gluable if $f$ and $g$ are gluable.
See Figure \ref{star figure} for an illustration of a gluable and non-gluable twist.
  \begin{figure}
    \centering
    \def\svgscale{0.9}
\begingroup%
  \makeatletter%
  \providecommand\color[2][]{%
    \errmessage{(Inkscape) Color is used for the text in Inkscape, but the package 'color.sty' is not loaded}%
    \renewcommand\color[2][]{}%
  }%
  \providecommand\transparent[1]{%
    \errmessage{(Inkscape) Transparency is used (non-zero) for the text in Inkscape, but the package 'transparent.sty' is not loaded}%
    \renewcommand\transparent[1]{}%
  }%
  \providecommand\rotatebox[2]{#2}%
  \newcommand*\fsize{\dimexpr\f@size pt\relax}%
  \newcommand*\lineheight[1]{\fontsize{\fsize}{#1\fsize}\selectfont}%
  \ifx\svgwidth\undefined%
    \setlength{\unitlength}{103.5395032bp}%
    \ifx\svgscale\undefined%
      \relax%
    \else%
      \setlength{\unitlength}{\unitlength * \real{\svgscale}}%
    \fi%
  \else%
    \setlength{\unitlength}{\svgwidth}%
  \fi%
  \global\let\svgwidth\undefined%
  \global\let\svgscale\undefined%
  \makeatother%
  \begin{picture}(1,0.95137839)%
    \lineheight{1}%
    \setlength\tabcolsep{0pt}%
    \put(0,0){\includegraphics[width=\unitlength,page=1]{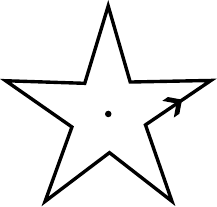}}%
    \put(0.51455268,0.45136181){\color[rgb]{0,0,0}\makebox(0,0)[lt]{\lineheight{1.25}\smash{\begin{tabular}[t]{l}$P$\end{tabular}}}}%
    \put(0,0){\includegraphics[width=\unitlength,page=2]{star_1.pdf}}%
  \end{picture}%
\endgroup%

    \hspace{0.1cm}\raisebox{1cm}{$\to$}\hspace{0.1cm}
    \def\svgscale{0.9}
\begingroup%
  \makeatletter%
  \providecommand\color[2][]{%
    \errmessage{(Inkscape) Color is used for the text in Inkscape, but the package 'color.sty' is not loaded}%
    \renewcommand\color[2][]{}%
  }%
  \providecommand\transparent[1]{%
    \errmessage{(Inkscape) Transparency is used (non-zero) for the text in Inkscape, but the package 'transparent.sty' is not loaded}%
    \renewcommand\transparent[1]{}%
  }%
  \providecommand\rotatebox[2]{#2}%
  \newcommand*\fsize{\dimexpr\f@size pt\relax}%
  \newcommand*\lineheight[1]{\fontsize{\fsize}{#1\fsize}\selectfont}%
  \ifx\svgwidth\undefined%
    \setlength{\unitlength}{103.68843499bp}%
    \ifx\svgscale\undefined%
      \relax%
    \else%
      \setlength{\unitlength}{\unitlength * \real{\svgscale}}%
    \fi%
  \else%
    \setlength{\unitlength}{\svgwidth}%
  \fi%
  \global\let\svgwidth\undefined%
  \global\let\svgscale\undefined%
  \makeatother%
  \begin{picture}(1,0.95286318)%
    \lineheight{1}%
    \setlength\tabcolsep{0pt}%
    \put(0,0){\includegraphics[width=\unitlength,page=1]{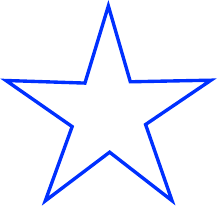}}%
    \put(0.51729668,0.41758298){\color[rgb]{0,0,0}\makebox(0,0)[lt]{\lineheight{1.25}\smash{\begin{tabular}[t]{l}$P$\end{tabular}}}}%
    \put(0,0){\includegraphics[width=\unitlength,page=2]{star_2.pdf}}%
  \end{picture}%
\endgroup%

    , 
    \def\svgscale{0.9}
\begingroup%
  \makeatletter%
  \providecommand\color[2][]{%
    \errmessage{(Inkscape) Color is used for the text in Inkscape, but the package 'color.sty' is not loaded}%
    \renewcommand\color[2][]{}%
  }%
  \providecommand\transparent[1]{%
    \errmessage{(Inkscape) Transparency is used (non-zero) for the text in Inkscape, but the package 'transparent.sty' is not loaded}%
    \renewcommand\transparent[1]{}%
  }%
  \providecommand\rotatebox[2]{#2}%
  \newcommand*\fsize{\dimexpr\f@size pt\relax}%
  \newcommand*\lineheight[1]{\fontsize{\fsize}{#1\fsize}\selectfont}%
  \ifx\svgwidth\undefined%
    \setlength{\unitlength}{103.5395032bp}%
    \ifx\svgscale\undefined%
      \relax%
    \else%
      \setlength{\unitlength}{\unitlength * \real{\svgscale}}%
    \fi%
  \else%
    \setlength{\unitlength}{\svgwidth}%
  \fi%
  \global\let\svgwidth\undefined%
  \global\let\svgscale\undefined%
  \makeatother%
  \begin{picture}(1,0.95137828)%
    \lineheight{1}%
    \setlength\tabcolsep{0pt}%
    \put(0,0){\includegraphics[width=\unitlength,page=1]{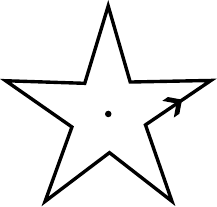}}%
    \put(0.51455284,0.45136168){\color[rgb]{0,0,0}\makebox(0,0)[lt]{\lineheight{1.25}\smash{\begin{tabular}[t]{l}$P$\end{tabular}}}}%
    \put(0,0){\includegraphics[width=\unitlength,page=2]{star_3.pdf}}%
  \end{picture}%
\endgroup%

    \hspace{0.1cm}\raisebox{1.0cm}{$\to$}\hspace{0.1cm}
    \raisebox{-0.30cm}{\def\svgscale{0.9}
\begingroup%
  \makeatletter%
  \providecommand\color[2][]{%
    \errmessage{(Inkscape) Color is used for the text in Inkscape, but the package 'color.sty' is not loaded}%
    \renewcommand\color[2][]{}%
  }%
  \providecommand\transparent[1]{%
    \errmessage{(Inkscape) Transparency is used (non-zero) for the text in Inkscape, but the package 'transparent.sty' is not loaded}%
    \renewcommand\transparent[1]{}%
  }%
  \providecommand\rotatebox[2]{#2}%
  \newcommand*\fsize{\dimexpr\f@size pt\relax}%
  \newcommand*\lineheight[1]{\fontsize{\fsize}{#1\fsize}\selectfont}%
  \ifx\svgwidth\undefined%
    \setlength{\unitlength}{103.74438308bp}%
    \ifx\svgscale\undefined%
      \relax%
    \else%
      \setlength{\unitlength}{\unitlength * \real{\svgscale}}%
    \fi%
  \else%
    \setlength{\unitlength}{\svgwidth}%
  \fi%
  \global\let\svgwidth\undefined%
  \global\let\svgscale\undefined%
  \makeatother%
  \begin{picture}(1,1.04957811)%
    \lineheight{1}%
    \setlength\tabcolsep{0pt}%
    \put(0,0){\includegraphics[width=\unitlength,page=1]{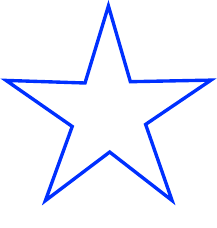}}%
    \put(0.50647722,0.54654397){\color[rgb]{0,0,0}\makebox(0,0)[lt]{\lineheight{1.25}\smash{\begin{tabular}[t]{l}$P$\end{tabular}}}}%
    \put(0,0){\includegraphics[width=\unitlength,page=2]{star_4.pdf}}%
  \end{picture}%
\endgroup%
}
    \caption{Let $W$ denote the $5$-pointed star. Left: An example of a symmetry of $T_PW$. We can view the transformation as a $72^\circ$ counterclockwise rotation of $T_PW$. Right: An example of a non-symmetry of $T_PW$. We can view the transformation is a $108^{\circ}$ counterclockwise rotation of $T_PW$. The dotted overlay shows how the non-symmetry would look if extended to the entire $W$. The transformation on the left can be extended, however. Therefore, a $72^\circ$ counterclockwise twist is gluable, but not a $108^\circ$ counterclockwise twist.}
    \label{star figure}
  \end{figure}
\begin{remark}
  Observe that gluable twists are closed under the $\diamond$ composition.
\end{remark}
Now let $P = (0,P')$ be the midpoint of $X = J\times \Pi^{\kappa-1}$. Since there is a natural Riemannian metric on both $J$ and $\Pi^{\kappa-1}$ (see Remark \ref{metric discussion}), we put the product metric on $T_{P} X$. We define special types of twists $\varphi_{i, j}: [0,1]\to SO(T_P X)$ which we will use frequently in this paper.
\begin{definition}
  Let $\omega\in S_\kappa$ be a permutation in the symmetric group. We define $P_\omega: \mathbb{A}^{\kappa-1}\to \mathbb{A}^{\kappa-1}$ as the orthogonal transformation that applies the $\omega$-permutation to the coordinates of a point $p$. Namely, $P_\omega(p_1,\ldots,p_{\kappa}) = (p_{\omega^{-1}(1)},\ldots,p_{\omega^{-1}(\kappa)})$. We also let $P_\omega$ denote the orthogonal transformation $T_p \mathbb{A}^{\kappa-1}\to T_p\mathbb{A}^{\kappa-1}$, similarly defined by the $\omega$-permutation to the coordinates of a vector $\mathbf{v}$.
\end{definition}

\begin{definition}
  For $1\leq i,j\leq \kappa-2$, $i\neq j$, we define the twist $\varphi_{i,j}$ as the following: take the vector $\mathbf{a}= (0,e_i-e_j), \partial_J = (e_J,0)\in T_0 J\times T_{P'} \Pi^{\kappa-1} \cong T_P X$.
  Now identify $T_{P}X\cong \langle \mathbf{a},\partial_J\rangle^\perp\times\langle \mathbf{a},\partial_J\rangle$ and define $\varphi_{i, j}:[0,1]\to SO(T_PX)$ to be the twist $\langle \mathbf{a},\partial_J\rangle^\perp\times \varphi$, where $\varphi$ turns the oriented basis $(\mathbf{a},\partial_J)$ $180^\circ$ clockwise (see Figure \ref{twist in plane picture}).
  \begin{figure}
    \begin{tabular}{m{4.5cm} m{3cm} m{5cm}}
\begingroup%
  \makeatletter%
  \providecommand\color[2][]{%
    \errmessage{(Inkscape) Color is used for the text in Inkscape, but the package 'color.sty' is not loaded}%
    \renewcommand\color[2][]{}%
  }%
  \providecommand\transparent[1]{%
    \errmessage{(Inkscape) Transparency is used (non-zero) for the text in Inkscape, but the package 'transparent.sty' is not loaded}%
    \renewcommand\transparent[1]{}%
  }%
  \providecommand\rotatebox[2]{#2}%
  \newcommand*\fsize{\dimexpr\f@size pt\relax}%
  \newcommand*\lineheight[1]{\fontsize{\fsize}{#1\fsize}\selectfont}%
  \ifx\svgwidth\undefined%
    \setlength{\unitlength}{155.71922002bp}%
    \ifx\svgscale\undefined%
      \relax%
    \else%
      \setlength{\unitlength}{\unitlength * \real{\svgscale}}%
    \fi%
  \else%
    \setlength{\unitlength}{\svgwidth}%
  \fi%
  \global\let\svgwidth\undefined%
  \global\let\svgscale\undefined%
  \makeatother%
  \begin{picture}(1,0.67334903)%
    \lineheight{1}%
    \setlength\tabcolsep{0pt}%
    \put(0,0){\includegraphics[width=\unitlength,page=1]{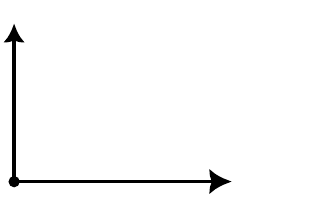}}%
    \put(0.0123519,0.02256319){\color[rgb]{0,0,0}\makebox(0,0)[lt]{\lineheight{1.25}\smash{\begin{tabular}[t]{l}$P$\end{tabular}}}}%
    \put(0,0){\includegraphics[width=\unitlength,page=2]{180_twist_part_1.pdf}}%
    \put(0.28029785,0.39067578){\color[rgb]{0,0,0}\makebox(0,0)[lt]{\lineheight{1.25}\smash{\begin{tabular}[t]{l}$180^{\circ}$\end{tabular}}}}%
    \put(0.6398857,0.01004213){\color[rgb]{0,0,0}\makebox(0,0)[lt]{\lineheight{1.25}\smash{\begin{tabular}[t]{l}$\mathbf{a}$\end{tabular}}}}%
    \put(-0.00267307,0.62857157){\color[rgb]{0,0,0}\makebox(0,0)[lt]{\lineheight{1.25}\smash{\begin{tabular}[t]{l}$e_J$\end{tabular}}}}%
  \end{picture}%
\endgroup%
 & \makecell{$\varphi_{i, j}$\\
      \begin{tikzpicture}
      [decoration=snake,
   line around/.style={decoration={pre length=#1,post length=#1}}]
      \draw[->, decorate, line around = 3pt] (0,0)--(1.2,0);
    \end{tikzpicture}} & 
\begingroup%
  \makeatletter%
  \providecommand\color[2][]{%
    \errmessage{(Inkscape) Color is used for the text in Inkscape, but the package 'color.sty' is not loaded}%
    \renewcommand\color[2][]{}%
  }%
  \providecommand\transparent[1]{%
    \errmessage{(Inkscape) Transparency is used (non-zero) for the text in Inkscape, but the package 'transparent.sty' is not loaded}%
    \renewcommand\transparent[1]{}%
  }%
  \providecommand\rotatebox[2]{#2}%
  \newcommand*\fsize{\dimexpr\f@size pt\relax}%
  \newcommand*\lineheight[1]{\fontsize{\fsize}{#1\fsize}\selectfont}%
  \ifx\svgwidth\undefined%
    \setlength{\unitlength}{149.40549835bp}%
    \ifx\svgscale\undefined%
      \relax%
    \else%
      \setlength{\unitlength}{\unitlength * \real{\svgscale}}%
    \fi%
  \else%
    \setlength{\unitlength}{\svgwidth}%
  \fi%
  \global\let\svgwidth\undefined%
  \global\let\svgscale\undefined%
  \makeatother%
  \begin{picture}(1,0.66472915)%
    \lineheight{1}%
    \setlength\tabcolsep{0pt}%
    \put(0,0){\includegraphics[width=\unitlength,page=1]{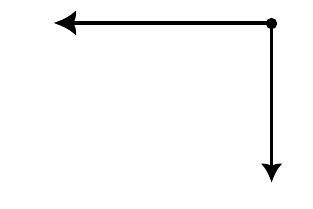}}%
    \put(0.89706697,0.61805928){\color[rgb]{0,0,0}\makebox(0,0)[lt]{\lineheight{1.25}\smash{\begin{tabular}[t]{l}$P$\end{tabular}}}}%
    \put(-0.00172493,0.4821868){\color[rgb]{0,0,0}\makebox(0,0)[lt]{\lineheight{1.25}\smash{\begin{tabular}[t]{l}$\varphi_{i, j}(1)(\mathbf{a})$\end{tabular}}}}%
    \put(0.7664739,0.01946873){\color[rgb]{0,0,0}\makebox(0,0)[lt]{\lineheight{1.25}\smash{\begin{tabular}[t]{l}$\varphi_{i, j}(1)(e_J)$\end{tabular}}}}%
  \end{picture}%
\endgroup%

    \end{tabular}
      \caption {An illustration of how the twist looks in the plane $\langle \mathbf{a}, \mathbf{\partial_J}\rangle$}
      \label{twist in plane picture}
    \end{figure}
  \end{definition}
See Figure \ref{hex twist} for an illustration of twist examples $\varphi_i$. Note the twist $\varphi_{i, j}$ is gluable. Indeed, $\varphi_{i, j}(1)$ is the isometry $X\to X$ that maps $(\mathbf{v},t)$ to $(P_{(ij)}(\mathbf{v}),-t)$. Observe that $\varphi_{i, j}(1)$ swaps $J\times G_{\{i\}}$ with $J\times G_{\{j\}}$, $J\times G_{\{i,k\}}$ with $J\times G_{\{j,k\}}$, and so on.
\begin{notation}
  We write $\varphi_i:= \varphi_{i, i+1}$, which lightens up the notation in our paper.
\end{notation}
\begin{figure}
  \centering
  \def\svgscale{0.8}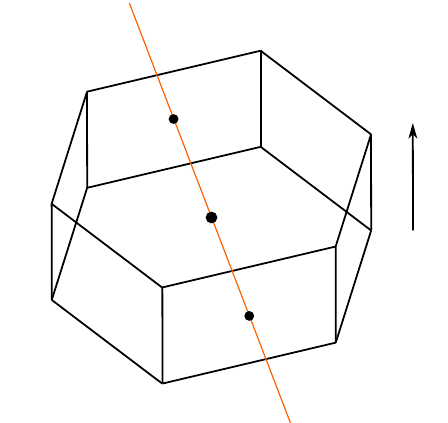\hspace{1cm}
  \def\svgscale{0.8}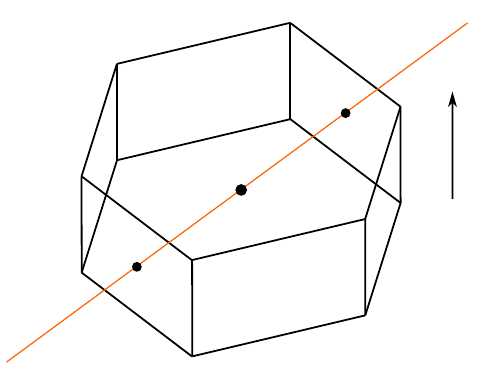
  \caption{Let $X = J\times \Pi^2$. Left: The gluable twist $\varphi_1$ in $T_PX$. Right: The gluable twist $\varphi_2$ in $T_PX$.}
  \label{hex twist}
\end{figure}

\begin{lemma}\label{commutators}
  Let $X = J\times\Pi^{\kappa-1}$, and let $1\leq i,j,k\leq \kappa-1$. We have the identities
  \begin{align*}
    \varphi_{j, k}\diamond \varphi_{i, j} &= \varphi_{i, k}^{-1}\diamond \varphi_{j, k} =  \varphi_{i, j} \diamond \varphi_{i, k}^{-1}\\
    &=  \varphi_{i, k}\diamond \varphi_{j, k}^{-1} =  \varphi_{i, j}^{-1} \diamond \varphi_{i, k}.
  \end{align*}
  In particular, we observe $\varphi_{i+1}\diamond \varphi_i = \varphi_{i, i+2}^{-1}\diamond \varphi_{i+1} = \varphi_i\diamond \varphi_{i, i+2}^{-1}$.
\end{lemma}
\begin{proof}
  We only prove the first equality; the other equality follows similarly. The first equality follows from the identity
  \begin{align*}
    \varphi_{j, k}\diamond \varphi_{i, j}\diamond \varphi_{j, k}^{-1} &\cong \varphi_{j, k}(1) \varphi_{i, j}(t) \varphi_{j, k}^{-1}(1) \\
                                                                                                                      &= (P_{(jk)}\tau_J)\varphi_{i, j}(t)(P_{(jk)}\tau_J)^{-1} \\
                                                                                                                      &= P_{(jk)}\left(\tau_J \varphi_{i, j}(t)\tau_J^{-1}\right)P_{(jk)}^{-1} \\
                                                                                                                      &= P_{(jk)}\varphi_{i, j}(t)^{-1}P_{(jk)}^{-1}\\
    &= \varphi_{i, k}(t)^{-1}.\qedhere
    \end{align*}
\end{proof}
\begin{definition}\label{full twist}
  Suppose $W$ is a manifold with corners equipped with a Riemannian metric, and that $\dim W = m \geq 4$. If $P\in \Int W$, then $\pi_1(SO(T_PW))\cong \pi_1(SO(m))\cong \mathbb{Z}/2$. We call a ``full twist'' to be any twist $g:[0,1]\to  SO(T_PW)$ such that $g(1) = \Id$ and $g$ defines the generator $1\in \pi_1(SO(T_PW))$.
\end{definition}
\begin{remark}\label{full twist fact}
  Let $W,P$ be as Definition \ref{full twist}. Note that all full twists are homotopic relative the endpoints $0,1$, and that a $\diamond$ composition of an even number of full twists is nullhomotopic. Note that by Lemma \ref{conjugation}, full twists in $SO(T_PW)$ must commute (under the $\diamond$ operation) with any twist $f$ up to homotopy. Indeed, if $g$ is such a full twist, then $g\diamond f\diamond g^{-1} \cong g(1)fg(1)^{-1} = f$.
\end{remark}

\begin{definition}\label{full twist notation}
  Let $W,P$ be as Definition \ref{full twist}, and suppose $g$ is a full twist. We denote the homotopy class of $g$ as $1$, and furthermore a composition $g_1\diamond\ldots g_l$ of full twists as $l$. We further write $f\diamond g_1\diamond\ldots g_l \cong f + l$. The ``$+$'' symbol is natural for us, since by Remark \ref{full twist fact}, full twists commute under $\diamond$ composition. \par
  Now let $\varphi$, $\phi$ be twists $[0,1]\to SO(T_P W)$ with the same endpoint $\varphi(1)=\phi(1)$. $\varphi-\phi$ defines an element $\omega \in \pi_1 (SO(T_PW))\cong \mathbb{Z}/2$ (see Definition \ref{subtracting twists}). We write $\varphi-\phi=(\omega)$. Observe the identities $\varphi-\phi=\phi-\varphi$ and $(\varphi-\phi) + (\phi-\psi) = \varphi-\psi$ if $\varphi(1)=\phi(1)=\psi(1)$. In view of Remark \ref{full twist fact}, we conclude that $\varphi = \phi + (\omega)$.
\end{definition}

\subsection{Adding twists to tubes}\label{twists to tubes}
Let us consider the case where we have an embedding $\Theta:X \times [0,1]\to Y$, where $X = J\times \Pi^{n-1}$ and $Y$ is a manifold with corners. We call $\Theta$, a \textit{tube}, with ends $X\times \{0\}$, $X\times \{1\}$. We define what it means to ``add a twist'' to $\Theta$.
\begin{definition}\label{twist definition}
  Suppose we are given a gluable twist $\varphi: [0,1] \to SO(T_PW)$. If $W$ is convex, each point $P'\in W$ is written as $P+v$ for a unique $v\in T_PW$. Given a twist $\varphi: [0,1]\to SO(T_PW)$, we can construct an embedding $\Omega:[0,1]\times P\to [0,1]\times P$ as the following. $\Omega(P+v,t) = (P + \lambda(t)\varphi_t(v),t)$, where $\lambda:[0,1]\to (0,1]$ is a continuous function satisfying:
  \begin{itemize}
  \item $\lambda(t)$ is the constant $1$ for $t$ near the endpoints $0,1$.
  \item $\lambda$ decays in the middle quickly enough so that $P + \lambda(t)\varphi(W)$ stays inside $W$.
  \end{itemize}
  To ``add'' the twist $\varphi$ to a tube $\Theta: W\times [0,1]\to Y$, simply precompose $\Theta$ with $\Omega$ (see Figure \ref{twist addition}). We call the resulting tube $\Theta \diamond \varphi$.
\end{definition}
The twist $\Theta \diamond \varphi$ depends on our choices for $\lambda$, but it is straightforward to check that it is well-defined up to homotopy (relative the endpoints $0,1\in [0,1]$).
\begin{lemma}\label{associativity}
  Let $\Theta:X\times [0,1]\to Y$ be a tube, and let $\varphi,\varphi': [0,1]\to SO(T_PX)$ be twists. $(\Theta\diamond\varphi)\diamond\varphi'\cong \Theta \diamond \varphi\diamond\varphi'$. 
\end{lemma}
\begin{proof}
  From Definition \ref{twist definition}, the tube $(\Theta\diamond\varphi)\diamond\varphi'$ is $(\Theta\circ \Omega)\circ\Omega'$, where $\Omega(P+v,t) = \left(P + \lambda(t)\varphi_t\cdot v,t\right)$ and $\Omega'(P+v,t) = \left(P + \lambda'(t)\varphi'_t\cdot v,t\right)$.
  \begin{align*}
    \left((\Theta\circ \Omega)\circ\Omega'\right)(P+v,t) = \Omega\left(P + \lambda'(t)\varphi'(t)\cdot v,t\right) &= \left(P + \lambda(t)\varphi(t)\cdot (\lambda'(t)\varphi'(t)\cdot v),t\right)\\
                                                                                                                                   &= \left(P + \lambda(t)\lambda'(t)(\varphi(t)\varphi'(t))\cdot v,t\right).
  \end{align*}
  
  Therefore, $(\Theta\diamond \Omega)\diamond\Omega'\cong \Theta \diamond (t\mapsto \varphi(t)\varphi'(t))\cong \Theta\diamond\varphi\diamond\varphi'$, with the last equivalence being from Lemma \ref{twist homotopy}.
\end{proof}
\begin{figure}
  {\centering\def\svgscale{0.8}
\begingroup%
  \makeatletter%
  \providecommand\color[2][]{%
    \errmessage{(Inkscape) Color is used for the text in Inkscape, but the package 'color.sty' is not loaded}%
    \renewcommand\color[2][]{}%
  }%
  \providecommand\transparent[1]{%
    \errmessage{(Inkscape) Transparency is used (non-zero) for the text in Inkscape, but the package 'transparent.sty' is not loaded}%
    \renewcommand\transparent[1]{}%
  }%
  \providecommand\rotatebox[2]{#2}%
  \newcommand*\fsize{\dimexpr\f@size pt\relax}%
  \newcommand*\lineheight[1]{\fontsize{\fsize}{#1\fsize}\selectfont}%
  \ifx\svgwidth\undefined%
    \setlength{\unitlength}{70.10353269bp}%
    \ifx\svgscale\undefined%
      \relax%
    \else%
      \setlength{\unitlength}{\unitlength * \real{\svgscale}}%
    \fi%
  \else%
    \setlength{\unitlength}{\svgwidth}%
  \fi%
  \global\let\svgwidth\undefined%
  \global\let\svgscale\undefined%
  \makeatother%
  \begin{picture}(1,0.9064481)%
    \lineheight{1}%
    \setlength\tabcolsep{0pt}%
    \put(0,0){\includegraphics[width=\unitlength,page=1]{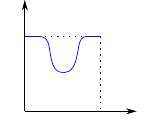}}%
    \put(-0.00795944,0.59481495){\color[rgb]{0,0,0}\makebox(0,0)[lt]{\lineheight{1.25}\smash{\begin{tabular}[t]{l}$1$\end{tabular}}}}%
    \put(0.64083381,0.00744596){\color[rgb]{0,0,0}\makebox(0,0)[lt]{\lineheight{1.25}\smash{\begin{tabular}[t]{l}$1$\end{tabular}}}}%
    \put(0.35745027,0.25021008){\color[rgb]{0.10196078,0,1}\makebox(0,0)[lt]{\lineheight{1.25}\smash{\begin{tabular}[t]{l}$\lambda$\end{tabular}}}}%
  \end{picture}%
\endgroup%
\par}
  \begin{tabular}{m{6.5cm} m{1.1cm} m{6cm}}
    \def\svgscale{0.8}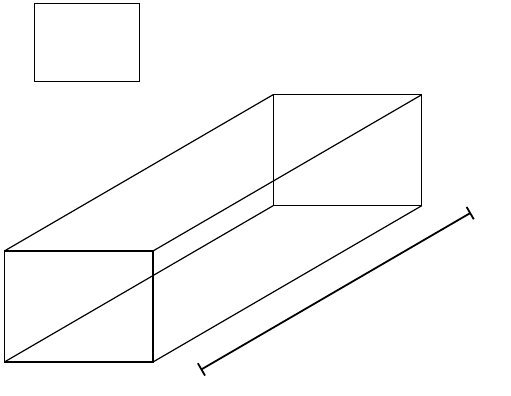 & \begin{tikzpicture}
      \draw[->] (0,0)--(1.2,0) node[midway,above] {$\Omega$};
    \end{tikzpicture} & \def\svgscale{0.8}
\begingroup%
  \makeatletter%
  \providecommand\color[2][]{%
    \errmessage{(Inkscape) Color is used for the text in Inkscape, but the package 'color.sty' is not loaded}%
    \renewcommand\color[2][]{}%
  }%
  \providecommand\transparent[1]{%
    \errmessage{(Inkscape) Transparency is used (non-zero) for the text in Inkscape, but the package 'transparent.sty' is not loaded}%
    \renewcommand\transparent[1]{}%
  }%
  \providecommand\rotatebox[2]{#2}%
  \newcommand*\fsize{\dimexpr\f@size pt\relax}%
  \newcommand*\lineheight[1]{\fontsize{\fsize}{#1\fsize}\selectfont}%
  \ifx\svgwidth\undefined%
    \setlength{\unitlength}{243.52982649bp}%
    \ifx\svgscale\undefined%
      \relax%
    \else%
      \setlength{\unitlength}{\unitlength * \real{\svgscale}}%
    \fi%
  \else%
    \setlength{\unitlength}{\svgwidth}%
  \fi%
  \global\let\svgwidth\undefined%
  \global\let\svgscale\undefined%
  \makeatother%
  \begin{picture}(1,0.6353901)%
    \lineheight{1}%
    \setlength\tabcolsep{0pt}%
    \put(0,0){\includegraphics[width=\unitlength,page=1]{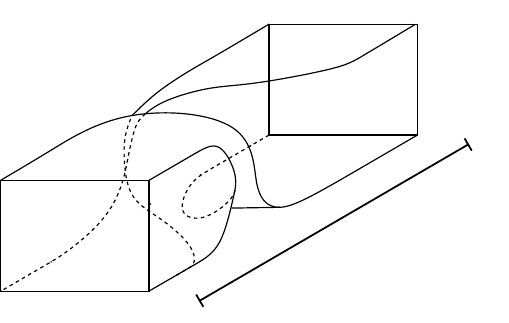}}%
    \put(0.41870809,0.00670763){\color[rgb]{0,0,0}\makebox(0,0)[lt]{\lineheight{1.25}\smash{\begin{tabular}[t]{l}$0$\end{tabular}}}}%
    \put(0.93807048,0.32993828){\color[rgb]{0,0,0}\makebox(0,0)[lt]{\lineheight{1.25}\smash{\begin{tabular}[t]{l}$1$\end{tabular}}}}%
    \put(0,0){\includegraphics[width=\unitlength,page=2]{twist_model_2.pdf}}%
    \put(0.01052636,0.00583912){\color[rgb]{0,0,0}\makebox(0,0)[lt]{\lineheight{1.25}\smash{\begin{tabular}[t]{l}$W\times \{0\}$\end{tabular}}}}%
    \put(0.54906456,0.60734022){\color[rgb]{0,0,0}\makebox(0,0)[lt]{\lineheight{1.25}\smash{\begin{tabular}[t]{l}$W\times \{1\}$\end{tabular}}}}%
  \end{picture}%
\endgroup%

  \end{tabular}
  \begin{tabular}{m{6.5cm} m{1cm} m{6cm}}
    \def\svgscale{0.8}
\begingroup%
  \makeatletter%
  \providecommand\color[2][]{%
    \errmessage{(Inkscape) Color is used for the text in Inkscape, but the package 'color.sty' is not loaded}%
    \renewcommand\color[2][]{}%
  }%
  \providecommand\transparent[1]{%
    \errmessage{(Inkscape) Transparency is used (non-zero) for the text in Inkscape, but the package 'transparent.sty' is not loaded}%
    \renewcommand\transparent[1]{}%
  }%
  \providecommand\rotatebox[2]{#2}%
  \newcommand*\fsize{\dimexpr\f@size pt\relax}%
  \newcommand*\lineheight[1]{\fontsize{\fsize}{#1\fsize}\selectfont}%
  \ifx\svgwidth\undefined%
    \setlength{\unitlength}{232.55820447bp}%
    \ifx\svgscale\undefined%
      \relax%
    \else%
      \setlength{\unitlength}{\unitlength * \real{\svgscale}}%
    \fi%
  \else%
    \setlength{\unitlength}{\svgwidth}%
  \fi%
  \global\let\svgwidth\undefined%
  \global\let\svgscale\undefined%
  \makeatother%
  \begin{picture}(1,0.67448453)%
    \lineheight{1}%
    \setlength\tabcolsep{0pt}%
    \put(0,0){\includegraphics[width=\unitlength,page=1]{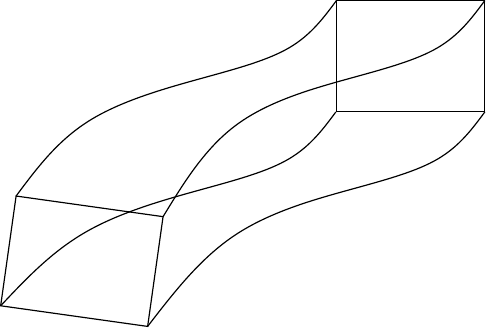}}%
    \put(0.27603846,0.51718687){\color[rgb]{0,0,0}\makebox(0,0)[lt]{\lineheight{1.25}\smash{\begin{tabular}[t]{l}$\Theta$\end{tabular}}}}%
    \put(0,0){\includegraphics[width=\unitlength,page=2]{twist_embedded_1.pdf}}%
  \end{picture}%
\endgroup%
 & \begin{tikzpicture}
      [decoration=snake,
   line around/.style={decoration={pre length=#1,post length=#1}}]
      \draw[->, decorate, line around = 3pt] (0,0)--(1.2,0);
    \end{tikzpicture} & \def\svgscale{0.8}
\begingroup%
  \makeatletter%
  \providecommand\color[2][]{%
    \errmessage{(Inkscape) Color is used for the text in Inkscape, but the package 'color.sty' is not loaded}%
    \renewcommand\color[2][]{}%
  }%
  \providecommand\transparent[1]{%
    \errmessage{(Inkscape) Transparency is used (non-zero) for the text in Inkscape, but the package 'transparent.sty' is not loaded}%
    \renewcommand\transparent[1]{}%
  }%
  \providecommand\rotatebox[2]{#2}%
  \newcommand*\fsize{\dimexpr\f@size pt\relax}%
  \newcommand*\lineheight[1]{\fontsize{\fsize}{#1\fsize}\selectfont}%
  \ifx\svgwidth\undefined%
    \setlength{\unitlength}{232.55824772bp}%
    \ifx\svgscale\undefined%
      \relax%
    \else%
      \setlength{\unitlength}{\unitlength * \real{\svgscale}}%
    \fi%
  \else%
    \setlength{\unitlength}{\svgwidth}%
  \fi%
  \global\let\svgwidth\undefined%
  \global\let\svgscale\undefined%
  \makeatother%
  \begin{picture}(1,0.67448473)%
    \lineheight{1}%
    \setlength\tabcolsep{0pt}%
    \put(0,0){\includegraphics[width=\unitlength,page=1]{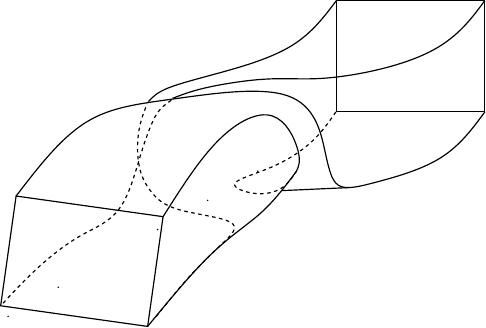}}%
    \put(0.20939063,0.52643815){\color[rgb]{0,0,0}\makebox(0,0)[lt]{\lineheight{1.25}\smash{\begin{tabular}[t]{l}$\Theta\circ\Omega$\end{tabular}}}}%
    \put(0,0){\includegraphics[width=\unitlength,page=2]{twist_embedded_2.pdf}}%
  \end{picture}%
\endgroup%
\\
  \end{tabular}
  \caption{Precomposing with the above map $\Omega$ defines the twisted tube $\Theta\circ \Omega$.}
  \label{twist addition}
\end{figure}
\subsection{Doubly specified tubes}
So far, we have focused on tubes $\Theta: X\times [0,1] \to Y$, where note that $\Theta$ has a canonical starting face $X\times \{0\}$, and ending face $X\times \{1\}$, and (gluable) twists are added twisting from the starting face in the direction of the ending face. We next define a type of tube $T$ in this paper that does not have canonical starting and ending face, where we can add twists that travel in both directions. We give a name to each side of $T$, and while there is no restriction to the names, we often make the arbitrary choice to name one side $\alpha$ and the other side $\beta$.
\begin{definition}\label{doubly param. tube}
  Let $T = \{(\Theta, \alpha),(\Lambda, \beta)\}$ be an unordered pair of tubes $W\times [0,1]\to Y$ such that $\Theta(x,t) = \Lambda(D(x),1-t)$ for some isometry $D: W\to W$ fixing $P\in W$. We call $T$, a \textit{doubly specified  tube} (abbreviated \textit{d.s.\ tube}). We denote the face $\Theta(X\times\{0\})$ as the \textit{$\alpha$-end}, and the face $\Lambda(X\times\{0\})$ as the \textit{$\beta$-end}.\par
  If $G=\Id$, then $\Lambda$ is just the reversed parametrization of $\Theta$. In this case, we say that $T$ is \textit{boundary-coherent}.\par
  A family $\{T_s\}_{0\leq s\leq 1}$, $T_s = \{(\Theta_s,\alpha_s),(\Lambda_s,\beta_s)\}$ of d.s.\ tubes is called a \textit{homotopy} if $\Theta_s$, $\Lambda_s$ themselves are homotopies relative $t=0,1$, and $\Theta_s(x,t) = \Lambda_s(D(x),1-t)$ for all $s$. (Note that $D$ is fixed.) The tubes $T_0$, $T_1$ are said to be \textit{homotopic}.
\end{definition}
\begin{example}\label{same face example}
  Consider the manifold with corners $Y:= J\times \Pi^{\kappa}$ living inside $J\times \mathbb{A}^{\kappa}$, and define $X:= J\times \Pi^{\kappa-1}$ We define a d.s.\ tube $T = \{(V_i,\alpha),(V'_i,\beta)\}$ with both ends on the boundary $\partial Y$ (see Figure \ref{Vi tube}). We first define $V_i: X \times [0,1]\to \mathbb{A}^{\kappa}$ as follows:
  \begin{figure}
    \centering
    \def\svgscale{0.6}
\begingroup%
  \makeatletter%
  \providecommand\color[2][]{%
    \errmessage{(Inkscape) Color is used for the text in Inkscape, but the package 'color.sty' is not loaded}%
    \renewcommand\color[2][]{}%
  }%
  \providecommand\transparent[1]{%
    \errmessage{(Inkscape) Transparency is used (non-zero) for the text in Inkscape, but the package 'transparent.sty' is not loaded}%
    \renewcommand\transparent[1]{}%
  }%
  \providecommand\rotatebox[2]{#2}%
  \newcommand*\fsize{\dimexpr\f@size pt\relax}%
  \newcommand*\lineheight[1]{\fontsize{\fsize}{#1\fsize}\selectfont}%
  \ifx\svgwidth\undefined%
    \setlength{\unitlength}{209.72634383bp}%
    \ifx\svgscale\undefined%
      \relax%
    \else%
      \setlength{\unitlength}{\unitlength * \real{\svgscale}}%
    \fi%
  \else%
    \setlength{\unitlength}{\svgwidth}%
  \fi%
  \global\let\svgwidth\undefined%
  \global\let\svgscale\undefined%
  \makeatother%
  \begin{picture}(1,0.73664744)%
    \lineheight{1}%
    \setlength\tabcolsep{0pt}%
    \put(0,0){\includegraphics[width=\unitlength,page=1]{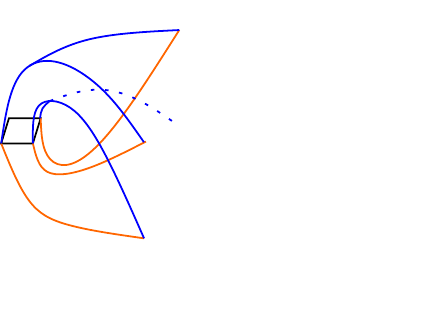}}%
    \put(0.43549968,0.37073037){\color[rgb]{0,0,0}\makebox(0,0)[lt]{\lineheight{1.25}\smash{\begin{tabular}[t]{l}$J\times G_b$\end{tabular}}}}%
    \put(0,0){\includegraphics[width=\unitlength,page=2]{tube_chain_2.pdf}}%
    \put(0.06662046,0.67013351){\color[rgb]{0,0,0}\makebox(0,0)[lt]{\lineheight{1.25}\smash{\begin{tabular}[t]{l}$T_b$\end{tabular}}}}%
  \end{picture}%
\endgroup%

    \caption{An illustration of the doubly specified tube $T_i = \{(V_i,\alpha),(V'_i,\beta)\}$. $V_i$ starts in the bottom (orange) portion and goes to the top (blue) portion.}
    \label{Vi tube}
  \end{figure}
  \begin{itemize}
  \item We define the starting position $(t=0)$ of $V_i$ by $V_i(s,a,0) = (s,f_i(a))$, where, recall, $f_i$ is the embedding from Definition \ref{f_i def}.
  \item The vectors $dV_i(\partial_t)$, $dV_i(\partial_J)$ should start $(t=0)$ pointing in the $-\partial_{\mathbf{n}}$, $\partial_J$ directions respectively, where $\mathbf{n}$ is the inward unit normal of $X$.
  \item As $t$ increases from $0$ to $1$, these vectors $dV_i(\partial_t)$, $dV_i(\partial_J)$ should rotate $180^\circ$ in the plane $\langle \mathbf{n},\partial_J\rangle$ so that at time $t=\frac{1}{2}$, they point in the directions $\partial_J$, $\mathbf{n}$ respectively.
  \item We define the ending position $(t=1)$ of $V_i$ by $V_i(s,a,0) = (-s,f_i(a))$.
  \end{itemize}
  We now define $V_i: X\times [0,1]\to J\times \mathbb{A}^{\kappa}$ by $V'_i(s,a,t) = V_i(-s,a,1-t)$. In effect, $V_i$ turns ``up and around'' back into $Y$, where $V'_i$ turns ``down and around.''
\end{example}

\begin{example}\label{X example}
  Let $X = J\times \Pi^{\kappa-1}$ be, and let $Y$ be as in the previous example. Consider the faces $J\times F_i$, $J\times F_j\subset Y$, $0\leq i<j\leq \kappa$. Let $E:=\conv((J\times F_i)\cup (J\times F_j))$ be the convex hull of the two faces (see Figure \ref{different convex tubes} for an illustration).
  \begin{figure}
    \centering
    \footnotesize
    \def\svgscale{0.55}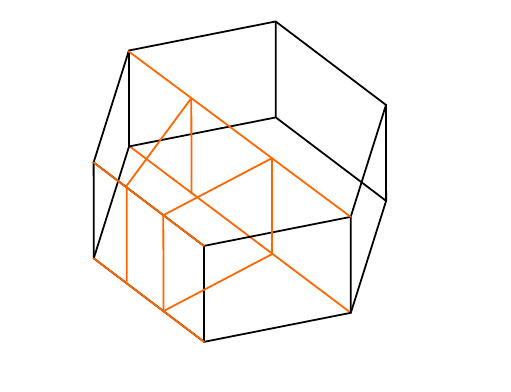\hspace{0.3cm}
    \def\svgscale{0.55}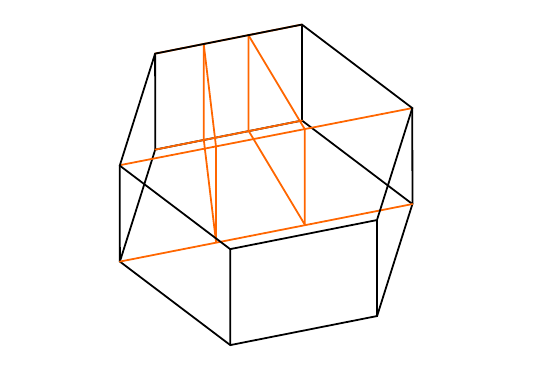\hspace{0.3cm}
    \def\svgscale{0.55}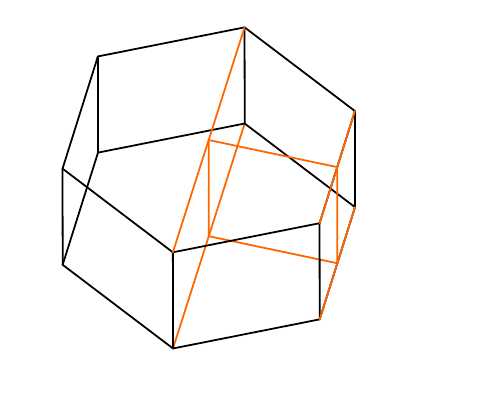
    \caption{The convex tubes $X\times [0,1]\to \Pi^2$, where $X = J\times \Pi^1$. Note that both the left and the middle tubes are boundary-coherent, while the right tube is boundary-incoherent.}
    \label{different convex tubes}
  \end{figure}
  $E$ is the image of a d.s.\ tube $T_{\{i,j\}}=\{(V_{ij},\alpha),(V_{ji},\beta)\}$ containing parametrizations $X\times [0,1]\to Y$. Our definition of $T_{\{i,j\}}$ is as follows:
  \begin{align*}
    V_{ij}(s,(a_1,\ldots,a_\kappa),t) &:= (s,(a_1,\ldots,a_i,\kappa+1,a_{i+1},\ldots,a_\kappa)) + t(\kappa+1-a_j)(e_{j+1}-e_{i+1})\\
    V_{ji}(s,(a_1,\ldots,a_\kappa),t) &:= (s,(a_1,\ldots,a_j,\kappa+1,a_{j+1},\ldots,a_\kappa)) + t(\kappa+1-a_{i+1})(e_{i+1}-e_{j+1}).
  \end{align*}
  We have intentionally defined $V_{ij},V_{ji}$ such that $V_{ij}(\cdot,\cdot,0)$ (resp.\ $V_{ji}(\cdot,\cdot,0)$) is the inclusion $J\times \Pi^{\kappa-1}\xhookrightarrow{\Id \times f_i}J \times \Pi^{\kappa}$ (resp.\ the inclusion $J\times \Pi^{\kappa-1}\xhookrightarrow{\Id \times f_j}J \times \Pi^{\kappa}$) (again, see Figure \ref{different convex tubes}). To check the behavior when $t=1$, we compute
  \begin{equation}\label{endpoint Vi Vj}
  \begin{aligned}
    V_{ij}(s,(a_1,\ldots,a_\kappa),1) &= (s,(a_1,\ldots,a_i,a_j,a_{i+1},\ldots a_{j-1},\kappa+1,a_{j+1},\ldots,a_\kappa)),\\
    V_{ji}(s,(a_1,\ldots,a_\kappa),1) &= (s,(a_1,\ldots,a_i,\kappa+1,a_{i+2},\ldots,a_j,a_{i+1},a_{j+1},\ldots,a_\kappa)),
  \end{aligned}
  \end{equation}
  so it is indeed the case that $V_{ij}(J\times \Pi^{\kappa-1}\times \{t\})$ goes from $J\times F_i$ to $J\times F_j$ and $V_{ji}(J\times\Pi^{\kappa-1}\times \{t\})$ goes from $J\times F_j$ to $J\times F_i$. However, note from (\ref{endpoint Vi Vj}) that $T_{\{i,j\}}$ is only boundary-coherent when $j = i+1$. Indeed, we have the relation $V_{i,j}(s,x,1) = V_{j,i}(s,P_{(i+1\ \ldots\ j)}(x),0)$, where $P_{(i+1\ \ldots\ j)}$ permutes the coordinates of $\Pi^{\kappa-1}$ as
  \[
    P_{(i+1\ \ldots\ j)}(a_1,\ldots,a_\kappa) = (a_1,\ldots,a_i,a_j,a_{i+1},\ldots,\widehat{a_j},\ldots,a_\kappa).
  \]
\end{example}
\begin{definition}[Adding twists]
  Suppose $T = \{(\Theta, \alpha),(\Lambda, \beta)\}$ is a d.s.\ tube with parametrizations $\Theta,\Lambda:W\times [0,1]\to Y$, and the relation $\Theta(x,t) = \Lambda(Dx,1-t)$. Let $\varphi$ be a gluable twist in $T_PW$, with $\varphi_1:T_PW\to T_PW$ induced by an isometry $\Phi:W\to W$. We define $T\diamond(\varphi,\alpha)$ to be the d.s. tube $T' = \{(\Theta',\alpha),(\Lambda',\beta)\}$, where
  \[
    \Theta' = \Theta \diamond \varphi,\qquad \Lambda'(x,t) = \Theta'(\Phi^{-1}D^{-1}(x),1-t)
  \]
  It is immediate from the property that $\Theta'(x,t) = \Lambda'(D \circ \Phi(x),1-t)$, so we indeed have that $T'$ is a d.s.\ tube (since $D \circ \Phi$ is an isometry). Furthermore, we have $\Lambda'(x,0) = \Lambda(x,0)$.
\end{definition}
Note how the names of the tube ends stay the same after adding twists. This makes it easier to continually add twists without having to keep track of the names of the ends.
\begin{notation}[Adding full twists]
  In view of Remark \ref{full twist fact}, if $\dim W \geq 4$, we can add a full twist $f$ along $T$ in either direction and we should get the same result (up to homotopy). So similar to Notation \ref{full twist notation}, we denote $T$ plus a $\diamond$ composition of $k$ full twists (in either direction) as $T+k$.
\end{notation}
\begin{example}
  Let $T_{\{i,j\}}=\{(V_{ij},i),(V_{ji},j)\}$ be as in Example \ref{X example}, and consider the tube $T:= T_{\{i,j\}} \diamond (\varphi_{j-1}\diamond\ldots\diamond \varphi_{i+1},i)$ (see Figures \ref{1 twist fig}, \ref{2 twists fig} for examples).
  \begin{figure}
    \centering
    \footnotesize
    \begin{tabular}{m{5cm} m{3.5cm} m{5cm}}
      \def\svgscale{0.6}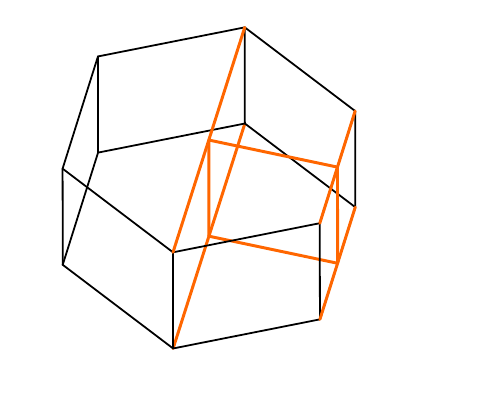 &
      \makecell{\begin{tikzpicture}
      [decoration=snake,
   line around/.style={decoration={pre length=#1,post length=#1}}]
      \draw[->, decorate, line around = 1pt] (0,0)--(1.5,0);
    \end{tikzpicture}\\
      \def\svgscale{0.6}
\begingroup%
  \makeatletter%
  \providecommand\color[2][]{%
    \errmessage{(Inkscape) Color is used for the text in Inkscape, but the package 'color.sty' is not loaded}%
    \renewcommand\color[2][]{}%
  }%
  \providecommand\transparent[1]{%
    \errmessage{(Inkscape) Transparency is used (non-zero) for the text in Inkscape, but the package 'transparent.sty' is not loaded}%
    \renewcommand\transparent[1]{}%
  }%
  \providecommand\rotatebox[2]{#2}%
  \newcommand*\fsize{\dimexpr\f@size pt\relax}%
  \newcommand*\lineheight[1]{\fontsize{\fsize}{#1\fsize}\selectfont}%
  \ifx\svgwidth\undefined%
    \setlength{\unitlength}{124.81198649bp}%
    \ifx\svgscale\undefined%
      \relax%
    \else%
      \setlength{\unitlength}{\unitlength * \real{\svgscale}}%
    \fi%
  \else%
    \setlength{\unitlength}{\svgwidth}%
  \fi%
  \global\let\svgwidth\undefined%
  \global\let\svgscale\undefined%
  \makeatother%
  \begin{picture}(1,0.70372131)%
    \lineheight{1}%
    \setlength\tabcolsep{0pt}%
    \put(0,0){\includegraphics[width=\unitlength,page=1]{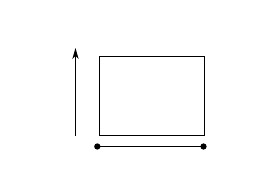}}%
    \put(0.30756322,0.00377511){\color[rgb]{0,0,0}\makebox(0,0)[lt]{\lineheight{1.25}\smash{\begin{tabular}[t]{l}$21$\end{tabular}}}}%
    \put(0.71319583,0.00836218){\color[rgb]{0,0,0}\makebox(0,0)[lt]{\lineheight{1.25}\smash{\begin{tabular}[t]{l}$12$\end{tabular}}}}%
    \put(0.15530505,0.40654592){\color[rgb]{0,0,0}\makebox(0,0)[lt]{\lineheight{1.25}\smash{\begin{tabular}[t]{l}$J$\end{tabular}}}}%
    \put(0,0){\includegraphics[width=\unitlength,page=2]{adding_1_twist.pdf}}%
    \put(0.44398539,0.21770293){\color[rgb]{0,0,0}\makebox(0,0)[lt]{\lineheight{1.25}\smash{\begin{tabular}[t]{l}$180^\circ$\end{tabular}}}}%
    \put(-0.00331676,0.64891905){\color[rgb]{0,0,0}\makebox(0,0)[lt]{\lineheight{1.25}\smash{\begin{tabular}[t]{l}add twist $(\varphi_1,\alpha)$ \end{tabular}}}}%
  \end{picture}%
\endgroup%
} &
                                                                                \def\svgscale{0.6}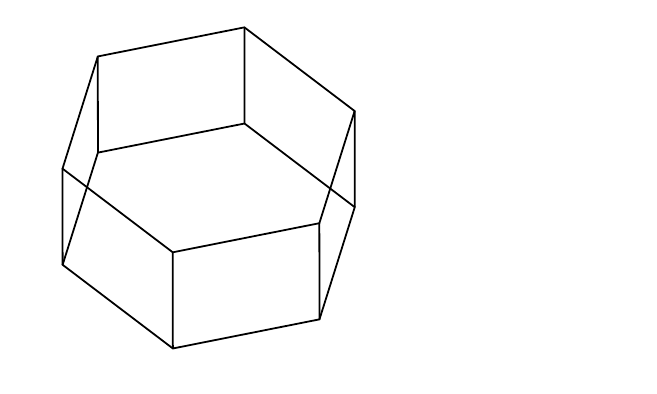
    \end{tabular}
  \caption{We illustrate $T':= T_{\{0,2\}}\diamond (\varphi_1,0)$, where $T=\{(V_{0,2},0),(V_{2,0},2)\}$. Here, we imagine $\kappa = 2$, and we only draw the $J\times\mathcal{M}_{\mathscr{C}_C(n)}(v,\overline{0})$ factor. $T'$ is not boundary-coherent, since the two sides differ by a flip in the $J$-factor}
  \label{1 twist fig}
\end{figure}
\begin{figure}
  \centering
  \def\svgscale{0.80}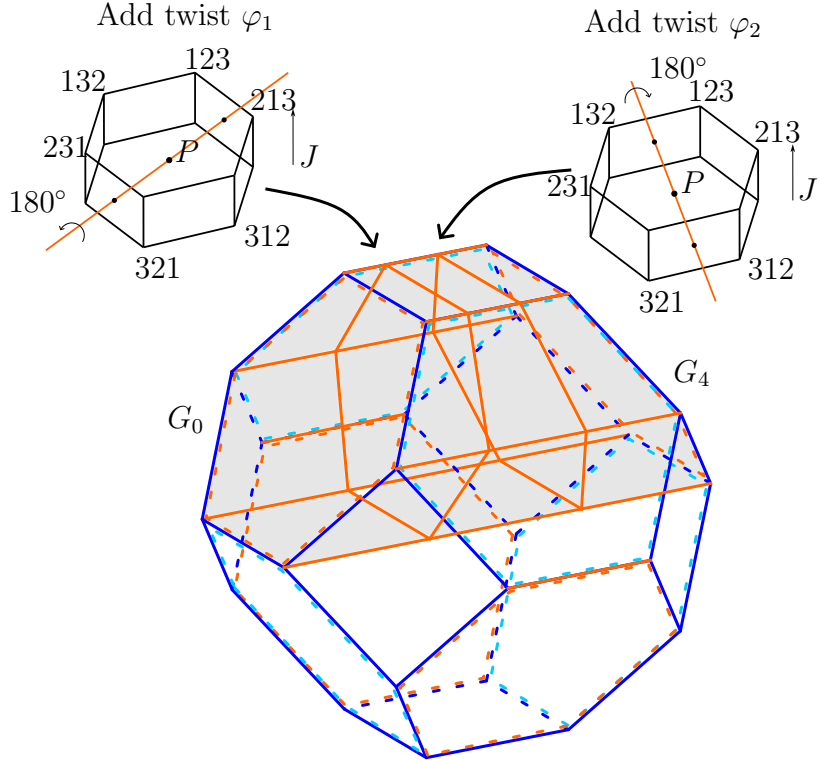
  \caption{An example of the tube $T\diamond (\varphi_2,0)\diamond(\varphi_1,0)$, where $T =\{(V_0,\alpha),(V_3,\beta)\}$ Here, we imagine $\kappa = 3$, and we only draw the $\mathcal{M}_{\mathscr{C}_C(n)}(v,\overline{0})$ factor (so not the $J$ factor). This tube $T$ happens to be boundary-coherent.}
  \label{2 twists fig}
\end{figure}
Observe
\begin{equation}\label{possibly incoherent}
  \begin{aligned}
    &(V_{ij}\diamond\varphi_{j-1}\diamond\ldots\diamond\varphi_{i+1})(s,(a_1,\ldots,a_\kappa),1)\\
    &=V_{ij}(\varphi_{j-1}(1)\ldots\varphi_{i+1}(1)(s,(a_1,\ldots,a_\kappa),1)\\
    &= V_{ij}\left((-1)^{j-i-1}s,P_{(j\ \ldots\ i+1)}(a_1,\ldots,a_\kappa),1\right)\\
    &= \left((-1)^{j-i-1}s,(a_1,\ldots a_{j},\kappa+1,a_{j+1},\ldots,a_\kappa)\right)\\
    &=V_{ji}((-1)^{j-i-1}s,(a_1,\ldots,a_\kappa),0),
  \end{aligned}
\end{equation}

  implying that $T\diamond(\varphi_{j-1}\diamond\ldots\diamond\varphi_{i+1},i)$ is boundary-coherent if and only if $j-i \equiv 1\mod 2$.
\end{example}
Now suppose we want to add a twist $\varphi$ to $T$ in the $\alpha$-direction and a twist $\varphi'$ in the $\beta$-direction. The following proposition shows that the order does not matter.
\begin{proposition}
  Let $T=\{(\Theta,\alpha),(\Lambda,\beta)\}$ be a d.s.\ tube as in Definition \ref{doubly param. tube}, and let $\varphi,\varphi'$ be twists.
  \[
    T \diamond (\varphi,\alpha)\diamond (\varphi',\beta) \cong T \diamond(\varphi',\beta)\diamond(\varphi,\alpha).
  \]
\end{proposition}
\begin{proof}
  Follows from a direct computation.
\end{proof}
We now consider ``sliding'' twists from one end of $T$ to another, whether it be from the $\alpha$-end to the $\beta$-end, or the other way (see Figure \ref{twist direction change} for an illustration).
\begin{figure}
  \centering
  \def\svgscale{1.0}
\begingroup%
  \makeatletter%
  \providecommand\color[2][]{%
    \errmessage{(Inkscape) Color is used for the text in Inkscape, but the package 'color.sty' is not loaded}%
    \renewcommand\color[2][]{}%
  }%
  \providecommand\transparent[1]{%
    \errmessage{(Inkscape) Transparency is used (non-zero) for the text in Inkscape, but the package 'transparent.sty' is not loaded}%
    \renewcommand\transparent[1]{}%
  }%
  \providecommand\rotatebox[2]{#2}%
  \newcommand*\fsize{\dimexpr\f@size pt\relax}%
  \newcommand*\lineheight[1]{\fontsize{\fsize}{#1\fsize}\selectfont}%
  \ifx\svgwidth\undefined%
    \setlength{\unitlength}{431.40287144bp}%
    \ifx\svgscale\undefined%
      \relax%
    \else%
      \setlength{\unitlength}{\unitlength * \real{\svgscale}}%
    \fi%
  \else%
    \setlength{\unitlength}{\svgwidth}%
  \fi%
  \global\let\svgwidth\undefined%
  \global\let\svgscale\undefined%
  \makeatother%
  \begin{picture}(1,0.30559451)%
    \lineheight{1}%
    \setlength\tabcolsep{0pt}%
    \put(0,0){\includegraphics[width=\unitlength,page=1]{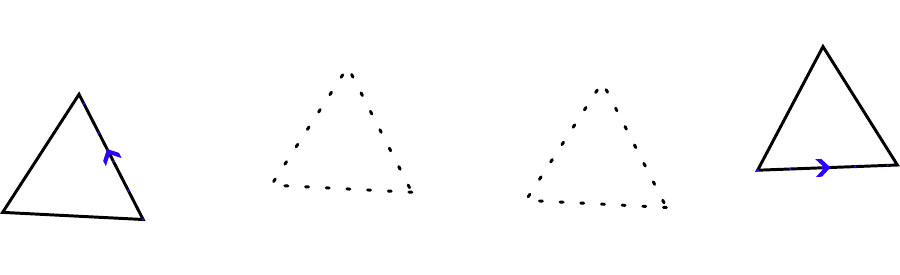}}%
    \put(0.06400776,0.12668884){\color[rgb]{0,0,0}\makebox(0,0)[lt]{\lineheight{1.25}\smash{\begin{tabular}[t]{l}$\alpha$\end{tabular}}}}%
    \put(0.90363631,0.1731453){\color[rgb]{0,0,0}\makebox(0,0)[lt]{\lineheight{1.25}\smash{\begin{tabular}[t]{l}$\beta$\end{tabular}}}}%
    \put(0.30472462,0.00538147){\color[rgb]{0,0,0}\makebox(0,0)[lt]{\lineheight{1.25}\smash{\begin{tabular}[t]{l}$(\varphi\diamond \varphi',\alpha)$\end{tabular}}}}%
    \put(0.82436877,0.01247961){\color[rgb]{0,0,0}\makebox(0,0)[lt]{\lineheight{1.25}\smash{\begin{tabular}[t]{l}$(\psi',\beta)$\end{tabular}}}}%
    \put(0,0){\includegraphics[width=\unitlength,page=2]{same_tube_different_twists.pdf}}%
    \put(0.21090371,0.28115704){\color[rgb]{0,0,0}\makebox(0,0)[lt]{\lineheight{1.25}\smash{\begin{tabular}[t]{l}$(\varphi',\alpha)$\end{tabular}}}}%
    \put(0.59817553,0.28936476){\color[rgb]{0,0,0}\makebox(0,0)[lt]{\lineheight{1.25}\smash{\begin{tabular}[t]{l}$(\psi\diamond\psi',\beta)$\end{tabular}}}}%
    \put(0,0){\includegraphics[width=\unitlength,page=3]{same_tube_different_twists.pdf}}%
  \end{picture}%
\endgroup%
\vspace{0.3cm}
  \caption{An illustration of how the same tube $T'$ can be constructed by adding a twist to $T$ in the $\alpha$-direction vs.\ adding a twist in the $\beta$-direction.}
  \label{twist direction change}
\end{figure}

\begin{proposition}\label{sliding}
  Let $T = \{(\Theta,\alpha),(\Lambda,\beta)\}$, with relation $\Theta(x,t) = \Lambda(Dx,1-t)$. We have
  \[
    T\diamond(\varphi\diamond\varphi',\alpha) \diamond (\psi',\beta) \cong T\diamond(\varphi',\alpha) \diamond (\psi\diamond \psi',\beta),
  \]
  where $\psi$ is the twist $\psi(t) = (dD)\varphi(1-t)\varphi(1)^{-1}(dD)^{-1}$. (Here, $dD$ is the tangent map $dD: T_PW\to T_PW$.)
\end{proposition}
\begin{proof}
  Follows from another direct computation.
\end{proof}
\begin{example}\label{Vij twists}
  Let $T_{\{i,j\}}=\{(V_{ij},i),(V_{ji},j)\}$ be as in Example \ref{X example}, with relation $V_{ij}(x,t) = V_{ji}(Dx,1-t)$, and consider again $T := T_{\{i,j\}}\diamond(\varphi_{j-1}\diamond\ldots\diamond\varphi_{i+1},i)$. After a repeated application of Proposition \ref{sliding}, we have the identity
  \begin{equation}\label{adding varphis}
    T = T_{\{i,j\}}\diamond(\varphi_{j-1}\diamond\ldots\diamond\varphi_{i+1},i) = T_{\{i,j\}}\diamond(\psi_{i+1}\diamond\ldots\diamond \psi_{j-1},j),
  \end{equation}
  where
  \begin{align*}
    \psi_{(n)}(t)
    &= (dD)\varphi_{n}(1-t)\varphi_{n}(1)^{-1}(dD)^{-1}\\
    &=P_{(i+1\ldots j)}\left(\varphi_{n}(t)^{-1}\right)\left(P_{(i+1\ldots j)}\right)^{-1}=
    \begin{cases*}
      \varphi_{j, i+1}^{-1}(t) & if $n=j-1$\\
      \varphi_{(n+1)}^{-1}(t) & otherwise.
    \end{cases*}    
  \end{align*}
  By a repeated application of Lemma \ref{commutators}, the right hand side of (\ref{adding varphis}) simplifies as
  \[
    T\diamond (\varphi_{(i+2)}^{-1}\diamond \ldots \diamond \varphi_{(j-1)}^{-1}\diamond\varphi_{j, i+1},j) \cong \ldots \cong T\diamond(\varphi_{(i+1)}\diamond \ldots\diamond \varphi_{(j-1)},j)
  \]
  by ``sliding'' the $\varphi_{j, i+1}$ to the other end of the composition.
  
\end{example}
We are allowed to compose twists with d.s.\ tubes, with Lemma \ref{associativity} showing that this composition rule is associative up to homotopy. Therefore, we can imagine that tubes are analogous with points in affine space and twists are analogous with vectors. With this analogy in mind, we define the difference $T-T'$ of tubes in terms of twists.
\begin{definition}\label{tube difference}
  Let $T= \{(\Theta,\alpha),(\Lambda,\beta)\}$, $T' = \{(\Theta',\alpha),(\Lambda',\beta)\}$ both be doubly specified tubes, with
  \begin{equation}\label{tube match up}
    \Theta(x,0)=\Theta'(x,0),\qquad \Lambda(x,0) = \Lambda'(x,0).
  \end{equation}
  We define $T-T'$ as the set $T-T' := \{\{(f,\alpha), (g,\beta)\}: T' \cong T \diamond (f,\alpha) \diamond (g,\beta)\}$. If $T' \cong T \diamond(f,\alpha) \diamond (g,\beta)$, then we write $T' - T \cong (f,\alpha)\diamond (g,\beta)$. Furthermore, we can compose sets of the form $T'-T''$, $T-T'$ as long as the end labels for $T$, $T'$, $T''$ all match up.
\end{definition}
Note how the criterion (\ref{tube match up}) that the ends of $T$, $T'$ match up is necessary, but not sufficient for $T-T'$ to be nonempty. For example, the cores of $T,T'$ may not even be homotopic, in which case there is no way to express $T'$ as a twisting of $T$. Now if $T,T'\subset Y$ are smooth d.s. tubes, where $Y$ is a smooth manifold, and the cores $T,T'$ are homotopic, then $T-T'$ is necessarily nonempty.
\begin{lemma}
  Whe have the relation $T-T'' = (T'-T'')\diamond (T-T')$ In particular, if $T'-T''\cong (f',\alpha)\diamond (g',\beta)$ and $T-T'\cong (f,\alpha)\diamond (g,\beta)$, we have $T-T''\cong (f'\diamond f,\alpha)\diamond (g'\diamond g,\beta)$.
\end{lemma}
\begin{proof}
  Follows from a direct computation using Definition \ref{tube difference}.
\end{proof}

\subsection{Composing doubly specified tubes}
Consider two doubly specified tubes $T,T'\subset Y$, and imagine that in our setup, the $\Lambda$-end of $T$ is identified with the $\Omega'$-end of $T'$. (See Figure \ref{combining tubes} for an illustration.)
\begin{figure}
  \centering
  \def\svgscale{0.9}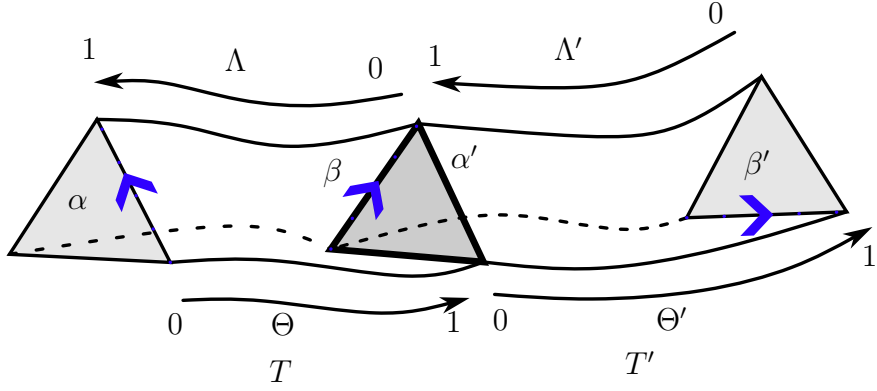
  \caption{The tube $T\cup T'$, where $T=\{(\Theta,\alpha),(\Lambda,\beta)\}$, $T' = \{(\Theta',\alpha'),(\Lambda',\beta')\}$. Here, $T\cup T'$ joins the $\beta$-end of $T$ with the $\alpha'$ end of $T'$.}
  \label{combining tubes}
\end{figure}
We wish to doubly parametrize $T\cup T'$, where the $T$-portion (resp.\ $T'$-portion) is still parametrized using the ``free'' $\Theta$-end (resp.\ $\Lambda'$-end).
\begin{definition}\label{tube compose operation}
  Let $T,T'$ be doubly specified tubes, with $T=\{(\Theta,\alpha),(\Lambda,\beta)\}$, $T' = \{(\Theta',\alpha'),(\Lambda',\beta')\}$, and end relations $\Theta(x,t) = \Lambda(D(x),1-t)$ and $\Theta'(x,t) = \Lambda'(D'(x),1-t)$. Suppose additionally that $\Lambda(x,0) = \Theta'(x,0)$. We define $T\cup T' = T\cup_{\beta,\alpha'} T' := \{(\widehat{\Theta},\alpha),(\widehat{\Lambda},\beta')\}$, with the parametrizations $\widehat{\Theta},\widehat{\Lambda}$ defined by
  \begin{align*}
  \widehat{\Theta}(x,t) &= 
    \left\{\begin{array}{@{}l@{}}
      \Theta(x,2t) \\[\jot]
      \Theta'(D(x),2t-1)
    \end{array}\right. &&\hspace{-5em}
    \begin{array}{@{}l@{}}
      t\in [0,1/2] \\[\jot]
      t\in [1/2,1],
    \end{array}\\[\jot]
    \widehat{\Lambda}(x,t) &= 
    \left\{\begin{array}{@{}l@{}}
      \Lambda'(x,2t) \\[\jot]
      \Lambda(D'(x),2t-1)
    \end{array}\right. &&\hspace{-5em}
    \begin{array}{@{}l@{}}
      t\in [0,1/2] \\[\jot]
      t\in [1/2,1].
    \end{array}
  \end{align*}
  Furthermore, if we are given a sequence of d.s.\ tubes $T_1,\ldots,T_n$, with $T_i = \{(\Theta_i,\alpha_i), (\Lambda_i,\beta_i)\}$, and $\Lambda_i(x,0) = \Theta_{i+1}(x,0)$, we can iterate the ``$\cup$'' operation to obtain a d.s.\ $T_1\cup\ldots \cup T_n$. The order of composition is invariant up to homotopy.
\end{definition}
We intentionally defined $\widehat{\Theta}$ and $\widehat{\Lambda}$ in order to satisfy the identities $\widehat{\Theta}(x,t) = \Theta(x,2t)$, $\widehat{\Lambda}(x,t) = \Lambda(x,2t)$ for $0\leq t\leq 1/2$. In other words, $\widehat{\Theta}|_{W\times [0,1/2]}$ parametrizes $T$ just like $\Theta$ (and $\widehat{\Lambda}|_{W\times [0,1/2]}$ parametrizes $T'$ just like $\Lambda'$).\par

Consider a concatenation $(T\diamond(\varphi,\alpha)) \cup T'$. We describe how to ``push'' $\varphi$ into $T'$ (see Figure \ref{intertube slide}).
\begin{figure}
  \centering
  \begin{tabular}{c}
\begingroup%
  \makeatletter%
  \providecommand\color[2][]{%
    \errmessage{(Inkscape) Color is used for the text in Inkscape, but the package 'color.sty' is not loaded}%
    \renewcommand\color[2][]{}%
  }%
  \providecommand\transparent[1]{%
    \errmessage{(Inkscape) Transparency is used (non-zero) for the text in Inkscape, but the package 'transparent.sty' is not loaded}%
    \renewcommand\transparent[1]{}%
  }%
  \providecommand\rotatebox[2]{#2}%
  \newcommand*\fsize{\dimexpr\f@size pt\relax}%
  \newcommand*\lineheight[1]{\fontsize{\fsize}{#1\fsize}\selectfont}%
  \ifx\svgwidth\undefined%
    \setlength{\unitlength}{362.4291031bp}%
    \ifx\svgscale\undefined%
      \relax%
    \else%
      \setlength{\unitlength}{\unitlength * \real{\svgscale}}%
    \fi%
  \else%
    \setlength{\unitlength}{\svgwidth}%
  \fi%
  \global\let\svgwidth\undefined%
  \global\let\svgscale\undefined%
  \makeatother%
  \begin{picture}(1,0.32601511)%
    \lineheight{1}%
    \setlength\tabcolsep{0pt}%
    \put(0,0){\includegraphics[width=\unitlength,page=1]{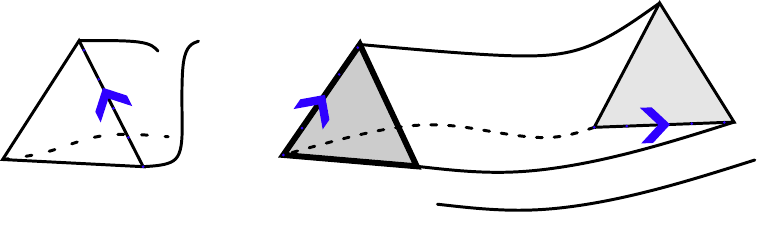}}%
    \put(0.07171879,0.16726244){\color[rgb]{0,0,0}\makebox(0,0)[lt]{\lineheight{1.25}\smash{\begin{tabular}[t]{l}$\alpha$\end{tabular}}}}%
    \put(0.43745411,0.18523903){\color[rgb]{0,0,0}\makebox(0,0)[lt]{\lineheight{1.25}\smash{\begin{tabular}[t]{l}$\beta$\end{tabular}}}}%
    \put(0.51331995,0.21531243){\color[rgb]{0,0,0}\makebox(0,0)[lt]{\lineheight{1.25}\smash{\begin{tabular}[t]{l}$\alpha'$\end{tabular}}}}%
    \put(0.85210231,0.20933372){\color[rgb]{0,0,0}\makebox(0,0)[lt]{\lineheight{1.25}\smash{\begin{tabular}[t]{l}$\beta'$\end{tabular}}}}%
    \put(0.25721056,0.00640557){\color[rgb]{0,0,0}\makebox(0,0)[lt]{\lineheight{1.25}\smash{\begin{tabular}[t]{l}$T\diamond (\varphi,\alpha)$\end{tabular}}}}%
    \put(0.77060457,0.01856146){\color[rgb]{0,0,0}\makebox(0,0)[lt]{\lineheight{1.25}\smash{\begin{tabular}[t]{l}$T'$\end{tabular}}}}%
    \put(0,0){\includegraphics[width=\unitlength,page=2]{one_tube_to_another_1.pdf}}%
  \end{picture}%
\endgroup%
\\
    \hspace{2cm}
\begingroup%
  \makeatletter%
  \providecommand\color[2][]{%
    \errmessage{(Inkscape) Color is used for the text in Inkscape, but the package 'color.sty' is not loaded}%
    \renewcommand\color[2][]{}%
  }%
  \providecommand\transparent[1]{%
    \errmessage{(Inkscape) Transparency is used (non-zero) for the text in Inkscape, but the package 'transparent.sty' is not loaded}%
    \renewcommand\transparent[1]{}%
  }%
  \providecommand\rotatebox[2]{#2}%
  \newcommand*\fsize{\dimexpr\f@size pt\relax}%
  \newcommand*\lineheight[1]{\fontsize{\fsize}{#1\fsize}\selectfont}%
  \ifx\svgwidth\undefined%
    \setlength{\unitlength}{430.02283292bp}%
    \ifx\svgscale\undefined%
      \relax%
    \else%
      \setlength{\unitlength}{\unitlength * \real{\svgscale}}%
    \fi%
  \else%
    \setlength{\unitlength}{\svgwidth}%
  \fi%
  \global\let\svgwidth\undefined%
  \global\let\svgscale\undefined%
  \makeatother%
  \begin{picture}(1,0.27022071)%
    \lineheight{1}%
    \setlength\tabcolsep{0pt}%
    \put(0,0){\includegraphics[width=\unitlength,page=1]{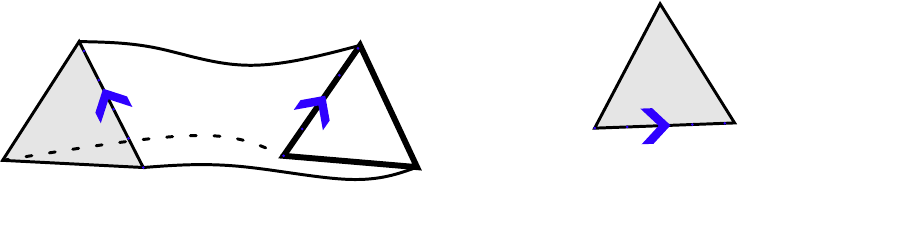}}%
    \put(0.05970592,0.1361845){\color[rgb]{0,0,0}\makebox(0,0)[lt]{\lineheight{1.25}\smash{\begin{tabular}[t]{l}$\alpha$\end{tabular}}}}%
    \put(0.3729232,0.15327064){\color[rgb]{0,0,0}\makebox(0,0)[lt]{\lineheight{1.25}\smash{\begin{tabular}[t]{l}$\beta$\end{tabular}}}}%
    \put(0.43263307,0.1759421){\color[rgb]{0,0,0}\makebox(0,0)[lt]{\lineheight{1.25}\smash{\begin{tabular}[t]{l}$\alpha'$\end{tabular}}}}%
    \put(0.71816346,0.17090316){\color[rgb]{0,0,0}\makebox(0,0)[lt]{\lineheight{1.25}\smash{\begin{tabular}[t]{l}$\beta'$\end{tabular}}}}%
    \put(0.27288603,0.00619363){\color[rgb]{0,0,0}\makebox(0,0)[lt]{\lineheight{1.25}\smash{\begin{tabular}[t]{l}$T$\end{tabular}}}}%
    \put(0.57608178,0.00539872){\color[rgb]{0,0,0}\makebox(0,0)[lt]{\lineheight{1.25}\smash{\begin{tabular}[t]{l}$T'\diamond ((dD)\varphi(dD)^{-1},\alpha')$\end{tabular}}}}%
    \put(0,0){\includegraphics[width=\unitlength,page=2]{one_tube_to_another_2.pdf}}%
  \end{picture}%
\endgroup%

  \end{tabular}
  \caption{Sliding the twist $\varphi$ from $T$ to $T'$.}
  \label{intertube slide}
\end{figure}
\begin{example}\label{gluing tubes homotopy}
  Let $X = J\times \Pi^{\kappa-1}$, $Y = J\times \Pi^{\kappa}$. Let $0\leq a,b,c\leq \kappa$ be distinct integers. $T_{\{a,b\}}\cup_{b,\alpha} T_b\cup_{\beta,b} T_{\{b,c\}}$ is homotopic to $T_{\{a,c\}}\diamond(\rho,a)$ for some twist $\rho$. (see Figure \ref{gluing tubes picture}).
  \begin{figure}
    \begin{tabular}{m{4.2cm} m{0.3cm} m{4.3cm} m{0.2cm} m{4.3cm}}
      \def\svgscale{0.6}
\begingroup%
  \makeatletter%
  \providecommand\color[2][]{%
    \errmessage{(Inkscape) Color is used for the text in Inkscape, but the package 'color.sty' is not loaded}%
    \renewcommand\color[2][]{}%
  }%
  \providecommand\transparent[1]{%
    \errmessage{(Inkscape) Transparency is used (non-zero) for the text in Inkscape, but the package 'transparent.sty' is not loaded}%
    \renewcommand\transparent[1]{}%
  }%
  \providecommand\rotatebox[2]{#2}%
  \newcommand*\fsize{\dimexpr\f@size pt\relax}%
  \newcommand*\lineheight[1]{\fontsize{\fsize}{#1\fsize}\selectfont}%
  \ifx\svgwidth\undefined%
    \setlength{\unitlength}{200.39341484bp}%
    \ifx\svgscale\undefined%
      \relax%
    \else%
      \setlength{\unitlength}{\unitlength * \real{\svgscale}}%
    \fi%
  \else%
    \setlength{\unitlength}{\svgwidth}%
  \fi%
  \global\let\svgwidth\undefined%
  \global\let\svgscale\undefined%
  \makeatother%
  \begin{picture}(1,0.85245788)%
    \lineheight{1}%
    \setlength\tabcolsep{0pt}%
    \put(0,0){\includegraphics[width=\unitlength,page=1]{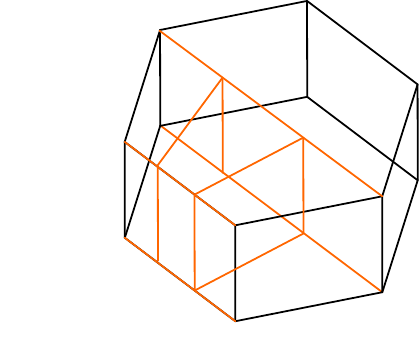}}%
    \put(-0.00278444,0.66945572){\color[rgb]{0,0,0}\makebox(0,0)[lt]{\lineheight{1.25}\smash{\begin{tabular}[t]{l}$J\times G_b$\end{tabular}}}}%
    \put(0,0){\includegraphics[width=\unitlength,page=2]{tube_chain_1.pdf}}%
    \put(0.15779271,0.12826626){\color[rgb]{0,0,0}\makebox(0,0)[lt]{\lineheight{1.25}\smash{\begin{tabular}[t]{l}$T_{\{a,b\}}$\end{tabular}}}}%
    \put(0.66329571,0.00853314){\color[rgb]{0,0,0}\makebox(0,0)[lt]{\lineheight{1.25}\smash{\begin{tabular}[t]{l}$J\times G_{a}$\end{tabular}}}}%
  \end{picture}%
\endgroup%
 & \raisebox{1.0cm}{$\bigcup$} & \def\svgscale{0.6} & \raisebox{1.0cm}{$\bigcup$} & \def\svgscale{0.6}
\begingroup%
  \makeatletter%
  \providecommand\color[2][]{%
    \errmessage{(Inkscape) Color is used for the text in Inkscape, but the package 'color.sty' is not loaded}%
    \renewcommand\color[2][]{}%
  }%
  \providecommand\transparent[1]{%
    \errmessage{(Inkscape) Transparency is used (non-zero) for the text in Inkscape, but the package 'transparent.sty' is not loaded}%
    \renewcommand\transparent[1]{}%
  }%
  \providecommand\rotatebox[2]{#2}%
  \newcommand*\fsize{\dimexpr\f@size pt\relax}%
  \newcommand*\lineheight[1]{\fontsize{\fsize}{#1\fsize}\selectfont}%
  \ifx\svgwidth\undefined%
    \setlength{\unitlength}{238.5129457bp}%
    \ifx\svgscale\undefined%
      \relax%
    \else%
      \setlength{\unitlength}{\unitlength * \real{\svgscale}}%
    \fi%
  \else%
    \setlength{\unitlength}{\svgwidth}%
  \fi%
  \global\let\svgwidth\undefined%
  \global\let\svgscale\undefined%
  \makeatother%
  \begin{picture}(1,0.70493788)%
    \lineheight{1}%
    \setlength\tabcolsep{0pt}%
    \put(0,0){\includegraphics[width=\unitlength,page=1]{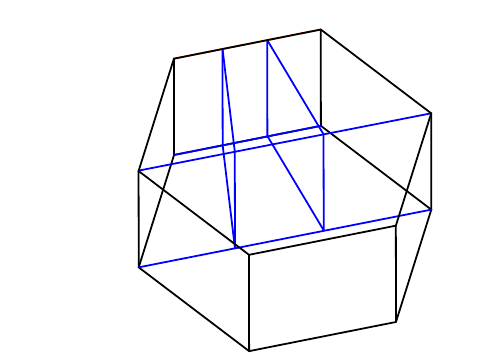}}%
    \put(-0.00233947,0.48745587){\color[rgb]{0,0,0}\makebox(0,0)[lt]{\lineheight{1.25}\smash{\begin{tabular}[t]{l}$J\times G_b$\end{tabular}}}}%
    \put(0.76446566,0.58027033){\color[rgb]{0,0,0}\makebox(0,0)[lt]{\lineheight{1.25}\smash{\begin{tabular}[t]{l}$J\times G_c$\end{tabular}}}}%
    \put(0,0){\includegraphics[width=\unitlength,page=2]{tube_chain_3.pdf}}%
    \put(0.38639683,0.67626043){\color[rgb]{0,0,0}\makebox(0,0)[lt]{\lineheight{1.25}\smash{\begin{tabular}[t]{l}$T_{\{b,c\}}$\end{tabular}}}}%
  \end{picture}%
\endgroup%

    \end{tabular}
    \begin{tabular}{m{0.5cm} m{4.3cm}}
      $\cong$ & \def\svgscale{0.6}
\begingroup%
  \makeatletter%
  \providecommand\color[2][]{%
    \errmessage{(Inkscape) Color is used for the text in Inkscape, but the package 'color.sty' is not loaded}%
    \renewcommand\color[2][]{}%
  }%
  \providecommand\transparent[1]{%
    \errmessage{(Inkscape) Transparency is used (non-zero) for the text in Inkscape, but the package 'transparent.sty' is not loaded}%
    \renewcommand\transparent[1]{}%
  }%
  \providecommand\rotatebox[2]{#2}%
  \newcommand*\fsize{\dimexpr\f@size pt\relax}%
  \newcommand*\lineheight[1]{\fontsize{\fsize}{#1\fsize}\selectfont}%
  \ifx\svgwidth\undefined%
    \setlength{\unitlength}{293.44369963bp}%
    \ifx\svgscale\undefined%
      \relax%
    \else%
      \setlength{\unitlength}{\unitlength * \real{\svgscale}}%
    \fi%
  \else%
    \setlength{\unitlength}{\svgwidth}%
  \fi%
  \global\let\svgwidth\undefined%
  \global\let\svgscale\undefined%
  \makeatother%
  \begin{picture}(1,0.59961121)%
    \lineheight{1}%
    \setlength\tabcolsep{0pt}%
    \put(0,0){\includegraphics[width=\unitlength,page=1]{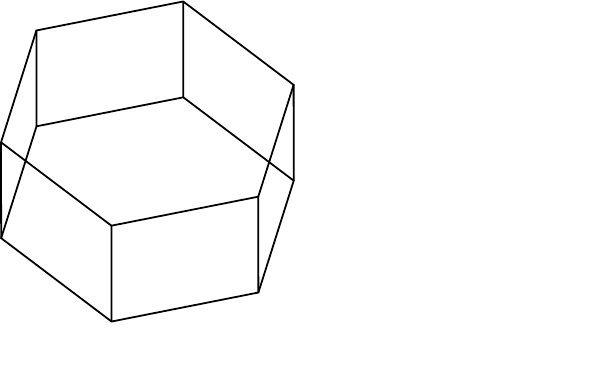}}%
    \put(0.28763026,0.00582731){\color[rgb]{0,0,0}\makebox(0,0)[lt]{\lineheight{1.25}\smash{\begin{tabular}[t]{l}$J\times G_a$\end{tabular}}}}%
    \put(0.41827313,0.53452417){\color[rgb]{0,0,0}\makebox(0,0)[lt]{\lineheight{1.25}\smash{\begin{tabular}[t]{l}$J\times G_c$\end{tabular}}}}%
    \put(0,0){\includegraphics[width=\unitlength,page=2]{tube_chain_pulled.pdf}}%
    \put(0.49820384,0.23281791){\color[rgb]{0,0,0}\makebox(0,0)[lt]{\lineheight{1.25}\smash{\begin{tabular}[t]{l}$T_{\{a,c\}}\diamond (\rho,\alpha)$\end{tabular}}}}%
  \end{picture}%
\endgroup%

    \end{tabular}
    \caption{The homotopic tubes $T_{\{a,b\}}\cup T_b\cup T_{\{b,c\}}\cong T_{\{a,c\}}\diamond(\rho,a)$ from Example \ref{gluing tubes homotopy}.}
    \label{gluing tubes picture}
  \end{figure}
  To see what $\rho$ must be, consult Table \ref{flipback table}.
  \begin{table}
    \centering
    \begin{tabular}{ m{4em} M{6cm} }
      Cases for \hextwist{$b$}{$c$}{$a$} & \makecell{Twist $\rho$ such that\\ $T_{\{a,c\}}+(\rho,\Theta_a) = T_{\{a,b\}}\cup T_b\cup T_{\{b,c\}}$}\\
      \hline
\hextwist{$j$}{$k$}{$i$}& $\varphi_{j, k}$ \\ 
\hextwist{$j$}{$i$}{$k$} & $\varphi_{j+1, i+1}$\\
\hextwist{$k$}{$j$}{$i$}& $\varphi_{k, j}$\\
\hextwist{$k$}{$i$}{$j$} & $\varphi_{k, i+1}$ \\
\hextwist{$i$}{$k$}{$j$} & $\varphi_{i+1, k}$ \\ 
\hextwist{$i$}{$j$}{$k$} & $\varphi_{i+1, j+1}$
    \end{tabular}
    \caption{How to read this table: $i = \min(a,b,c)$, $j = \Mid(a,b,c)$, $k = \max(a,b,c)$}
    \label{flipback table}
  \end{table}
\end{example}

\begin{proposition}\label{twist through concat}
  Suppose $T$, $T'$ are doubly specified tubes, with $T=\{(\Theta,\alpha),(\Lambda,\beta)\}$, $T'=\{(\Theta',\alpha),(\Lambda',\beta')\}$, with the relation $\Theta(x,t) = \Lambda(D(x),1-t)$. Suppose $\Lambda(x,0)=\Theta'(x,0)$. Then
  \[
    (T\diamond(\varphi,\alpha))\cup T' \cong T\cup(T'\diamond((dD)\varphi (dD)^{-1},\alpha')
  \]
\end{proposition}
\begin{proof}
  We can either directly compute this identity using the definitions, or we can use Proposition \ref{sliding}.
\end{proof}
\section{Flow categories, the cube flow category, and signed cubical realizations}
We give a brief summary of the cube flow category, introduced by Lawson, Lipshitz, Sarkar \cite{MR4153651}.
\subsection{The cube category and sign assignments. }
\begin{definition}[\cite{MR4153651}]
For $n\in \mathbb{N}$, $n>0$, we define the cube category $\underline{2}^n$ as follows:
\begin{itemize}
\item $\Ob(\underline{2}^n) = 2^{\{1,\ldots, n\}}$, that is, subsets of $\{1,\ldots, n\}$.
\item For two objects $u,v\in \underline{2}^n$, $\text{Hom}(u,v)$ is empty unless $u \supseteq v$, in which case there is a single morphism $\phi_{u,v}:u\to v$.
\end{itemize}
\end{definition}
We often use the notation  $u<v$ as shorthand for $u\subset v$ (and $u\leq v$ as shorthand for $u\subseteq v$). Thre is a grading on the objects of $\underline{2}^n$ given by $\gr(u)=|u|$. If $u>v$ and $\gr(u)=\gr(v)+i$, we use the notation $u>_i v$. In the case where $i=1$ ($i=2$), we call $\phi_{u,v}$ an \textit{edge} (a \textit{face}). If we want to emphasize that $u,v$ are subsets of $\{1,\ldots, n\}$, we often use labels $S,T$.
\begin{definition}\label{cube sign assignment}
  A \textit{sign assignment} $\widetilde{s}$ is a function from the edges $u>_1v$ of the cube category $\underline{2}^n$ to $\mathbb{F}_2$ such that, for any face $u>_2w$ with intermediate vertices $v_1,v_2$, we have $\widetilde{s}_{u,v_1}+\widetilde{s}_{u,v_2}+\widetilde{s}_{v_1,w}+\widetilde{s}_{v_2,w} = 1$. (We abbreviate $\widetilde{s}_{u,v}:= \widetilde{s}(\phi_{u,v})$).\par
  The \textit{standard sign assignment} $s$ is defined as follows: For an edge $T>_1 S$, $T \backslash S = \{j\}$, we define
  \[
    s_{T,S}:= \# \{i \in S: i<j\}\pmod 2\in \mathbb{F}_2.
  \]
  When viewing a sign assignment $\tilde{s}$ as a cochain in $C^1_{\cell}([0,1]^n,\mathbb{F}_2)$, $\delta s \in C^2_{\cell}([0,1]^n,\mathbb{F}_2)$ is the constant map that maps all faces to $1$. Note that any sign assignment $\tilde{s}$ differs from $s$ by a cocycle $\delta c\in C^2_{\cell}([0,1]^n;\mathbb{F}_2)$.
\end{definition}
\begin{definition}
  The \textit{index assignment} $s_\mathbb{Z}$ is defined as follows: For an edge $T>_1 S$, $T \backslash S = \{j\}$, we define $s_\mathbb{Z}(T,S):= \# \{i \in S: i<j\}\in \mathbb{Z}$.
\end{definition}

\subsection{Manifolds with corners and $\langle n\rangle$-manifolds.}
\begin{definition}
  A \textit{facet} of $X$ is the closure of a codimension-$1$ boundary-component of $X$. A \textit{multifacet} of $X$ is a (possibly empty) union of disjoint facets of $X$. A manifold with corners $X$ is a \textit{multifaced manifold} if every $x\in X$ belongs to exactly $c(x)$ facets of $X$. We define an $\langle n\rangle$-manifold to be a multifaceted manifold $X$ with an ordered n-tuple $(\partial_1 X, \ldots, \partial_n X)$ of multifacets $X$ satisfying
  \begin{itemize}
  \item $\bigcup_i\partial_i X = \partial X$.
  \item For all distinct $i,j$, $\partial_i X\cap \partial_j X$ is a multifacet of both $\partial_i X$ and $\partial_j X$.
  \end{itemize}  
\end{definition}
\begin{lemma}\label{permutohedron as n-manifold}
  The space $\Pi^{n-1}$ can be viewed as an $\langle n-1 \rangle$-manifold by defining
  \[
    \partial_i\Pi^{n-1}=\bigcup_{\substack{S\\|S|=i}}F_S
  \]
  for $1\leq i\leq n-1$.
\end{lemma}
\begin{proof}
  The proof is due to \cite{MR4153651}
  We check the following:
  \begin{enumerate}[label = (\arabic*), ref = (\arabic*)]
  \item Each $\partial_i \Pi^{n-1}$ is a multifacet\label{n-manifold 1}
  \item Every point $x$ belongs to $c(x)$ facets\label{n-manifold 2}
  \item $\bigcup_i\partial_i\Pi^{n-1}=\partial \Pi^{n-1}$\label{n-manifold 3}
  \item For each $i\neq j$, $\partial_i\Pi^{n-1}\cap \partial_j \Pi^{n-1}$ is a multifacet of $\partial_i\Pi^{n-1}$ (and $\partial_j\Pi^{n-1}$).\label{n-manifold 4}
    \end{enumerate}
    To prove \ref{n-manifold 1}, we use Lemma \ref{face intersection}: if $F_S\cap F_T\neq\emptyset$, then either $S\subseteq T$ or $T\subseteq S$. Therefore, $\partial_i\Pi^{n-1}$ is indeed a \textit{disjoint} union of the faces $S$, where $|S|=i$.\par
    To prove \ref{n-manifold 2}, fix a point $x\in \Pi^{n-1}$ and note that $x$ is possibly contained in facets, which are of the form $F_S$. For convenience, we call these facets $F_{S_1},\ldots, F_{S_m}$, where $S_1<\ldots < S_m$ (the collection of $S_i$ is possibly empty). From the construction $\Pi^{n-1}=\bigcap_{\emptyset\neq S\subset \{1,\ldots, n\}}H_S$, where $H_S$ are half-spaces, we see that $x$ is contained in a coordinate neighborhood diffeomorphic to $\bigcap_{1\leq i\leq m}H_{S_i}\cong \{y\in \mathbb{R}^{n-1}: y_1\geq 0, \ldots, y_m\geq 0\}$.\par
    Point \ref{n-manifold 3} is immediate from the definitions.\par
    For part \ref{n-manifold 4}, let $F_S$ be a facet of $\partial_i\Pi^{n-1}$. it suffices to show that $F_S\cap \partial_j \Pi^{n-1}$ is a multifacet of $F_S$. We simply need to show that $F_S\cap F_T$ is either empty or a facet of $F_S$. But Lemma \ref{face intersection} says that $F_S\cap F_T$ is empty unless $T\subset S$ or $S\subset T$. The map $f_S: F_S\to \Pi^{i-1}\times \Pi^{k-i-1}$ of Lemma \ref{face identification} identifies $F_S\cap F_T$ with
    \[
      \begin{cases*}
        F_T'\times \Pi^{k-i-1} & if $T\subset S$\\
        \Pi^{i-1}\times F_T' & if $S \subset T$
      \end{cases*},
    \]
    where  if $T\subset S$, then $F_T'\cong \Pi^{j-1}\times \Pi^{i-j-1}$ is a facet of $\Pi^{i-1}$ and if $S\subset T$, then $F_T'\cong \Pi^{j-i-1}\times \Pi^{k-j-1}$ is a facet of $\Pi^{k-i-1}$.
\end{proof}
\subsection{Signed flow categories}
We review some notation about flow categories from Cohen-Jones-Segal \cite{MR1362832}.
\begin{definition}[\cite{MR1362832}]
  A \textit{flow category} $\mathscr{C}$ is a topological category such that the objects $\Ob(\mathscr{C})$ form a discrete space, and the morphisms satisfy the following:
  \begin{enumerate}[label = (FC-\arabic*), ref = (FC-\arabic*)]
  \item For any $x\in \Ob(\mathscr{C})$, $\text{Hom}(x,x) = \{\text{Id}\}$; The identity morphisms in a flow category are a somewhat special case and it is often convenient to disregard them. We define the \textit{moduli space} $\mathcal{M}(x,y)$ from $x$ to $y$ to be $\text{Hom}(x,y)$ if $x\neq y$ and empty if $x=y$.
  \item For any $x,y\in \Ob(\mathscr{C})$ with $\gr(x)-\gr(y) = k$, $\mathcal{M}(x,y)$ is a (possibly empty) compact ($k-1$)-dimensional $\langle k-1\rangle$-manifold; and
  \item The composition maps combine to form a diffeomorphism of $\langle k-2\rangle$-manifolds
    \begin{equation}\label{composition}
      \coprod_{\substack{y\in \Ob(\mathscr{C})\backslash \{x,z\}\\ \gr(y)-\gr(x) = i}} \mathcal{M}(y,z)\times \mathcal{M}(x,y) \cong \partial_i\mathcal{M}(x,z).
    \end{equation}
  \end{enumerate}
  For any flow category $\mathscr{C}$, define $\Sigma^k\mathscr{C}$ to be the flow category obtained by increasing the gradings of $\mathscr{C}$ up by $k$.
\end{definition}
\begin{definition}
  Let $\mathscr{C}$ be a flow category. For $x,y\in \Ob(\mathscr{C})$, we define $\widehat{A}_{x,y}$ to be the (possibly empty) set of path components of $\mathcal{M}(x,y)$. The composition maps in (\ref{composition}) descend to composition maps $\widehat{A}_{y,z}\times \widehat{A}_{x,y}\to \widehat{A}_{x,z}$, which we denote by $(\gamma,\xi)\mapsto \gamma\circ \xi$.
\end{definition}
We generalize the notion of a flow category slightly by introducing ($\mathbb{F}_2$-valued) signs on components of the moduli spaces $\mathcal{M}(x,y)$:
\begin{definition}\label{signed flow category}
  A \textit{signed flow category} $\mathscr{C}$ is a flow category $\mathscr{C}$ equipped with a sign map $\sigma: \coprod_{x,y\in \Ob(\mathscr{C})}\widehat{A}_{x,y}\to \mathbb{F}_2$ such that $\sigma(\gamma\circ \xi) = \sigma(\gamma)+\sigma(\xi)$ for all $\gamma \in \widehat{A}_{y,z}$, $\xi\in \widehat{A}_{x,y}$. We call $\sigma$ a \textit{sign map}. Every (unsigned) flow category $\mathscr{C}'$ comes with a trivial sign map $\sigma_0$: simply define $\sigma_0(\gamma)=0$ for all $\gamma\in \widehat{A}_{x,y}$, $x,y\in \Ob(\mathscr{C})$.\par
  We often say \textit{unsigned flow category} when referring to a flow category to avoid any confusion.
\end{definition}
\begin{notation}
  In the case that $\gr(y) = \gr(x) + 1$ and $p\in \mathcal{M}(y,x)$, we use the notation $\sigma(p) := \sigma(\{p\})$.
\end{notation}
\begin{definition}\label{morse chain complex}
  Given a (signed or unsigned) flow category $\mathscr{C}$, we can define a cochain complex with $\mathbb{F}_2$ coefficients $C_{\mathcal{M}}^*(\mathscr{C};\mathbb{F}_2)$  as follows:
  \begin{enumerate}[label = (K-\arabic*), ref = (K-\arabic*)]
  \item The set of generators is $\Ob(\mathscr{C})$, with a generator $x$ having cohomological grading $\gr(x)$.
    \item For $\gr(y)=\gr(x)+1$, the coefficient of $y$ in $\delta x$ is $\# (\mathcal{M}(y,x))$.
  \end{enumerate}
\end{definition}

\subsection{The cube flow category}
We introduce the cube flow category, which records the moduli space of ``flowlines'' between vertices in the cube.
\begin{definition}\label{cube flow definition}
  Fix an integer $n>0$, The objects of the cube flow category $\mathscr{C}_C(n)$ are the same as the objects of the cube category $\underline{2}^n$, that is, subsets of  $\{1,\ldots,n\}$. The grading on the objects is the same as the grading in $\underline{2}^n$ and the partial ordering $\geq$ is also inherited from $\underline{2}^n$.\par
  The space $\mathcal{M}(u,v)$ is defined to be empty unless $u>v$. In the case $u>v$ and $|u|-|v|=k>0$, we define $\mathcal{M}(u,v)=\Pi^{k-1}$, the $(k-1)$-dimensional permutohedron. The composition map $\mathcal{M}(v,w)\times\mathcal{M}(u,v)$ is defined as follows: Assume $u>v>w$ and $|u|-|v| = k$, $|v|-|w|= l$ Let $u\backslash w = \{a_1,\ldots, a_{k+l}\}$, where $a_1<\ldots<a_{k+l}$. Let $S$ be the set of $s\in \{1,\ldots,k+l\}$ satisfying $a_{s}\in v$. By Lemma \ref{face identification}, there is a corresponding facet $F_S\subset \Pi^{k+l-1}=\mathcal{M}(u,w)$, and we define the composition by
  \[
    \mathcal{M}(v,w)\times\mathcal{M}(u,v) \xhookrightarrow{f_S^{-1}} F_S\hookrightarrow \mathcal{M}(u,w).
  \]
\end{definition}
\begin{lemma}
  Definition \ref{cube flow definition} defines a flow category.
\end{lemma}
\begin{proof}
  See \cite{MR4153651}, Lemma 3.17.
\end{proof}

\subsection{Signed cubical flow categories}
\begin{definition}[\cite{MR4153651}, Definition 3.21]
  A \textit{cubical flow category} is a flow category $\mathscr{C}$ equipped with a grading-preserving functor $\mathfrak{f}:\Sigma^k\mathscr{C}\to \mathscr{C}_{C}(n)$ for some $k\in \mathbb{Z}$, $n\in \mathbb{N}$ so that for each $x,y\in \Ob(\mathscr{C})$, $\mathfrak{f}: \mathcal{M}(x,y)\to \mathcal{M}(\mathfrak{f}(x),\mathfrak{f}(y))$ is a (trivial) covering map.\par
  Note that if $(\mathscr{C},\mathfrak{f})$ is a cubical flow category, then $\mathcal{M}(x,y)$ can only be nonempty ($x>y$) if $\mathfrak{f}(x)>\mathfrak{f}(y)$, and in this case, must be a (possibly empty) disjoint union of permutohedra.
\end{definition}
\begin{definition}
  A \textit{signed cubical flow category} $(\mathscr{C},\mathfrak{f},\sigma)$ is a cubical flow category $(\mathscr{C},\mathfrak{f})$ equipped with the additional structure of a signed flow category, that is, a sign map $\sigma$.
\end{definition}
The cube category $\mathscr{C}_C(n)$ is not a signed cubical flow category, so there should be no expectation that the sign map $\sigma$ satisfies any sort of compatibility condition with respect to the covering map $\mathfrak{f}: \mathscr{C}\to \mathscr{C}_C(n)$.
\begin{example}
  The Khovanov flow category $\mathscr{C}_K(L)$ (\cite{MR3230817}, Definition 5.3) associates to a link diagram $L$ with $n$ crossings a cubical flow category $\mathscr{C}_K(L)$. For any $v\in \underline{2}^n$, the subset of $\Ob(\mathscr{C}_K(L))$ that maps to $v$ are precisely the Khovanov generators $x$ that lie over the $v$-resolution of $L$.\par
  For any $u,v\in \underline{2}^n$ with $|u|-|v| = 1$, and any $x\in \mathfrak{f}^{-1}(u)=F(u)$, $y\in \mathfrak{f}^{-1}(v)=F(v)$, the moduli space is
  \[
    \mathscr{C}_K(L)(x,y)=
    \begin{cases*}
      \{\text{pt}_{x,y}\} & if $x$ appears in $\delta_{Kh}(y)$\\
      \emptyset & otherwise.\\
    \end{cases*}
  \] 
\end{example}
The $1$-dimensional moduli spaces takes some work to construct, and the higher-dimensional spaces are defined inductively. The $\mathbb{F}_2$-coefficient Khovanov chain complex $KC^*(L;\mathbb{F}_2)$ is canonically identified with $C_{\mathcal{M}}^*(\mathscr{C}_K(L);\mathbb{F}_2)$.
\begin{example}\label{odd khovanov flow cat}
  The odd Khovanov flow category $\mathscr{C}_{K,o}(L)$ is a signed flow category, which as a flow category, equals $\mathscr{C}_K(L)$ precisely. Now $\mathscr{C}_{K,o}(L)$ has a sign map $\sigma$ which arises from the odd Khovanov functor $\mathfrak{F}_o: (\underline{2}^n)^{\text{op}}\to \mathbb{Z}\text{--Mod}$ in \cite{MR4078823}, Section 5.1. Namely, for each edge $u>_1 v$ in $\underline{2}^n$ and Khovanov generator $y$ with $\mathfrak{f}(y) = v$, we write
  \[
    \mathfrak{F}_o(\phi^{\text{op}}_{v,u})(y) = \sum_{x\in F_o(u)}\epsilon_{x,y}x,
  \]
  where by definition of $\mathfrak{F}_o$, $\epsilon_{x,y}\in\{-1,0,1\}$ for each $x\in F_o(u)$, and more specifically, $\epsilon_{x,y}\in\{-1,1\}$ for each $x\in F_o(u)$ with $\mathscr{C}_{K,o}(L)(x,y)\neq \emptyset$. We impose the identity $\epsilon_{x,y} = (-1)^{\sigma_{x,y}(\{\text{pt}_{x,y}\})}$, which defines $\sigma_{x,y}$ for $\gr(x)-\gr(y)=1$. For sequences $x_m,\ldots x_0$ such that $\gr(x_i)-\gr(x_{i-1})=1$, $\mathcal{M}(x_i,x_{i-1}) = \{\text{pt}_{x_i,x_{i-1}}\}$, we define
  \[
    \sigma(\{\text{pt}_{x_m,x_{m-1}}\}\circ\ldots\circ \{\text{pt}_{x_1,x_0}\}):= \sigma(\{\text{pt}_{x_m,x_{m-1}}\})+\ldots + \sigma(\{\text{pt}_{x_1,x_0}\}).
  \]
  This definition allows us to fully extend $\sigma$ to a sign map on $\mathscr{C}_{K,o}(L)$.
\end{example}
\begin{definition}
  There is also the notion of a sign assignment on a cubical flow category $\mathscr{C}$. A \textit{cubical sign assignment} $\widetilde{S}$ is a function
  \[
    \widetilde{S}:\coprod_{\substack{x,y\\ \gr(y) = \gr(x)+1}}\mathcal{M}(y,x)\to \mathbb{F}_2
  \]
  from the set of $0$-dimensional moduli spaces to $\mathbb{F}_2$ such that if $I$ is an interval in a $2$D moduli space $\mathcal{M}(z,x)$ and $\partial I = \{p\circ q,p'\circ q'\}$, then
  \begin{equation}\label{sign criterion}
    \widetilde{S}(p)+ \widetilde{S}(q) + \widetilde{S}(p') + \widetilde{S}(q') = 1.
  \end{equation}
  Note how this notion of a sign assignment is a generalization of the sign assignment on the cube category (see Definition \ref{cube sign assignment}). Indeed, if we identify edges $v>_1 u$ of $\underline{2}^n$ with the corresponding (unique) morphisms $p\in \mathcal{M}_{\mathscr{C}_C(n)}(v,u)$, the criterion $\widetilde{s}_{u,v_1}+\widetilde{s}_{u,v_2}+\widetilde{s}_{v_1,w}+\widetilde{s}_{v_2,w} = 1$ of Definition \ref{cube sign assignment} is equivalent to the criterion (\ref{sign criterion}) for the cube flow category.
\end{definition}

\begin{definition}\label{canonical sign assignment}
  Let $\mathfrak{f}:\mathscr{C}\to \mathscr{C}_C(n)$ be a signed cubical flow category. We define the \textit{standard cubical sign assignment} $S$ as the pullback of the standard sign assignment $s$ along $\mathfrak{f}$, plus the sign map $\sigma$. Namely, if $p \in \mathfrak{f}^{-1}(\mathcal{M}(v,u))$ for $v>_1 u$, we define  $S(p) := s(v,u) + \sigma(p)$. If $\mathscr{C}$ is an unsigned cubical flow category, we simply omit the $\sigma$-component, defining $S(p) := s(v,u)$. Similarly, we define the \textit{cubical index assignment} by $S_{\mathbb{Z}}(p) := s_{\mathbb{Z}}(v,u)$.\par
  In practice, we simply say \textit{standard sign assignment} and \textit{index assignment} if the context is clear.
\end{definition}

\begin{example}\label{Z morse chain complex}
  Given a cubical flow category $\mathfrak{f}: \mathscr{C}\to \mathscr{C}_{C}(n)$ with sign map $\sigma$, we define a cochain complex with $\mathbb{Z}$ coefficients $C_{\mathcal{M}}^*(\mathscr{C};\mathbb{Z})$  as follows:
  \begin{enumerate}[label = (K'-\arabic*), ref = (K'-\arabic*)]
  \item The set of generators is $\Ob(\mathscr{C})$, with a generator $x$ having cohomological grading $\gr(x)$.
  \item The differential is defined on generators $x$ by
    \begin{equation}\label{complex coeff}
      dx = \sum_{\substack{p\in \mathcal{M}(y,x)\\ \gr(y) = \gr(x) + 1}}(-1)^{S(p)}y,
    \end{equation}
  \end{enumerate}
  where $S$ is the standard cubical sign assignment defined in Example \ref{canonical sign assignment}.
\end{example}
\begin{remark}
  Note that given the odd Khovanov flow category $\mathscr{C}_{K,o}(L)$ from Example \ref{odd khovanov flow cat}, the cochain complex $C_{\mathcal{M}}^*(\mathscr{C}_{K,o}(L);\mathbb{Z})$ recovers the odd Khovanov chain complex $Kc_o^*(L)$ (see \cite{MR3071132}). If we use the (unsigned) Khovanov flow category $\mathscr{C}_{K}(L)$, the cochain complex $C_{\mathcal{M}}^*(\mathscr{C}_{K}(L);\mathbb{Z})$ recovers the even Khovanov chain complex.
\end{remark}

\subsection{Signed cubical flow categories are functors from the cube to the signed Burnside category}
The construction of the Odd Khovanov homotopy type in \cite{MR4078823} does not rely on cubical flow categories, but instead unitary, lax functors $\underline{2}^n\to \mathscr{B}_\sigma$ from the cube category to the signed Burnside category. We want to show that these objects are equivalent in study, since we want to show that our definition of the odd Khovanov homotopy type is equivalent to the construction in \cite{MR4078823}. The following constructions that we define follow \cite{MR4153651}, Section 4.3, with the additional detail of adding sign maps:
\begin{construction}\label{flow category to functor}
  Fix a signed cubical flow category $\mathfrak{f}:\mathscr{C}\to \mathscr{C}_C(n)$ with a sign map $\sigma: \coprod_{x,y}\widehat{A}_{x,y}\to \mathbb{F}_2$. We construct a unitary, lax $2$-functor $F: \underline{2}^n\to \mathscr{B}_\sigma$ from the cube to the signed Burnside category. Namely, we need to define the sets $X_u:=F(u)$, the correspondences $(A_{u,w},s_{A_{u,v}},t_{A_{u,v}},\sigma_{A_{u,v}}):=F(\phi_{u,v})$, and the isomorphisms of correspondences $F_{u,v,w}:A_{v,w}\times_{X_v}A_{u,v}\to A_{u,w}$.
  \begin{itemize}
  \item For $u\in \underline{2}^n$, define $X_u:= \mathfrak{f}^{-1}(u)$.
  \item For $u,v\in \underline{2}^n$, $u>v$, define
    \[
      A_{u,v}:= \coprod_{\substack{x\in X_u \\ y\in X_v}}\widehat{A}_{x,y}.
    \]
    The sources $s_{A_{u,v}}$ and targets $t_{A_{u,v}}$, when restricted to the subset $\widehat{A}_{x,y}$, are simply constant maps to $x$ and $y$ respectively. Similarly, $\sigma_{A_{u,v}}:=(-1)^{\sigma}$.
  \item For $\gamma\in \widehat{A}_{y,z}\subseteq A_{v,w}$, $\xi\in \widehat{A}_{x,y}\subseteq A_{u,v}$, define $F_{u,v,w}(\gamma,\xi) = \gamma\circ \xi$.
  \end{itemize}
\end{construction}
\begin{construction}\label{functor to flow category}
  Fix a strictly unitary lax $2$-functor $\underline{2}^n\to \mathscr{B}_\sigma$. We build a signed cubical flow category $\mathfrak{f}:\mathscr{C}\to \mathscr{C}_C(n)$ with a sign map $\sigma$.
  \begin{itemize}
  \item $\Ob(\mathscr{C}) = \coprod_{u\in \underline{2}^n}F(u)$. For $x\in F(u)$, we define $\mathfrak{f}(x)=u$.
  \item Let $x,y\in \Ob(\mathscr{C})$, with $\mathfrak{f}(x) = u$, $\mathfrak{f}(y) = v$. Consider the subset $\widehat{B}_{x,y}=s^{-1}(x)\cap t^{-1}(y)\subseteq A_{u,v}=F(\phi_{u,v})$. For $x,y\in \Ob(\mathscr{C})$, we define
    \[
      \mathcal{M}(x,y)=\widehat{B}_{x,y}\times \mathcal{M}_{\mathscr{C}_C(n)}(u,v).
    \]
    As usual, $\widehat{A}_{x,y}$ denotes the set of components of $\mathcal{M}(x,y)$.
  \item For a component $\gamma=\{p\}\times \mathcal{M}_{\mathscr{C}_C(n)}(u,v)\in \widehat{A}_{x,y}$, define $\sigma(\gamma)$ implicitly by $(-1)^{\sigma(\gamma)} = \sigma_{A_{u,v}}(p)$.
  \item For $\mathfrak{f}(x) = u$, $\mathfrak{f}(y) = v$, $\mathfrak{f}(z) = w$, we define the composition map
    \[
      \circ:\left(\widehat{B}_{y,z}\times \mathcal{M}_{\mathscr{C}_C(n)}(v,w)\right)\times\left(\widehat{B}_{x,y}\times \mathcal{M}_{\mathscr{C}_C(n)}(u,v)\right)\to \widehat{B}_{x,z}\times \mathcal{M}_{\mathscr{C}_C(n)}(u,w)
    \]
    to be $F_{u,v,w}$ on the $\widehat{B}$ factors and $\circ_{\mathscr{C}_C(n)}$ on the $\mathcal{M}$ factors.
  \end{itemize}
\end{construction}
\begin{lemma}\label{equivalent constructions}
  Construction \ref{flow category to functor} defines a unitary, lax $2$-functor $F: \underline{2}^n\to \mathscr{B}_\sigma$. Similarly, Construction \ref{functor to flow category} defines a signed flow category $\mathfrak{f}:\mathscr{C}\to \mathscr{C}_C(n)$ with a sign map $\sigma$. 
\end{lemma}
\begin{proof}
  We certainly have from \cite{MR4153651}, Lemma 4.18 that forgetting the sign maps $\sigma_{A_{u,v}}$, $F$ is a unitary, lax $2$-functor $\underline{2}^n\to \mathscr{B}$. The compatibility of the sign maps $\sigma_{A_{u,v}}$ under composition follows from the naturality of the sign map $\sigma$.\par
  We also have from \cite{MR4153651}, Lemma 4.20 that Construction \ref{functor to flow category} defines a flow category. And similarly, the compatibility of $\sigma$ under composition follows from the naturality of the sign maps $\sigma_{A_{u,v}}$.
\end{proof}
\subsection{Cubical neat embeddings}
Fix a cube flow category $\mathscr{C}_C(n)$ and fix a tuple $\mathbf{d} = (d_0,\ldots,d_{n-1})\in \mathbb{N}^n$ and a real number $R>0$ (we think of $R$ as a ``large'' number). For any $u>v$ in $\Ob(\mathscr{C}_C(n))= \underline{2}^n$, define
\[
  E_{u,v} = \left[\prod_{i=|v|}^{|u|-1}(-R,R)^{d_i}\right]\times \mathcal{M}_{\mathscr{C}_C(n)}(u,v).
\]
We can think of $E_{u,v}$ as a ``thickened'' moduli space. For $u>v>w$ in $\Ob(\mathscr{C}_C(n))$, there is a map $E_{v,w}\times E_{u,v}\to E_{u,w}$ given by:
\begingroup
  \allowdisplaybreaks
\begin{align*}
  E_{v,w}\times E_{u,v} &= \prod_{i=|w|}^{|v|-1}(-R,R)^{d_i}\times \mathcal{M}_{\mathscr{C}_C(n)}(v,w)\times \prod_{i=|v|}^{|u|-1}(-R,R)^{d_i}\times \mathcal{M}_{\mathscr{C}_C(n)}(u,v)\\
                        &\cong \prod_{i=|w|}^{|u|-1}(-R,R)^{d_i}\times \mathcal{M}_{\mathscr{C}_C(n)}(v,w)\times \mathcal{M}_{\mathscr{C}_C(n)}(u,v)\\
                        &\xhookrightarrow{\text{Id}\times\circ} \prod_{i=|w|}^{|u|-1}(-R,R)^{d_i}\times \mathcal{M}_{\mathscr{C}_C(n)}(u,w).
\end{align*}
\endgroup
Now we discuss cubical neat embeddings: a way to fit the moduli spaces $\mathcal{M}(x,y)$ into our thickened moduli spaces $E_{u,v}$.
\begin{definition}\label{cubical neat embedding def}
  A \textit{cubical neat embedding} $\iota$ of a (signed or unsigned) flow category $\mathscr{C}$ relative to a tuple $\mathbf{d} = (d_0,\ldots,d_{n-1})\in \mathbb{N}^n$ consists of neat embeddings $\iota_{x,y}:\mathcal{M}_{\mathscr{C}}(x,y)\to E_{\mathfrak{f}(x),\mathfrak{f}(y)}$ satisfying:
    \begin{enumerate}[label = (\arabic*), ref = \arabic*]
    \item For each $x,y\in \Ob(\mathscr{C})$, the following diagram commutes:
      \begin{equation*}
        \begin{tikzpicture}[scale=2]
  \path[draw, ->, shorten <=1.0cm,shorten >=0.7cm, thick] (0,1) -- (2,1) node [midway, label={[label distance=-0.2cm]90:$\iota_{x,y}$}] {};
  \path[draw, ->>, shorten <=0.3cm,shorten >=0.7cm, thick] (2,1) -- (2,0) {};
  \path[draw, ->, shorten <=0.7cm,shorten >=0.5cm, thick] (0,1) -- (2,0) node [midway, label={[label distance=-0.2cm]-125:$f_{x,y}$}] {};
  \node at (0, 1) {$\mathcal{M}_\mathscr{C}(x,y)$};
  \node at (2,1) {$E_{\mathfrak{f}(x),\mathfrak{f}(y)}$};
  \node at (2,0) {$\mathcal{M}_{\mathscr{C}_C(n)}(\mathfrak{f}(x),\mathfrak{f}(y))$)};
   \end{tikzpicture}
      \end{equation*}
      
    \label{lift of moduli space}
    \item For each $u,v\in\Ob(\mathscr{C})$, the induced map
    \begin{equation*}
      \coprod_{\substack{x,y\\\mathfrak{f}(x)=u,\mathfrak{f}(y)=v}} \iota_{x,y}:\coprod_{\substack{x,y\\\mathfrak{f}(x)=u,\mathfrak{f}(y)=v}}\mathcal{M}_\mathscr{C}(x,y)\to E_{u,v}
    \end{equation*}
    is a neat embedding \label{embedding criterion}
  \item For each $x,y,z\in\Ob(\mathscr{C})$, the following diagram commutes:
    \begin{equation*}
      \begin{tikzpicture}[scale=2]
  \path[draw, ->, shorten <=2.1cm,shorten >=1.0cm, thick] (0,1) -- (3,1) node [midway, label={[label distance=-0.2cm]90:$\cdot \circ \cdot$}] {};
  \path[draw, ->, shorten <=1.9cm,shorten >=1.0cm, thick] (0,0) -- (3,0) node [midway, label={[label distance=-0.2cm]90:$\cdot \circ\cdot$}] {};
  \path[draw, ->, shorten <=0.3cm,shorten >=0.3cm, thick] (0,1) -- (0,0);
  \path[draw, ->, shorten <=0.3cm,shorten >=0.3cm, thick] (3,1) -- (3,0) {};
     \node at (0, 1) {$\mathcal{M}_\mathscr{C}(y,z)\times\mathcal{M}_\mathscr{C}(x,y)$};
   \node at (3, 1) {$\mathcal{M}_\mathscr{C}(x,z)$};
   \node at (0, 0) {$E_{\mathfrak{f}(y),\mathfrak{f}(z)}\times E_{\mathfrak{f}(x),\mathfrak{f}(y)}$};
   \node at (3, 0) {$E_{\mathfrak{f}(x),\mathfrak{f}(z)}$};
   \end{tikzpicture}
    \end{equation*}
    
  \end{enumerate}
\end{definition}
For our construction in Section \ref{cubical realization}, we need to fit a ``slightly thickened'' version of the moduli spaces $\mathcal{M}(x,y)$ into our ``greatly thickened'' moduli spaces $E_{u,v}$.
\begin{definition}
  We can, in fact, extend the embedding $\iota_{x,y}$ to an embedding
  \begin{align*}
    \overline{\iota}_{x,y}:  \left[\prod_{i=|v|}^{|u|-1}(-\epsilon,\epsilon)^{d_i}\right]\times \mathcal{M}_{\mathscr{C}}(x,y)& \to  \left[\prod_{i=|v|}^{|u|-1}(-R,R)^{d_i}\right]\times \mathcal{M}_{\mathscr{C}_C(n)}(u,v),\\
    (a,p)\mapsto \iota_{x,y}(p)+(a,0)&\quad \text{for }a\in \prod_{i=|v|}^{|u|-1}(-\epsilon,\epsilon)^{d_i}\text{, } p\in \mathcal{M}_{\mathscr{C}}(x,y)
  \end{align*}
\end{definition}
for some small $\epsilon>0$. By Condition \ref{lift of moduli space}, $\iota(\mathcal{M}(x,y))$ is transverse to the fibers $\left[\prod_{i=|v|}^{|u|-1}(-R,R)^{d_i}\right]\times \{p\}$. So for small enough $\epsilon$, $\overline{\iota}_{x,y}$ is indeed an embedding.
\subsection{The signed cubical realization}\label{cubical realization}
\begin{definition}\label{cubical realization def}
Fix a cubical neat embedding $\iota$ of a signed cubical flow category $(\mathscr{C},\mathfrak{f}:\Sigma^N\mathscr{C}\to \mathscr{C}_C(n),\sigma)$ relative to a tuple $\mathbf{d}=(d_0,\ldots,d_{n-1})$. From $\mathscr{C},\iota$, we build a CW complex $||\mathscr{C}||=||\mathscr{C}||_{\iota}$ satisfying:
\begin{enumerate}[label = (C-\arabic*), ref = (C-\arabic*)]
  \item The CW complex $||\mathscr{C}||$ has one cell for each $x\in \Ob(\mathscr{C})$. Letting $u$ denote $\mathfrak{f}(x)$, this cell is given by
	\[
	  \mathcal{C}(x)=\prod_{i=0}^{|u|-1}[-R,R]^{d_i}\times \prod_{i=|u|}^{n-1}[-\epsilon,\epsilon]^{d_i}\times J \times \widetilde{\mathcal{M}}_{\mathscr{C}_C(n)}(u,\overline{0}),
	\]
        where $J$ is the usual unit interval $[0,1]$, and
        \[
          \widetilde{\mathcal{M}}_{\mathscr{C}_C(n)}(u,\overline{0}) =
          \begin{cases*}
            [0,1]\times \mathcal{M}_{\mathscr{C}_C(n)}(u,\overline{0}) & if $u \neq \overline{0}$\\
            \{0\} & if $u = \overline{0}$.
          \end{cases*}
        \]
    \item For any $x,y\in \Ob(\mathscr{C})$ with $\mathfrak{f}(x)=u>\mathfrak{f}(y)=v$, the cubical neat embedding $\iota$ gives an embedding
      \begingroup
      \allowdisplaybreaks
      \begin{align*}
        \jmath_\gamma:& \mathcal{C}(y) \times \gamma\\
        &\xhookrightarrow{\tau^{\sigma(\gamma)}\times\text{Id}}\mathcal{C}'(y)\times \gamma\\
        &\subseteq \mathcal{C}(y) \times \mathcal{M}_{\mathscr{C}}(x,y)\\
        &=\prod_{i=0}^{|v|-1}[-R,R]^{d_i}\times \prod_{i=|v|}^{n-1}[-\epsilon,\epsilon]^{d_i}\times J \times \widetilde{\mathcal{M}}_{\mathscr{C}_C(n)}(v,\overline{0})\times \mathcal{M}_\mathscr{C}(x,y)\\
        &\cong \prod_{i=0}^{|v|-1}[-R,R]^{d_i}\times \prod_{i=|u|}^{n-1}[-\epsilon,\epsilon]^{d_i}\times J\times\widetilde{\mathcal{M}}_{\mathscr{C}_C(n)}(v,\overline{0})\times \left(  \prod_{i=|v|}^{|u|-1}[-\epsilon,\epsilon]^{d_i} \times \mathcal{M}_\mathscr{C}(x,y)\right)\\
        &\xhookrightarrow{\text{Id} \times \tilde{\iota}_{x,y}}\prod_{i=0}^{|v|-1}[-R,R]^{d_i}\times \prod_{i=|u|}^{n-1}[-\epsilon,\epsilon]^{d_i}\times J\times \widetilde{\mathcal{M}}_{\mathscr{C}_C(n)}(v,\overline{0})\times \left(  \prod_{i=|v|}^{|u|-1}[-R,R]^{d_i} \times \mathcal{M}_\mathscr{C}(u,v)\right)\\
        &\cong \prod_{i=0}^{|u|-1}[-R,R]^{d_i}\times \prod_{i=|u|}^{n-1}[-\epsilon,\epsilon]^{d_i} \times J\times \widetilde{\mathcal{M}}_{\mathscr{C}_C(n)}(v,\overline{0})\times \mathcal{M}_{\mathscr{C}_C(n)}(u,v)\\
        &\xhookrightarrow{} \prod_{i=0}^{|u|-1}[-R,R]^{d_i}\times \prod_{i=|u|}^{n-1}[-\epsilon,\epsilon]^{d_i}\times J\times  \partial(\widetilde{\mathcal{M}}_{\mathscr{C}_C(n)}(u,\overline{0}))\\
        &\subset \partial\mathcal{C}(x),
	\end{align*}
        \endgroup
	where $\gamma \in \widehat{A}_{x,y}$, and $\tau: \mathcal{C}(y)\to \mathcal{C}(y)$ denotes a flip $t\mapsto -t$ in the $J$-factor, and $\text{Id}$ on all other factors (so $\tau^{\sigma(\gamma)}=\text{Id}$ if and only if $\sigma(\gamma)=0$). We call the image of this map $\mathcal{C}_y(x)$.
  \item \label{attaching realization} The attaching map for $\mathcal{C}(x)$ sends the subspace $\mathcal{C}_y(x)\cong \mathcal{C}(y)\times\mathcal{M}_\mathscr{C}(x,y)$ of the boundary $\partial \mathcal{C}'(x)$ by the projection map to $\mathcal{C}(y)$, and sends the complement of $\bigcup_y \mathcal{C}_y(x)$ to the basepoint.
  \end{enumerate}
  The \textit{signed cubical realization} $\mathcal{X}_1(\mathscr{C})$ is defined to be the formal desuspension
  \[
    \mathcal{X}_1(\mathscr{C}):= \Sigma^{-(N+|\mathbf{d}|+1)}||\mathscr{C}||,
  \]
  where $|\mathbf{d}|$ denotes $d_0+\ldots+d_{n-1}$. (The desuspension ensures that the gradings $\gr(x)$ of objects $x\in \Ob(\mathscr{C})$ agree with the dimensions of the corresponding ``cells'' in $\mathcal{X}_1(\mathscr{C})$.)
\end{definition}
\begin{remark}\label{sign importance}
  Our construction of $\mathcal{X}_1(\mathscr{C})$ appears similar to the cubical realization $\mathcal{X}(\mathscr{C})$ from \cite{MR4153651}, Section 3.7. The difference of our definition is our introduction of the $J$-factor in the cells $\mathcal{C}(x)$. The $J$-factor's role in the attaching maps $\partial\mathcal{C}(x)\to \mathcal{C}(y)$ is characterized by the potential flips $\tau^{\sigma(W)}$, which themselves depend on the signs $\sigma(W)$. Observe that if $\sigma(W)=0$ for all $W$, the $J$-factor simply acts like a suspension factor, and the signed cubical realization $\mathcal{X}_1(\mathscr{C})$ agrees with $\mathcal{X}(\mathscr{C})$.
\end{remark}
\begin{definition}
  On the other hand, given an unsigned cubical flow category $\mathscr{C}$, we can equip to $\mathscr{C}$ the trivial sign map $\sigma:=0$ as in Definition \ref{signed flow category}. We can then define the signed cubical realization $\mathcal{X}_1(\mathscr{C})$ of $\mathscr{C}$. We observe, as in Remark \ref{sign importance}, $\mathcal{X}_1(\mathscr{C})$ is stably equivalent to $\mathcal{X}(\mathscr{C})$.
\end{definition}
\begin{lemma}
  $C_{\mathcal{M}}^*(\mathscr{C};\mathbb{F}_2)$ is canonically isomorphic to $C_{\cell}^*(\mathcal{X}_1(\mathscr{C});\mathbb{F}_2)$, where generators $x\in \Ob(\mathscr{C})$ map to the generators $\mathcal{C}(x)$ of $C_{\cell}^*(\mathcal{X}_1(\mathscr{C});\mathbb{F}_2)$. In particular, $C_{\mathcal{M}}^*(\mathscr{C}_K(L);\mathbb{F}_2)\cong KC^*(L;\mathbb{F}_2)$ for any link diagram $L$.
\end{lemma}
\begin{proof}
  The one-to-one correspondence of generators is immediate from Definitions \ref{morse chain complex} and \ref{cubical realization def}. The fact that the differentials agree is immediate.
\end{proof}
\begin{remark}
  The space $\mathcal{X}_1(\mathscr{C})$ is equivalent to  $\mathcal{X}(\mathscr{C})$ for $\mathscr{C}$ unsigned, but the addition of the $J$-factor streamlines our computation of $\Sq^2$ on $\mathcal{X}_1(\mathscr{C})$. In the signed case, the $J$ factor will be essential.
\end{remark}
\subsection{The signed cubical realization is stably equivalent to the signed realization}
\begin{definition}
  We define a \textit{stable Burnside functor} $\Sigma^r F$ to be a pair $(F:\underline{2}^n\to \mathscr{B}_\sigma,r)$ of a functor $F:\underline{2}^n\to \mathscr{B}_\sigma$, and an integer $r$. Given such a stable Burnside functor we fix a $k$-dimensional spatial refinement $\widetilde{F}_k$, denote $\widetilde{F}_k^+$ for the extended diagram (see \cite{MR4153651}, Section 4.4 for a definition), and then define $||F||_k = \hocolim \widetilde{F}_k^+$ (see \cite{MR4153651}, Definition 4.9 for a definition). Define the realization $|\Sigma^rF|$ as the finite CW spectrum $(||F||_k,r-k)$.
\end{definition}

\begin{proposition}[\cite{MR4078823}]
  Let $F:\underline{2}^n\to \mathscr{B}_\sigma$ be a strictly unitary, lax $2$-functor and let $\widetilde{F}_k$ be a $k$-dimensional \textit{spatial refinement} of $F$. The realization $||F||_k$ carries a CW complex structure whose nonbasepoint cells correspond to the elements of the set $\coprod_{u\in\underline{2}^n}F(u)$. Further, the equivalences $\Sigma \hocolim \widetilde{F}_k^+\cong \hocolim \widetilde{F}_{k+1}^+$ can be chosen to be cellular, so $\hocolim \widetilde{F}^+$ inherits the structure of a CW spectrum.
\end{proposition}
\begin{proof}
  The proof generalizes without changes from \cite{MR4153651}.
\end{proof}
\begin{theorem}
  Let $(\mathscr{C},\mathfrak{f}:\Sigma^N\mathscr{C}\to\mathscr{C}_C(n),\sigma)$ be a signed cubical flow category, and let $F:\underline{2}^n\to \mathscr{B}_\sigma$ be its corresponding unitary, lax $2$-functor (Lemma \ref{equivalent constructions}). The signed cubical realization is stably homotopy equivalent to the signed realization $|\Sigma^{-N}F| = (||F||_k,-N-k)$.
\end{theorem}
The proof, which we shall outline, generalizes from \cite{MR4153651}, Theorem 8. 
\begin{proof}
  First  fix a cubical neat embedding $\iota$ of $\mathscr{C}$, where $\iota$ is relative $\mathbf{d}=(d_0,\ldots,d_{n-1})$, defining $||\mathscr{C}|| := ||\mathscr{C}||_\iota$. Now let $k = 1 + \sum_i d_i $ (the ``$1$'' accounts for the extra $J$-factor).
  \paragraph{\textbf{Step 1}: Build a spatial refinement $\widetilde{F}_k$ of $F$ and define $||F||_k$. }
  The cubical neat embedding $\iota$ determines the cells of $||\mathscr{C}||$, which are of the following form: if $u\in \Ob(\underline{2}^n)$ and $x\in F(u)$, then
  \[
    \mathcal{C}(x) =
    \begin{cases*}
      \prod_{i=0}^{|u|-1}[-R,R]^{d_i}\times \prod_{i=|u|}^{n-1}[-\epsilon,\epsilon]^{d_i}\times J\times [0,1]\times \mathcal{M}_{\mathscr{C}_C(n)}(u,\overline{0}) & if $u\neq \overline{0}$,\\
      \prod_{i=0}^{n-1}[-\epsilon,\epsilon]^{d_i}\times J\times \{0\} & if $u = \overline{0}$.
    \end{cases*}
  \]
  Now for $u\in \Ob(\underline{2}^n)$, we define $\widetilde{F}^+(u) = \bigvee_{x\in F(u)}B_x/\partial{B_x}$, where the box $B_x$ associated to $x$ is defined by
  \[
    B_x = J\times \prod_{i=0}^{|u|-1}[-R,R]^{d_i}\times \prod_{i=|u|}^{n-1}[-\epsilon,\epsilon]^{d_i}.
  \]
  For the remaining object $*\in \underline{2}^n_+$, simply define $\widetilde{F}^+(*) = \{\text{pt}\}$. Now we define our family of box maps using $\iota$. For $v>u$ in $\underline{2}^n$ and a chain of non-identity morphisms $v = v^0\xrightarrow{f_1}\ldots \xrightarrow{f_m}v^m=u$, we define the corresponding spatial refinement map
  \[
    \widetilde{F}^+_k(f_m,\ldots,f_1): [0,1]^{m-1}\times \widetilde{F}^+_k(v)\to \widetilde{F}^+_k(u),
  \]
  as follows: For all $\gamma\in F(\phi_{v,u})$, $\gamma\subset \mathcal{M}_{\mathscr{C}}(y,x)$, we define a map $e_\gamma: \mathcal{M}_{\mathscr{C}_C(n)}(v,u)\times B_x\to B_y$. We denote the box $B_\gamma = J\times \prod_{i=0}^{|u|-1}[-R,R]^{d_i}\times\prod_{i=|u|}^{n-1}[-\epsilon,\epsilon]^{d_i}$. Define an embedding $\imath_\gamma: \gamma\times B_\gamma\to B_y$ by the composition
  \begingroup
  \allowdisplaybreaks
  \begin{align*}
    \gamma\times B_\gamma &\xhookrightarrow{(\Id\times\tau^{\sigma(\gamma)})}\gamma\times B_\gamma\\
    &\subset \mathcal{M}_{\mathscr{C}}(y,x)\times B_\gamma\\
                                                       &=\mathcal{M}_{\mathscr{C}_C(n)}(y,x) \times J\times \prod_{i=0}^{|u|-1}[-R,R]^{d_i}\times\prod_{i=|u|}^{n-1}[-\epsilon,\epsilon]^{d_i}\\
                                                       & \cong J\times \prod_{i=0}^{|u|-1}[-R,R]^{d_i}\times  \left( \prod_{i=|u|}^{|v|-1}[-\epsilon,\epsilon]^{d_i}\times \mathcal{M}_{\mathscr{C}_C(n)}(y,x)\right) \times\prod_{i=|v|}^{n-1}[-\epsilon,\epsilon]^{d_i}\\
                                                       & \xhookrightarrow{(\Id,\overline{\iota}_{x,y},\Id)} J\times \prod_{i=0}^{|u|-1}[-R,R]^{d_i}\times  \left( \prod_{i=|u|}^{|v|-1}[-R,R]^{d_i}\times \mathcal{M}_{\mathscr{C}_C(n)}(v,u)\right) \times\prod_{i=|v|}^{n-1}[-\epsilon,\epsilon]^{d_i}\\
                                                       &\xtwoheadrightarrow{(\Id,\pi^R,\Id)} J\times \prod_{i=0}^{|u|-1}[-R,R]^{d_i}\times \prod_{i=|u|}^{|v|-1}[-R,R]^{d_i} \times\prod_{i=|v|}^{n-1}[-\epsilon,\epsilon]^{d_i}\\
                                                       &\cong J\times \prod_{i=0}^{|v|-1}[-R,R]^{d_i}\times \prod_{i=|v|}^{n-1}[-\epsilon,\epsilon]^{d_i} = B_y
  \end{align*}
  \endgroup
  and define $e_\gamma: \mathcal{M}_{\mathscr{C}_C(n)}(v,u) \times B_\gamma \to B_y$ by the composition
  \[
    \mathcal{M}_{\mathscr{C}_C(n)}(v,u) \times B_\gamma \xhookrightarrow{(\mathfrak{f}_\gamma^{-1},\Id)} \gamma \times B_\gamma \xhookrightarrow{\imath_\gamma} B_y.
  \]
  Consider the induced map
  \begin{equation}\label{box maps}
    \mathcal{M}_{\mathscr{C}_C(n)}(v,u) \times \coprod_{\substack{\gamma\in F(\phi_{v,u})\\ s(\gamma) = y}} B_\gamma \to B_y.
  \end{equation}
  Since $\iota$ is a cubical neat embedding (in particular, see Definition \ref{cubical neat embedding def}, (\ref{embedding criterion})), it follows that for any point $\text{pt}\in \mathcal{M}_{\mathscr{C}_C(n)}(v,u)$, the restriction $\{\text{pt}\}\times \coprod_{\gamma\in F(\phi_{v,u})\mid s(\gamma) = y}B_\gamma\to B_y$ is an inclusion of disjoint sub-boxes.\par
  The chain $v=v^0>\ldots > v^m = u$ corresponds to a subcube $[0,1]^{m-1}\subset \mathcal{M}_{\mathscr{C}_C(n)}(v,u)$ in the cubical complex structure of $\mathcal{M}_{\mathscr{C}_C(n)}(v,u)$ (see \cite{MR4153651}, Lemma 3.20). Restrict the map from (\ref{box maps}) to $[0,1]^{m-1}\times \coprod_{\substack{\gamma\in F(\phi_{v,u})\\ s(\gamma) = y}} B_\gamma$ to obtain our family $e$ of sub-boxes $\coprod_{\substack{\gamma\in F(\phi_{v,u})\\ s(\gamma) = y}} B_\gamma \subset B_y$.
  
  \paragraph{\textbf{Step 2}: Define a cellular map $||F||_k\to ||\mathscr{C}||$.}
  For any $u\in \Ob(\underline{2}^n)$ and any $x\in F(u)$, the cell associated to $x$ in $||F||_k$ is
  \[
    \mathcal{C}''(x) =
    \begin{cases*}
      \mathcal{M}_{\mathscr{C}_C(n)}(u,\overline{0}) \times [0,2] \times B_x& if $u\neq \overline{0}$,\\
      \{0\} \times B_x & if $u = \overline{0}$.
    \end{cases*}
  \]
  \[
    \mathcal{C}(x) =
    \begin{cases*}
      \mathcal{M}_{\mathscr{C}_C(n)}(u,\overline{0}) \times [0,1] \times B_x & if $u\neq \overline{0}$,\\
      \{0\} \times B_x & if $u = \overline{0}$.
    \end{cases*}
  \]
  Map $\mathcal{C}''(x)\to \mathcal{C}'(x)$ by the quotient map $[0,2]\to [0,2]/[1,2]\cong [0,1]$, and the identity map on all other factors. This map is degree $\pm 1$ on each cell, and is thus a stable equivalence as long as our map is well-defined, that is, compatible with the attaching maps of $||F||_k$ and $||\mathscr{C}||$.
  \paragraph{Step 3: Prove that our cellular map is well-defined.}
  Suppose $p\in \partial\mathcal{C}''(y)$, $p\in N_v\times B_y$, and $p$ is glued to some point $q\in \mathcal{C}''(x)$ under the CW complex attaching map of $\hocolim(\widetilde{F}^+_k)$. We prove that $p$, now viewed as a point in $\mathcal{C}(y)$, is glued to $q$, where now $q$ is viewed as a point in $\mathcal{C}(y)$. Just as in \par
  We first outline how $p$ is glued in $\hocolim(\widetilde{F}^+_k)$. Let $p = (p_1,p_2)$, where $p_1\in N_v$, $p_2\in B_y$, and also assume $p_1$ lies in the cube $[0,1]^m$ associated to the chain $v=v^0>\ldots > v^m$; we say $p_1$ has coordinates $(p_{1,1},\ldots,p_{1,l})$. Now suppose (since $p_1\in N_v$) that $p_{1,l}=0$. Denote $\widetilde{F}_k(\phi_{v^{l-1},v^l},\ldots,\phi_{v^1,v^0})$ by $\psi$. The point $p$ is glued to a point
  \[
    ((p_{1,l+1},\ldots, p_{1,m}),\psi((p_{1,1},\ldots,p_{1,l-1}),p_2)),
  \]
  which we now denote as $q$. To lighten up the notation, we introduce the terms $q_1 = (p_{1,l+1},\ldots, p_{1,m})\in\widetilde{M}_{v^l}$, $q_2 = \psi((p_{1,1},\ldots,p_{1,l-1}),p_2)\in B_x$, $q' = (p_{1,1},\ldots,p_{1,l-1})\in M_{v,v^l}$. We include these terms in the below equation for reference:
  \[
    q = \bigl(\overbrace{(p_{1,l+1},\ldots, p_{1,m})}^{q_1\in \widetilde{M}_{v^l}},\overbrace{\psi(\underbrace{(p_{1,1},\ldots,p_{1,l-1})}_{q'\in M_{v,v^l}},p_2)}^{q_2\in B_x}\bigr).
  \]
  The following equalities show that $p$ and $q$ (both written as points in $\mathcal{C}''(y)$, $\mathcal{C}''(x)$ respectively) are identified under the gluing map in $||\mathscr{C}||$:
  \begin{align*}
    \jmath_{\gamma}\left(\mathfrak{f}_\gamma^{-1}(q'),q\right) &= \jmath_{\gamma}\left(\mathfrak{f}_\gamma^{-1}(q'),q_1,q_2\right)\\
                                                               &=\left(q_1\circ q',\imath_\gamma\left(\mathfrak{f}_\gamma^{-1}(q'),q_2\right)\right)\\
    &=\left(q_1\circ q',e_\gamma(q',q_2)\right)\\
                                                               &= \left((p_{1,1},\ldots, p_{1,l-1},0,p_{1,l+1},\ldots, p_{1,m}),p_2\right)\\
    &=(p_1,p_2)=p.
  \end{align*}
  The second to last equality is justified in Figure \ref{box map}.
  \begin{figure}
    \centering
    \def\svgscale{0.8}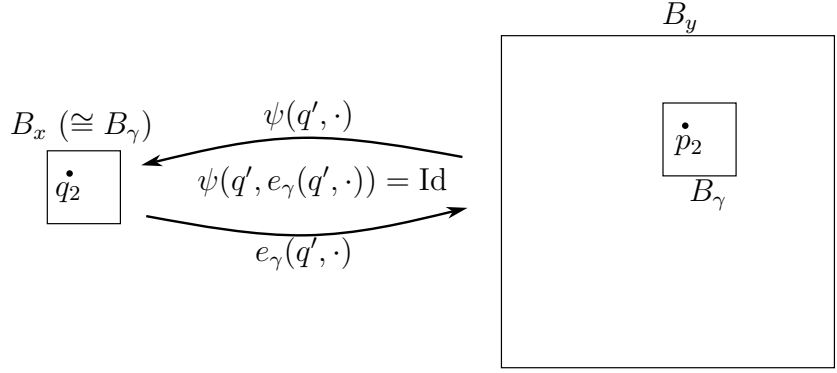
    \caption{We have $e_\gamma(q',q_2) = p_2$ because $\psi(q',p_2) = q_2$.}
    \label{box map}
  \end{figure}
\end{proof}

\subsection{Truncating the cubical realization to three adjacent dimensions}\label{truncation}
We denote the cubical realization $||\mathscr{C}||$ of a flow category $\mathscr{C}$ by $Y$, and we let $M = |\mathbf{d}|+1+l+N$. This choice of $M$ comes from the identity
\[
  \Sigma^{-l}C^*(\mathcal{X}(\mathscr{C});\mathbb{F}_2) = C_{\cell}^*(\Sigma^{-(|\mathbf{d}|+1+l+N)}Y;\mathbb{F}_2) = \Sigma^{-M} C_{\cell}^*(Y;\mathbb{F}_2).
\]
and the following fact that $\Sq^2: H^l(\mathcal{X}(\mathscr{C});\mathbb{F}_2)\to H^l(\mathcal{X}(\mathscr{C});\mathbb{F}_2)$ is precisely $\Sq^2: H^M(Y';\mathbb{F}_2)\to H^M(Y';\mathbb{F}_2)$. We will show that now all the cells of $Y$ are relevant to us, namely that we only care about the $M$, $(M+1)$, and $(M+2)$-dimensional cells of $Y$. These cells can be repackaged into another CW complex, which is the cubical realization of a ``truncated'' flow category.
\begin{definition}
  Define the \textit{truncated} signed cubical flow category $(\mathscr{C}',\mathfrak{f}',\sigma')$ as follows:
  \begin{itemize}
  \item $\mathscr{C}'$ is the full subcategory of $\mathscr{C}$ containing precisely the objects of grading $l,l+1,l+2$.
  \item The grading-preserving functor $\mathfrak{f}':\Sigma^N\mathscr{C}'\to \mathscr{C}_C(n)$ is the restriction of $\mathfrak{f}$ to $\mathscr{C}'$.
  \item Similarly, $\sigma'$ is the restriction of $\sigma$ to $\coprod_{x,y\in \Ob(\mathscr{C}')}\widehat{A}_{x,y}$.
  \end{itemize}
  By an abuse of notation, we identify $\mathfrak{f}'$ with $\mathfrak{f}$ and $\sigma'$ with $\sigma$.\par
  We define a cubical neat embedding $\iota'$ relative to the truncated tuple
  \[
    \mathbf{d}' := (0,\ldots,0,A,B,0,\ldots, 0), \qquad A = d_\kappa,\ B = d_{\kappa+1}.
  \]
  The embeddings $\iota'_{y,x}$ are defined exactly the same as the embeddings $\iota_{y,x}$ of $\mathscr{C}$. We define the \textit{truncated} cubical realization of $\mathscr{C}$ to be $||\mathscr{C}'||=||\mathscr{C}'||_{\iota}$, and we define $\mathcal{X}(\mathscr{C}') = \Sigma^{-(N+|\mathbf{d}+1)}||\mathscr{C}'||$. We let $\mathcal{C}'(x)$ denote the cell of $||\mathscr{C}'||$ corresponding to $x\in\Ob(\mathscr{C}')$.\par
\end{definition}
Objects of grading $l$ (resp.\ $l+1$, $l+2$), we often name $x$ (resp.\ $y$, $z$). In this spirit, the cells of $Y'$ are of the type
\begin{alignat*}{3}
	\mathcal{C}'(z)&=[-R,R]^A\times [-R,R]^B\times J\times \widetilde{\mathcal{M}}_{\mathscr{C}_C(n)}(w,\overline{0})&&\qquad \mathfrak{f}(z)=w,&&\quad |w|=\kappa+2\\[\jot]
	\mathcal{C}'(y)&=[-R,R]^A\times [-\epsilon,\epsilon]^B\times J\times \widetilde{\mathcal{M}}_{\mathscr{C}_C(n)}(v,\overline{0}) &&\qquad \mathfrak{f}(y)=v,&&\quad |v|=\kappa+1\\[\jot]
	\mathcal{C}'(x)&=[-\epsilon,\epsilon]^A\times [-\epsilon,\epsilon]^B\times J\times \widetilde{\mathcal{M}}_{\mathscr{C}_C(n)}(u,\overline{0}) &&\qquad \mathfrak{f}(x)=u,&&\quad |u|=\kappa,
\end{alignat*}
where we define $\kappa := l+N$, $A=d_\kappa$, $B=d_{\kappa+1}$. The cells are $Y'$ are of three consecutive dimensions; namely, if we define $m:= A+B+1+l+N$, then $\mathcal{C}(x)$, $\mathcal{C}(y)$, $\mathcal{C}(z)$ are $m$-dimensional, $m+1$-dimensional, and $m+2$-dimensional cells.\par
Now $Y^{(M+2)}/Y^{(M-1)}$ also consists of cells of three consecutive dimensions; the $M$, $(M+1)$, and $(M+2)$-dimensional cells of $Y$, which we write as
\begin{align*}
	\mathcal{C}(z)&=\prod_{i=0}^{(\kappa+2)-1}[-R,R]^{d_i}\times \prod_{i=\kappa+2}^{n-1}[-\epsilon,\epsilon]^{d_i} \times J\times \widetilde{\mathcal{M}}_{\mathscr{C}_C(n)}(w,\overline{0}) \qquad \mathfrak{f}(z)=w,\quad |w|=\kappa+2\\
	\mathcal{C}(y)&=\prod_{i=0}^{(\kappa+1)-1}[-R,R]^{d_i}\times \prod_{i=\kappa+1}^{n-1}[-\epsilon,\epsilon]^{d_i} \times J\times \widetilde{\mathcal{M}}_{\mathscr{C}_C(n)}(v,\overline{0}) \qquad \mathfrak{f}(y)=v,\quad |v|=\kappa+1\\
	\mathcal{C}(x)&=\prod_{i=0}^{\kappa-1}[-R,R]^{d_i}\times \prod_{i=\kappa}^{n-1}[-\epsilon,\epsilon]^{d_i} \times J\times \widetilde{\mathcal{M}}_{\mathscr{C}_C(n)}(u,\overline{0}) \qquad \mathfrak{f}(x)=u,\quad |u|=\kappa,
\end{align*}
We observe that 
\begin{proposition}\label{suspension identification}
  There is a CW structure on $\Sigma^{|\mathbf{d}|-|\mathbf{d}'|}Y'$, where the cells of $\Sigma^{|\mathbf{d}|-|\mathbf{d}'|}Y'$ corresponding to $x$, $y$, $z$ are copies of the cells $\mathcal{C}(x)$, $\mathcal{C}(y)$, $\mathcal{C}(z)$. In fact, identifying the cells of $\Sigma^{|\mathbf{d}|-|\mathbf{d}'|}Y'$ with their corresponding cells $\mathcal{C}(x)$, $\mathcal{C}(y)$, $\mathcal{C}(z)$ induces a homeomorphism
  \[
    \Sigma^{|\mathbf{d}|-|\mathbf{d}'|}Y' \cong Y^{(M+2)}/Y^{(M-1)}.
  \]
\end{proposition}
\begin{proof}
  The identification of cells respects the attaching maps, because the cubical neat embedding $\iota'$ of $\mathscr{C}'$ can be viewed as a ``restriction'' of the cubical neat embedding $\iota$. Therefore, the identification of CW complexes is well-defined.
\end{proof}

By making sure $Y:= ||\mathscr{C}||$ is defined witih  $d_\kappa,d_{\kappa+1}\geq 2$ if necessary, we can assume that $A,B >1$.
    $\mathcal{C}(x)$, $\mathcal{C}(y)$, $\mathcal{C}(z)$ are $m$-dimensional, $m+1$-dimensional, and $m+2$-dimensional cells, where $m:= A+B+1+l+N$. Note that $m>2$.
\begin{procedure}\label{computation idea}
  We can compute $\Sq^2: H^M(Y;\mathbb{F}_2)\to H^{M+2}(Y;\mathbb{F}_2)$ as follows: For any element $[\mathbf{c}] \in H^M(Y';\mathbb{F}_2)$, $\mathbf{c} = \sum_x \mu_x \cdot \mathcal{C}(x)$, take the corresponding cycle $\mathbf{c}' = \sum_x \mu_x \cdot \mathcal{C}'(x)\in C_{\cell}^m(Y';\mathbb{F}_2)$, compute $\Sq^2([\mathbf{c}']) = [\mathbf{r}']\in H^{m+2}(Y';\mathbb{F}_2)$, and pull $\mathbf{r}'$ back to $C_{\cell}^{M+2}(Y;\mathbb{F}_2)$. This computation allows us to compute $\Sq^2: H^l(\mathcal{X}(\mathscr{C});\mathbb{F}_2)\to H^{l+2}(\mathcal{X}(\mathscr{C});\mathbb{F}_2)$.
\end{procedure}
\begin{proof}
  Consider the following commutative diagram:
  \[
  \begin{tikzpicture}[scale=2]
  \path[draw, ->, shorten <=1.3cm,shorten >=1.7cm, thick] (0,1) -- (3,1) node [midway, label={[label distance=-0.2cm]90:$\cong$}] {};
  \path[draw, ->>, shorten <=1.3cm,shorten >=1.7cm, thick] (0,0) -- (3,0);
  \path[draw, ->, shorten <=0.3cm,shorten >=0.3cm, thick] (0,0) -- (0,1) node [midway, label={[label distance=-0.1cm]180:$\Sq^2$}] {};
  \path[draw, ->, shorten <=0.3cm,shorten >=0.3cm, thick] (3,0) -- (3,1) node [midway, label={[label distance=-0.1cm]180:$\Sq^2$}]{};
  \path[draw, ->, shorten <=1.3cm,shorten >=1.7cm, thick] (6,0) -- (3,0) node [midway, label={[label distance=-0.2cm]90:$\cong$}] {};
  \path[draw, {Hooks[left]}->, shorten <=1.3cm,shorten >=1.7cm, thick] (6,1) -- (3,1) {};
  \path[draw, ->, shorten <=0.3cm,shorten >=0.3cm, thick] (6,0) -- (6,1) node [midway, label={[label distance=-0.1cm]180:$\Sq^2$}] {};
    \node at (0, 1) {$H^{m+2}(Y';\mathbb{F}_2)$};
    \node at (3, 1) {$H^{M+2}(Y^{(M+2)};\mathbb{F}_2)$};
    \node at (0, 0) {$H^m(Y';\mathbb{F}_2)$};
    \node at (3, 0) {$H^{M+2}(Y^{(M+2)};\mathbb{F}_2)$};
    \node at (6, 1) {$H^{M+2}(Y;\mathbb{F}_2)$};
    \node at (6, 0) {$H^{M}(Y;\mathbb{F}_2)$};
   \end{tikzpicture}
\]
  
    where the first set of horizontal arrows are induced by the homeomorphism in Proposition \ref{suspension identification} Imagine $[\mathbf{c}]$ starting in the bottom right corner. We can move $[\mathbf{c}]$ ``clockwise'' around the perimeter of the diagram to compute $\Sq^2([\mathbf{c}])\in H^{M+2}(Y;\mathbb{F}_2)$.
  \end{proof}
  We include the chronological order of spaces we constructed for reference:
  \[
\begin{tikzpicture}[scale=2]
  \path[draw, ->, thick, decorate,
  decoration={snake, amplitude = 0.3mm}] (0,0) -- (1.2,0) node [midway, label={[label distance=-0.2cm]90:subquotient}] {};
  \path[draw, ->, thick, decorate,
  decoration={snake, amplitude = 0.3mm}] (2.8,0) -- (3.8,0) node [midway, label={[label distance=-0.2cm]90:$\Sigma^{|\mathbf{d}'|-|\mathbf{d}|}$}]  {};
    \node at (-0.2, 0) {$Y$};
    \node at (2, 0) {$Y^{(M+2)}/Y^{(M-1)}$};
    \node at (4, 0) {$Y'$};
 \end{tikzpicture}
\]

We focus on the CW complex $Y'$ for the next few chapters.\par
The attaching maps for these cells are defined using the cubical neat embedding $\iota'$. Our cubical neat embedding $\iota'$ consists of only three classes of embeddings. Namely, there are the embeddings of the $0$-dimensional moduli spaces
\begin{equation}\label{yx,zy}
  \begin{aligned}
  \iota'_{y,x}:\mathcal{M}(y,x)&\to [-R,R]^{A}\times \mathcal{M}(v,u)\\
  p & \mapsto (a_p,\text{pt}),\\
  \iota'_{z,y}:\mathcal{M}(z,y)&\to [-R,R]^{B}\times \mathcal{M}(w,v)\\
    q & \mapsto (b_q,\text{pt}),
  \end{aligned}
\end{equation}
and finally, there are the embeddings of the $1$-dimensional moduli spaces
\begin{equation}\label{cubical embedding}
  \iota'_{z,x}:\mathcal{M}(z,x)\to [-R,R]^{A}\times [-R,R]^{B}\times \mathcal{M}(w,u),
\end{equation}
defined as follows: $\mathcal{M}(z,x)$ is a finite union of line segments $I$. For each segment $I\subset \mathcal{M}(z,x)$ $\iota_{x,z}$ maps its boundary points to $(p_1,q_1)$ and $(p_2,q_2)$, where $(p_i,q_i)$ is a point in $\mathcal{M}(y_i,x)\times \mathcal{M}(z,y_i)$. $\iota_{z,x}$ maps the rest of $I$ to an embedded path joining
\begin{equation*}
	(a_{p_1},b_{q_1},\mathcal{M}(v_1,u)\times \mathcal{M}(w,v_1))\text{ to } (a_{p_2},b_{q_2},\mathcal{M}(v_2,u)\times \mathcal{M}(w,v_2))
\end{equation*}
in $[-R,R]^{A}\times [-R,R]^{B}\times \mathcal{M}(w,u)$. The extended embeddings $\overline{\iota}'_{y,x}$, $\overline{\iota}'_{z,y}$, $\overline{\iota}'_{z,x}$, obtained from (\ref{yx,zy}), (\ref{cubical embedding}), and $\sigma$, induce embeddings
\begin{align}
    &\jmath'_{y,x}:\mathcal{C}'(x) \times \mathcal{M}(y,x) \to \mathcal{C}'(y)\label{embedding yx}\\
    &\jmath'_{z,y}:\mathcal{C}'(y)\times \mathcal{M}(z,y)\to \mathcal{C}'(z)\label{embedding zy}\\
    &\jmath'_{z,x}:\mathcal{C}'(x)\times \mathcal{M}(z,x)\to \mathcal{C}'(z)\label{embedding zx},
\end{align}
and through \ref{attaching realization}, we get our attaching maps. See Figure \ref{attaching pictures}.
\begin{figure}
  \begin{tabular} {m{6cm} m{5.0cm} m{3.0cm}}
    \def\svgscale{0.7}
\begingroup%
  \makeatletter%
  \providecommand\color[2][]{%
    \errmessage{(Inkscape) Color is used for the text in Inkscape, but the package 'color.sty' is not loaded}%
    \renewcommand\color[2][]{}%
  }%
  \providecommand\transparent[1]{%
    \errmessage{(Inkscape) Transparency is used (non-zero) for the text in Inkscape, but the package 'transparent.sty' is not loaded}%
    \renewcommand\transparent[1]{}%
  }%
  \providecommand\rotatebox[2]{#2}%
  \newcommand*\fsize{\dimexpr\f@size pt\relax}%
  \newcommand*\lineheight[1]{\fontsize{\fsize}{#1\fsize}\selectfont}%
  \ifx\svgwidth\undefined%
    \setlength{\unitlength}{268.73566395bp}%
    \ifx\svgscale\undefined%
      \relax%
    \else%
      \setlength{\unitlength}{\unitlength * \real{\svgscale}}%
    \fi%
  \else%
    \setlength{\unitlength}{\svgwidth}%
  \fi%
  \global\let\svgwidth\undefined%
  \global\let\svgscale\undefined%
  \makeatother%
  \begin{picture}(1,1.39492551)%
    \lineheight{1}%
    \setlength\tabcolsep{0pt}%
    \put(0,0){\includegraphics[width=\unitlength,page=1]{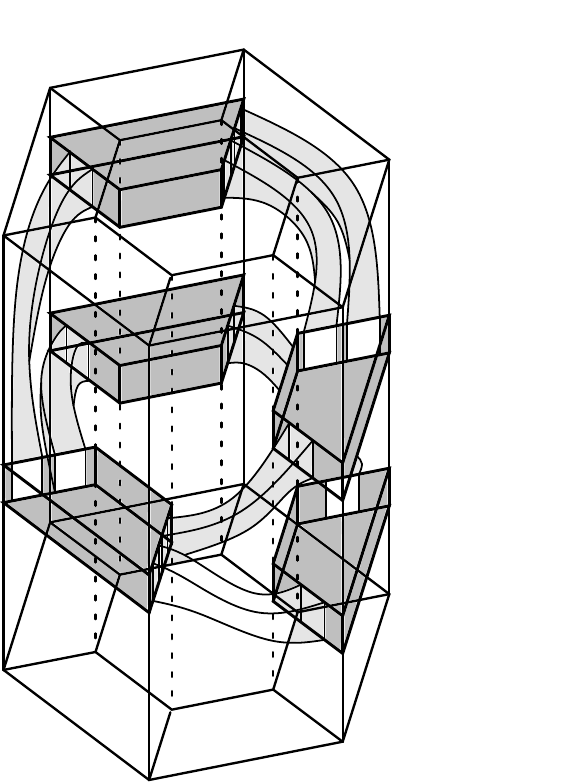}}%
    \put(0.24028133,1.36947298){\color[rgb]{0,0,0}\makebox(0,0)[lt]{\lineheight{1.25}\smash{\begin{tabular}[t]{l}$\partial\mathcal{C}'(z)$\end{tabular}}}}%
    \put(0,0){\includegraphics[width=\unitlength,page=2]{Cz_1.pdf}}%
    \put(0.35442915,0.20017376){\color[rgb]{0,0,0}\makebox(0,0)[lt]{\lineheight{1.25}\smash{\begin{tabular}[t]{l}$\mathcal{C}'(x)$\end{tabular}}}}%
    \put(0.73338487,0.52935282){\color[rgb]{0,0,0}\makebox(0,0)[lt]{\lineheight{1.25}\smash{\begin{tabular}[t]{l}$\mathcal{C}'(y)$\end{tabular}}}}%
    \put(0,0){\includegraphics[width=\unitlength,page=3]{Cz_1.pdf}}%
  \end{picture}%
\endgroup%
 & \def\svgscale{0.7}
\begingroup%
  \makeatletter%
  \providecommand\color[2][]{%
    \errmessage{(Inkscape) Color is used for the text in Inkscape, but the package 'color.sty' is not loaded}%
    \renewcommand\color[2][]{}%
  }%
  \providecommand\transparent[1]{%
    \errmessage{(Inkscape) Transparency is used (non-zero) for the text in Inkscape, but the package 'transparent.sty' is not loaded}%
    \renewcommand\transparent[1]{}%
  }%
  \providecommand\rotatebox[2]{#2}%
  \newcommand*\fsize{\dimexpr\f@size pt\relax}%
  \newcommand*\lineheight[1]{\fontsize{\fsize}{#1\fsize}\selectfont}%
  \ifx\svgwidth\undefined%
    \setlength{\unitlength}{269.61098018bp}%
    \ifx\svgscale\undefined%
      \relax%
    \else%
      \setlength{\unitlength}{\unitlength * \real{\svgscale}}%
    \fi%
  \else%
    \setlength{\unitlength}{\svgwidth}%
  \fi%
  \global\let\svgwidth\undefined%
  \global\let\svgscale\undefined%
  \makeatother%
  \begin{picture}(1,1.38982547)%
    \lineheight{1}%
    \setlength\tabcolsep{0pt}%
    \put(0,0){\includegraphics[width=\unitlength,page=1]{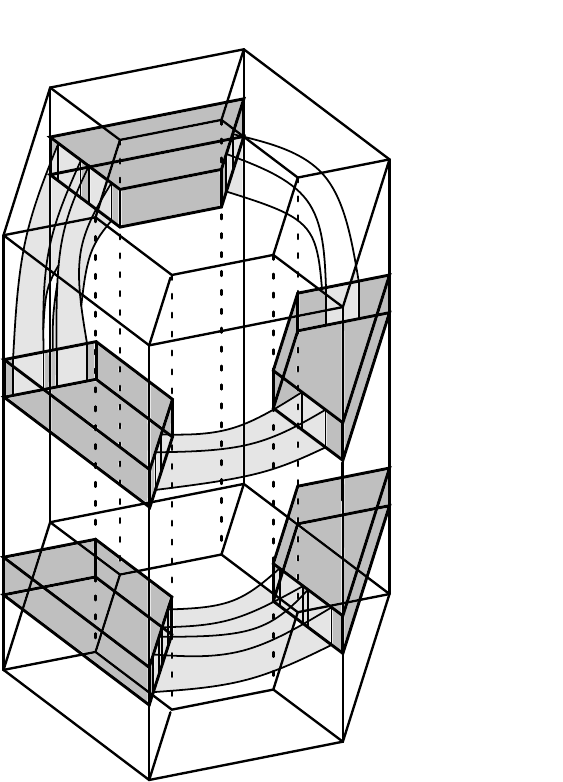}}%
    \put(0.49621922,1.14271754){\color[rgb]{0,0,0}\makebox(0,0)[lt]{\lineheight{1.25}\smash{\begin{tabular}[t]{l}$\mathcal{C}'(x')$\end{tabular}}}}%
    \put(0,0){\includegraphics[width=\unitlength,page=2]{Cz_2.pdf}}%
    \put(0.72540448,0.87728264){\color[rgb]{0,0,0}\makebox(0,0)[lt]{\lineheight{1.25}\smash{\begin{tabular}[t]{l}$\mathcal{C}'(y')$\end{tabular}}}}%
    \put(0.22619202,1.36445559){\color[rgb]{0,0,0}\makebox(0,0)[lt]{\lineheight{1.25}\smash{\begin{tabular}[t]{l}$\partial\mathcal{C}'(z')$\end{tabular}}}}%
  \end{picture}%
\endgroup%
 & \def\svgscale{0.7}
\begingroup%
  \makeatletter%
  \providecommand\color[2][]{%
    \errmessage{(Inkscape) Color is used for the text in Inkscape, but the package 'color.sty' is not loaded}%
    \renewcommand\color[2][]{}%
  }%
  \providecommand\transparent[1]{%
    \errmessage{(Inkscape) Transparency is used (non-zero) for the text in Inkscape, but the package 'transparent.sty' is not loaded}%
    \renewcommand\transparent[1]{}%
  }%
  \providecommand\rotatebox[2]{#2}%
  \newcommand*\fsize{\dimexpr\f@size pt\relax}%
  \newcommand*\lineheight[1]{\fontsize{\fsize}{#1\fsize}\selectfont}%
  \ifx\svgwidth\undefined%
    \setlength{\unitlength}{147.10252176bp}%
    \ifx\svgscale\undefined%
      \relax%
    \else%
      \setlength{\unitlength}{\unitlength * \real{\svgscale}}%
    \fi%
  \else%
    \setlength{\unitlength}{\svgwidth}%
  \fi%
  \global\let\svgwidth\undefined%
  \global\let\svgscale\undefined%
  \makeatother%
  \begin{picture}(1,0.94187211)%
    \lineheight{1}%
    \setlength\tabcolsep{0pt}%
    \put(0,0){\includegraphics[width=\unitlength,page=1]{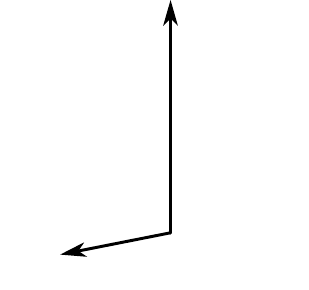}}%
    \put(-0.00379317,0.00966666){\color[rgb]{0,0,0}\makebox(0,0)[lt]{\lineheight{1.25}\smash{\begin{tabular}[t]{l}$[-R,R]^A$\end{tabular}}}}%
    \put(0.62747954,0.87184182){\color[rgb]{0,0,0}\makebox(0,0)[lt]{\lineheight{1.25}\smash{\begin{tabular}[t]{l}$[-R,R]^B$\end{tabular}}}}%
    \put(0,0){\includegraphics[width=\unitlength,page=2]{Cz_axes.pdf}}%
    \put(0.5925291,0.13965945){\color[rgb]{0,0,0}\makebox(0,0)[lt]{\lineheight{1.25}\smash{\begin{tabular}[t]{l}$J\times [0,1]$\end{tabular}}}}%
  \end{picture}%
\endgroup%

  \end{tabular}
  
  \caption{A picture of the attaching map for the boundary $\partial \mathcal{C}'(z)$. Here we imagine $\mathcal{C}'(z)$ as a thickened $2$-dimensional permutohedron ($|z|=|w|=3$). We do not draw the $J$-axis nor the $[0,1]$-component in $\widetilde{\mathcal{M}}_{\mathscr{C}_C(n)}(w,\overline{0})$.}
  \label{attaching pictures}
\end{figure}
\begin{definition}\label{G_i definition}
 We define the \textit{facets} $\mathbf{G}_S$ of $\mathcal{C}'(y)$, indexed by nonempty, proper subsets $S\subset\{0,\ldots, \kappa\}$, as follows:
  \begin{align*}
    \intertext {If $\kappa = 0$, define}
      \mathbf{G}_0 &:= [-R,R]^A\times [-\epsilon,\epsilon]^B\times J\times \{0\}\times\mathcal{M}_{\mathscr{C}_C(n)}(v,\overline{0}) \subset \partial\mathcal{C}'(y)
  \intertext{If $\kappa \geq 1$, define}
      \mathbf{G}_S&:= [-R,R]^A\times [-\epsilon,\epsilon]^B\times J\times [0,1]\times G_S\subset \partial\mathcal{C}'(y)
  \intertext{(See Figure \ref{G_i illustration} for a picture.) Similarly, for nonempty, proper subsets $S\subset \{0,\ldots,\kappa+1\}$ we define the \textit{facets} of $\mathcal{C}'(z)$ by}
    \mathbf{G}_S&:= [-R,R]^A\times [-R,R]^B\times J\times [0,1]\times G_S\subset \partial\mathcal{C}'(y)
  \end{align*}
  
\end{definition}
\begin{figure}
  \centering
  \begin{tabular} {m{7cm} m{3cm}}
    \def\svgscale{0.7}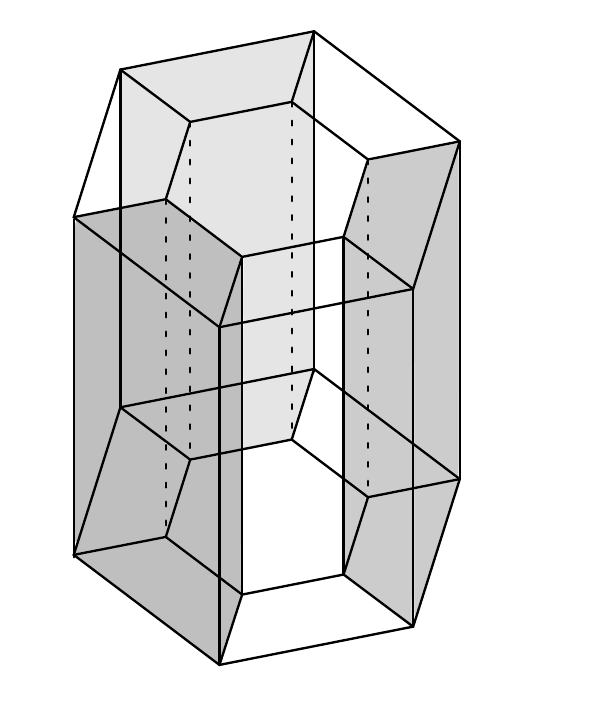 & \def\svgscale{0.7}
\begingroup%
  \makeatletter%
  \providecommand\color[2][]{%
    \errmessage{(Inkscape) Color is used for the text in Inkscape, but the package 'color.sty' is not loaded}%
    \renewcommand\color[2][]{}%
  }%
  \providecommand\transparent[1]{%
    \errmessage{(Inkscape) Transparency is used (non-zero) for the text in Inkscape, but the package 'transparent.sty' is not loaded}%
    \renewcommand\transparent[1]{}%
  }%
  \providecommand\rotatebox[2]{#2}%
  \newcommand*\fsize{\dimexpr\f@size pt\relax}%
  \newcommand*\lineheight[1]{\fontsize{\fsize}{#1\fsize}\selectfont}%
  \ifx\svgwidth\undefined%
    \setlength{\unitlength}{107.69055104bp}%
    \ifx\svgscale\undefined%
      \relax%
    \else%
      \setlength{\unitlength}{\unitlength * \real{\svgscale}}%
    \fi%
  \else%
    \setlength{\unitlength}{\svgwidth}%
  \fi%
  \global\let\svgwidth\undefined%
  \global\let\svgscale\undefined%
  \makeatother%
  \begin{picture}(1,1.28487984)%
    \lineheight{1}%
    \setlength\tabcolsep{0pt}%
    \put(0,0){\includegraphics[width=\unitlength,page=1]{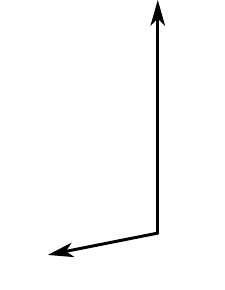}}%
    \put(-0.00518135,0.01320461){\color[rgb]{0,0,0}\makebox(0,0)[lt]{\lineheight{1.25}\smash{\begin{tabular}[t]{l}$[-R,R]^A$\end{tabular}}}}%
    \put(0.8028502,1.18922096){\color[rgb]{0,0,0}\makebox(0,0)[lt]{\lineheight{1.25}\smash{\begin{tabular}[t]{l}$J$\end{tabular}}}}%
    \put(0.7360769,0.17417508){\color[rgb]{0,0,0}\makebox(0,0)[lt]{\lineheight{1.25}\smash{\begin{tabular}[t]{l}$[0,1]$\end{tabular}}}}%
    \put(0,0){\includegraphics[width=\unitlength,page=2]{Cy_axes.pdf}}%
  \end{picture}%
\endgroup%

  \end{tabular}
  \caption{What we refer to as faces $\mathbf{G}_0$, $\mathbf{G}_1$, $\mathbf{G}_2$ in the case $\mathcal{C}'(y)$ is a thickened $2$-dimensional permutohedron ($|y| = |v| = 3$). We do not draw the extra $[0,1]$-factor.}
  \label{G_i illustration}
\end{figure}
\begin{remark}
  Given a cell $\mathcal{C}'(y)$ of $Y'$, there are embedded subsets $\mathcal{C}'_x(y)\cong \mathcal{C}'(x)\times \mathcal{M}(y,x)$ of $\partial\mathcal{C}'(y)$ for each $x$, where each component $\mathcal{C}'(x)\times \{p\}$ of $\mathcal{C}'_x(y)$ gets identified with $\mathcal{C}'(x)$ and everything outside of $\bigcup_xC_x(y)$ gets identified to the basepoint.\par
  We make an analogous observation for cells $\mathcal{C}'(z)$ of $Y$. There are embedded subsets $\mathcal{C}'_y(z)\cong \mathcal{C}'(y)\times \mathcal{M}(z,y)$, where each component $\mathcal{C}'(y)\times \{q\}$ of $\mathcal{C}'_y(z)$ gets identified with $\mathcal{C}'(y)$.\par
  It is important to note that our indexing of facets is compatible with our cubical index assignment $S_{\mathbb{Z}}$ in the following sense: if $p\in \mathcal{M}(y,x)$, then $S_{\mathbb{Z}}(p)$ equals the index $i$ of the facet $\mathbf{G}_i\subset \partial\mathcal{C}'(y)$ that contains $\mathcal{C}'(x)\times \{p\}$. Similarly, $\mathcal{C}'(y)\times\{q\}\subset \mathbf{G}_j$ for $q\in \mathcal{M}(z,y)$, $j = S_{\mathbb{Z}}(q)$.
\end{remark}
\section{Constructing map to the Eilenberg-MacLane space and boundary matching}\label{boundary matching}
\subsection{Defining the truncated Eilenberg-MacLane space $K_m^{(m+2)}$}
We are going to more explicitly define $K_m^{(m+2)}$. We start with a single $0$-cell. Then we define the $m$-cell of $K_m^{(m+2)}$ as
\begin{equation}\label{m-cell}
e^{m}=[-\epsilon,\epsilon]^A\times [-\epsilon,\epsilon]^B\times J\times \widetilde{\mathcal{M}}_{\mathscr{C}_C(\kappa)}(\overline{1},\overline{0}),
\end{equation}
with the entire boundary attached to the basepoint. Define the $(m+1)$-cell as $e^{m+1}=e^m\times [0,1]$. The map $e^{m+1}\to K_m^{(m+1)}$ maps $e^m\times \{0\}$ by Id to $e^m$ and maps $e^m\times\{1\}$ to $e^m$ by the following map:
\begin{equation}
  e^m\times\{1\} \cong e^m\xrightarrow{\tau} e^m,\qquad \tau(a,b,t,p) = (a,b,-t,p)\label{tau definition},
\end{equation}
which simply flips the $J$ factor. The rest of the boundary of $e^{m+1}$ gets mapped to the basepoint.
Also, $K_m^{(m-1)}$ has the characterization
\begin{equation}\label{m-cell redone}
  K_m^{(m+1)} = \bigslant{\left( [-\epsilon,\epsilon]^A\times [-\epsilon,\epsilon]^B\times \widetilde{\mathcal{M}}_{\mathscr{C}_C(\kappa)}(\overline{1},\overline{0})\right)}{\partial}\wedge \bigslant{\left(\frac{J\times [0,1]}{((t,0)\sim (-t,1)}\right)}{\partial},
\end{equation}
which is mainly different from the construction using (\ref{m-cell}) in the sense that the $J$ factor is moved to the right.\par 
Now we move to the $(m+2)$-cells. We define $e^{m+2} = D^{m+2}$ and define the attaching map $\partial e^{m+2}\to K_m^{(m+1)}$ as a suspension of the Hopf map $\eta$.
\[
  \partial e^{m+1}\cong S^{m+1}\xrightarrow{\Sigma^{m-2}\eta} S^m \cong e^{m}/\partial e^{m}\subset K_m^{(m+1)}.
\]
\par
Let $\mathbf{c}\in C_{\cell}^m(Y;\mathbb{F}_2)$ be an $m$-dimensional cocycle denoted by
\begin{equation*}
  \mathbf{c} = \sum_{x}\mu_x\cdot \mathcal{C}(x),
\end{equation*}
and let
\[
  \mu = \sum_x \mu_x \cdot x
\]
be the corresponding cocycle in $C_\mathcal{M}^l(\mathscr{C};\mathbb{F}_2)\cong C_{\cell}^m(Y;\mathbb{F}_2)$. The $\mu_x$-terms are the $\mathbb{F}_2$-coefficients. We say $x$ \textit{appears in} $\mu$ if $\mu_x=1$.
We would like to construct a map $\mathfrak{c}:Y\to K_m^{(m+2)}$ that pulls $\iota$ back to $\mathbf{c}$ (that is, $\mathbf{c}=\mathfrak{c}^*\iota$, where $\iota\in C_{\cell}^m(K_m^{(m+2)};\mathbb{F}_2)$ is the fundamental class $e^m$). Let the $x_1,\ldots,x_r$ be the generators of $C^m(Y)$ appears in $\mu$ (having coefficients $1$, not $0$). On the $m$-skeleton of $Y$, we ask $\mathfrak{c}$ to map the cells $\mathcal{C}(x)$ as follows:
\begin{itemize}
\item $\mathfrak{c}$ maps $\mathcal{C}(x)=[-\epsilon,\epsilon]^{d_\kappa}\times [-\epsilon,\epsilon]^{d_{\kappa+1}}\times \widetilde{\mathcal{M}}_{\mathscr{C}_C(n)}(u,\overline{0})$ by identity to $e^m$ if $x$ appears in $\mu$. ($\widetilde{\mathcal{M}}_{\mathscr{C}_C(n)}(u,\overline{0})$, by definition, is equal to $\widetilde{\mathcal{M}}_{\mathscr{C}_C(\kappa)}(\overline{1},\overline{0})$.)
\item $\mathfrak{c}$ maps $\mathcal{C}(x)$ to the basepoint if $\mu_x=0$.
\end{itemize}
Any $\mathfrak{c}$ satisfying the above condition will pull $\iota$ back to $\mathbf{c}$ by $\mathfrak{c}^*$. Now we must specify how $\mathfrak{c}$ maps the $(m+1)$-skeleton of $Y$, that is, specify how it maps cells of type $\mathcal{C}(y)$.\par
\begin{notation}
  For a generator $y\in \Ob(\mathscr{C})$, define
  \[
    \mathcal{M}(y,\mu):= \coprod_{\substack{x\text{ appears}\\\text{in } \mu}}\mathcal{M}(y,x).
  \]
  Note that since $\mu$ is a cocycle, $\#(\mathcal{M}(y,\mu))=0\mod 2$.
\end{notation}
Because $\mu$ is a cocycle, $\mathfrak{c}$ can only map an even number of these components homeomorphically onto $K_m^{(m)}\cong S^m$; everything outside of these components maps to the basepoint. We choose to group this set of these ``boundary components'' into pairs in a process called ``boundary matching:'' each component of some $C_x(y)$ for $x$ appears in $\mu$ is ``matched'' with some other arbitrarily chosen component of another $C_{x'}(y)$ (see Figure \ref{facewise picture} for an illustration).
\begin{figure}
  \begin{tabular} {m{7cm} m{3cm}}
    \def\svgscale{0.7}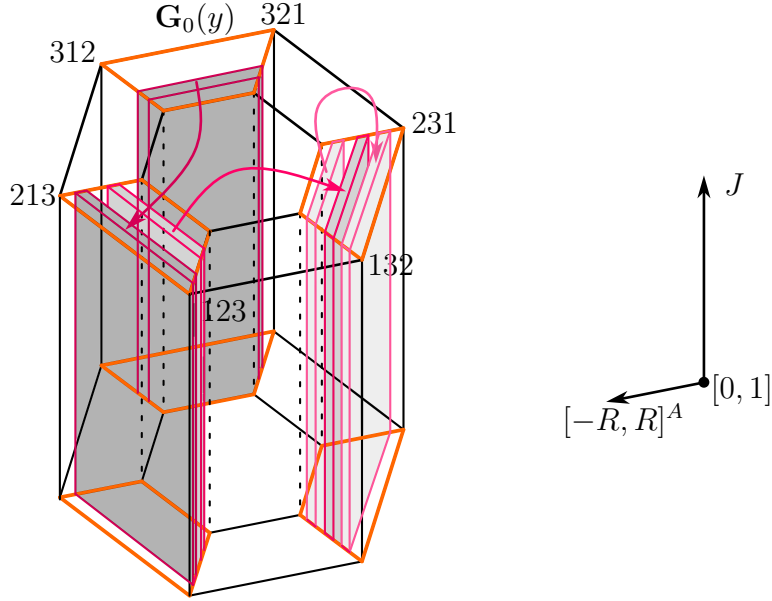 & \def\svgscale{0.7}
  \end{tabular}
    \caption{An illustration of a (facewise) boundary matching $\mathfrak{b}_y$ of $\mathcal{M}(y,\mu)$ in terms of embeddings of cells $\mathcal{C}(x)$ into $\mathcal{C}(y)$.  There are an even number of embedded cells of the form $\mathcal{C}(x)$ where $x\in \mu$, so we pair them up and connect them through tubes $\eta\subset \mathcal{C}(y)$. (The faces $\mathbf{G}_i$ are thickenings of $F_i\subset \Pi^{(\kappa+1)-1}$, defined in Definition \ref{G_i definition}.) Note that the edge from $\mathbf{G}_0$ to $\mathbf{G}_2$ (green) and $\mathbf{G}_1$ to $\mathbf{G}_2$ (light blue) must be oriented from the lower index face to the higher index face. But the edge from $\mathbf{G}_0$ to itself could have been oriented either way.}
    \label{facewise picture}
  \end{figure}
We use this boundary matching (matching an embedded $\mathcal{C}(x)\times\{p\}$ with an embedded $\mathcal{C}(x')\times\{p'\}$) to extend $\mathfrak{c}$ to the cell $\mathcal{C}(y)$. Our process is as follows:
\begin{enumerate}[label = (B-\arabic*), ref = B-\arabic*]
  \item The restriction $\mathfrak{c}|_{\partial\mathcal{C}(y)}$ to the boundary $\partial \mathcal{C}(y)$ is predetermined by $\partial\mathcal{C}(y)\xrightarrow{\text{attaching}} Y^{(m)}\xrightarrow{\mathfrak{c}|_{Y^{(m)}}}K_m^{(m)}$. In particular, if $x$ is represented in $\mu$, $\mathcal{C}_x(y)\subset \partial\mathcal{C}(y)$ is mapped by $\mathcal{C}_x(y)\cong \mathcal{C}(x)\times \mathcal{M}(y,x)\to \mathcal{C}(x)\cong e^m\to K_m^{(m)}$.\label{boundary condition}
\item Construct d.s.\ tubes $\eta := \{(\Theta, 0),(\Lambda, 1)\}: e^m\times [0,1]\xhookrightarrow{} \mathcal{C}(y)$, where $\Theta$ identifies the endpoints $e^m\times \{i\}$ ($i=0,1$) with the boundary components $\mathcal{C}(x)\times \{p\}$, $\mathcal{C}(x')\times \{p'\}$ in $\mathcal{C}_{x}(y)$, $\mathcal{C}_{x'}(y)$. (This identification may not be canonical.)
  \label{tube step}
\item Have $\mathfrak{c}$ map everywhere in $\mathcal{C}(y)$ outside of the tubes $\eta$ to the basepoint.\label{basepoint}
\item Continuously extend $\mathfrak{c}$ to $\eta$ for each $\eta$. That is, define an extension $\mathfrak{c}|_{\eta}:\eta\to K_m^{(m+1)}$ that agrees with $\mathfrak{c}|_{\partial\mathcal{C}(y)}$ on the boundary components $\mathcal{C}(x_0)\times\{p_0\}$, $\mathcal{C}(x_1)\times \{p_1\}$. \label{extension step}
\end{enumerate}
\subsection{Constructing the boundary matching tubes $\eta$}\label{eta construction}
Here, we outline (\ref{tube step}), which first involves examining a pair of matched boundary components. Suppose a component $\mathcal{C}'(x)\times \{p\}\subset \mathcal{C}(y)$ is matched with another component $\mathcal{C}'(x')\times \{p'\}\subset \mathcal{C}'(y)$. We first construct an ``unsigned'' version $\eta_\emptyset:= \{(\Theta_\emptyset,0),(\Lambda_\emptyset,1)\}$ of our eventual tube $\eta$.
\subsubsection*{Case 0: The matched boundary components are on the same face $\mathbf{G}_i\subset \partial \mathcal{C}(y)$.}\label{case0}
Let us first focus on $\Theta_\emptyset$. Among the two boundary components $\mathcal{C}'(x)\times \{p\}\subset \mathcal{C}(y)$, $\mathcal{C}'(x')\times \{p'\}\subset \mathcal{C}(y)$ we arbitrarily choose a starting $(t=0)$ component of $\Theta_\emptyset$, and an ending $(t=1)$ component. This choice is made arbitrarily, and after renaming $x,x',p,p'$, we intend that $\Theta_\emptyset$ goes from $\mathcal{C}'(x_0)\times \{p_0\}$ to $\mathcal{C}'(x_0)\times \{p_0\}$ as $t$ goes from $0$ to $1$.\par
The tube $\Theta_\emptyset: e^m\times [0,1]\xhookrightarrow{} \mathcal{C}(y)$ is characterized by its projection to the $\mathcal{C}(y)$ factors $[-R,R]^A$, $[-\epsilon,\epsilon]^B$, and $J\times \widetilde{\mathcal{M}}_{\mathscr{C}_C(n)}(v,\overline{0})$. In other words, $\Theta_\emptyset$ is determined by its projections:
\begin{enumerate}[label={(\arabic*)}]
\item $\pi_\epsilon\circ \Theta_\emptyset: e^m\times [0,1]\to [-\epsilon,\epsilon]^B$
\item $\pi_R\circ \Theta_\emptyset: e^m\times [0,1]\to [-R,R]^A$
\item $\pi_{\mathcal{M}}\circ \Theta_\emptyset: e^m\times [0,1] \to J\times \widetilde{\mathcal{M}}_{\mathscr{C}_C(n)}(v,\overline{0})$.
\end{enumerate}
\paragraph{Defining $\Theta_{\epsilon}:= \pi_\epsilon\circ \Theta_\emptyset$:}
For $a\in [-\epsilon,\epsilon]^A$, $b\in [-\epsilon,\epsilon]^B$, $s\in J$, $x\in \widetilde{\mathcal{M}}_{\mathscr{C}_C(n)}(u,\overline{0})$, and $t\in [0,1]$, we define $\Theta_\epsilon (a, b, s, x, t) = b$.
\paragraph{Defining $\Theta_R:=\pi_R\circ \Theta_\emptyset$:}
See Figure \ref{R projection}. Let $\gamma(t):[0,1]\to [-R,R]^A$ be a smooth embedded path from $a_{p_0}$ to $a_{p_1}$ (the points $a_{p_0},a_{p_1}$ are defined in (\ref{yx,zy})). We define $\Theta_R (t, a, b, x) = \gamma(t)+a$.
\begin{figure}
  \centering
\begingroup%
  \makeatletter%
  \providecommand\color[2][]{%
    \errmessage{(Inkscape) Color is used for the text in Inkscape, but the package 'color.sty' is not loaded}%
    \renewcommand\color[2][]{}%
  }%
  \providecommand\transparent[1]{%
    \errmessage{(Inkscape) Transparency is used (non-zero) for the text in Inkscape, but the package 'transparent.sty' is not loaded}%
    \renewcommand\transparent[1]{}%
  }%
  \providecommand\rotatebox[2]{#2}%
  \newcommand*\fsize{\dimexpr\f@size pt\relax}%
  \newcommand*\lineheight[1]{\fontsize{\fsize}{#1\fsize}\selectfont}%
  \ifx\svgwidth\undefined%
    \setlength{\unitlength}{350.14742633bp}%
    \ifx\svgscale\undefined%
      \relax%
    \else%
      \setlength{\unitlength}{\unitlength * \real{\svgscale}}%
    \fi%
  \else%
    \setlength{\unitlength}{\svgwidth}%
  \fi%
  \global\let\svgwidth\undefined%
  \global\let\svgscale\undefined%
  \makeatother%
  \begin{picture}(1,0.44119873)%
    \lineheight{1}%
    \setlength\tabcolsep{0pt}%
    \put(0,0){\includegraphics[width=\unitlength,page=1]{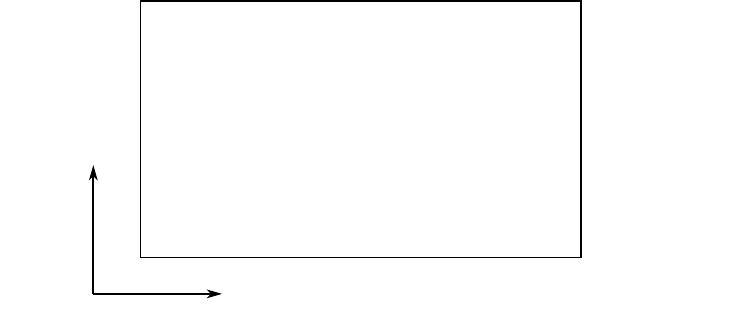}}%
    \put(0.27564241,0.00149079){\color[rgb]{0,0,0}\makebox(0,0)[lt]{\lineheight{1.25}\smash{\begin{tabular}[t]{l}$t$\end{tabular}}}}%
    \put(-0.00159357,0.19620048){\color[rgb]{0,0,0}\makebox(0,0)[lt]{\lineheight{1.25}\smash{\begin{tabular}[t]{l}$[-R,R]^A$\end{tabular}}}}%
    \put(0,0){\includegraphics[width=\unitlength,page=2]{Theta_A_projection.pdf}}%
    \put(0.58697139,0.22613377){\color[rgb]{0,0,0}\makebox(0,0)[lt]{\lineheight{1.25}\smash{\begin{tabular}[t]{l}image of\\$\pi_R(\Theta|_{e^m\times \{s\}})$\end{tabular}}}}%
    \put(0.56773896,0.05737061){\color[rgb]{0,0,0}\makebox(0,0)[lt]{\lineheight{1.25}\smash{\begin{tabular}[t]{l}$s$\end{tabular}}}}%
  \end{picture}%
\endgroup%

  \caption{An illustration showing how the embedding $\Theta$ behaves when projected to the $[-R,R]^A$-component of $\mathcal{C}(y)$.}
  \label{R projection}
\end{figure}
\paragraph{Defining $\Theta_\mathcal{M} := \pi_{\mathcal{M}}\circ \Theta_\emptyset$:}
Our construction of $\Theta_\mathcal{M} $ depends on a delicate treatment of the cases $u=\overline{0}$ and $u\neq \overline{0}$, since the definition of $\widetilde{\mathcal{M}}_{\mathscr{C}_C(n)}(u,\overline{0})$ is exceptional in the case $u=\overline{0}$ (see Definition \ref{cubical realization def}). Now suppose we match two components $\mathcal{C}(x_0)\times \{p_0\}\subset \mathbf{G}_i$, $\mathcal{C}(x_1)\times \{p_1\}\subset \mathbf{G}_j$.
\label{case0}
See Figure \ref{same face figure} for an illustration. We choose a tube $\Theta_{\mathcal{M}}$ where as $t$ moves from $0$ to $1$, the slices $e^m\times \{t\}$ move straight into $\Int(J\times \widetilde{\mathcal{M}}_{\mathscr{C}_C(n)}(v,\overline{0}))$, but then pull up and around back into the face $J\times G_i$, but with the $J$-direction flipped (see Figure \ref{same face figure}). This should look a lot like an ``inward'' version of our construction in Example \ref{same face example}. Let us give a precise construction. Denote values in $e^m\times [0,1]$ by $(a,b,s,p,t)$, where when $\kappa > 0$ $a\in [-\epsilon,\epsilon]^A$, $b\in [-\epsilon,\epsilon]^B$, $s\in J$, $p=(r, (a_1,\ldots,a_\kappa))\in \widetilde{\mathcal{M}}_{\mathscr{C}_C(n)}(u,\overline{0})$, and $t\in [0,1]$. In the exceptional case $\kappa = 0$, we denote $p = 0$.
\begin{itemize}
\item We define the starting position ($t=0$) of $\Theta_{\mathcal{M}}$ by
  \begin{equation}
    \Theta_{\mathcal{M}}(a,b,s,p,0) =
    \begin{cases*}
      (s,r,(a_1\ldots,a_{i-1},\kappa+1,a_{i},\ldots,a_\kappa))\subset \mathbf{G}_i & in the case $\kappa > 0$\\
      (s,0,(1))\in \mathbf{G}_1 & in the case $\kappa=0$.
    \end{cases*}
  \end{equation}
\item The vectors $d\Theta_\mathcal{M} (\partial_t)$, $d\Theta_\mathcal{M}(\partial_{J})$ should start ($t=0$) pointing in the $\partial_{\mathbf{n}}$, $\partial_J$ directions respectively, where $\mathbf{n}$ is the inward unit normal.
\item As $t$ increases from $0$ to $1$, these vectors $d\Theta_\mathcal{M} (\partial_t)$, $d\Theta_\mathcal{M} (\partial_J)$ should rotate $180^\circ$ in the plane $\langle\mathbf{n},\partial_J\rangle$ so that at time $t=\frac{1}{2}$, they point in the directions of $\partial_J,-\mathbf{n}$ respectively.
\item At the end ($t=1$), $d\Theta_\mathcal{M} (\partial_t)$, $d\Theta_\mathcal{M}(\partial_{J})$ should be pointing in the $-\partial_{\mathbf{n}}$, $-\partial_J$ directions respectively.
\item We define the ending position of $\Theta_{\mathcal{M}}$ by
  \begin{equation}\label{same face t=1}
    \Theta_{\mathcal{M}}(a,b,s,p,1) = 
    \begin{cases*}
      (-s,r,(a_1\ldots,a_{i-1},\kappa+1,a_{i},\ldots,a_\kappa))\in \mathbf{G}_i & in the case $\kappa > 0$\\
      (-s,0,(1))\in \mathbf{G}_1 & in the case $\kappa=0$.
    \end{cases*}
  \end{equation}
\end{itemize}
\begin{figure}
  \centering
  \def\svgscale{0.8}
\begingroup%
  \makeatletter%
  \providecommand\color[2][]{%
    \errmessage{(Inkscape) Color is used for the text in Inkscape, but the package 'color.sty' is not loaded}%
    \renewcommand\color[2][]{}%
  }%
  \providecommand\transparent[1]{%
    \errmessage{(Inkscape) Transparency is used (non-zero) for the text in Inkscape, but the package 'transparent.sty' is not loaded}%
    \renewcommand\transparent[1]{}%
  }%
  \providecommand\rotatebox[2]{#2}%
  \newcommand*\fsize{\dimexpr\f@size pt\relax}%
  \newcommand*\lineheight[1]{\fontsize{\fsize}{#1\fsize}\selectfont}%
  \ifx\svgwidth\undefined%
    \setlength{\unitlength}{221.48684044bp}%
    \ifx\svgscale\undefined%
      \relax%
    \else%
      \setlength{\unitlength}{\unitlength * \real{\svgscale}}%
    \fi%
  \else%
    \setlength{\unitlength}{\svgwidth}%
  \fi%
  \global\let\svgwidth\undefined%
  \global\let\svgscale\undefined%
  \makeatother%
  \begin{picture}(1,0.69753302)%
    \lineheight{1}%
    \setlength\tabcolsep{0pt}%
    \put(0,0){\includegraphics[width=\unitlength,page=1]{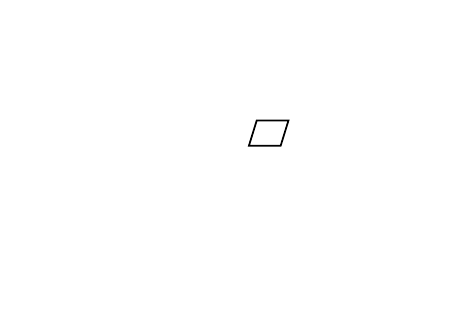}}%
    \put(-0.0025193,0.49760345){\color[rgb]{0,0,0}\makebox(0,0)[lt]{\lineheight{1.25}\smash{\begin{tabular}[t]{l}$J\times G_i$\end{tabular}}}}%
    \put(0,0){\includegraphics[width=\unitlength,page=2]{Theta_M_same_face.pdf}}%
    \put(0.94469643,0.46384237){\color[rgb]{0,0,0}\makebox(0,0)[lt]{\lineheight{1.25}\smash{\begin{tabular}[t]{l}$J$\end{tabular}}}}%
    \put(0,0){\includegraphics[width=\unitlength,page=3]{Theta_M_same_face.pdf}}%
    \put(0.37405468,0.49630316){\color[rgb]{0,0,0}\makebox(0,0)[lt]{\lineheight{1.25}\smash{\begin{tabular}[t]{l}$\Theta_{\mathcal{M}}$\end{tabular}}}}%
  \end{picture}%
\endgroup%

  \caption{An illustration of $\Theta_{\mathcal{M}}$ for when $\mathcal{C}'(x)\times\{p\}$, $\mathcal{C}'(x)\times\{p\}$ are in the same face $\mathbf{G}_i$}
  \label{same face figure}
\end{figure}
We now define $\Lambda_\emptyset(x,t) = \Theta_\emptyset(\tau(x),1-t)$. Essentially, $\Lambda_\emptyset$ moves in reverse, starting at $\mathcal{C}'(x_1)\times \{p_1\}\subset \mathcal{C}(y)$, and ending at $\mathcal{C}'(x_0)\times \{p_0\}\subset \mathcal{C}(y)$. We observe that $\Lambda_\emptyset$ turns ``down'' and around back into $\mathbf{G}_i$, instead of ``up'' and around.
\subsubsection*{Case 1: The matched boundary components are on faces $\mathbf{G}_i, \mathbf{G}_j\subset \partial \mathcal{C}(y)$, $i<j$.}
$\partial \mathcal{C}(y)$ contains distinct faces $\mathbf{G}_i,\mathbf{G}_{j}$. This assumption can only hold when $\kappa > 0$, which is an important distinction, since we can say $J\times \widetilde{\mathcal{M}}_{\mathscr{C}_C(n)}(v,\overline{0}) = J\times [0,1]\times \Pi^{\kappa}$, where $|v|=\kappa+1\geq 2$\par
We first construct a ``convex'' version $\eta^{\conv}:=\{(\Theta^{\conv},0),(\Lambda^{\conv},1)\}$ of the d.s.\ tube $\eta_\emptyset$, and we do so by first specifying $\Theta^{\conv}$.\par
Unlike in Case 0, where we could arbitrarily choose the starting component, there is a designated starting component for $\Theta^{\conv}$: the component in the lower index face. So after a relabeling of $x,x',p,p'$, we have two boundary components $\mathcal{C}(x_0)\times \{p_0\}$, $\mathcal{C}(x_1)\times \{p_1\}$, where importantly, $S_\mathbb{Z}(p_0)=i<S_\mathbb{Z}(p_1)=j$.\par
 $\Theta^{\conv}$ is similarly determined by its projections:
\begin{enumerate}[label={(\arabic*)}]
\item $\Theta_\epsilon^{\conv}: e^m\times [0,1]\to [-\epsilon,\epsilon]^B$\label{epsilon conv}
\item $\Theta_R^{\conv}: e^m\times [0,1]\to [-R,R]^A$\label{R conv}
\item $\Theta_{\mathcal{M}}^{\conv}: e^m\times [0,1] \to J\times \widetilde{\mathcal{M}}_{\mathscr{C}_C(n)}(v,\overline{0})$.\label{M conv}
\end{enumerate}
The projections $\Theta_\epsilon^{\conv}$, $\Theta_R^{\conv}$ are defined exactly the same as in Case 0, but \ref{M conv} is different.
Consider the doubly-parametrized tube $\{V_{ij},V_{ji}\}$ from Example \ref{X example} (see Figure \ref{1 apart 2D}), with
\begin{align*}
  V_{ij}:\left(J\times \Pi^{\kappa-1}\right)\times [0,1]&\hookrightarrow J\times \Pi^{\kappa}\\
  (s,(a_1,\ldots,a_\kappa),t)\mapsto (s,(a_1,\ldots,a_{i-1},\kappa+1,& a_i,a_{i+1},\ldots,a_\kappa))+t(\kappa+1-a_i)(e_{i+1}-e_{i})).
\end{align*}
We define the tube $\Theta^{\conv}_\mathcal{M}$ by
\begin{equation}\label{M conv equation}
\begin{aligned}
  \Theta^{\conv}_\mathcal{M}: e^m\times [0,1]
  &\xrightarrow{\pi\times \text{Id}} \left(J\times [0,1] \times \Pi^{\kappa-1}\right)\times [0,1]\\
  & \xrightarrow{(s,r,p,t)\mapsto (s,p,t,r)} J\times \Pi^{\kappa-1}\times [0,1]\times [0,1]\\
  & \xrightarrow{V_{i,j}\times \text{Id}_{[0,1]}} J\times \Pi^{\kappa}\times [0,1]\\
  &\xhookrightarrow{(s,p,r)\mapsto (s,r,p)} J\times [0,1]\times \Pi^{\kappa},
\end{aligned}
\end{equation}

where $\pi: e^m\to J\times [0,1] \times \Pi^{\kappa-1}$ is the projection onto the factor $J\times \widetilde{\mathcal{M}}_{\mathscr{C}_C(\kappa)}(\overline{1},\overline{0})$. Now to define $\Lambda^{\conv}$, we define $\Lambda_\epsilon^{\conv}(x,t) = \Theta_\epsilon^{\conv}(x,1-t)$, $\Lambda_R^{\conv}(x,t) = \Theta_R^{\conv}(x,1-t)$, and we define $\Lambda^{\conv}$ as the composition in (\ref{M conv equation}), but with $V_{ij}$ replaced by $V_{ji}$.
\begin{figure}
  \scriptsize
  \begin{tabular}{m{4.8cm} m{5.4cm} m{4.8cm} m{0.5cm}}
    \def\svgscale{0.51}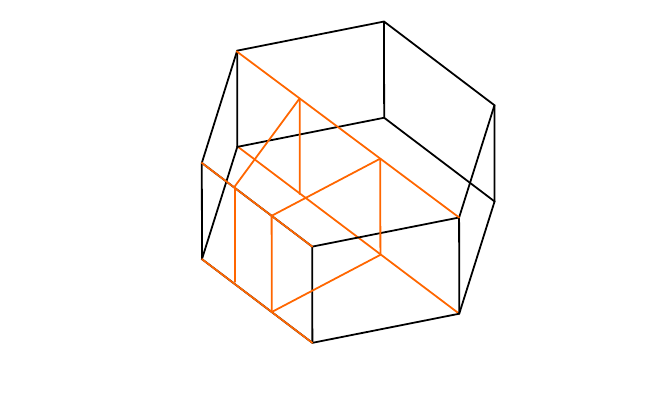 & \def\svgscale{0.51}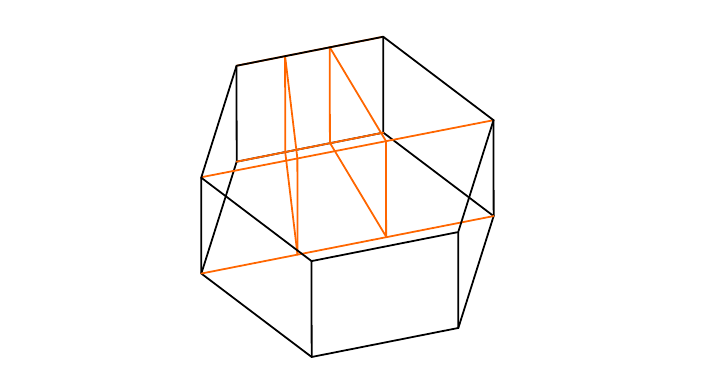 & \def\svgscale{0.51}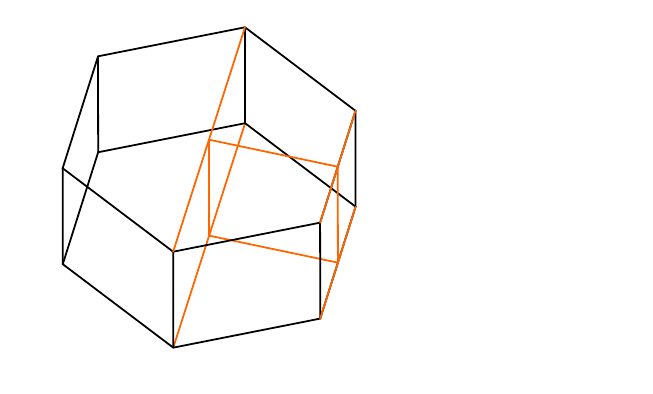 & \def\svgscale{0.51}
\begingroup%
  \makeatletter%
  \providecommand\color[2][]{%
    \errmessage{(Inkscape) Color is used for the text in Inkscape, but the package 'color.sty' is not loaded}%
    \renewcommand\color[2][]{}%
  }%
  \providecommand\transparent[1]{%
    \errmessage{(Inkscape) Transparency is used (non-zero) for the text in Inkscape, but the package 'transparent.sty' is not loaded}%
    \renewcommand\transparent[1]{}%
  }%
  \providecommand\rotatebox[2]{#2}%
  \newcommand*\fsize{\dimexpr\f@size pt\relax}%
  \newcommand*\lineheight[1]{\fontsize{\fsize}{#1\fsize}\selectfont}%
  \ifx\svgwidth\undefined%
    \setlength{\unitlength}{18.39761344bp}%
    \ifx\svgscale\undefined%
      \relax%
    \else%
      \setlength{\unitlength}{\unitlength * \real{\svgscale}}%
    \fi%
  \else%
    \setlength{\unitlength}{\svgwidth}%
  \fi%
  \global\let\svgwidth\undefined%
  \global\let\svgscale\undefined%
  \makeatother%
  \begin{picture}(1,2.78879863)%
    \lineheight{1}%
    \setlength\tabcolsep{0pt}%
    \put(0,0){\includegraphics[width=\unitlength,page=1]{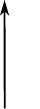}}%
    \put(0.33420436,2.03314602){\color[rgb]{0,0,0}\makebox(0,0)[lt]{\lineheight{1.25}\smash{\begin{tabular}[t]{l}$J$\end{tabular}}}}%
  \end{picture}%
\endgroup%

  \end{tabular}
  \caption{Both pictures: examples of $\Theta^{\conv}_{\mathcal{M}}$ for $|v|=\kappa+1 = 3$}
  \label{1 apart 2D}
\end{figure}
\paragraph{Case 1a: $j=i+1$:}
In this case, simply define $\eta_\emptyset := \eta^{\conv}$.
\paragraph{Case 1b: $j>i+1$:}
In this case, we need to add some twists to $\eta^{\conv}$ in the $J\times \Pi^{\kappa-1}$-plane (see Figures \ref{1 twist fig}, \ref{2 twists fig}).
Let $P\in e^m$ be the midpoint and consider the twists in
\[
  T_P e^m \cong T_{P_1}\left([-\epsilon,\epsilon]^{A+B}\times [0,1]\right)\times T_{P_2} \left(J\times \Pi^{\kappa-1}\right)
\]
that are constant in the first tangent plane, but $\varphi_r$ in the second tangent plane. We call these twists $\boldvarphi_r$. Using this notation, we define $\eta_\emptyset := \eta^{\conv} \diamond(\boldvarphi_{j-1}\diamond\ldots\diamond\boldvarphi_{i+1},0)$.\par
We are finally set up to define our boundary matching tube. The last step is incorporate the signs $\sigma(p_0)$, $\sigma(p_1)$:
\begin{definition}
  Define the doubly specified \textit{boundary matching} tube $\eta = \{\Theta,\Lambda\}$ by
  \begin{equation}
    \Theta(x,t) = \Theta_\emptyset(\tau^{\sigma(p_0)}(x),t),\quad \Lambda(x,t) = \Lambda_\emptyset(\tau^{\sigma(p_1)}(x),t),
  \end{equation}
\end{definition}
We constructed $\eta$ intentionally so it has the property that the restrictions $\Theta|_{e^m\times \{0\}}$ $\Lambda|_{e^m\times \{0\}}$  are the inclusions $e^m\cong C(x_0)\times \{p_0\}\hookrightarrow \mathcal{C}(y)$, $e^m\cong C(x_0)\times \{p_0\}\hookrightarrow \mathcal{C}(y)$. Note that in the unsigned flow category setting, $\eta = \eta_{\emptyset}$.

\subsection{Extending $\mathfrak{c}$ to the tubes $\eta$.}
Here, we outline Step (\ref{extension step}) using our tubes $\eta$ constructed in the previous section, thus finishing our extension of $\mathfrak{c}$ to $\mathcal{C}(y)$. Observe that $\Theta^{-1}|_{\mathcal{C}(x_0)\times \{p_0\}}: \mathcal{C}(x_0)\times \{p_0\}\to e^m\times \{0\}$ agrees with $\mathfrak{c}|_{\mathcal{C}(x_0)\times\{p_0\}}: \mathcal{C}(x_0)\times \{p_0\}\to e^m$. It might seem reasonable to extend $\mathfrak{c}$ to the slices $\Theta(e^m\times \{t\})$ similarly, but recall from equations (\ref{same face t=1}), (\ref{possibly incoherent}) that if the faces $\mathbf{G}_i$, $\mathbf{G}_k$ are even index apart, $(\Theta_{e^m\times \{1\}})^{-1}$ differs from $\mathfrak{c}|_{\mathcal{C}(x_1)\times \{p_1\}}$ by a flip in the $J$ factor. In these cases, $\mathfrak{c}$ must also map to the $(m+1)$-cell of $K_m^{(m+1)}$.
\begin{remark}\label{eta coherence}
  The boundary matching tube $\eta$ is boundary-coherent if and only if the ``flipping'' quantity $\omega = j-i+1+\sigma(p_0)+\sigma(p_1)=0 \in \mathbb{F}_2$.
\end{remark}

\begin{definition}
  Suppose $\eta$ is a boundary matching tube with ends $\mathcal{C}(x_0)\times \{p_0\}\subset \mathbf{G}_i$, $\mathcal{C}(x_1)\times \{p_1\}\subset \mathbf{G}_j$.
  \begin{enumerate}[label = (T-\arabic*), ref = \arabic*]
  \item If $\eta$ is boundary-coherent (see Remark \ref{eta coherence}), we define
    \[
      \mathfrak{c}|_{\eta}: \eta\xrightarrow{\Theta_\sigma^{-1}} e^m\times [0,1]\to  e^m\hookrightarrow K_m^{(m)}.
    \]
    We call $\eta$ (also $\Theta$) a \textit{boundary-coherent} tube in this case. \label{coherent}
  \item If $\eta$ is boundary-incoherent, we define
    \[
      \mathfrak{c}|_{\eta}: \eta\xrightarrow{\Theta_\sigma^{-1}} e^m\times[0,1] \cong e^{m+1}\to K_m^{(m+1)}.
    \]
  \end{enumerate}
  Whether $\eta$ is boundary-coherent or boundary-incoherent, our definition of  $\mathfrak{c}|_{\eta}$ agrees with the prescribed $\mathfrak{c}|_{\mathcal{C}(y)}$ from (\ref{boundary condition}) on the intersection $\eta\cap \partial\mathcal{C}(y) = (\mathcal{C}(x_0)\times \{p_0\})\cup (\mathcal{C}(x_1)\times \{p_1\})$, which confirms our extension to $\eta$ is well-defined.
\end{definition}
\begin{remark}
  The terms ``boundary-coherent'' and ``boundary-incoherent'' tubes $\eta$ should allude to the boundary-coherent and boundary-incoherent \textit{arcs} $\eta$ from \cite{MR3252965}. The two concepts are linked; whereas the identifications $\mathcal{C}(x_i)\times \{p_i\}\xrightarrow{\Theta^{-1}} e^m\times\{i\}$, $i=0,1$ agree with $\mathfrak{c}$ if $\eta$ is a boundary-coherent \textit{tube}, the framing of a boundary-incoherent \textit{arc} in \cite{MR3252965} agrees with the framing of its endpoints. Similarly, $\mathcal{C}(x_1)\times \{p_1\}\xrightarrow{\Theta^{-1}} e^m\times\{1\}$ disagrees with $\mathfrak{c}$ by a flip if $\eta$ is a boundary-incoherent tube, mirroring how the framing of a boundary-incoherent arc disagrees with the framing of one of its endpoints by a ``flip.''
\end{remark}
Let us examine the choices we could have made when following Steps (\ref{boundary condition})-(\ref{extension step}). Steps (\ref{boundary condition}), (\ref{basepoint}), and (\ref{extension step}), were entirely predetermined, but we could have made a lot of choices in Step (\ref{tube step}), affecting our extension of $\mathfrak{c}$ to cells $\partial\mathcal{C}(z)$. In particular, we made pairs of boundary components $\{\mathcal{C}(x)\times \{p\},\mathcal{C}(x')\times \{p'\}\}$, and furthermore designated a starting $(t=0)$ component and ending $(t=1)$ component for each pair. This designation is predetermined if the components lie in different index faces $\mathbf{G}_i,\mathbf{G}_j$ of $\partial\mathcal{C}(y)$, but the designation is arbitrary if $\mathcal{C}(x)\times \{p\},\mathcal{C}(x')\times \{p'\}$ lie in the same face $\mathbf{G}_i$. We encode these choices in what we call a \textit{facewise boundary matching}.
\begin{definition}\label{facewise def}
  Given a cycle $\mu \in C_{\mathcal{M}}^l(\mathscr{C};\mathbb{F}_2)$, we define a \textit{facewise boundary-matching} $\mathfrak{m}=(\mathfrak{b}_y,\mathfrak{s}_y)$, where:
  \begin{itemize}
  \item $\mathfrak{b}_y$ is a fixed point free involution of $\mathcal{M}(y,\mu)$. We can also think of $\mathfrak{b}_y$ as a partition of $\mathcal{M}(y,\mu)$ into disjoint pairs $\{p,\mathfrak{b}_y(p)\}$.
  \item $\mathfrak{s}_y$ is an ordering for each pair $\{p,p'\}\in \mathfrak{b}_y$. We require that if $S_{\mathbb{Z}}(p)<S_{\mathbb{Z}}(p')$, then $\{p,p'\}$ is ordered as $(p,p')$. In other words, if $p$ and $p'$ are in moduli spaces $\mathcal{M}(y,x)$ and $\mathcal{M}(y,x')$ respectively, and $\mathcal{C}(x)\times \{p\}$ and $\mathcal{C}(x')\times \{p'\}$ embed in faces $\mathbf{G}_i,\mathbf{G}_j\subset \partial \mathcal{C}(y)$, $i<j$, then $\mathfrak{s}_y$ orders the edge $\{p,p'\}$ as $(p,p')$.
  \end{itemize}
  Again, the process of facewise boundary matching is illustrated in Figure \ref{facewise picture}.
\end{definition}
Note that if $S_{\mathbb{Z}}(p)=S_{\mathbb{Z}}(p')$, in a matched pair $\{p,p\}$, then $\mathfrak{s}_y$ can order $\{p,p'\}$ as $(p,p')$ or $(p',p)$. This completely arbitrary choice encodes the choice we made in the beginning of Section \ref{eta construction} to designate a ``$t=0$'' component $\mathcal{C}(x_0)\times \{p_0\}$ and a ``$t=1$'' component $\mathcal{C}(x_1)\times \{p_1\}$.
\section{Cycles and their homotopy classes}
We have defined $\mathfrak{c}$ on the $(m+1)$-skeleton of $Y$, and pur next goal is to extend $\mathfrak{c}$ to the $(m+2)$-skeleton, which we can do one cell at a time. Following this idea, we fix an $(m+2)$-cell $\mathcal{C}'(z)$ of $Y'$. ``Most'' of the boundary $\partial\mathcal{C}'(z)$ must map to the basepoint under $\mathfrak{c}$; to explain what we mean, there are inclusions of type
\begin{enumerate}[label = (E-\arabic*), ref = E-\arabic*]
\item $\jmath'_{z,x}:\mathcal{C}'(x)\times \mathcal{M}_{\mathscr{C}'}(z,x)\hookrightarrow \partial\mathcal{C}'(z)$ (see Section \ref{truncation}, Equation (\ref{embedding zx})),\label{zx embedding}
\item $\jmath'_{y,x}:\mathcal{C}'(y)\times \mathcal{M}_{\mathscr{C}'}(z,y)\hookrightarrow \partial\mathcal{C}'(z)$ (see Section \ref{truncation}, Equation (\ref{embedding zy})),\label{zy embedding}
\end{enumerate}
and $\mathfrak{c}$ must map any point outside of these images to the basepoint. Embeddings of type-(\ref{zx embedding}) look like a collection of $\mathcal{C}'(x)$-tubes and embeddings of type-(\ref{zy embedding}) look like a collection of cells $\mathcal{C}'(y)$ (see Figure \ref{attaching pictures} for a visualization).\par
In fact, ``most'' of a type-(\ref{zy embedding}) image must map to the basepoint. Indeed, if we consider an embedded $\mathcal{C}'(y)\times \{q\}\subset \partial\mathcal{C}'(z)$, any point outside of the embedded boundary matching tubes $\eta\times\{q\}$ gets mapped to the basepoint as well.\par
We have narrowed our attention to two types of ``tubes,'' that is, tubes of the form (\ref{zx embedding}), and tubes of the form $\eta\times\{q\}$. We make these notions precise in the following definitions:
\begin{definition}
  Let $\mathcal{C}'(x)$ be an $m$-cell of $Y'$. If we consider a component $I \subset \mathcal{M}_{\mathscr{C}'}(z,x)$, we can define a tube
  \[
    \Theta_\sigma:e^m\times [0,1]\xhookrightarrow{\Id\times \theta} \mathcal{C}'(x)\times I\xhookrightarrow{\jmath'_I} \mathcal{C}'(z),
  \]
  where $\theta:[0,1]\to I$ is an arbitrary parametrization of $I$. By parametrizing $I$ from the other direction, we obtain another tube $\Lambda_\sigma(x,t) = \Theta_\sigma(x,1-t)$, giving us a coherent d.s.\ $\zeta = \{\Theta_\sigma,\Lambda_\sigma\}$. We call $\zeta$ a \textit{Pontrjagin-Thom tube}.
\end{definition}\label{boundary matching in z}
Observe that the restrictions $\Theta_\sigma|_{e^m\times\{0\}}$, $\Theta_\sigma|_{e^m\times\{1\}}$ are precisely the inclusions
\[
  \mathcal{C}'(x)\times \{p_0\}\times\{q_0\}\hookrightarrow \mathcal{C}'(z),\qquad \mathcal{C}'(x)\times \{p_1\}\times\{q_1\}\hookrightarrow \mathcal{C}'(z)
\]
for $\{p_i\}\times\{q_i\} = \theta(i)$, $(i=0,1)$, keeping in mind that $\mathcal{C}'(x)=e^m$.\par
\begin{definition}
  Let $q\in \mathcal{M}_{\mathscr{C}'}(z,y)$, and let $\eta = \{\Theta_\sigma,\Lambda_\sigma\}$ be a boundary matching tube in $\mathcal{C}'(y)$. We ``embed'' $\eta$ in $\partial\mathcal{C}'(z)$ by postcomposing $\Theta_\sigma,\Lambda_\sigma$ with the inclusion $\jmath'_q:\mathcal{C}'(y)\times\{q\}\hookrightarrow \partial\mathcal{C}'(z)$. Namely, define $\widetilde{\Theta}_\sigma:= \jmath'_q\circ \Theta_\sigma$, $\widetilde{\Lambda}_\sigma:= \jmath'_q\circ \Lambda_\sigma$, giving us $\tilde{\eta} = \{\widetilde{\Theta}_\sigma,\widetilde{\Lambda}_\sigma\}$. We reuse our terminology and also call $\tilde{\eta}$ a \textit{boundary matching tube}.
\end{definition}
\begin{notation}
  We defined tubes $\eta$ the previous section and tubes $\tilde{\eta}$ in Definition \ref{boundary matching in z}, both of which we call boundary mataching tubes. We always use the $\sim$ symbol to distinguish the two.
\end{notation}
Let $K\subset \bigl(\bigcup_{\substack{\tilde{\eta} \text{ boundary}\\ \text{matching}}}\tilde{\eta}\bigr) \cup \bigl(\bigcup_{\substack{\zeta \text{ Pontrjagin}\\\text{-Thom}}} \zeta\bigr)$
be a connected component, which we view (for now) as a topological space. $K$ is composed by piecing together a sequence of tubes that alternate between Pontrjagin-Thom and boundary matching. (See Figure \ref{cycle example arc union} for an illustration.)
\begin{figure}
  \begin{tabular}{m{7cm} m{2cm}}
    \def\svgscale{0.7}
\begingroup%
  \makeatletter%
  \providecommand\color[2][]{%
    \errmessage{(Inkscape) Color is used for the text in Inkscape, but the package 'color.sty' is not loaded}%
    \renewcommand\color[2][]{}%
  }%
  \providecommand\transparent[1]{%
    \errmessage{(Inkscape) Transparency is used (non-zero) for the text in Inkscape, but the package 'transparent.sty' is not loaded}%
    \renewcommand\transparent[1]{}%
  }%
  \providecommand\rotatebox[2]{#2}%
  \newcommand*\fsize{\dimexpr\f@size pt\relax}%
  \newcommand*\lineheight[1]{\fontsize{\fsize}{#1\fsize}\selectfont}%
  \ifx\svgwidth\undefined%
    \setlength{\unitlength}{277.37154503bp}%
    \ifx\svgscale\undefined%
      \relax%
    \else%
      \setlength{\unitlength}{\unitlength * \real{\svgscale}}%
    \fi%
  \else%
    \setlength{\unitlength}{\svgwidth}%
  \fi%
  \global\let\svgwidth\undefined%
  \global\let\svgscale\undefined%
  \makeatother%
  \begin{picture}(1,1.35149491)%
    \lineheight{1}%
    \setlength\tabcolsep{0pt}%
    \put(0,0){\includegraphics[width=\unitlength,page=1]{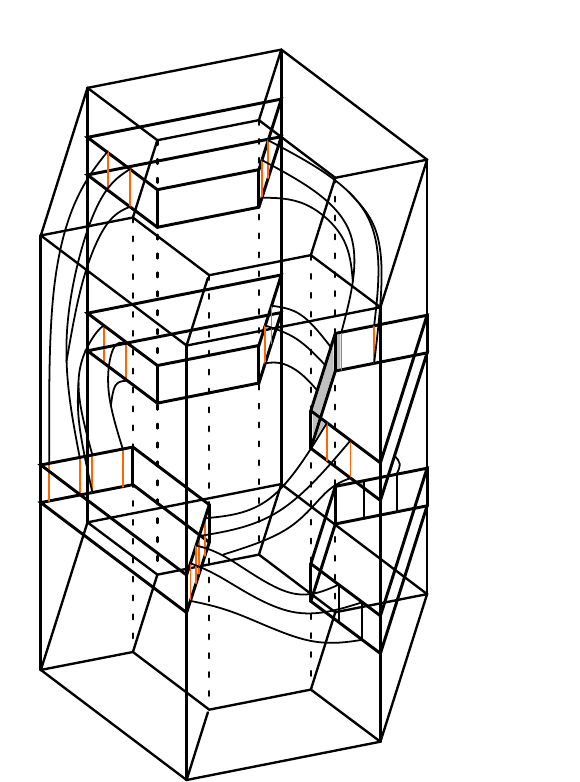}}%
    \put(0.29719168,1.32683491){\color[rgb]{0,0,0}\makebox(0,0)[lt]{\lineheight{1.25}\smash{\begin{tabular}[t]{l}$\partial\mathcal{C}(z)$\end{tabular}}}}%
    \put(0,0){\includegraphics[width=\unitlength,page=2]{6-cycle.pdf}}%
    \put(0.7252351,0.17553794){\color[rgb]{0,0,0}\makebox(0,0)[lt]{\lineheight{1.25}\smash{\begin{tabular}[t]{l}$\mathbf{G}_1(z)$\end{tabular}}}}%
    \put(0.15359897,1.23787694){\color[rgb]{0,0,0}\makebox(0,0)[lt]{\lineheight{1.25}\smash{\begin{tabular}[t]{l}$\mathbf{G}_0(z)$\end{tabular}}}}%
    \put(-0.00201171,0.0700821){\color[rgb]{0,0,0}\makebox(0,0)[lt]{\lineheight{1.25}\smash{\begin{tabular}[t]{l}$\mathbf{G}_2(z)$\end{tabular}}}}%
  \end{picture}%
\endgroup%
 & \def\svgscale{1.3}
\begingroup%
  \makeatletter%
  \providecommand\color[2][]{%
    \errmessage{(Inkscape) Color is used for the text in Inkscape, but the package 'color.sty' is not loaded}%
    \renewcommand\color[2][]{}%
  }%
  \providecommand\transparent[1]{%
    \errmessage{(Inkscape) Transparency is used (non-zero) for the text in Inkscape, but the package 'transparent.sty' is not loaded}%
    \renewcommand\transparent[1]{}%
  }%
  \providecommand\rotatebox[2]{#2}%
  \newcommand*\fsize{\dimexpr\f@size pt\relax}%
  \newcommand*\lineheight[1]{\fontsize{\fsize}{#1\fsize}\selectfont}%
  \ifx\svgwidth\undefined%
    \setlength{\unitlength}{87.04518055bp}%
    \ifx\svgscale\undefined%
      \relax%
    \else%
      \setlength{\unitlength}{\unitlength * \real{\svgscale}}%
    \fi%
  \else%
    \setlength{\unitlength}{\svgwidth}%
  \fi%
  \global\let\svgwidth\undefined%
  \global\let\svgscale\undefined%
  \makeatother%
  \begin{picture}(1,0.718687)%
    \lineheight{1}%
    \setlength\tabcolsep{0pt}%
    \put(0,0){\includegraphics[width=\unitlength,page=1]{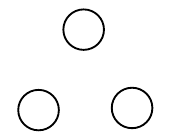}}%
    \put(-0.0055382,0.09677979){\color[rgb]{0,0,0}\makebox(0,0)[lt]{\lineheight{1.25}\smash{\begin{tabular}[t]{l}$0$\end{tabular}}}}%
    \put(0.85171286,0.1008887){\color[rgb]{0,0,0}\makebox(0,0)[lt]{\lineheight{1.25}\smash{\begin{tabular}[t]{l}$2$\end{tabular}}}}%
    \put(0.30281325,0.65088597){\color[rgb]{0,0,0}\makebox(0,0)[lt]{\lineheight{1.25}\smash{\begin{tabular}[t]{l}$1$\end{tabular}}}}%
    \put(0,0){\includegraphics[width=\unitlength,page=2]{6-cycle_parametrization.pdf}}%
  \end{picture}%
\endgroup%

  \end{tabular}
  \caption{Left: a cycle $K$ in $\partial \mathcal{C}(z)\cong \partial \mathcal{C}_{3}$, seen in Figure \ref{cycle example arc union}. Right: the facet cycle $Z$ that parametrizes $K$.}
  \label{cycle example arc union}
\end{figure}
Indeed, the following procedure gives us a sequence $\zeta_1,\tilde{\eta}_1,\zeta_2,\tilde{\eta}_2\ldots\zeta_l,\tilde{\eta}_l$ of consecutive tubes that alternate between Pontrjagin-Thom boundary matching such that $\bigcup_i \tilde{\eta}_i\cup\bigcup_i \zeta_i = K$:
\begin{enumerate}[label = (P-\arabic*), ref = P-\arabic*]
  \item Start with a Pontrjagin-Thom tube $\zeta_1\subset K$, which has end components $\mathcal{C}(x_1)\times\{p_1\}\times\{q_1\}$ and $\mathcal{C}(x_1)\times\{p'_2\}\times\{q_2\}$.\label{step 1}
  \item $\mathcal{C}(x_1)\times\{p'_2\}\subset \mathcal{C}(y_1)$ is boundary-matched with $\mathcal{C}(x_2)\times\{p_2\}\subset \mathcal{C}(y_1)$. There is a boundary matching tube $\tilde{\eta}_1$, which has ends $\mathcal{C}(x_1)\times\{p'_2\}\times\{q_2\}$ and $\mathcal{C}(x_2)\times\{p_2\}\times\{q_2\}$.
  \item The interval $I\subset \mathcal{M}_{\mathscr{C}}(z,x_2)$ containing the endpoint $\{p_2\}\times\{q_2\}$ has the other endpoint $\{p'_3\}\times\{q_3\}$. There exists a Pontrjagin-Thom tube $\zeta_2$ which has ends $\mathcal{C}(x_2)\times\{p'_2\}\times\{q_2\}$ and $\mathcal{C}(x_2)\times\{p_3\}\times\{q_3\}$.
  \item Continue this process until we find a Pontrjagin-Thom tube connected to $\mathcal{C}(x_1)\times\{p_1\}\times\{q_1\}$.\label{step 4}
  \end{enumerate}
  \begin{definition}
    Recall from Lemma \ref{suspension RP2} that $[\partial\mathcal{C}(z),K_m^{(m+1)}]\cong [S^{m+1},K_m^{(m+1)}]\cong \mathbb{Z}/2$, since we have already fixed $m>2$ in this paper. For a cycle $K$ of alternating Pontrjagin-Thom tubes and boundary matching tubes, the restriction $\mathfrak{c}|_K$ specifies an element $[\mathfrak{c}|_K] \in [\partial \mathcal{C}(z),K_m^{(m+1)}]\cong \mathbb{Z}/2$, defined by sending $K$ to $K_m^{(m+1)}$ by $\mathfrak{c}$ and $\partial \mathcal{C}(z)\backslash K$ to the basepoint. By summing along all cycles $K$, we obtain the identity $\sum_{K\subset \partial\mathcal{C}(z)}[\mathfrak{c}|_K] = [\mathfrak{c}|_{\partial\mathcal{C}(z)}]\in \mathbb{Z}/2$.
  \end{definition}
  This quantity $[\mathfrak{c}|_{\partial\mathcal{C}(z)}]$ will end up being our coefficient for $\mathcal{C}(z)$ in $\Sq^2(\mathbf{c})$.\par
  Upon piecing together the whole $K$, we observe $K$ looks like a tubular neighborhood of some closed curve, which we shall define.
  \begin{definition}\label{0-section construction}
    Let $P$ denote the midpoint of $e^m$ and $P_x$ denote the midpoint of $\mathcal{C}(x)$. We define the simple closed curve $\mathcal{K}\subset K$ to be the union of the $\{P\}\times [0,1]$-identified subsets of the $\eta$'s and the $\{P_x\}\times \mathcal{M}_{\mathscr{C}}(z,x)$-identified subsets of the $\zeta$'s.
  \end{definition}
  We observe that $K$ is a fiber bundle over $\mathcal{K}$ with fibers $e^m$, and $\mathcal{K}$ looks like a ``$0$-section'' of $K$. We ask if we can trivialize this fiber bundle, and Proposition \ref{trivialization} answers this question in the affirmative. The style of our proof is similar to \cite{MR3252965} (the discussion preceding Lemma 3.9). But first we need a lemma.
  \begin{lemma}\label{even incoherent}
    In a cycle $K$, there are an even number of boundary-incoherent tubes $\tilde{\eta}$.
  \end{lemma}
  We postpone the proof of this lemma to Section \ref{special graph intro}.
  
\begin{proposition}\label{trivialization}
  There exist exactly two trivializations $\Phi: K\hookrightarrow e^m\times \mathcal{K}$ that satisfy the following:
  \begin{itemize}
  \item For each Pontrjagin-Thom tube $\zeta\subset K$, $\Phi|_{\zeta}$ is equal to
    \[
      \zeta\cong e^m\times \mathcal{K}_\zeta \xhookrightarrow{\tau^\omega \times \Id} e^m\times \mathcal{K}_\zeta,
    \]
    where  $\tau^\omega$ is either $\tau^0 = \Id$ or $\tau^1=\tau$, and $\mathcal{K}_\zeta = \mathcal{K}\cap \zeta$.
  \item For each boundary matching tube $\tilde{\eta}\subset K$, $\Phi|_{\tilde{\eta}}$ is equal to
    \[
      \tilde{\eta} \xhookrightarrow{\Theta^{-1}} e^m\times [0,1] \xhookrightarrow{\tau^\omega \times \Theta(P,\cdot)} e^m\times \mathcal{K}_{\tilde{\eta}},
    \]
    where $\tau^\omega\in\{\Id,\tau\}$, $\mathcal{K}_{\tilde{\eta}} = \mathcal{K}\cap \tilde{\eta}$.  
  \end{itemize}
  Furthermore, these two trivializations $\Phi,\Phi'$ are related by a flip in the $J$-factor $(\Phi' = (\tau\times \Id_{\mathcal{K}})\circ \Phi)$. 
\end{proposition}
\begin{proof}
  Use Steps (\ref{step 1})--(\ref{step 4}) to find a sequence $\zeta_1,\tilde{\eta}_1,\zeta_2,\tilde{\eta}_2,\ldots,\zeta_l,\tilde{\eta}_l$ of consecutive d.s.\ tubes whose union is $K$. Following Definition \ref{tube compose operation}, we construct a single d.s.\ tube $T = \{\overline{\Theta}_\sigma,\overline{\Lambda}_\sigma\} := \zeta_1\cup\tilde{\eta}_1\cup\zeta_2\cup\tilde{\eta}_2\cup\ldots\cup\zeta_l\cup\tilde{\eta}_l$, with free ends being an end of $\zeta_1$ and an end of $\tilde{\eta}_l$. All of the $\zeta_i$'s are boundary-coherent, and by Lemma \ref{even incoherent}, an even number of the $\tilde{\eta}_i$'s are boundary-incoherent, so in total, an even number of the component tubes of $T$ are boudary-incoherent. By induction, we find that $T$ must be boundary-coherent. Therefore, the tube $\overline{\Theta}_\sigma:e^m\times [0,1]\to K$ is identical on the ends $e^m\times\{0\},e^m\times\{1\}$, and so factors through a homeomorphism $e^m\times S^1\xhookrightarrow{\sim} K$. We define the $e^m$-component $\Phi_{e^m}$ of $\Phi$ by the composition $K\xhookrightarrow{\overline{\Theta}_\sigma^{-1}}e^m\times S^1\xrightarrow{\pi} e^m$, and we define the $\mathcal{K}$-component $\Phi_{\mathcal{K}}$ of $\Phi$ to be the canonical projection to the ``$0$-section.'' Now simply define $\Phi':K\to e^m\times \mathcal{K}$ by $\Phi' = (\tau\times \Id) \circ \Phi$.\par
  The proof of uniqueness of the pair $\Phi,\Phi'$ is a consequence that a trivialization is determined by its value on any given subtube $\zeta\subset K$.
  \end{proof}
  \begin{definition}\label{K element}
  Let $\Phi,\Phi'$ denote the canonical pair of trivializations of $K$ from Proposition \ref{trivialization}. The pair $(K,\Phi)$ (resp.\ $(K,\Phi')$) specifies a class $[K,\Phi]$ (resp.\ $[K,\Phi']$) in $[\partial \mathcal{C}(z),e^m/\partial e^m] \cong [S^{m+1},S^m]\cong \mathbb{Z}/2$, defined by the composition $\pi \circ \Phi$ (resp.\ $\pi \circ \Phi'$) on $K$ and a constant map to the basepoint on $\partial C(z)\backslash K$. Since $\pi\circ \Phi$ differs from $\pi\circ \Phi'$ by a flip in the $J$-factor, $[K,\Phi]$ and $[K,\Phi']$ are actually the same element, which we call $[K]$.
\end{definition}
We now prove that the classes $[\mathfrak{c}|_K],[K]$ are identical. We start by building a diagram. First consider the projection $\pi:e^m\times \mathcal{K}\to e^m$ to the $e^m$-factor. Fix $\mathcal{C}(x_1)\times\{p_1\}\times\{q_1\}\subset K$, and let $\Phi:K\to e^m\times \mathcal{K}$ be the trivialization of $K$ defined in Proposition \ref{trivialization} such that $\pi\circ\Phi$ agrees with $\mathfrak{c}|_K$ on $\mathcal{C}(x_1)\times\{p_1\}\times\{q_1\}$. These maps are summarized in the following diagram:
\[
  \begin{tikzpicture}[scale=2]
  \path[draw, ->, shorten <=0.3cm,shorten >=1.0cm, thick] (0,1) -- (2,1) node [midway, label={[label distance=-0.2cm]90:$\mathfrak{c}$}] {};
  \path[draw, ->, shorten <=0.7cm,shorten >=0.7cm, thick] (0,0) -- (2,0) node [midway, label={[label distance=-0.2cm]90:$\pi$}] {};
  \path[draw, ->, shorten <=0.3cm,shorten >=0.3cm, thick] (0,0) -- (0,1) node [midway, label={[label distance=-0.1cm]180:$\Phi^{-1}$}] [midway, label={[label distance=-0.3cm, rotate=-90]45:$\cong$}] {};
  \path[draw, {Hooks[right]}->, shorten <=0.3cm,shorten >=0.3cm, thick] (2,0) -- (2,1) {};
     \node at (0, 1) {$K$};
   \node at (2, 1) {$K_m^{(m+1)}$};
   \node at (0, 0) {$e^m\times \mathcal{K}$};
   \node at (2, 0) {$e^m$};
   \end{tikzpicture}
\]

We will show that the above diagram commutes up to homotopy.
\begin{lemma}\label{trivialization equivalent}
  The restriction $\mathfrak{c}|_{K}$ is homotopic to $\pi\circ \Phi$ relative $\mathcal{C}(x_1)\times\{p_1\}\times\{q_1\}\cup \partial K$. In other words, the above square commutes up to homotopy relative $(e^m\times\left(\{P\}\times\{p_1\}\times\{q_1\}\right))\cup (\partial e^m\times \mathcal{K})$.
\end{lemma}
\begin{proof}
  We consider the composition $\mathfrak{c}\circ\Phi^{-1}:e^m\times \mathcal{K}\to K_m^{(m+1)}$ and the map $\pi:e^m\times \mathcal{K}\to K_m^{(m+1)}$. Using the characterization (\ref{m-cell redone}) of $K_m^{(m+1)}$ (and a reshuffling of the $J$ component in $e^m\times \mathcal{K}$), we can think of both $\mathfrak{c}\circ\Phi^{-1}$ and $\pi$ as maps
  \begin{align}
    &\left( [-\epsilon,\epsilon]^A\times [-\epsilon,\epsilon]^B \times \widetilde{\mathcal{M}}_{\mathscr{C}_C(\kappa)}(\overline{1},\overline{0}) \right)\times \left( J\times \mathcal{K}\right)\label{domain m+1 skeleton}\\
                                                    &\to \bigslant{\left( [-\epsilon,\epsilon]^A\times [-\epsilon,\epsilon]^B\times \widetilde{\mathcal{M}}_{\mathscr{C}_C(\kappa)}(\overline{1},\overline{0})\right)}{\partial}\wedge \bigslant{\left(\frac{J\times [0,1]}{((t,0)\sim (-t,1)}\right)}{\partial}\label{range m+1 skeleton}
  \end{align}
  Note that in both $\mathfrak{c}\circ\Phi^{-1},\pi$, the first factor of (\ref{domain m+1 skeleton}) maps by quotient to the first factor of (\ref{range m+1 skeleton}). Furthermore, the boundary of the second factor of (\ref{domain m+1 skeleton}) maps to the basepoint of (\ref{range m+1 skeleton}). Therefore, both maps factor through maps
  \begin{align*}
    &\bigslant{\left( [-\epsilon,\epsilon]^A\times [-\epsilon,\epsilon]^B \times \widetilde{\mathcal{M}}_{\mathscr{C}_C(\kappa)}(\overline{1},\overline{0}) \right)}{\partial}\wedge \bigslant{\left( J\times \mathcal{K}\right)}{\partial} \\
                                                    &\to \bigslant{\left( [-\epsilon,\epsilon]^A\times [-\epsilon,\epsilon]^B\times \widetilde{\mathcal{M}}_{\mathscr{C}_C(\kappa)}(\overline{1},\overline{0})\right)}{\partial}\wedge \bigslant{\left(\frac{J\times [0,1]}{((t,0)\sim (-t,1)}\right)}{\partial},
  \end{align*}
  which are both $(m-1)$-fold suspensions of maps
  \[
    f_1, f_2:\bigslant{\left(J\times \mathcal{K}\right)}{\partial}\to \bigslant{\left(\frac{J\times [0,1]}{((t,0)\sim (-t,1)}\right)}{\partial}\cong \mathbb{R}P^2,
  \]
  where $f_1$ is induced from $\mathfrak{c}\circ\Phi^{-1}$ and $f_2$ is induced from $g\circ \pi_2$. $f_1$ and $f_2$ agree on the subset $\{p_1\}\times \{q_1\}\times J\in \mathcal{K}\times J$, which implies that their homotopy classes (relative boundary) differ by an element $c\in \pi_2(\mathbb{R}P^2)\cong \mathbb{Z}$. Taking suspensions, we see that $\Sigma^{m-1}f_1 = \Sigma^{m-1}f_2 + \Sigma^{m-1}c$. But the suspension map $\Sigma^i:\pi_2(\mathbb{R}P^2)\to \pi_{i+2}(\Sigma^i\mathbb{R}P^2)$ is nullhomotopic for $i\geq 2$ by Lemma \ref{suspension RP2}. Since $m \geq 3$, we have $\Sigma^{m-1}f_1 \cong \Sigma^{m-1}f_2$ relative $(\{p_1\}\times \{q_1\}\times J)\wedge e^m$. Therefore, $\mathfrak{c}\circ\Phi^{-1}$ and $g\circ \pi$ are homotopic relative $(e^m\times\{p_1\}\times\{q_1\})\cup (\mathcal{K}\times \partial e^m)$.
\end{proof}
\begin{proposition}
  For any cycle $K\subset \partial \mathcal{C}(z)$, $[K] = [\mathfrak{c}|_K]$.
\end{proposition}
\begin{proof}
  Fix $\mathcal{C}(x_1)\times\{p_1\}\times\{q_1\}\subset K$ as in Lemma \ref{trivialization equivalent}, and let $\Phi,\Phi':K\to e^m\times\mathcal{K}$ denote the canonical trivializations. By Lemma \ref{trivialization equivalent}, either $[\mathfrak{c}|_K]=[K,\Phi]$ or $[\mathfrak{c}|_K]=[K,\Phi']$, which are in both cases equal to $[K]$.
\end{proof}
\section{Introduction to special graph structures}\label{special graph intro}
 We define a special graph structure in a manner similar to a structure of the same name in \cite{MR4165986}, but not identical:
\begin{definition}[compare \cite{MR4165986}]\label{special graph structure def}
  A \textit{special graph structure} $\Gamma =
  \Gamma(V,E,E',E'',S)$ consists of a set of vertices $V$ together with a function $S: V\to \mathbb{F}_2$, a subset $E$ of edges, a subset $E'\subset E$ of edges, and a subset $E''\subset E - E'$ of directed edges.
  Furthermore, $\Gamma$ must satisfy the following criteria:
  \begin{enumerate}[label = (G-\arabic*):, ref = G-\arabic*]
  \item Each vertex is contained in two edges, with exactly one of the edges $e(v)$ being in $E'$.
  \item If $e\in E'$ and $e=\{v_1,v_2\}$, then $S(v_1)\neq S(v_2)$.
  \end{enumerate}
  A \textit{cubical special graph structure} $\Gamma =
  \Gamma(V,E,E',E'',S, S_\mathbb{Z},\sigma_2)$ is a special graph structure $\Gamma$ equipped with maps $S_\mathbb{Z}: V\to \mathbb{Z}$ $\sigma_2:V\to \mathbb{F}_2$ satisfying:
  \begin{enumerate}[label = (G-\arabic*):, ref = G-\arabic*]
    \setcounter{enumi}{2}
  \item $S_\mathbb{Z}(v_1) = S_\mathbb{Z}(v_2)$ if $\{v_1,v_2\}\in E\backslash E'$.
  \item $\sigma_2(v_1) = \sigma_2(v_2)$ if $\{v_1,v_2\}\in E\backslash E'$.
  \item $S_\mathbb{Z}(v_1) \neq S_\mathbb{Z}(v_2)$ if $\{v_1,v_2\}\in E'$.
  \end{enumerate}
\end{definition}
\begin{example}\label{R graph}
  Let $\mathscr{C}$ be a signed cubical flow category, $\mu\in C^l(\mathscr{C};\mathbb{F}_2)$ a cocycle, and $(\mathfrak{b}_y,\mathfrak{s}_y)$ a facewise boundary matching for $\mu$. Given $z\in \Ob(\mathscr{C})$ of grading $\gr(z) = \gr(x) + 2$, we define a cubical special graph structure $\Gamma(z,\mu)$ as follows. The vertex set $V$ is the disjoint union
  \[
    V = \coprod_{\gr(y) = l+1} \mathcal{M}(y,\mu)\times \mathcal{M}(z,y),
  \]
  (Think of $V$ as the set of chains $z\xrightarrow{q} y\xrightarrow{p} x$, $\gr(x) = \gr(y)-1 = \gr(z)-2$, where $x\in \mu$ and $y\in \Ob(\mathscr{C})$ is some intermediate object.) Each interval component $I\subset \mathcal{M}(z,x)$, $x\in \mu$ defines an edge in $E'$: if $\partial I = \{p_1\circ q_1, p_2\circ q_2\}$, then we have the edge $e = \{(p_1,q_1),(p_2,q_2)\}\in E$.\par
  We define the remaining edges $E\backslash E'$ as follows: If $p\in \mathcal{M}(y,x)$ is boundary matched with $p'\in\mathcal{M}(y,x')$, then for all $q\in\mathcal{M}(z,y)$, we have an edge $e = \{(p,q),(p',q)\}\in E\backslash E'$. Furthermore, if $(p,p')\in \mathfrak{s}_y$, then we require that $e$ is directed from $(p,q)$ to $(p',q)$. Our construction directs every edge in $E\backslash E'$, so we must define $E'':= E\backslash E'$. For an example of a cycle in $\widetilde{\Gamma}(z,\mu)$, see Figure \ref{graph counting incoherent}.\par
  We define $S:V\to\mathbb{F}_2$ by $S(p,q) := S(p) + S(q)$, where the latter $S$ denotes the cubical sign assignment $S(p) = s(\mathfrak{f}(p))+\sigma(p)$ from Example \ref{canonical sign assignment}. Finally, we define $S_{\mathbb{Z}}(p,q) = S_{\mathbb{Z}}(q)$, where the latter $S_{\mathbb{Z}}$ denotes the index assignment. $\sigma_2:V\to \mathbb{F}_2$ is defined by $\sigma_2(p,q) = \sigma(q)$; in other words, $\sigma_2(v)$ is the sign map $\sigma$ applied to the second factor of $V$.
\end{example}
\begin{remark}\label{graph correspondence}
  The cubical special graph structure $\Gamma(z,\mu)$ is a tool that records the relevant behavior of each cycle $K$. Let us start with the vertices $v=(p,q)$, which correspond with the embedded $\mathcal{C}(x)\times \{p\}\times\{q\}$. Furthermore, edges $\{v,v'\}$ correspond to tubes, with the ends $v,v'$ corresponding to tube ends. Edges $e \in E'$ correspond with Pontrjagin-Thom tubes $\mathcal{C}(x)\times I$ and the edges $e\in E\backslash E'$ corresponding with boundary matching tubes $\tilde{\eta}$. These correspondences piece together so that every cycle $K$ constructed in (\ref{step 1})--(\ref{step 4}) indeed corresponds with a graph cycle $C = C(K)$ in $\Gamma(z,\mu)$. Furthermore, the orientation of the edges in $E''$ encodes the embedding of tubes $\tilde{\eta}$ that match a face with itself.\par
  Finally, the map $S_\mathbb{Z}$ encodes the facets $\mathbf{G}_i$ each embedded $\mathcal{C}(x)\times \{p\}\times\{q\}$ lives in the facets $\mathbf{G}_{ij}$ each Pontrjagin-Thom tube $\zeta$ lives in.
\end{remark}
  In view of Remark \ref{graph correspondence}, we say $C_K$ \textit{parametrizes} $K$.\par
  \paragraph{Proof of Lemma \ref{even incoherent}:} Consider the graph $\Gamma(z,\mu)$ (see Figure \ref{graph counting incoherent} for an example).
  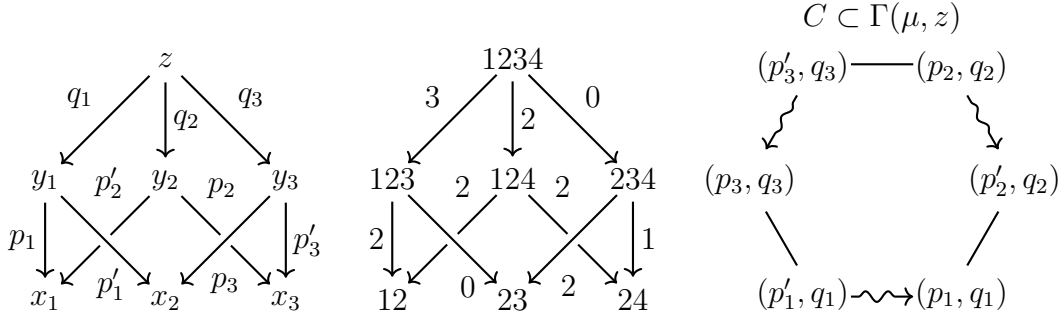
\begin{figure}
    \begin{tabular}{c c c}
      \begin{tikzpicture}[scale=1.6]
  \path[draw, <-, shorten <=\T,shorten >=\T, thick] (0,1) -- (1,2) node [midway, label={[label distance=-0.2cm]135:$q_1$}] {};
  \path[draw, <-, shorten <=\T,shorten >=\T, thick] (1,1) -- (1,2) node [midway, label={[label distance=-0.2cm]0:$q_2$}] {};
  \path[draw, <-, shorten <=\T,shorten >=\T, thick] (2,1) -- (1,2) node [midway, label={[label distance=-0.2cm]45:$q_3$}] {};
  \path[draw, <-, shorten <=\T,shorten >=\T, thick] (0,0) -- (0,1) node [midway, label={[label distance=-0.2cm]180:$p_1$}] {};
  \path[, draw, <-, shorten <=\T,shorten >=\T, thick] (2,0) -- (2,1) node [midway, label={[label distance=-0.2cm]0:$p_3'$}] {};
  \path[draw, <-, shorten <=\T,shorten >=\T, thick] (1,0) -- (2,1) node [pos=0.3, label={[label distance=-0.25cm]-45:$p_3$}] {};
  \path[draw, <-, shorten <=\T,shorten >=\T, thick] (1,0) -- (0,1) node [pos=0.3, label={[label distance=-0.35cm]225:$p_1'$}] {};
  \path[draw, <-, shorten <=\T,shorten >=\P, thick] (0,0) -- (0.5,0.5);
  \path[draw, <-, shorten <=\T,shorten >=\P, thick] (2,0) -- (1.5,0.5);
  \draw[shorten <=\T,shorten >=\P, thick] (1,1) -- (0.5,0.5) node [midway, label={[label distance=-0.2cm]135:$p_2'$}] {};
  \draw[shorten <=\T,shorten >=\P, thick] (1,1) -- (1.5,0.5) node [midway, label={[label distance=-0.2cm]45:$p_2$}] {};
   \node at (1, 2) {$z$};
   \node at (0, 1) {$y_1$};
   \node at (1, 1) {$y_2$};
   \node at (2, 1) {$y_3$};
   \node at (0, 0) {$x_1$};
   \node at (1, 0) {$x_2$};
   \node at (2, 0) {$x_3$};
 \end{tikzpicture} &
 \begin{tikzpicture}[scale=1.6]
  \path[draw, <-, shorten <=\T,shorten >=\T, thick] (0,1) -- (1,2) node [midway, label={[label distance=-0.2cm]135:$3$}] {};
  \path[draw, <-, shorten <=\T,shorten >=\T, thick] (1,1) -- (1,2) node [midway, label={[label distance=-0.2cm]0:$2$}] {};
  \path[draw, <-, shorten <=\T,shorten >=\T, thick] (2,1) -- (1,2) node [midway, label={[label distance=-0.2cm]45:$0$}] {};
  \path[draw, <-, shorten <=\T,shorten >=\T, thick] (0,0) -- (0,1) node [midway, label={[label distance=-0.2cm]180:$2$}] {};
  \path[draw, <-, shorten <=\T,shorten >=\T, thick] (2,0) -- (2,1) node [midway, label={[label distance=-0.2cm]0:$1$}] {};
  \path[draw, <-, shorten <=\T,shorten >=\T, thick] (1,0) -- (2,1) node [pos=0.3, label={[label distance=-0.2cm]-45:$2$}] {};
  \path[draw, <-, shorten <=\T,shorten >=\T, thick] (1,0) -- (0,1) node [pos=0.25, label={[label distance=-0.3cm]225:$0$}] {};
  \path[draw, <-, shorten <=\T,shorten >=\P, thick] (0,0) -- (0.5,0.5);
  \path[draw, <-, shorten <=\T,shorten >=\P, thick] (2,0) -- (1.5,0.5);
  \draw[shorten <=\T,shorten >=\P, thick] (1,1) -- (0.5,0.5) node [midway, label={[label distance=-0.2cm]135:$2$}] {};
  \draw[shorten <=\T,shorten >=\P, thick] (1,1) -- (1.5,0.5) node [midway, label={[label distance=-0.2cm]45:$2$}] {};
  \node at (1, 2) {$1234$};
  \node at (0, 1) {$123$};
  \node at (1, 1) {$124$};
  \node at (2, 1) {$234$};
  \node at (0, 0) {$12$};
  \node at (1, 0) {$23$};
  \node at (2, 0) {$24$};
\end{tikzpicture} &
\begin{tikzpicture}[scale = 0.8]
  \path[thick, draw,->,decorate,
  decoration={snake,amplitude=.7mm,segment length=3mm,post length=1mm, pre length = 1mm}] (255:2)--(285:2);
   \path[thick, draw,->,decorate,
  decoration={snake,amplitude=.4mm,segment length=3mm,post length=1mm, pre length = 1mm}] (45:2) -- (15:2);
  \path[thick, draw,->,decorate,
  decoration={snake,amplitude=.4mm,segment length=3mm,post length=1mm, pre length = 1mm}] (135:2)--(165:2);
  \draw[thick] (75:2)--(105:2);
  \draw[thick] (195:2)--(225:2);
  \draw[thick] (315:2)--(345:2);
  \node at (90: 2.7) {$C\subset \Gamma(\mu,z)$};
  \node at (0:2.2) {$(p_2',q_2)$};
  \node at (55:2.3) {$(p_2,q_2)$};
  \node at (125:2.3) {$(p_3',q_3)$};
  \node at (180:2.2) {$(p_3,q_3)$};
  \node at (235:2.3) {$(p_1',q_1)$};
  \node at (305:2.3) {$(p_1,q_1)$};
\end{tikzpicture}
    \end{tabular}
    \caption{Left: An example of a subset of chains $z\xrightarrow{q} y\xrightarrow{p} x$, $x\in \mu$.  Middle: the functor $\mathfrak{f}$ applied to the left diagram (the arrows are labeled by the signs $S_{\mathbb{Z}}$). Right: A cycle $C\subset \Gamma_{\mathfrak{m}}(\mu, z)$ formed by these chains. We assume the squares in the left diagram correspond to components $I\subset \mathcal{M}(z,x)$ (equivalently, from Pontrjagin-Thom tubes), and thus an edge $e\in E$ which we draw as a straight line. We also assume the pairs $\{(p,p'\}$ are boundary-matched meaning there are edges $e\in E\backslash E'$ (which we draw as squiggly) between vertices $(p_i,q_i)$, $(p_i',q_i)$. Note how in this example, $S_{\mathbb{Z}}(p_1)>S_{\mathbb{Z}}(p_1')$ and $S_{\mathbb{Z}}(p_3)>S_{\mathbb{Z}}(p_3')$, so we must direct the edges $\{(q_1,p_1'),(q_1,p_1)\}$, $\{(q_3,p_3'),(q_3,p_3)\}$ as $(q_1,p_1')\rightsquigarrow (q_1,p_1)$ and $(q_3,p_3')\rightsquigarrow (q_1,p_3)$.}
    \label{graph counting incoherent}
  \end{figure}
  \begingroup
  \allowdisplaybreaks
  \begin{align*}
    \#&\{\text{boundary matching tubes in $K$}\}\\
    &=\#\{\text{edges $e\in E'$ in $C_K$}\} \\
    &= \sum_{\substack{\text{edges $e\in E'$}\\\text{in $C_K$}}} 1\\
    &= \sum_{\substack{\text{edges}\\\{v,v'\}\in E'}} S(v')+S(v')\\
    &=\sum_{\text{vertices $v$}}S(v)\\
    &= \sum_{\substack{\text{edges}\\\{v,v'\}\in E\backslash E'}}S(v)+S(v')\\
    &=\#\{\text{edges $\{v,v'\}\in E\backslash E'$ in $C_K$}: S(v)\neq S(v')\}\\
    &=\#\{\text{boundary-coherent tubes in $K$}\}\mod 2
  \end{align*}
  \endgroup
  Subtracting both ends of the equation by $\#\{\text{boundary-coherent tubes in $K$}\}$, we observe the number of boundary-incoherent arcs must be even.
  \section{Parametrizing cycles $K$}
Consider cycles $K$ living in $\partial\mathcal{C}(z)$. Boundary matching components lie in the faces $\mathbf{G}_0,\ldots \mathbf{G}_{\kappa+1}$ and the Pontrjagin-Thom components straddle two distinct faces. Recall that the our choice of boundary matching determines our cycles $K$, and this data is recorded in our cubical special graph structure $\Gamma(z,\mu)$.
\begin{notation}\label{graph notation}
  If $\Gamma$ is a cubical special graph structure, $e\in E\backslash E'$, and $S_\mathbb{Z}(e) = b$, we write $e$ as $e_b$.\par
  We will have to write out cycles $C$ in $\Gamma$, so we introduce notation for these cycles. We abbreviate a portion $v_1'\xsquigline{e_a} v_1\longline v_2'\allowbreak\xsquigline{e_b}v_2\longline v_3'\xsquigline{e_c}v_3$ of $C$ as $e_a\longline e_b\longline e_c$, essentially treating the edges in $E\backslash E'$ as vertices.
  If, the edge $e_b$ is oriented, we can write $e_a\longline\overrightarrow{e_b}\longline e_c$ or $e_a\longline\overleftarrow{e_b}\longline e_c$ depending on the orientation of $e_b$. Let us clarify that while the arrow in $\overrightarrow{e_b}$ is drawn to emphasize orientation, we are not required to draw an arrow under an oriented $e_b$.
\end{notation}
\begin{definition}\label{facet cycles}
  Let $\Gamma = \Gamma(V,E,E',E'',S, S_\mathbb{Z})$ be a cubical special graph structure, and let $C\subset \Gamma$ be a cycle, denoted by $e_{a_1}\longline e_{a_2}\longline\ldots\longline e_{a_r}\longline e_{a_1}$. We define the \textit{facet cycle} of $C$ as the $\mathbb{Z}$-valued cycle $a_1\longline a_2\longline \ldots\longline a_r\longline a_1$, which we call $Z(C)$. The indices of $Z(C)$ inherit an orientation from $C$, appearing as $a\longline \overrightarrow{b}\longline c$ if $e_a\longline \overrightarrow{e_b}\longline e_c$.\par
  If $\Gamma = \Gamma(V,E,E',E'',S, S_\mathbb{Z},\sigma_2)$ is signed, we wish to include the data of the map $\sigma_2:V\to \mathbb{F}_2$ in this cycle. Let $\omega_i=\sigma_2(v)$, where $v$ is either vertex in $e_{a_i}$. We define the \textit{signed facet cycle} of $C$ as the cycle $Z(C)$ denoted by $(a_1,\omega_1)\longline (a_2,\omega_2)\longline\ldots \longline (a_r,\omega_r)\longline (a_1,\omega_1)$.
Just as in the previous paragraph, the indices of $Z(C)$ inherit their orientations from the corresponding cycle $C$ in $\Gamma$.
\end{definition}
Let us consider the case $\Gamma = \Gamma(z,\mu)$, with $C$ a cycle in $\Gamma$ parametrizing $K$. The corresponding facet cycle $Z=Z(C)$ records the essential information of $K$. Indeed, if $C$ is such a cycle, then $Z(C)$ records the sequences of facets $\mathbf{G}_i$ that $K$ passes through (see Figure \ref{cycle example arc union} for an example illustration).
And furthermore, the orientation of the indices $\overrightarrow{b}$, $\overleftarrow{b}$ records how $K$ behaves in turnarounds. For instance, in a portion $a\longline\overrightarrow{b}\longline a$ of $Z$, we know that the corresponding portion of $K$ travels through $\mathbf{G}_a$ into $\mathbf{G}_b$, turns up (in the $+J$ direction) and around in $\mathbf{G}_b$, and travels back towards $\mathbf{G}_a$ (see Figure \ref{parametrization turnarounds} for an example illustration).\par
\begin{figure}
  \begin{tabular}{m{9cm} c}
    \def\svgscale{0.70}
\begingroup%
  \makeatletter%
  \providecommand\color[2][]{%
    \errmessage{(Inkscape) Color is used for the text in Inkscape, but the package 'color.sty' is not loaded}%
    \renewcommand\color[2][]{}%
  }%
  \providecommand\transparent[1]{%
    \errmessage{(Inkscape) Transparency is used (non-zero) for the text in Inkscape, but the package 'transparent.sty' is not loaded}%
    \renewcommand\transparent[1]{}%
  }%
  \providecommand\rotatebox[2]{#2}%
  \newcommand*\fsize{\dimexpr\f@size pt\relax}%
  \newcommand*\lineheight[1]{\fontsize{\fsize}{#1\fsize}\selectfont}%
  \ifx\svgwidth\undefined%
    \setlength{\unitlength}{308.42898632bp}%
    \ifx\svgscale\undefined%
      \relax%
    \else%
      \setlength{\unitlength}{\unitlength * \real{\svgscale}}%
    \fi%
  \else%
    \setlength{\unitlength}{\svgwidth}%
  \fi%
  \global\let\svgwidth\undefined%
  \global\let\svgscale\undefined%
  \makeatother%
  \begin{picture}(1,0.98586898)%
    \lineheight{1}%
    \setlength\tabcolsep{0pt}%
    \put(0,0){\includegraphics[width=\unitlength,page=1]{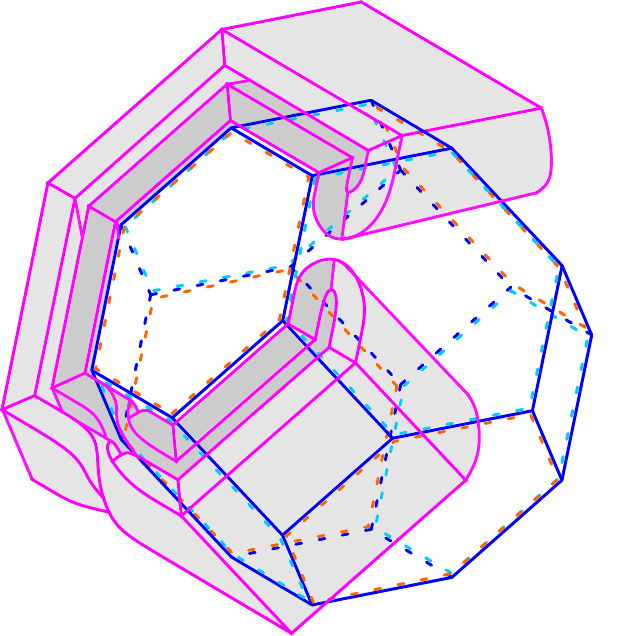}}%
    \put(0.62504889,0.55871334){\color[rgb]{0,0,0}\makebox(0,0)[lt]{\lineheight{1.25}\smash{\begin{tabular}[t]{l}$G_0$\end{tabular}}}}%
    \put(0.14576019,0.86120178){\color[rgb]{0,0,0}\makebox(0,0)[lt]{\lineheight{1.25}\smash{\begin{tabular}[t]{l}$G_3$\end{tabular}}}}%
    \put(0.08472301,0.1494855){\color[rgb]{0,0,0}\makebox(0,0)[lt]{\lineheight{1.25}\smash{\begin{tabular}[t]{l}$G_1$\end{tabular}}}}%
    \put(0.91736184,0.22952456){\color[rgb]{0,0,0}\makebox(0,0)[lt]{\lineheight{1.25}\smash{\begin{tabular}[t]{l}$G_2$\end{tabular}}}}%
    \put(0,0){\includegraphics[width=\unitlength,page=2]{cycle_with_turnarounds.pdf}}%
  \end{picture}%
\endgroup%
 & \def\svgscale{1.4}
\begingroup%
  \makeatletter%
  \providecommand\color[2][]{%
    \errmessage{(Inkscape) Color is used for the text in Inkscape, but the package 'color.sty' is not loaded}%
    \renewcommand\color[2][]{}%
  }%
  \providecommand\transparent[1]{%
    \errmessage{(Inkscape) Transparency is used (non-zero) for the text in Inkscape, but the package 'transparent.sty' is not loaded}%
    \renewcommand\transparent[1]{}%
  }%
  \providecommand\rotatebox[2]{#2}%
  \newcommand*\fsize{\dimexpr\f@size pt\relax}%
  \newcommand*\lineheight[1]{\fontsize{\fsize}{#1\fsize}\selectfont}%
  \ifx\svgwidth\undefined%
    \setlength{\unitlength}{108.41028529bp}%
    \ifx\svgscale\undefined%
      \relax%
    \else%
      \setlength{\unitlength}{\unitlength * \real{\svgscale}}%
    \fi%
  \else%
    \setlength{\unitlength}{\svgwidth}%
  \fi%
  \global\let\svgwidth\undefined%
  \global\let\svgscale\undefined%
  \makeatother%
  \begin{picture}(1,0.62663395)%
    \lineheight{1}%
    \setlength\tabcolsep{0pt}%
    \put(0,0){\includegraphics[width=\unitlength,page=1]{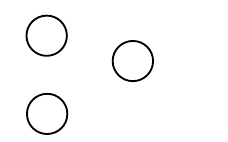}}%
    \put(-0.0044468,0.03397103){\color[rgb]{0,0,0}\makebox(0,0)[lt]{\lineheight{1.25}\smash{\begin{tabular}[t]{l}$1$\end{tabular}}}}%
    \put(0.69052556,0.33525541){\color[rgb]{0,0,0}\makebox(0,0)[lt]{\lineheight{1.25}\smash{\begin{tabular}[t]{l}$0$\end{tabular}}}}%
    \put(0.08228542,0.57219486){\color[rgb]{0,0,0}\makebox(0,0)[lt]{\lineheight{1.25}\smash{\begin{tabular}[t]{l}$3$\end{tabular}}}}%
    \put(0,0){\includegraphics[width=\unitlength,page=2]{facet_cycle_with_turnarounds.pdf}}%
    \put(0.88093655,0.13157747){\color[rgb]{0,0,0}\makebox(0,0)[lt]{\lineheight{1.25}\smash{\begin{tabular}[t]{l}$2$\end{tabular}}}}%
    \put(0,0){\includegraphics[width=\unitlength,page=3]{facet_cycle_with_turnarounds.pdf}}%
  \end{picture}%
\endgroup%

  \end{tabular}
  \caption{Left: A cycle $K$ in $\partial \mathcal{C}(z) \subset \partial \mathcal{C}_4$. Right: the facet cycle $c$ that parametrizes $K$. Note that this facet cycle, unlike the facet cycle in Figure \ref{cycle example arc union}, has turnarounds.}
  \label{parametrization turnarounds}
\end{figure}
So far, we have discussed how a cycle $C$ in $\Gamma(z,\mu)$ has signed facet cycle $Z(C)$, which takes values in $\{0\ldots \kappa+1\}\times \mathbb{F}_2$. The following construction examines a ``reverse'' procedure: given a cycle which takes values in $\{0\ldots \kappa+1\}\times \mathbb{F}_2$ at each index, we create a ``cycle'' $K$ living in $\partial\mathcal{C}(z)$. This $K$ is \textit{contrived}, in the sense that the Pontrjagin-Thom tubes in $K$ do not arise from a cubical neat embedding, nor do the boundary matching tubes do not arise from a boundary matching $\mathfrak{m}$ (see Definition \ref{facewise def}).
\begin{definition}
  For $r \geq 0$, define
  \begin{align*}
    E(r)
    &:= [-\epsilon,\epsilon]^A\times [-\epsilon,\epsilon]^B\times J\times \widetilde{\mathcal{M}}_{\mathscr{C}_C(r)}(\overline{1},\overline{0})\\
    \mathcal{C}(r)
    &:= [-R,R]^A\times [-R,R]^B\times J\times \widetilde{\mathcal{M}}_{\mathscr{C}_C(r)}(\overline{1},\overline{0}).
  \end{align*}
  This definition gives a higher dimensional analogs of $e^m$ and $\mathcal{C}(z)$. Indeed, fixing $r = \kappa$, we recover $e^m = E(\kappa)$ and $\mathcal{C}(z) = \mathcal{C}(\kappa+2)$.
\end{definition}
\begin{definition}
  Let $Z$ be a cycle denoted by
  \[
    (a_1,\omega_1)\longline (a_2,\omega_2)\longline \ldots \longline (a_n,\omega_n) \longline (a_1,\omega_1) \qquad (\text{resp.\ } a_1\longline a_2\longline \ldots \longline a_n \longline a_1),
  \]
  where:
  \begin{enumerate}[label = (\arabic*)]
    \item At each vertex, $Z$ is valued in $\mathbb{Z}_{\geq 0}$ (resp.\ $\mathbb{Z}_{\geq 0}\times \mathbb{F}_2$).\label{facet cycle 1}
    \item If $(a,\omega')\longline (b,\omega)$ (resp.\ $a\longline b$) is an edge in $c$, then $a\neq b$,
    \item Each index in the cycle is oriented.
    \item We have $(a,\omega')\longline \overrightarrow{(b,\omega)} \longline (c,\omega'')$ (resp.\ $a\longline \overrightarrow{b}\longline c$) if $a<c$.\label{facet cycle 4}
    \end{enumerate}
    We call $Z$ a \textit{signed facet cycle} (resp.\ \textit{unsigned facet cycle}).\par
    Similarly, let $D$ be a chain denoted by
    \[
      (a_1,\omega_1)\longline (a_2,\omega_2)\longline \ldots \longline (a_n,\omega_n) \qquad (\text{resp.\ } a_1\longline a_2\longline \ldots \longline a_k),
    \]
    satisfying \ref{facet cycle 1}-\ref{facet cycle 4}. We call $D$ a \textit{facet chain}.\par
    Sometimes, we drop the signs when writing out subchains of a facet cycle.
  \end{definition}
  \begin{construction}\label{cycle building}
    Suppose we are given a signed facet cycle $Z$ (or an unsigned facet cycle $Z$, in which case we treat $Z$ as a signed facet cycle with all signs $\omega=0$). We construct a tube cycle $K\cong E(r) \times S^1$ living in $\partial\mathcal{C}(r+2)$ as follows:
  \begin{enumerate}[label = (T-\arabic*)]
  \item \label{PT construction} For every $(a,\omega')\longline (b,\omega)$, we define a d.s.\ \textit{Pontrjagin-Thom} tube $\zeta$ in $\mathbf{G}_{ab}\subset \partial\mathcal{C}(r+2)$ by defining first an embedding $\Theta: E(r)\times [0,1]\to \partial \mathcal{C}(r+2)$:
    \begingroup
    \allowdisplaybreaks
    \begin{align*}
      \Theta: E(r)\times [0,1]
      &\xhookrightarrow{\Id\times f} E(r)\times \Pi^1\\
      &= [-\epsilon,\epsilon]^{A+B}\times J\times \widetilde{\mathcal{M}}_{\mathscr{C}_C(\kappa)}(\overline{1},\overline{0})\times \Pi^1\\
      &= [-\epsilon,\epsilon]^{A+B}\times J\times \widetilde{\mathcal{M}}_{\mathscr{C}_C(\kappa+2)}(\overline{1}-\{a,b\},\overline{0})\times \mathcal{M}_{\mathscr{C}_C(\kappa+2)}(\overline{1},\overline{1}-\{a,b\})\\
      &\xhookrightarrow{\widehat{\iota}} [-R,R]^{A+B}\times J\times \widetilde{\mathcal{M}}_{\mathscr{C}_C(\kappa+2)}(\overline{1},\overline{0})\\
      &= \mathcal{C}(r+2).
    \end{align*}
    \endgroup
     Here, $f:[0,1]\to \Pi^1$ is one of the two canonical identifications, $\widehat{\iota}(a,t,p,q) = (\gamma(q)+a,t,p\circ q)$, where $\gamma:\Pi^1\to [-R,R]^{A+B}$ is a smoothly embedded path. We now define the second parametrization $\Lambda(x,t):=\Theta(x,1-t)$, thus defining our d.s.\ tube $\zeta = \{(\Theta,\alpha),(\Lambda,\beta)\}$. By possibly perturbing the $\gamma$'s for each $(a,\omega')\longline (b,\omega)$, we can ensure that the $\zeta$'s do not intersect.\par
    We label the edge as $(a,\omega')\xline{\zeta}(b,\omega)$.
    
  \item \label{bm construction}For every $(a,\omega')\xline{\zeta'} \overrightarrow{(b,\omega)}\xline{\zeta''} (c,\omega'')$, we define a \textit{boundary matching} tube $\tilde{\eta}=\{(\tilde{\Theta},0),(\tilde{\Lambda},1)\}$ in $\mathbf{G}_b\subset \partial\mathcal{C}(z)$. We construct $\tilde{\eta}$ so that it connects $\zeta'$ to $\zeta''$, matching $\zeta'\cap \mathbf{G}_b\cong E(r)$ with $\zeta''\cap \mathbf{G}_b\cong E(r)$. The construction is analogous to Section \ref{eta construction}, as it follows a similar outline that begins by defining an ``unsigned'' version $\tilde{\eta}_\emptyset = \{(\tilde{\Theta}_\emptyset,0),(\tilde{\Lambda}_\emptyset,1)\}$.\par
    Indeed, if $a = c$, we define $\tilde{\Theta}_\emptyset$ by defining the its projections
    \begin{enumerate}[label={(\arabic*)}]
    \item $\tilde{\Theta}_R, \tilde{\Lambda}_R: E(r) \times [0,1]\to [-R,R]^{A+B}$,
    \item $\tilde{\Theta}_{\mathcal{M}}, \tilde{\Lambda}_{\mathcal{M}}: E(r)\times [0,1]\to J\times \mathcal{M}_{\mathscr{C}_C(r+2)}(\overline{1},\overline{0})$.
    \end{enumerate}
    The map $\tilde{\Theta}_R$ is defined by $\tilde{\Theta}_R(a,s,x,t)=\gamma(t)+a$ for $a\in [-\epsilon,\epsilon]^{A+B}$, $s\in J$, $x\in \mathcal{M}_{\mathscr{C}_C(\kappa)}(\overline{1},\overline{0})$, where $\gamma: [0,1]\to [-R,R]^{A+B}$ with the endpoints $\gamma(0),\gamma(1)$
    chosen to satisfy the equations
    \begin{equation}\label{matches boundary}
      \tilde{\Theta}_R(E(r)\times \{0\})=\zeta'\cap \mathbf{G}_b, \qquad \tilde{\Theta}_R(E(r)\times \{1\})=\zeta''\cap \mathbf{G}_b.
    \end{equation}
    This ensures that we can ``continuously'' travel from $\zeta'$ to $\tilde{\eta}$, and then through $\zeta''$. The projection $\tilde{\Lambda}_R$ is defined by $\tilde{\Lambda}_R(x,t) = \tilde{\Theta}_R(x,-t)$.\par
    The projection $\tilde{\Theta}_{\mathcal{M}}$ is defined in a way similar to the projection $\Theta_{\mathcal{M}}$ in Section \ref{eta construction}, where the slices $E(r)$ move directly into $\Int(J\times \mathcal{M}_{\mathscr{C}_C(r+2)}(\overline{1},\overline{0}))$, but then pull up in the $J$-direction and back around into the same face $J\times G_i$. The projection $\tilde{\Lambda}_{\mathcal{M}}$ is defined by $\tilde{\Lambda}_{\mathcal{M}}(x,t) = \tilde{\Theta}_{\mathcal{M}}(\tau(x),-t)$\par
    If $a\neq c$, we first define a convex version $\tilde{\eta}^{\conv} = \{(\tilde{\Theta}^{\conv},0),(\tilde{\Lambda}^{\conv},1)\}$, where its comprising tubes $\tilde{\Theta}^{\conv},\tilde{\Lambda}^{\conv}$ are determined by the following maps:
    \begin{enumerate}[label={(\arabic*)}]
    \item $\tilde{\Theta}_R^{\conv}, \tilde{\Lambda}_R^{\conv}: E(r) \times [0,1]\to [-R,R]^{A+B}$,
    \item $\tilde{\Theta}_{\mathcal{M}}^{\conv},\tilde{\Lambda}_{\mathcal{M}}^{\conv}: E(r)\times [0,1]\to J\times \mathcal{M}_{\mathscr{C}_C(r+2)}(\overline{1},\overline{0})$.
    \end{enumerate}
    The map $\tilde{\Theta}^{\conv}_R$ is defined in a similar way to before; namely, $\tilde{\Theta}^{\conv}_R(a,s,x,t)=\gamma(t)+a$, where $\gamma$ is constructed to solve the boundary-matching conditions analogous to (\ref{matches boundary}). We define $\tilde{\Theta}_{\mathcal{M}}^{\conv}$ as the following composition:
    \begin{equation}\label{contrived Theta M}
    \begin{aligned}
      \tilde{\Theta}_{\mathcal{M}}^{\conv}: E(r)\times [0,1]&\xhookrightarrow{\Theta_{\mathcal{M}}^{\conv}}J\times \widetilde{\mathcal{M}}_{\mathscr{C}_C(\kappa)}(\overline{1},\overline{0})\\
                     &\cong J\times \widetilde{\mathcal{M}}_{\mathscr{C}_C(r+1)}(\overline{1}-\{b\},\overline{0})\times \mathcal{M}_{\mathscr{C}_C(r+1)}(\overline{1},\overline{1}-\{b\})\\
                     &\xhookrightarrow{} J\times \widetilde{\mathcal{M}}_{\mathscr{C}_C(r+2)}(\overline{1},\overline{0})\\
      &= \partial\mathcal{C}(r+2),
    \end{aligned}
    \end{equation}
    where the embedding $\Theta_{\mathcal{M}}^{\conv}$ is defined in a similar way to the embedding $\Theta^{\conv}_{\mathcal{M}}$ from Section \ref{eta construction}, moving from $J\times F_j$ directly to $J\times F_j$ ($i=a_b$, $j=c_b$). $\Lambda_{\mathcal{M}}^{\conv}$ is defined by a composition similar to (\ref{contrived Theta M}), but with $\Theta_{\mathcal{M}}^{\conv}$ replaced with $\Lambda_{\mathcal{M}}^{\conv}$.
    \par
    And similar to Section \ref{eta construction}, we define $\tilde{\Lambda}_R^{\conv}(x,t) :=\tilde{\Theta}_R^{\conv}(x,1-t)$, and $\eta_{\emptyset} = \tilde{\eta}^{\conv}\diamond(\boldvarphi_{j-1}\diamond\ldots\diamond \boldvarphi_{i+1},0)$\par
    Finally, we define $\tilde{\eta}$ as the conjugation  $\tau^{\omega}(\eta_\emptyset)\tau^{-\omega}:= \{(\tau^{\omega}\tilde{\Theta}_\emptyset\tau^{-\omega},0),(\tau^{\omega}\tilde{\Lambda}_\emptyset\tau^{-\omega},1)\}$ of $\eta_{\emptyset}$, where the comprising tubes get conjugated by the $J$-factor isometry $\tau^{\omega}$. \end{enumerate}
  See Figure \ref{contrived example} for an example of this construction.\par
  \begin{figure}
    \centering
    \begin{tabular}{m{4cm} m{1cm} m{2cm} m{1.3cm} m{4.0cm}}
      \def\svgscale{0.65}
\begingroup%
  \makeatletter%
  \providecommand\color[2][]{%
    \errmessage{(Inkscape) Color is used for the text in Inkscape, but the package 'color.sty' is not loaded}%
    \renewcommand\color[2][]{}%
  }%
  \providecommand\transparent[1]{%
    \errmessage{(Inkscape) Transparency is used (non-zero) for the text in Inkscape, but the package 'transparent.sty' is not loaded}%
    \renewcommand\transparent[1]{}%
  }%
  \providecommand\rotatebox[2]{#2}%
  \newcommand*\fsize{\dimexpr\f@size pt\relax}%
  \newcommand*\lineheight[1]{\fontsize{\fsize}{#1\fsize}\selectfont}%
  \ifx\svgwidth\undefined%
    \setlength{\unitlength}{243.10423603bp}%
    \ifx\svgscale\undefined%
      \relax%
    \else%
      \setlength{\unitlength}{\unitlength * \real{\svgscale}}%
    \fi%
  \else%
    \setlength{\unitlength}{\svgwidth}%
  \fi%
  \global\let\svgwidth\undefined%
  \global\let\svgscale\undefined%
  \makeatother%
  \begin{picture}(1,1.47969672)%
    \lineheight{1}%
    \setlength\tabcolsep{0pt}%
    \put(0,0){\includegraphics[width=\unitlength,page=1]{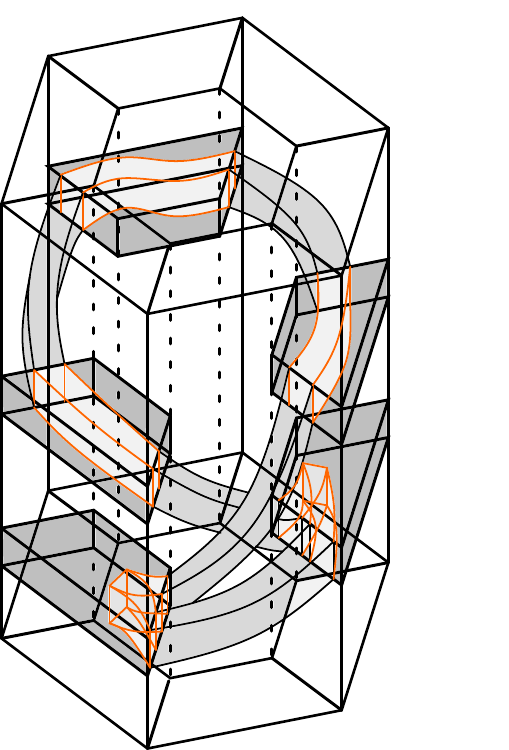}}%
    \put(0.16558838,1.44904353){\color[rgb]{0,0,0}\makebox(0,0)[lt]{\lineheight{1.25}\smash{\begin{tabular}[t]{l}$\mathbf{G}_1$\end{tabular}}}}%
    \put(0.01380642,0.09109884){\color[rgb]{0,0,0}\makebox(0,0)[lt]{\lineheight{1.25}\smash{\begin{tabular}[t]{l}$\mathbf{G}_2$\end{tabular}}}}%
    \put(0.72378546,0.15786729){\color[rgb]{0,0,0}\makebox(0,0)[lt]{\lineheight{1.25}\smash{\begin{tabular}[t]{l}$\mathbf{G}_0$\end{tabular}}}}%
  \end{picture}%
\endgroup%
 & \begin{tikzpicture}
      [decoration=snake,
      line around/.style={decoration={pre length=#1,post length=#1}}]
      \draw[->, decorate, line around = 1pt] (0,0)--(1.5,0);
    \end{tikzpicture} & \def\svgscale{0.9}
\begingroup%
  \makeatletter%
  \providecommand\color[2][]{%
    \errmessage{(Inkscape) Color is used for the text in Inkscape, but the package 'color.sty' is not loaded}%
    \renewcommand\color[2][]{}%
  }%
  \providecommand\transparent[1]{%
    \errmessage{(Inkscape) Transparency is used (non-zero) for the text in Inkscape, but the package 'transparent.sty' is not loaded}%
    \renewcommand\transparent[1]{}%
  }%
  \providecommand\rotatebox[2]{#2}%
  \newcommand*\fsize{\dimexpr\f@size pt\relax}%
  \newcommand*\lineheight[1]{\fontsize{\fsize}{#1\fsize}\selectfont}%
  \ifx\svgwidth\undefined%
    \setlength{\unitlength}{87.27717482bp}%
    \ifx\svgscale\undefined%
      \relax%
    \else%
      \setlength{\unitlength}{\unitlength * \real{\svgscale}}%
    \fi%
  \else%
    \setlength{\unitlength}{\svgwidth}%
  \fi%
  \global\let\svgwidth\undefined%
  \global\let\svgscale\undefined%
  \makeatother%
  \begin{picture}(1,0.60185387)%
    \lineheight{1}%
    \setlength\tabcolsep{0pt}%
    \put(0,0){\includegraphics[width=\unitlength,page=1]{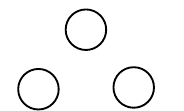}}%
    \put(-0.00552403,0.09652246){\color[rgb]{0,0,0}\makebox(0,0)[lt]{\lineheight{1.25}\smash{\begin{tabular}[t]{l}$a$\end{tabular}}}}%
    \put(0.86030314,0.09860293){\color[rgb]{0,0,0}\makebox(0,0)[lt]{\lineheight{1.25}\smash{\begin{tabular}[t]{l}$c$\end{tabular}}}}%
    \put(0.31556104,0.53414435){\color[rgb]{0,0,0}\makebox(0,0)[lt]{\lineheight{1.25}\smash{\begin{tabular}[t]{l}$b$\end{tabular}}}}%
    \put(0,0){\includegraphics[width=\unitlength,page=2]{honest_cycle_parametrization.pdf}}%
  \end{picture}%
\endgroup%
 & \begin{tikzpicture}
      [decoration=snake,
   line around/.style={decoration={pre length=#1,post length=#1}}]
      \draw[->, decorate, line around = 1pt] (0,0)--(1.5,0);
    \end{tikzpicture} & \def\svgscale{0.65}
\begingroup%
  \makeatletter%
  \providecommand\color[2][]{%
    \errmessage{(Inkscape) Color is used for the text in Inkscape, but the package 'color.sty' is not loaded}%
    \renewcommand\color[2][]{}%
  }%
  \providecommand\transparent[1]{%
    \errmessage{(Inkscape) Transparency is used (non-zero) for the text in Inkscape, but the package 'transparent.sty' is not loaded}%
    \renewcommand\transparent[1]{}%
  }%
  \providecommand\rotatebox[2]{#2}%
  \newcommand*\fsize{\dimexpr\f@size pt\relax}%
  \newcommand*\lineheight[1]{\fontsize{\fsize}{#1\fsize}\selectfont}%
  \ifx\svgwidth\undefined%
    \setlength{\unitlength}{240.91341514bp}%
    \ifx\svgscale\undefined%
      \relax%
    \else%
      \setlength{\unitlength}{\unitlength * \real{\svgscale}}%
    \fi%
  \else%
    \setlength{\unitlength}{\svgwidth}%
  \fi%
  \global\let\svgwidth\undefined%
  \global\let\svgscale\undefined%
  \makeatother%
  \begin{picture}(1,1.48869665)%
    \lineheight{1}%
    \setlength\tabcolsep{0pt}%
    \put(0,0){\includegraphics[width=\unitlength,page=1]{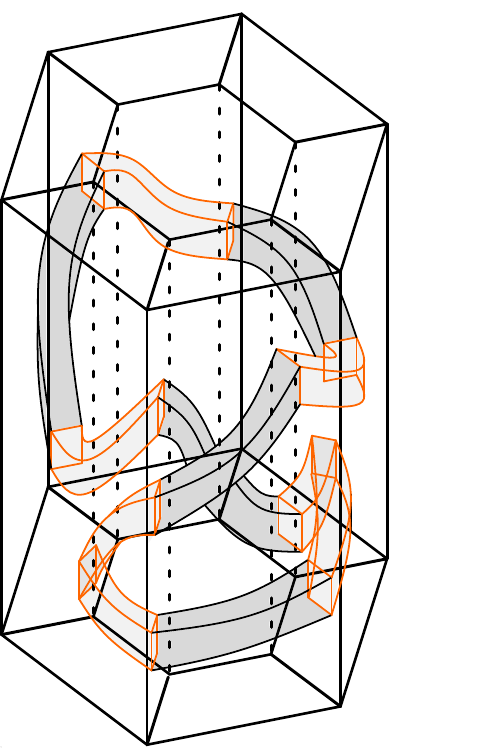}}%
    \put(0.16599437,1.4577647){\color[rgb]{0,0,0}\makebox(0,0)[lt]{\lineheight{1.25}\smash{\begin{tabular}[t]{l}$\mathbf{G}_1$\end{tabular}}}}%
    \put(0.00580409,0.10949399){\color[rgb]{0,0,0}\makebox(0,0)[lt]{\lineheight{1.25}\smash{\begin{tabular}[t]{l}$\mathbf{G}_2$\end{tabular}}}}%
    \put(0.72127354,0.14561716){\color[rgb]{0,0,0}\makebox(0,0)[lt]{\lineheight{1.25}\smash{\begin{tabular}[t]{l}$\mathbf{G}_0$\end{tabular}}}}%
  \end{picture}%
\endgroup%

  \end{tabular}
    \caption{Left: A cycle $K\subset \partial\mathcal{C}(z)$. Middle The facet cycle $c$ of $K$. Right: a contrived cycle $K'$ parametrized by $K$. The blue Pontrjagin-Thom tubes were constructed in (\ref{PT construction}), and the magenta boundary matching tubes were constructed in (\ref{bm construction}). Notice here that boundary-matched components in $K$ have to have the same $[-R,R]^B$ coordinates, but there is no such requirement for $K'$.}
    \label{contrived example}
  \end{figure}
  Similarly, if we are given a (signed or unsigned) facet chain
  \[
    D = (a_1,\omega_1)\longline (a_2,\omega_2)\longline \ldots\longline (a_n,\omega_n),
  \]
  we can construct a d.s.\ tube $U$, using Steps \ref{PT construction}, \ref{bm construction}, but with a slight subtlety: we do not include the Pontrjagin-Thom tubes corresponding to $a_1\longline a_2$, $a_{n-1}\longline a_n$. Indeed, the d.s.\ tube $U$ is constructed as a union
  \[
    U = \tilde{\eta}_2\cup\zeta_{2,3}\cup\widetilde{\eta_3}\cup\ldots\cup\zeta_{n-2,n-1}\cup \tilde{\eta}_{n-1}.
  \]
  For an illustration of an example (in the unsigned case), see the left side of Figure \ref{short to long comparison} to see a tube $U$ parametrizing $a\longline b\longline c\longline d\longline e$.
\end{construction}

\begin{definition}
  Let $\tilde{\eta}$ be a boundary matching tube constructed in \ref{bm construction}, Construction \ref{cycle building}. We will denote the tube $\tilde{\eta}^{\conv}$ used in the construction as the \textit{convex version} of $\tilde{\eta}$.\par
  Furthermore, if $U$ is a d.s.\ tube parametrized by a facet chain, we denote the \textit{convex version} $U^{\conv}$ of $U$ to be a modification of $U$ where we replace every boundary matching tube $\tilde{\eta}\subset U$ with its convex version $\tilde{\eta}^{\conv}$. See Figure \ref{convex long} for an illustration.
\end{definition}
\begin{definition}
  Let $Z$ be a signed facet cycle and let $K$ be as in Construction \ref{cycle building}. We say that $Z$ \textit{parametrizes} the tube cycle $K$, and that $K$ is a \textit{(contrived) cycle}. If $K$ is a cycle in $\partial\mathcal{C}(z)$ arising from our cubical neat embedding $\iota'$ and boundary matching tubes $\eta\subset \partial\mathcal{C}(y)$, we call $K$ an \textit{honest cycle}.\par
  Just as in Definition \ref{0-section construction}, we define the \textit{core} $\mathcal{K}_\zeta$, $\mathcal{K}_{\tilde{\eta}}$ of a Pontrjagin-Thom tube $\zeta$ and a boundary matching tube $\tilde{\eta}$. For $\zeta\subset K$, define $\mathcal{K}_\zeta$ to be the $\{P\}\times [0,1]$-identified subset of $\zeta$, and for $\tilde{\eta}\subset K$, define $\mathcal{K}_{\tilde{\eta}}$ to be the $\{P\}\times [0,1]$-identified subset of $\tilde{\eta}$. Now define the \textit{core} $\mathcal{K}$ of $K$ to be the union of the cores of its comprising $\zeta$ and $\tilde{\eta}$ tubes.
\end{definition}
Just as in Lemma \ref{trivialization}, there are two canonical ways to locate a point in $K$ in terms of its ``core'' coordinate and its ``fiber'' coordinate.
\begin{lemma}
  Let $Z$ be a signed facet cycle, and let $Z$ parametrize the cycle $K=\bigcup_{\tilde{\eta}}\tilde{\eta} \cup \bigcup_\zeta \zeta$. Let $\mathcal{K}\subset K$ be constructed as in Definition \ref{0-section construction}. There are exactly two trivializations of the form $\Phi: K\to E(r)\times\mathcal{K}$ that satisfy the following conditions:
  \begin{itemize}
  \item For every Pontrjagin-Thom tube $\zeta$, $\Phi|_\zeta$ is equal to
    \[
      \zeta \xhookrightarrow{\Theta^{-1}} E(r)\times [0,1] \xhookrightarrow{\tau^\omega \times \Theta(P,\cdot)}E(r)\times \mathcal{K}_\zeta,
    \]
    where $\tau^\omega\in \{\Id,\tau\}$.
    \item For every boundary matching tube $\tilde{\eta}$, $\Phi|_{\tilde{\eta}}$ is equal to
    \[
      \tilde{\eta} \xhookrightarrow{\Theta^{-1}} E(r)\times [0,1] \xhookrightarrow{\tau^\omega \times \Theta(P,\cdot)}E(r)\times \mathcal{K}_{\tilde{\eta}},
    \]
    where $\tau^\omega\in \{\Id,\tau \}$.
  \end{itemize}
  These trivializations, which we call $\Phi,\Phi'$, are related by a flip $\tau$ in the $J$-factor in $E(r)\times \mathcal{K}$.
\end{lemma}
\begin{proof}
  Let the components of $K$, ordered so that one tube follows another, be $\zeta_1,\tilde{\eta}_1,\ldots,\zeta_l,\tilde{\eta}_l$. Note that all the $\zeta$ tubes are boundary-coherent. One tricky part is to prove there are an even number of $\eta$ boundary components, but the proof is along the same lines as the proof of Lemma \ref{even incoherent}. The rest of the proof is similar to our proof of Lemma \ref{trivialization}, so we leave it to the reader to fill in the details.
\end{proof}
Similar to Definition \ref{K element}, both $\Phi,\Phi'$ define classes $[K,\Phi],[K,\Phi']\in [\partial\mathcal{C}_{\kappa+2},E(r)/\partial E(r)]$, which are identical, and which will call $[K]$.
\begin{lemma}\label{suspension agreement}
  Let a facet cycle $Z$ parametrize cycles $K\subset \partial \mathcal{C}(r+2)$, $K'\subset \partial\mathcal{C}(s+2)$, $r,s\geq 0$. Then $[K]=[K']$.
\end{lemma}
\begin{proof}
  Suppose $r=s$, and thus $K$, $K'$ both live in the same set $\partial\mathcal{C}(r)$. Then since $K$, $K'$ are both parametrized by the same facet cycle $Z$, we can isotope $K$ to $K'$ through a continuous family of cycles, all parametrized by $Z$. Thus $[K]=[K']$.\par
  By the previous paragraph and an induction argument, it remains to prove the following statement: If $Z$ parametrizes a cycle $K\subset \partial \mathcal{C}(r)$, then $Z$ also parametrizes a cycle $K'\subset \partial \mathcal{C}(r+1)$ with $[K] = [K']$. So we fix $K$, which we see as a union of tubes $\zeta_{01}\cup \tilde{\eta}_1\cup \ldots \cup \zeta_{n0}\cup \tilde{\eta}_0$. With the Pontrjagin-Thom construction in mind, we describe a ``suspension'' of $K$ to a cycle $K'$.\par
  Now $K$ is an embedded $S^1$-family of $\mathcal{C}(r) = [-\epsilon,\epsilon]^A\times [-\epsilon,\epsilon]^B\times J \times \widetilde{\mathcal{M}}_{\mathscr{C}_C(r)}(\overline{1},\overline{0})$ in $\mathcal{C}(r+2) = [-R,R]^A\times [-R,R]^B\times J \times \widetilde{\mathcal{M}}_{\mathscr{C}_C(r+2)}(\overline{1},\overline{0})$. Now consider $\mathbf{f}(K)$, where $\mathbf{f}$ is the embedding $\mathbf{f}:\mathcal{C}(r+2)\hookrightarrow \mathbf{G}_{\{0,\ldots,r+1\}}(\mathcal{C}(r+3))$ defined by.
  \begin{align*}
    \mathbf{f}:&[-R,R]^A\times [-R,R]^B\times J \times \widetilde{\mathcal{M}}_{\mathscr{C}_C(r+2)}(\overline{1},\overline{0})\\
    &=  [-R,R]^A\times [-R,R]^B\times J \times \Pi^0\times \widetilde{\mathcal{M}}_{\mathscr{C}_C(r+2)}(\overline{1},\overline{0})\\
    &\xhookrightarrow{\Id_R \times\Id_R\times \Id_J\times f} [-R,R]^A\times [-R,R]^B\times J \times \widetilde{\mathcal{M}}_{\mathscr{C}_C(r+3)}(\overline{1},\overline{0}),
  \end{align*}
  where $f$ is the embedding defined, for $r>0$, as
  \[
    \widetilde{\mathcal{M}}_{\mathscr{C}_C(r+2)}(\overline{1},\overline{0}) = \Pi^{r+1} \times [0,1] \xhookrightarrow{f_{\{n\}}\times \Id} \Pi^{r+2} \times [0,1] = \widetilde{\mathcal{M}}_{\mathscr{C}_C(r+3)}(\overline{1},\overline{0}),
  \]
  and for $r=0$, $f$ is defined as the embedding
  \[
    \widetilde{\mathcal{M}}_{\mathscr{C}_C(2)}(\overline{1},\overline{0}) = \Pi^1\times[0,1]\xhookrightarrow{} \Pi^2\times \{0\} \hookrightarrow \Pi^2\times [0,1] = \widetilde{\mathcal{M}}_{\mathscr{C}_C(2)}(\overline{1},\overline{0}).
  \]
  See Figure \ref{r=0 embedding} for an illustration of $\mathbf{f}$ for $r + 2 = 3$. 
  \begin{figure}
    \centering
    \begin{tabular}{m{5.5cm} m{1.3cm} m{4.0cm}}
    \def\svgscale{0.7}
\begingroup%
  \makeatletter%
  \providecommand\color[2][]{%
    \errmessage{(Inkscape) Color is used for the text in Inkscape, but the package 'color.sty' is not loaded}%
    \renewcommand\color[2][]{}%
  }%
  \providecommand\transparent[1]{%
    \errmessage{(Inkscape) Transparency is used (non-zero) for the text in Inkscape, but the package 'transparent.sty' is not loaded}%
    \renewcommand\transparent[1]{}%
  }%
  \providecommand\rotatebox[2]{#2}%
  \newcommand*\fsize{\dimexpr\f@size pt\relax}%
  \newcommand*\lineheight[1]{\fontsize{\fsize}{#1\fsize}\selectfont}%
  \ifx\svgwidth\undefined%
    \setlength{\unitlength}{395.21230623bp}%
    \ifx\svgscale\undefined%
      \relax%
    \else%
      \setlength{\unitlength}{\unitlength * \real{\svgscale}}%
    \fi%
  \else%
    \setlength{\unitlength}{\svgwidth}%
  \fi%
  \global\let\svgwidth\undefined%
  \global\let\svgscale\undefined%
  \makeatother%
  \begin{picture}(1,0.43379461)%
    \lineheight{1}%
    \setlength\tabcolsep{0pt}%
    \put(0,0){\includegraphics[width=\unitlength,page=1]{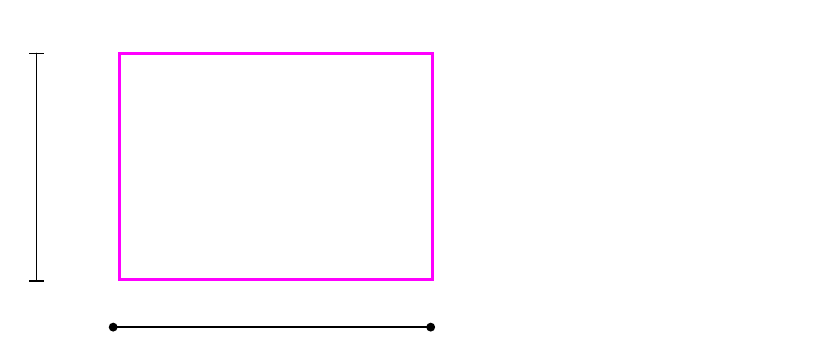}}%
    \put(0.11782497,0.00132084){\color[rgb]{0,0,0}\makebox(0,0)[lt]{\lineheight{1.25}\smash{\begin{tabular}[t]{l}$21$\end{tabular}}}}%
    \put(0.50612504,0.00701451){\color[rgb]{0,0,0}\makebox(0,0)[lt]{\lineheight{1.25}\smash{\begin{tabular}[t]{l}$12$\end{tabular}}}}%
    \put(0,0){\includegraphics[width=\unitlength,page=2]{lowest_dim_Cy.pdf}}%
    \put(0.00113404,0.08977352){\color[rgb]{0,0,0}\makebox(0,0)[lt]{\lineheight{1.25}\smash{\begin{tabular}[t]{l}$0$\end{tabular}}}}%
    \put(-0.00141187,0.35887791){\color[rgb]{0,0,0}\makebox(0,0)[lt]{\lineheight{1.25}\smash{\begin{tabular}[t]{l}$1$\end{tabular}}}}%
    \put(0,0){\includegraphics[width=\unitlength,page=3]{lowest_dim_Cy.pdf}}%
    \put(0.22181627,0.41648755){\color[rgb]{0,0,0}\makebox(0,0)[lt]{\lineheight{1.25}\smash{\begin{tabular}[t]{l}$\widetilde{\mathcal{M}}_{\mathscr{C}_C(r+2)}(\overline{1},\overline{0})$\end{tabular}}}}%
  \end{picture}%
\endgroup%
 &  $\lhook\joinrel\longline\joinrel\longrightarrow$ & \def\svgscale{0.7}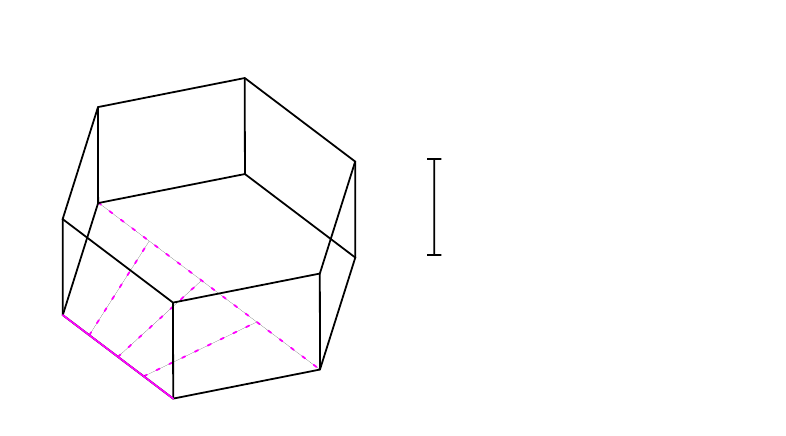
  \end{tabular}
    \caption{The embedding $\widetilde{\mathcal{M}}_{\mathscr{C}_C(r+2)}(\overline{1},\overline{0}) \hookrightarrow \widetilde{\mathcal{M}}_{\mathscr{C}_C(r+3)}(\overline{1},\overline{0})$ for $r=0$}
    \label{r=0 embedding}
  \end{figure}
  Observe that $\mathbf{f}$ restricts to embeddings $\mathbf{G}_i(\mathcal{C}(r+2))\hookrightarrow \mathbf{G}_a(\mathcal{C}(r+3))$, $\mathbf{G}_{ab}(\mathcal{C}(r+2))\hookrightarrow \mathbf{G}_{ab}(\mathcal{C}(r+3))$, meaning that if $\eta\subset K$ is in $\mathbf{G}_{ab}(\mathcal{C}(r+2))$, then $\mathbf{f}(\eta)\subset \mathbf{G}_{ab}(\mathcal{C}(r+3))$, and (similarly for $\zeta\subset K$). So we can view $\mathbf{f}(K)$ as a ``thinned out'' version of a cycle in $\partial \mathcal{C}(r+3)$.\par
  In the case $r>0$, we ``thicken'' $\mathbf{f}(K)$ in the direction outwardly normal to $\mathbf{G}_{\{0,\ldots, r+1\}}(\mathcal{C}(r+3))\subset \partial \mathcal{C}(r+3)$ to get a $S^1$-family of $E(r)$, which we view as a cycle $K'\subset \mathcal{C}(r+3)$. In the case $r=0$, we simply ``thicken'' $\mathbf{f}(K)$ uniformly in the $[0,1]$-coordinate of $\mathcal{C}(r+3)$ to obtain $K'$. (See Figure \ref{cycle suspension} for an illustration of an example of $K'$.)\par
  The class $[K]$ is an element of $[\partial\mathcal{C}(r+2), \mathcal{C}(r)]\cong [S^{A+B+r+2},S^{A+B+\kappa'+r+1}]\cong\mathbb{Z}/2$ and $[K']$ is an element of $[\partial\mathcal{C}(r+3), \mathcal{C}(r+1)]\cong [S^{A+B+r+3},S^{A+B+r+1}]\cong \mathbb{Z}/2$. Since $K'$ was defined from $K$ through a Pontrjagin-Thom like construction, we see $[K'] = \Sigma [K] = [K]$.
  \begin{figure}
    \begin{tabular}{p{8cm} p{7cm}}
    \def\svgscale{1}
\begingroup%
  \makeatletter%
  \providecommand\color[2][]{%
    \errmessage{(Inkscape) Color is used for the text in Inkscape, but the package 'color.sty' is not loaded}%
    \renewcommand\color[2][]{}%
  }%
  \providecommand\transparent[1]{%
    \errmessage{(Inkscape) Transparency is used (non-zero) for the text in Inkscape, but the package 'transparent.sty' is not loaded}%
    \renewcommand\transparent[1]{}%
  }%
  \providecommand\rotatebox[2]{#2}%
  \newcommand*\fsize{\dimexpr\f@size pt\relax}%
  \newcommand*\lineheight[1]{\fontsize{\fsize}{#1\fsize}\selectfont}%
  \ifx\svgwidth\undefined%
    \setlength{\unitlength}{82.13873075bp}%
    \ifx\svgscale\undefined%
      \relax%
    \else%
      \setlength{\unitlength}{\unitlength * \real{\svgscale}}%
    \fi%
  \else%
    \setlength{\unitlength}{\svgwidth}%
  \fi%
  \global\let\svgwidth\undefined%
  \global\let\svgscale\undefined%
  \makeatother%
  \begin{picture}(1,0.76449719)%
    \lineheight{1}%
    \setlength\tabcolsep{0pt}%
    \put(0,0){\includegraphics[width=\unitlength,page=1]{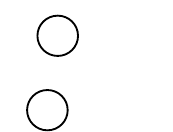}}%
    \put(-0.00586958,0.00549065){\color[rgb]{0,0,0}\makebox(0,0)[lt]{\lineheight{1.25}\smash{\begin{tabular}[t]{l}$0$\end{tabular}}}}%
    \put(0.17113609,0.69264623){\color[rgb]{0,0,0}\makebox(0,0)[lt]{\lineheight{1.25}\smash{\begin{tabular}[t]{l}$1$\end{tabular}}}}%
    \put(0,0){\includegraphics[width=\unitlength,page=2]{facet_3-cycle_unsuspended.pdf}}%
    \put(0.84285422,0.39993235){\color[rgb]{0,0,0}\makebox(0,0)[lt]{\lineheight{1.25}\smash{\begin{tabular}[t]{l}$2$\end{tabular}}}}%
    \put(0,0){\includegraphics[width=\unitlength,page=3]{facet_3-cycle_unsuspended.pdf}}%
  \end{picture}%
\endgroup%
 &\def\svgscale{1}
\begingroup%
  \makeatletter%
  \providecommand\color[2][]{%
    \errmessage{(Inkscape) Color is used for the text in Inkscape, but the package 'color.sty' is not loaded}%
    \renewcommand\color[2][]{}%
  }%
  \providecommand\transparent[1]{%
    \errmessage{(Inkscape) Transparency is used (non-zero) for the text in Inkscape, but the package 'transparent.sty' is not loaded}%
    \renewcommand\transparent[1]{}%
  }%
  \providecommand\rotatebox[2]{#2}%
  \newcommand*\fsize{\dimexpr\f@size pt\relax}%
  \newcommand*\lineheight[1]{\fontsize{\fsize}{#1\fsize}\selectfont}%
  \ifx\svgwidth\undefined%
    \setlength{\unitlength}{82.21929535bp}%
    \ifx\svgscale\undefined%
      \relax%
    \else%
      \setlength{\unitlength}{\unitlength * \real{\svgscale}}%
    \fi%
  \else%
    \setlength{\unitlength}{\svgwidth}%
  \fi%
  \global\let\svgwidth\undefined%
  \global\let\svgscale\undefined%
  \makeatother%
  \begin{picture}(1,0.92990414)%
    \lineheight{1}%
    \setlength\tabcolsep{0pt}%
    \put(0,0){\includegraphics[width=\unitlength,page=1]{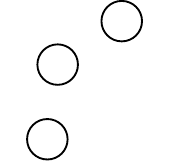}}%
    \put(-0.00586333,0.00548545){\color[rgb]{0,0,0}\makebox(0,0)[lt]{\lineheight{1.25}\smash{\begin{tabular}[t]{l}$0$\end{tabular}}}}%
    \put(0.84300891,0.77912046){\color[rgb]{0,0,0}\makebox(0,0)[lt]{\lineheight{1.25}\smash{\begin{tabular}[t]{l}$3$\end{tabular}}}}%
    \put(0.17096857,0.69196771){\color[rgb]{0,0,0}\makebox(0,0)[lt]{\lineheight{1.25}\smash{\begin{tabular}[t]{l}$1$\end{tabular}}}}%
    \put(0,0){\includegraphics[width=\unitlength,page=2]{facet_3-cycle_suspended.pdf}}%
    \put(0.84202871,0.39954065){\color[rgb]{0,0,0}\makebox(0,0)[lt]{\lineheight{1.25}\smash{\begin{tabular}[t]{l}$2$\end{tabular}}}}%
    \put(0,0){\includegraphics[width=\unitlength,page=3]{facet_3-cycle_suspended.pdf}}%
  \end{picture}%
\endgroup%
\\
    \def\svgscale{0.80}
\begingroup%
  \makeatletter%
  \providecommand\color[2][]{%
    \errmessage{(Inkscape) Color is used for the text in Inkscape, but the package 'color.sty' is not loaded}%
    \renewcommand\color[2][]{}%
  }%
  \providecommand\transparent[1]{%
    \errmessage{(Inkscape) Transparency is used (non-zero) for the text in Inkscape, but the package 'transparent.sty' is not loaded}%
    \renewcommand\transparent[1]{}%
  }%
  \providecommand\rotatebox[2]{#2}%
  \newcommand*\fsize{\dimexpr\f@size pt\relax}%
  \newcommand*\lineheight[1]{\fontsize{\fsize}{#1\fsize}\selectfont}%
  \ifx\svgwidth\undefined%
    \setlength{\unitlength}{222.30056402bp}%
    \ifx\svgscale\undefined%
      \relax%
    \else%
      \setlength{\unitlength}{\unitlength * \real{\svgscale}}%
    \fi%
  \else%
    \setlength{\unitlength}{\svgwidth}%
  \fi%
  \global\let\svgwidth\undefined%
  \global\let\svgscale\undefined%
  \makeatother%
  \begin{picture}(1,0.75943412)%
    \lineheight{1}%
    \setlength\tabcolsep{0pt}%
    \put(0,0){\includegraphics[width=\unitlength,page=1]{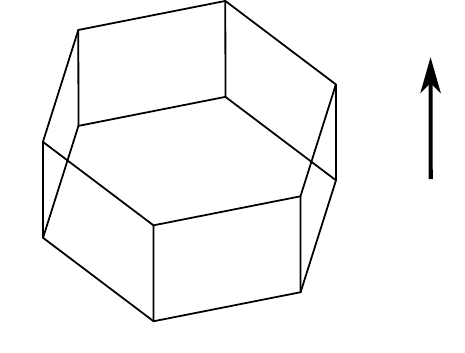}}%
    \put(0.94489885,0.54150816){\color[rgb]{0,0,0}\makebox(0,0)[lt]{\lineheight{1.25}\smash{\begin{tabular}[t]{l}$J$\end{tabular}}}}%
    \put(0.41291856,0.47160246){\color[rgb]{0,0,0}\makebox(0,0)[lt]{\lineheight{1.25}\smash{\begin{tabular}[t]{l}$123$\end{tabular}}}}%
    \put(0.18125981,0.4414636){\color[rgb]{0,0,0}\makebox(0,0)[lt]{\lineheight{1.25}\smash{\begin{tabular}[t]{l}$132$\end{tabular}}}}%
    \put(-0.00251005,0.18313469){\color[rgb]{0,0,0}\makebox(0,0)[lt]{\lineheight{1.25}\smash{\begin{tabular}[t]{l}$231$\end{tabular}}}}%
    \put(0.26292995,0.0023482){\color[rgb]{0,0,0}\makebox(0,0)[lt]{\lineheight{1.25}\smash{\begin{tabular}[t]{l}$321$\end{tabular}}}}%
    \put(0.63784973,0.06579296){\color[rgb]{0,0,0}\makebox(0,0)[lt]{\lineheight{1.25}\smash{\begin{tabular}[t]{l}$312$\end{tabular}}}}%
    \put(0.75509758,0.32270247){\color[rgb]{0,0,0}\makebox(0,0)[lt]{\lineheight{1.25}\smash{\begin{tabular}[t]{l}$213$\end{tabular}}}}%
    \put(0,0){\includegraphics[width=\unitlength,page=2]{3-cycle_unsuspended.pdf}}%
  \end{picture}%
\endgroup%
 &\def\svgscale{0.6}
\begingroup%
  \makeatletter%
  \providecommand\color[2][]{%
    \errmessage{(Inkscape) Color is used for the text in Inkscape, but the package 'color.sty' is not loaded}%
    \renewcommand\color[2][]{}%
  }%
  \providecommand\transparent[1]{%
    \errmessage{(Inkscape) Transparency is used (non-zero) for the text in Inkscape, but the package 'transparent.sty' is not loaded}%
    \renewcommand\transparent[1]{}%
  }%
  \providecommand\rotatebox[2]{#2}%
  \newcommand*\fsize{\dimexpr\f@size pt\relax}%
  \newcommand*\lineheight[1]{\fontsize{\fsize}{#1\fsize}\selectfont}%
  \ifx\svgwidth\undefined%
    \setlength{\unitlength}{290.31692144bp}%
    \ifx\svgscale\undefined%
      \relax%
    \else%
      \setlength{\unitlength}{\unitlength * \real{\svgscale}}%
    \fi%
  \else%
    \setlength{\unitlength}{\svgwidth}%
  \fi%
  \global\let\svgwidth\undefined%
  \global\let\svgscale\undefined%
  \makeatother%
  \begin{picture}(1,1.01024505)%
    \lineheight{1}%
    \setlength\tabcolsep{0pt}%
    \put(0,0){\includegraphics[width=\unitlength,page=1]{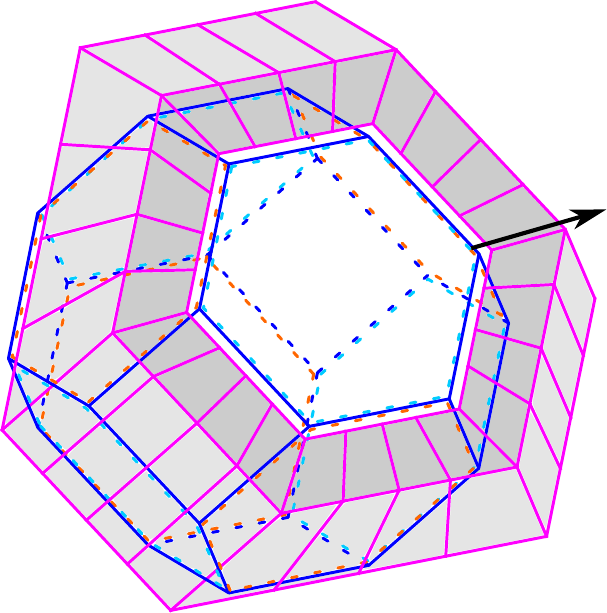}}%
    \put(0.95255098,0.72383802){\color[rgb]{0,0,0}\makebox(0,0)[lt]{\lineheight{1.25}\smash{\begin{tabular}[t]{l}$J$\end{tabular}}}}%
    \put(0.52043452,0.70719431){\color[rgb]{0,0,0}\makebox(0,0)[lt]{\lineheight{1.25}\smash{\begin{tabular}[t]{l}$2341$\end{tabular}}}}%
    \put(0.63337974,0.56321315){\color[rgb]{0,0,0}\makebox(0,0)[lt]{\lineheight{1.25}\smash{\begin{tabular}[t]{l}$3241$\end{tabular}}}}%
    \put(0.61185613,0.37652893){\color[rgb]{0,0,0}\makebox(0,0)[lt]{\lineheight{1.25}\smash{\begin{tabular}[t]{l}$4231$\end{tabular}}}}%
    \put(0.47403481,0.34868785){\color[rgb]{0,0,0}\makebox(0,0)[lt]{\lineheight{1.25}\smash{\begin{tabular}[t]{l}$4321$\end{tabular}}}}%
    \put(0.35632099,0.50230187){\color[rgb]{0,0,0}\makebox(0,0)[lt]{\lineheight{1.25}\smash{\begin{tabular}[t]{l}$3421$\end{tabular}}}}%
    \put(0.38460053,0.67925445){\color[rgb]{0,0,0}\makebox(0,0)[lt]{\lineheight{1.25}\smash{\begin{tabular}[t]{l}$2431$\end{tabular}}}}%
  \end{picture}%
\endgroup%

  \end{tabular}
    \caption{Left: A loop in the 3-vertex complete graph and the cycle $K$ it parametrizes. Right: The same loop embedded in the 4-vertex complete graph and the cycle $K'$ it parametrizes. As proved in Lemma \ref{suspension agreement}, the map $S^{(A+B)+2}\to S^{(A+B)+1}$ from the right figure is the suspension of the map from the left figure $S^{(A+B)+3}\to S^{(A+B)+2}$.}
    \label{cycle suspension}
  \end{figure}
\end{proof}
\section{Computing $[K]$ for $3$-cycles}
We assume for the next few sections that facet cycles $Z$ are unsigned.
\begin{lemma}\label{3-cycle}
  Let $K\subset \partial\mathcal{C}_{\kappa+2}$ be a $3$-cycle: that is, a cycle with three boundary matching tubes and three Pontrjagin-Thom tubes. Assume $K$ is parametrized by the facet cycle $a\longline b\longline c\longline a$. We have $[K]=ab+bc+ac+a+b+c+1$. 
\end{lemma}
\begin{proof}
  See Figure \ref{3-cycle compared to convex hull}.
  \begin{figure}
    \begin{tabular}{ m{8cm} m{8cm} }
      \def\svgscale{0.55}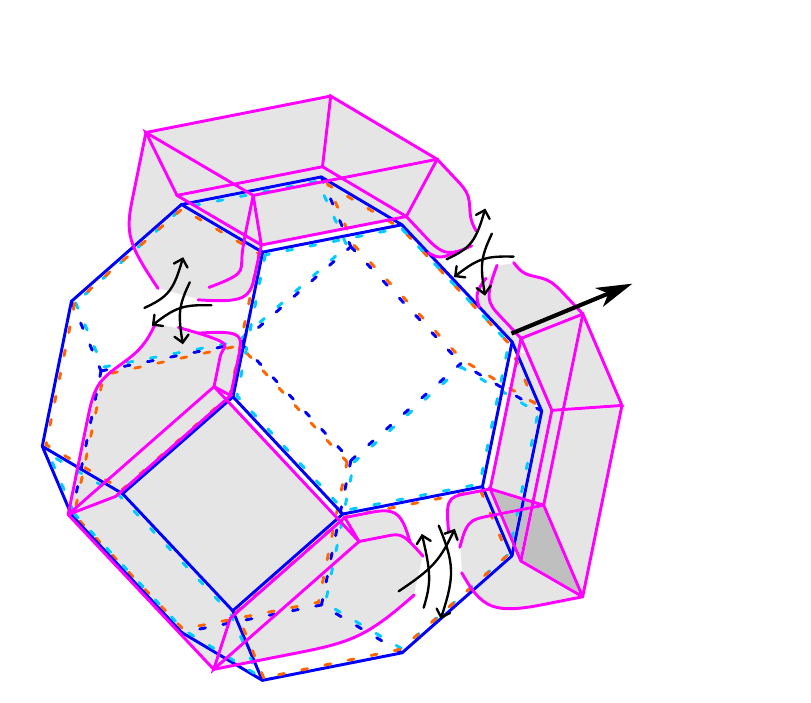 & \def\svgscale{0.55}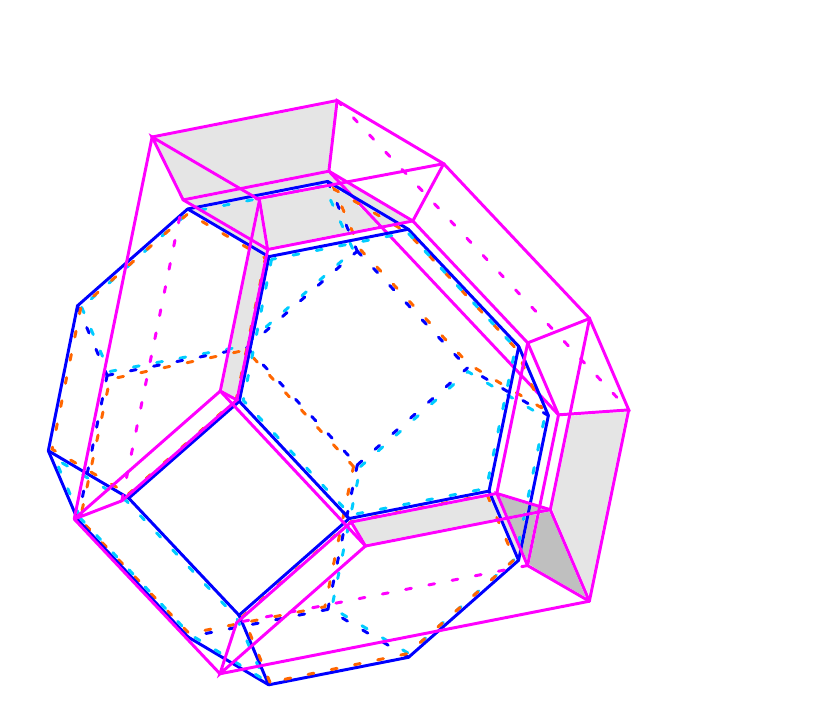
    \end{tabular}
    \caption{Left: an illustration of a $3$-cycle $K$. Right: an illustration of $K^{\conv}$. In the proof of Lemma \ref{3-cycle}, we view both cycles as tubes by splitting the cycles along the darkly shaded section. write $K$ as $K^{\conv}$ plus some twists.}
    \label{3-cycle compared to convex hull}
  \end{figure}
  Without loss of generality, we assume that $a<b<c$. $K$ is composed of boundary matching tubes $\tilde{\eta}_a,\tilde{\eta}_b,\tilde{\eta}_c$ lying in facets $\mathbf{G}_a$, $\mathbf{G}_b$, $\mathbf{G}_c$, and Pontrjagin-Thom tubes $\zeta_{ab},\zeta_{bc},\zeta_{ac}$ lying in facets $\mathbf{G}_{ab}$, $\mathbf{G}_{bc}$, $\mathbf{G}_{ac}$. Before we investigate $K$, we construct a new tube $K^{\conv}$ by modifying the boundary matching tubes. In place of the $\tilde{\eta}$ tubes defined in Construction \ref{cycle building}, we define tubes $T_a,T_b,T_c$ as follows: For each $\tilde{\eta}_r$, replace the $\{\Theta_r,\Lambda_r\}$, used to define $\tilde{\eta}_r$ in Section \ref{eta construction}, with the convex tube $T_r = \{\Theta_r^{\conv},\Lambda_r^{\conv}\}$. This should result in a more ``flat'' looking cycle $K^{\conv}$. Furthermore, if we view $K^{\conv}$ as a d.s.\ tube
  \[
    K^{\conv} = T_a\cup \zeta_{ac}\cup T_c\cup \zeta_{bc}\cup T_b\cup \zeta_{ab},
  \]
  with both ends meeting in $T_a\cap \zeta_{ab}$ (the darkly shaded section of Figure \ref{3-cycle compared to convex hull}), then $K^{\conv}$ is boundary-coherent, meaning we can trivialize $K^{\conv}$ as in Proposition \ref{trivialization}, and define an element $[K^{\conv}]\in \mathbb{Z}/2$. Note that $[K^{\conv}] = 0$, since $K^{\conv}$ looks like an $S^1$ family of fibers $E(r)$ that circles once around a hexagon without any further twisting.\par
  Now if we view $K$ as a d.s.\ tube
  \[
    K = \tilde{\eta}_a\cup \zeta_{ac}\cup \tilde{\eta}_c\cup \zeta_{bc}\cup \tilde{\eta}_b\cup \zeta_{ab},
  \]
  with both ends meeting in $\tilde{\eta}_a\cap \zeta_{ab}$, then we can obtain $K$ by adding twists to $K^{\conv}$ (see Figure \ref{3-cycle diagram}).
    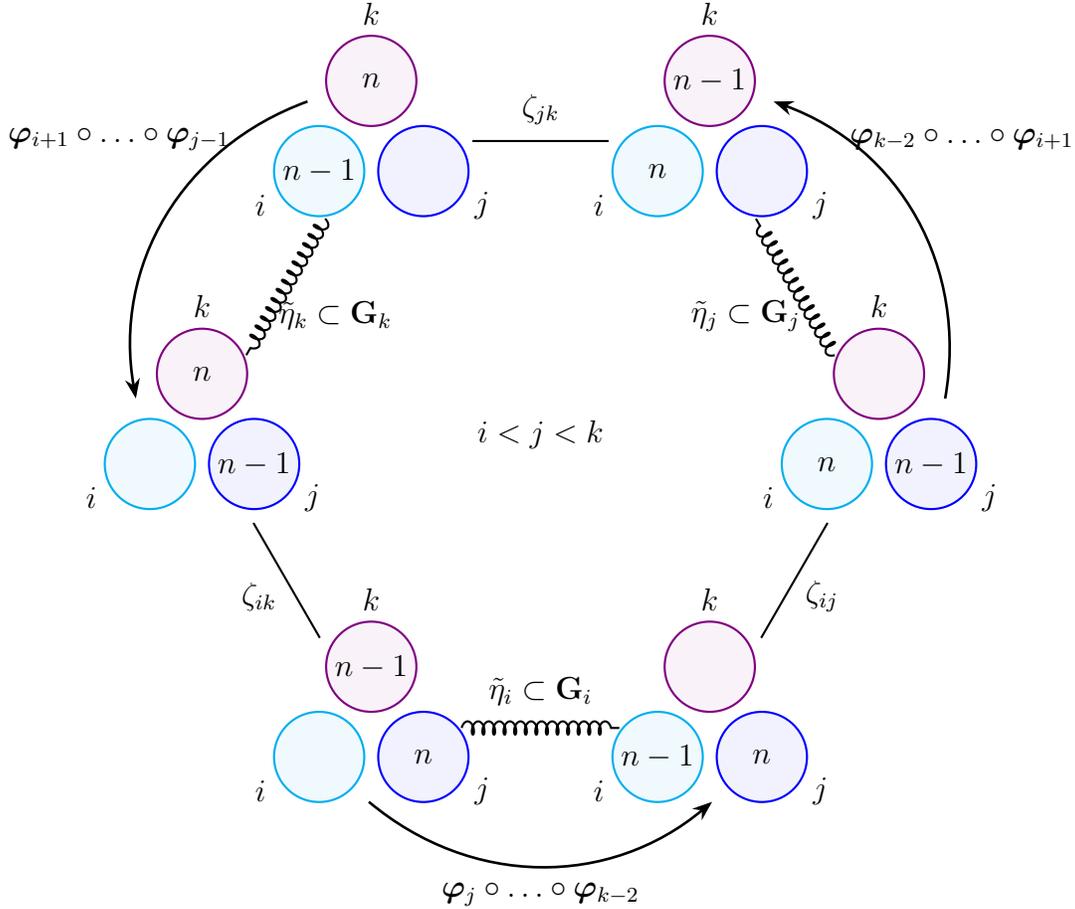
\begin{figure}
    \begin{tikzpicture}[scale=0.9]
  \draw (60:5) pic{mynode1={A, $n-1$, , $n$ , , }};
  \draw (120:5) pic{mynode1={B, $n$, , $n-1$ , }};
  \draw (180:5) pic{mynode1={C, $n$ , $n-1$ , , }};
  \draw (240:5) pic{mynode1={D, $n-1$ ,$n$ , , }};
  \draw (300:5) pic{mynode1={E, , $n$ , $n-1$ , }};
  \draw (0:5) pic{mynode1={F, , $n-1$ , $n$ , }};
  \draw[shorten <=1.2cm,shorten >=1.2cm, thick] (A)--(B);
  \draw[shorten <=1.2cm,shorten >=1.2cm, thick] (C)--(D);
  \draw[shorten <=1.2cm,shorten >=1.2cm, thick] (E)--(F);
  \draw[thick, decorate,decoration={coil,segment length=4pt}] (15:4.5)--(45:4.5);
  \draw[thick, decorate,decoration={coil,segment length=4pt}] (135:4.5)--(165:4.5);
  \draw[thick, decorate,decoration={coil,segment length=4pt}] (255:4.5)--(285:4.5
  );
  \draw [arrows = {-Stealth},line width=1pt] (5:6) to [bend right=40] (55:6);
  \draw [arrows = {-Stealth},line width=1pt] (125:6) to [bend right=40] (175:6);
  \draw [arrows = {-Stealth},line width=1pt] (245:6) to [bend right=40] (295:6);
  \node at (0,0) {$i<j<k$};
  \node at (30:3.5) {$\tilde{\eta}_j\subset \mathbf{G}_j$};
  \node at (150:3.5) {$\tilde{\eta}_k\subset \mathbf{G}_k$};
  \node at (270:3.8) {$\tilde{\eta}_i\subset \mathbf{G}_i$};
  \node at (90:4.8) {$\zeta_{jk}$};
  \node at (210:4.8) {$\zeta_{ik}$};
  \node at (-30:4.8) {$\zeta_{ij}$};
  \node at (35:7.6) {$\boldvarphi_{k-2}\circ\ldots\circ\boldvarphi_{i+1}$};
  \node at (145:7.6) {$\boldvarphi_{i+1}\circ\ldots\circ\boldvarphi_{j-1}$};
  \node at (-90:6.8) {$\boldvarphi_{j}\circ\ldots\circ\boldvarphi_{k-2}$};
  \end{tikzpicture}
    \caption{A diagram of the boundary matching and Pontrjagin-Thom neighborhoods behave in a 3-cycle. The straight edges represent Pontrjagin-Thom tubes and the coiled edges represent boundary matching tubes. The coils represent the possibly twists that have been added.}
    \label{3-cycle diagram}
  \end{figure}
  Indeed, $\tilde{\eta}_a = T_a\diamond(\boldvarphi_{c-2}\ldots\boldvarphi_b,\alpha_a)$, $\tilde{\eta}_b = T_b\diamond(\boldvarphi_{c-2}\ldots\boldvarphi_{a+1},\alpha_b)$, $\tilde{\eta}_c = T_c\diamond(\boldvarphi_{b-1}\ldots\boldvarphi_{a+1},\alpha_c)$. Therefore, we can use Proposition \ref{twist through concat} to compare $K$ with $K^{\conv}$:
  \begingroup
  \allowdisplaybreaks
  \begin{align*}
    K &= (T_a\diamond (\boldvarphi_{c-2}\ldots\boldvarphi_b,\alpha_a)) \cup \zeta_{ab} \cup
        (T_b\diamond(\boldvarphi_{c-2}\ldots\boldvarphi_{a+1},\alpha_b)) \cup \zeta_{bc}\\
    &\quad \cup(T_c\diamond(\boldvarphi_{b-1}\ldots\boldvarphi_{a+1},\alpha_c)) \cup \zeta_{ac}\\[\jot]
    &\cong  (T_a\diamond(\boldvarphi_{b}\ldots\boldvarphi_{c-2},\alpha_a)) \cup \zeta_{ab} \cup
      (T_b\diamond(\boldvarphi_{c-2}\ldots\boldvarphi_{a+1},\beta_b) \cup \zeta_{bc}\\
    &\quad \cup(T_c\diamond(\boldvarphi_{a+1}\ldots\boldvarphi_{b-1},\beta_c)) \cup \zeta_{ac}\\
    &\pushright{(\text{using Example \ref{Vij twists} to rewrite the boundary matching tubes})}\\[\jot]
    &\cong (T_a\diamond (\boldvarphi_{b}\ldots\boldvarphi_{c-2},\beta_a) \cup \zeta_{ab} \cup
      ((T_b\diamond(\boldvarphi_{c-2}\ldots\boldvarphi_{a+1},\alpha_b)\diamond(\tau^{-\omega}(\boldvarphi_{a+1}\ldots\boldvarphi_{b-1})\tau^{\omega},\alpha_b)\cup\zeta_{bc}\\
    &\quad \cup T_c \cup \zeta_{ac}\\
      &\pushright{(\text{by Proposition \ref{twist through concat}, setting } \omega = c+a \mod 2)}\\[\jot]
    &\cong (T_a\diamond(\boldvarphi_{b}\ldots\boldvarphi_{c-2},\beta_a) \cup \zeta_{ab} \cup
      (T_b\diamond(\boldvarphi_{c-2}\ldots\boldvarphi_{a+1}\tau^{-\omega}(\boldvarphi_{a+1}\ldots\boldvarphi_{b-1})\tau^{\omega},\alpha_b) \cup \zeta_{bc}\\
    &\quad \cup T_c \cup \zeta_{ac}\\[\jot]
    &\cong (T_a\diamond(\boldvarphi_{b}\ldots\boldvarphi_{c-2},\beta_a) \cup \zeta_{ab} \cup
      (T_b\diamond(\boldvarphi_{c-2}\ldots\boldvarphi_{a+1}(\boldvarphi_{a+1}^{-1}\ldots\boldvarphi_{b-1}^{-1}),\alpha_b)\cup \zeta_{bc} \cup T_c \cup \zeta_{ac}\\
       &\quad + ((\omega+1)(b-a-1))\\[\jot]
    &\cong (T_a\diamond(\boldvarphi_{b}\ldots\boldvarphi_{c-2},\beta_a)) \cup \zeta_{ab} \cup
    (T_b\diamond(\boldvarphi_{c-2}\ldots\boldvarphi_b,\alpha_b)) \cup \zeta_{bc} \cup
      T_c \cup \zeta_{ac}\\
    &\quad + ((\omega+1)(b-a-1))\\[\jot]
    &\cong (T_a\diamond(\boldvarphi_{b}\ldots\boldvarphi_{c-2}\tau^{-\omega'}(\boldvarphi_{c-2}\ldots\boldvarphi_b)\tau^{\omega'},\beta_a)) \cup \zeta_{ab} \cup
       T_b \cup \zeta_{bc} \cup
      T_c \cup \zeta_{ac}\\
    &\quad + ((\omega+1)(b-a-1))\\
    &\pushright{(\text{by Proposition \ref{twist through concat}, setting } \omega = b+c \mod 2)}\\[\jot]
    &\cong T_a \cup \zeta_{ab} \cup
       T_b \cup \zeta_{bc} \cup
       T_c \cup \zeta_{ac} + ((\omega'+1)(c-b-1))+((\omega+1)(b-a-1))\\
    &= K^{\conv} + (ab+bc+ac+a+b+c+1),
  \end{align*}
  \endgroup
  implying $[K] = [K^{\conv}] + ab+bc+ac+a+b+c+1 = ab+bc+ac+a+b+c+1$.
\end{proof}
\section{Simplifying facet cycles $Z$ without turnarounds}\label{simplifying}
We again only consider unsigned facet cycles $Z$ in this section.\par
To motivate our strategy, we recall Remark \ref{omnitruncation}, which says a loop in $\Pi^{n-1}$ occupying facets of the form $G_a$, $G_{ab}$ can nullhomotope by only traveling through facets of the form $G_{abc}$. We adapt this remark by noting $K$ is a tubular cycle that only passes through facets of the form $\mathbf{G}_a$, $\mathbf{G}_{ab}$, meaning we can isotope $K$ to a small loop by only passing $K$ through facets of the form $\mathbf{G}_{abc}$. It is important, therefore, for us to know how $K$ looks after each traversal. We explore this behavior on the level of facet cycles.\par
Consider a (unsigned) facet cycle $Z$ that has no turnarounds, that is, there are no portions that look like $a\longline \overrightarrow{b}\longline a$ (or $a\longline \overleftarrow{b}\longline a$). (See Figure \ref{backtrack comparison} for illustrations of these cycles.)
\begin{figure}
  \centering
  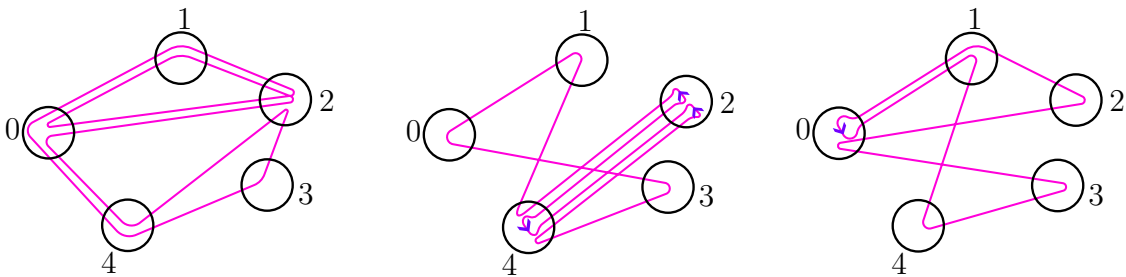
  \caption{Left: a cycle that does not backtrack. Middle and Right: cycles that backtrack. Cycles that backtrack are not allowed in Section \ref{simplifying}.}
  \label{backtrack comparison}
\end{figure}
This is equivalent to saying that if $Z$ parametrizes $K$, then $K$ does not contain any tube $\tilde{\eta}$ that starts and ends at the same face $\mathbf{G}_b$. We first investigate these types of facet cycles because we do not have to worry about how the orientation of a turnaround $a\longline \overrightarrow{b}\longline a$ affects the parametrization of $K$; orientations do not factor at all here.
Now suppose $Z$ contains a facet chain $a\longline b\longline c \longline d\longline e$, which we shall call $D$, where $b\neq d$, $a\neq c$, $c\neq e$, $a\neq d$, $b\neq e$. Then if we take $Z$ and replace $D$ with the facet chain $a\longline b \longline d\longline e$, which we call $D'$, then we obtain a new facet cycle $Z'$, which also lacks turnarounds (see Figure \ref{path move} for a schematic).
\begin{figure}
  \centering
\begingroup%
  \makeatletter%
  \providecommand\color[2][]{%
    \errmessage{(Inkscape) Color is used for the text in Inkscape, but the package 'color.sty' is not loaded}%
    \renewcommand\color[2][]{}%
  }%
  \providecommand\transparent[1]{%
    \errmessage{(Inkscape) Transparency is used (non-zero) for the text in Inkscape, but the package 'transparent.sty' is not loaded}%
    \renewcommand\transparent[1]{}%
  }%
  \providecommand\rotatebox[2]{#2}%
  \newcommand*\fsize{\dimexpr\f@size pt\relax}%
  \newcommand*\lineheight[1]{\fontsize{\fsize}{#1\fsize}\selectfont}%
  \ifx\svgwidth\undefined%
    \setlength{\unitlength}{118.90288826bp}%
    \ifx\svgscale\undefined%
      \relax%
    \else%
      \setlength{\unitlength}{\unitlength * \real{\svgscale}}%
    \fi%
  \else%
    \setlength{\unitlength}{\svgwidth}%
  \fi%
  \global\let\svgwidth\undefined%
  \global\let\svgscale\undefined%
  \makeatother%
  \begin{picture}(1,0.86469752)%
    \lineheight{1}%
    \setlength\tabcolsep{0pt}%
    \put(0,0){\includegraphics[width=\unitlength,page=1]{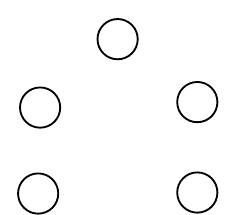}}%
    \put(0.42693939,0.8150624){\color[rgb]{0,0,0}\makebox(0,0)[lt]{\lineheight{1.25}\smash{\begin{tabular}[t]{l}$c$\end{tabular}}}}%
    \put(-0.00405442,0.07084909){\color[rgb]{0,0,0}\makebox(0,0)[lt]{\lineheight{1.25}\smash{\begin{tabular}[t]{l}$a$\end{tabular}}}}%
    \put(0.89196667,0.06941966){\color[rgb]{0,0,0}\makebox(0,0)[lt]{\lineheight{1.25}\smash{\begin{tabular}[t]{l}$e$\end{tabular}}}}%
    \put(0.85689837,0.52995148){\color[rgb]{0,0,0}\makebox(0,0)[lt]{\lineheight{1.25}\smash{\begin{tabular}[t]{l}$d$\end{tabular}}}}%
    \put(0.40850891,0.27759257){\color[rgb]{0,0,0}\makebox(0,0)[lt]{\lineheight{1.25}\smash{\begin{tabular}[t]{l}$D$\end{tabular}}}}%
    \put(0.04941217,0.50043962){\color[rgb]{0,0,0}\makebox(0,0)[lt]{\lineheight{1.25}\smash{\begin{tabular}[t]{l}$b$\end{tabular}}}}%
    \put(0,0){\includegraphics[width=\unitlength,page=2]{path_move_1.pdf}}%
  \end{picture}%
\endgroup%

  \hspace{0.5cm}\raisebox{1.5cm}{$\rightsquigarrow$}\hspace{0.5cm}
\begingroup%
  \makeatletter%
  \providecommand\color[2][]{%
    \errmessage{(Inkscape) Color is used for the text in Inkscape, but the package 'color.sty' is not loaded}%
    \renewcommand\color[2][]{}%
  }%
  \providecommand\transparent[1]{%
    \errmessage{(Inkscape) Transparency is used (non-zero) for the text in Inkscape, but the package 'transparent.sty' is not loaded}%
    \renewcommand\transparent[1]{}%
  }%
  \providecommand\rotatebox[2]{#2}%
  \newcommand*\fsize{\dimexpr\f@size pt\relax}%
  \newcommand*\lineheight[1]{\fontsize{\fsize}{#1\fsize}\selectfont}%
  \ifx\svgwidth\undefined%
    \setlength{\unitlength}{118.92079511bp}%
    \ifx\svgscale\undefined%
      \relax%
    \else%
      \setlength{\unitlength}{\unitlength * \real{\svgscale}}%
    \fi%
  \else%
    \setlength{\unitlength}{\svgwidth}%
  \fi%
  \global\let\svgwidth\undefined%
  \global\let\svgscale\undefined%
  \makeatother%
  \begin{picture}(1,0.86456732)%
    \lineheight{1}%
    \setlength\tabcolsep{0pt}%
    \put(-0.00405384,0.06569435){\color[rgb]{0,0,0}\makebox(0,0)[lt]{\lineheight{1.25}\smash{\begin{tabular}[t]{l}$a$\end{tabular}}}}%
    \put(0.42253095,0.26418227){\color[rgb]{0,0,0}\makebox(0,0)[lt]{\lineheight{1.25}\smash{\begin{tabular}[t]{l}$D'$\end{tabular}}}}%
    \put(0,0){\includegraphics[width=\unitlength,page=1]{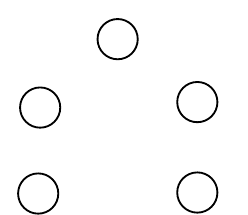}}%
    \put(0.42702555,0.81493967){\color[rgb]{0,0,0}\makebox(0,0)[lt]{\lineheight{1.25}\smash{\begin{tabular}[t]{l}$c$\end{tabular}}}}%
    \put(0.89198281,0.06940924){\color[rgb]{0,0,0}\makebox(0,0)[lt]{\lineheight{1.25}\smash{\begin{tabular}[t]{l}$e$\end{tabular}}}}%
    \put(0.85691979,0.52987166){\color[rgb]{0,0,0}\makebox(0,0)[lt]{\lineheight{1.25}\smash{\begin{tabular}[t]{l}$d$\end{tabular}}}}%
    \put(0.04955518,0.50036425){\color[rgb]{0,0,0}\makebox(0,0)[lt]{\lineheight{1.25}\smash{\begin{tabular}[t]{l}$b$\end{tabular}}}}%
    \put(0,0){\includegraphics[width=\unitlength,page=2]{path_move_2.pdf}}%
  \end{picture}%
\endgroup%

  \caption{For this path move to be valid, all triples of adjacent numbers in both the left and right diagram should consist of different numbers (e.g. $c,d,e$ must be different, but so must $a,b,d$).}
  \label{path move}
\end{figure}
We call this act of replacing $D$ with $D'$, thus replacing $Z$ with $Z'$, an \textit{elementary move}. Our goal in this section is to study the difference $[K]-[K']$ if $K$ (resp.\ $K'$) is parametrized by $Z$ (resp.\ $Z'$), and $Z$, $Z'$ differ by an elementary move.\par
We briefly outline this section's argument: Observe that we can obtain $K'$ by replacing a tubular cutout $U\subset K$ (parametrized by $D$) with another d.s.\ tube $U'$ (parametrized by $D'$). (See Figure \ref{short to long comparison} for an illustration.)
\begin{figure}
  \begin{tabular}{ m{7.8cm} m{6cm} }
      \def\svgscale{0.55}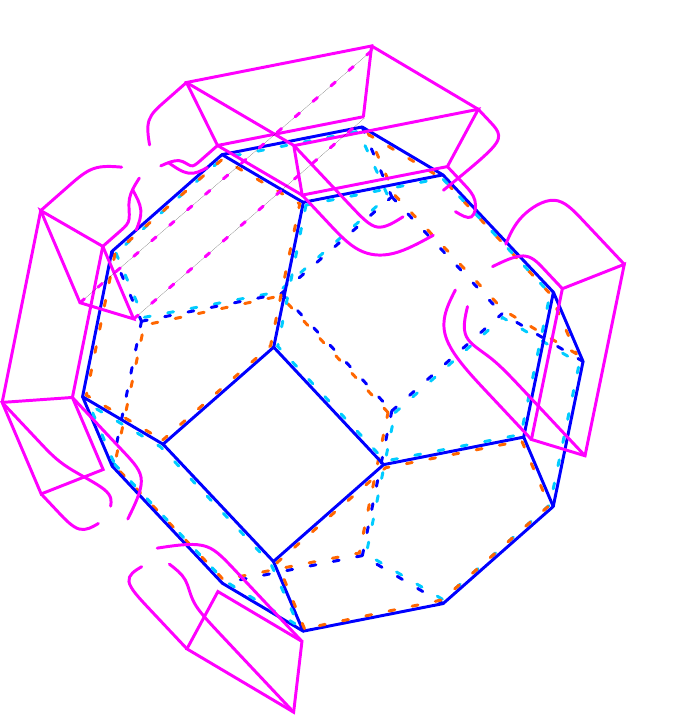 & \def\svgscale{0.55}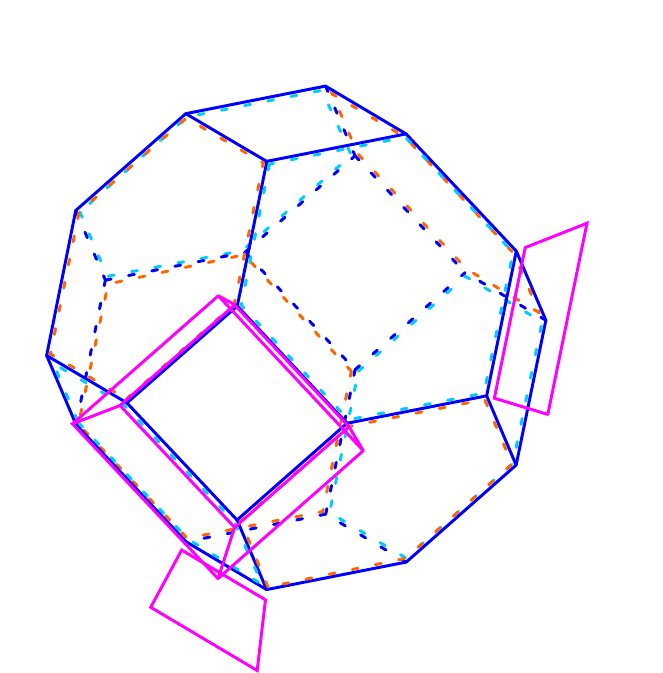
    \end{tabular}
  \caption{Left: The tubular chain $U\subset K$ parametrized by $a \text{---} b \text{---} c \text{---} d \text{---} e$. , Right: The tubular chain $U'\subset K'$ parametrized by $a \text{---} b \text{---} d \text{---} e$. We can imagine that $U$ is homotopic to $U'$ plus a twist.}
  \label{short to long comparison}
\end{figure}
We observe that if $U-U' = (\omega)$, then $[K]-[K'] = \omega \in \mathbb{Z}/2$.
\begin{figure}
  \centering
  \def\svgscale{0.55}
\begingroup%
  \makeatletter%
  \providecommand\color[2][]{%
    \errmessage{(Inkscape) Color is used for the text in Inkscape, but the package 'color.sty' is not loaded}%
    \renewcommand\color[2][]{}%
  }%
  \providecommand\transparent[1]{%
    \errmessage{(Inkscape) Transparency is used (non-zero) for the text in Inkscape, but the package 'transparent.sty' is not loaded}%
    \renewcommand\transparent[1]{}%
  }%
  \providecommand\rotatebox[2]{#2}%
  \newcommand*\fsize{\dimexpr\f@size pt\relax}%
  \newcommand*\lineheight[1]{\fontsize{\fsize}{#1\fsize}\selectfont}%
  \ifx\svgwidth\undefined%
    \setlength{\unitlength}{335.34071074bp}%
    \ifx\svgscale\undefined%
      \relax%
    \else%
      \setlength{\unitlength}{\unitlength * \real{\svgscale}}%
    \fi%
  \else%
    \setlength{\unitlength}{\svgwidth}%
  \fi%
  \global\let\svgwidth\undefined%
  \global\let\svgscale\undefined%
  \makeatother%
  \begin{picture}(1,1.02305784)%
    \lineheight{1}%
    \setlength\tabcolsep{0pt}%
    \put(0,0){\includegraphics[width=\unitlength,page=1]{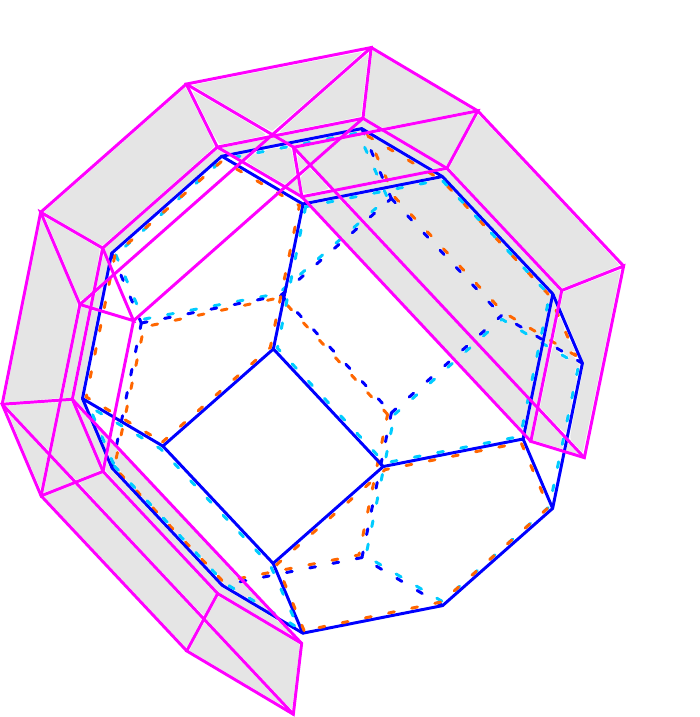}}%
    \put(0.72904939,0.15175196){\color[rgb]{0,0,0}\makebox(0,0)[lt]{\lineheight{1.25}\smash{\begin{tabular}[t]{l}$\mathbf{G}_a$\end{tabular}}}}%
    \put(0.80166425,0.22250457){\color[rgb]{0,0,0}\makebox(0,0)[lt]{\lineheight{1.25}\smash{\begin{tabular}[t]{l}$\mathbf{G}_e$\end{tabular}}}}%
    \put(0.05058598,0.17334579){\color[rgb]{0,0,0}\makebox(0,0)[lt]{\lineheight{1.25}\smash{\begin{tabular}[t]{l}$\mathbf{G}_b$\end{tabular}}}}%
    \put(0.0673205,0.8376632){\color[rgb]{0,0,0}\makebox(0,0)[lt]{\lineheight{1.25}\smash{\begin{tabular}[t]{l}$\mathbf{G}_c$\end{tabular}}}}%
    \put(0.75022556,0.79962327){\color[rgb]{0,0,0}\makebox(0,0)[lt]{\lineheight{1.25}\smash{\begin{tabular}[t]{l}$\mathbf{G}_d$\end{tabular}}}}%
    \put(0.4464506,1.00268744){\color[rgb]{0,0,0}\makebox(0,0)[lt]{\lineheight{1.25}\smash{\begin{tabular}[t]{l}$U^{\conv}$\end{tabular}}}}%
    \put(0.15027413,0.93880679){\color[rgb]{0,0,0}\makebox(0,0)[lt]{\lineheight{1.25}\smash{\begin{tabular}[t]{l}$J$\end{tabular}}}}%
    \put(0,0){\includegraphics[width=\unitlength,page=2]{U_conv_tube.pdf}}%
  \end{picture}%
\endgroup%

  \caption{The convex version $U^{\conv}$ of $U$, where $U$ is parametrized by $a \text{---} b \text{---} c \text{---} d \text{---} e$. All three of the boundary matching tubes $\tilde{\eta}_b$, $\tilde{\eta}_c$, $\tilde{\eta}_d$ in $U$ have been replaced with their convex versions.}
  \label{convex long}
\end{figure}
\begin{lemma}\label{pulling along cycle}
  Let the facet chain $D$, denoted by $a\longline b\longline c\longline d\longline e$, parametrize $U = \{(\Theta_a,\alpha),(\Theta_e,\beta)\}$. Now let $U^{\conv}$ be the convex version of $U$ (see Figure \ref{convex long}). Table \ref{pulling along cycle table}
  \begin{table}
    \centering
    \hspace{-1cm}
    \begin{tabular}{ M{4.5cm} m{3cm} m{3cm} m{3cm} }
\makecell{Change\\ $\hexchangeDots{$c_b$}{$d_b$}{$a_b$}$ \raisebox{0.8cm}{or} $\hexchangeDots{$c_d$}{$b_d$}{$e_d$}$} & \makecell{Twist $g_b$\\ (resp. $g_d$)} & \makecell{Twist $g'_b$\\ (resp. $g'_d$)} & \makecell{Twist $f_b$ (resp.\ $f_d$)} \\
\hexchangeDots{$j$}{$k$}{$i$}& $\boldvarphi_{j-1}\ldots \boldvarphi_{i+1}$ & $\boldvarphi_{k-1}\ldots \boldvarphi_{j+1}$ & $\boldvarphi_{k-1}\ldots\widehat{\boldvarphi_{j}}\ldots\boldvarphi_{i+1}$ \\ 
\hexchangeDots{$j$}{$i$}{$k$}& $\boldvarphi_{j+1}\ldots \boldvarphi_{k-1}$ & $\boldvarphi_{i+1}\ldots \boldvarphi_{j-1}$ & $\boldvarphi_{i+1}\ldots\widehat{\boldvarphi_{j}}\ldots\boldvarphi_{k-1}$ \\
\hexchangeDots{$k$}{$j$}{$i$}& $\boldvarphi_{k-1}\ldots \boldvarphi_{i+1}$ & $\boldvarphi_{j+1}\ldots \boldvarphi_{k-1}$ & $\boldvarphi_{j, k}\boldvarphi_{j-1}\ldots\boldvarphi_{i+1}$ \\
\hexchangeDots{$k$}{$i$}{$j$} & $\boldvarphi_{k-1}\ldots \boldvarphi_{j+1}$ & $\boldvarphi_{i+1}\ldots \boldvarphi_{k-1}$ & $\boldvarphi_{i+1}\ldots\boldvarphi_{j-1}\boldvarphi_{j, k}$ \\
\hexchangeDots{$i$}{$k$}{$j$} &$\boldvarphi_{i+1}\ldots \boldvarphi_{j-1}$ & $\boldvarphi_{k-1}\ldots \boldvarphi_{i+1}$ & $\boldvarphi_{k-1}\ldots\boldvarphi_{j+1}\boldvarphi_{i+1, j + 1}$ \\ 
\hexchangeDots{$i$}{$j$}{$k$} &$\boldvarphi_{i+1}\ldots \boldvarphi_{k-1}$ & $\boldvarphi_{j-1}\ldots \boldvarphi_{i+1}$& $\boldvarphi_{i+1, j+1}\boldvarphi_{j+1}\ldots\boldvarphi_{k-1}$
    \end{tabular}
    \caption{The table referenced in Lemma \ref{pulling along cycle}}
    \label{pulling along cycle table}
  \end{table}
  lists twists $f_b$, $f_d$ associated to either change based on case output $\vcenter{\hbox{\DchangeDash{$a$}{$b$}{$c$}{$d$}}}$, $\vcenter{\hbox{\DchangeDash{$e$}{$d$}{$c$}{$b$}}}$ Then,
  \begin{equation}\label{long to readjusted eq}
    U \cong U^{\conv} \diamond (f_b,\alpha)\diamond (f_d,\beta) + (bc+cd+bd+b+c+d+1).
  \end{equation}
\end{lemma}
See Figure \ref{long to readjusted fig} for an illustration of the two homotopic tubes in Equation (\ref{long to readjusted eq}).
\begin{figure}
  \begin{tabular}{ m{6.3cm} m{1cm} m{6.5cm} }
      \def\svgscale{0.55}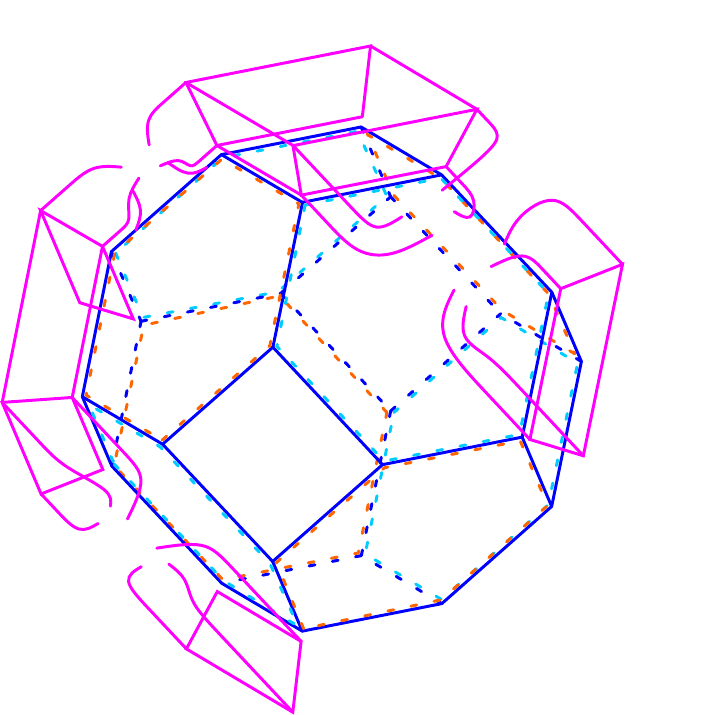 & $\cong$ & \def\svgscale{0.55}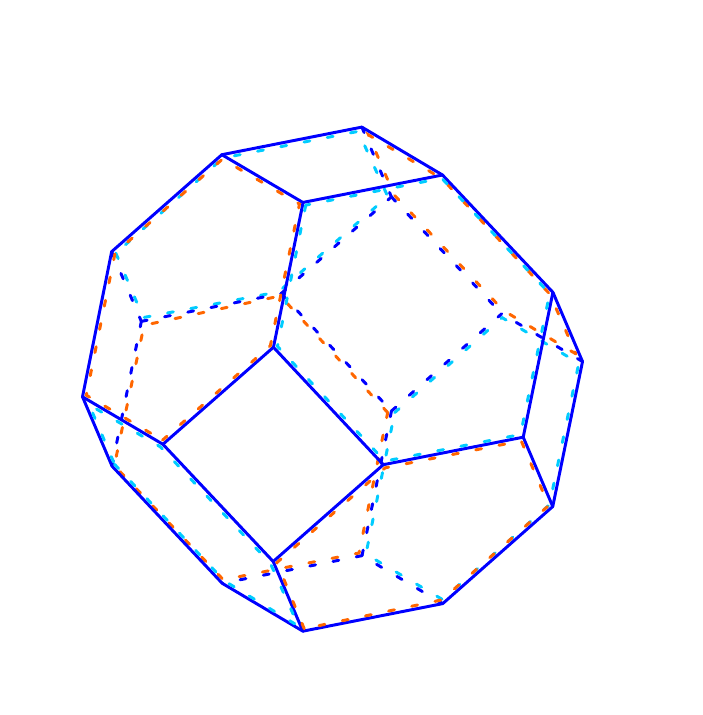
    \end{tabular}
  \caption{Left: the tube $U$. Right: the tube $U^{\conv}\diamond (f_b,\alpha)\diamond (f_d,\beta)$, which is homotopic to $U$. The twists in the darkly shaded part of $U$ get redistributed, while the twists in the lightly shaded part stay the same.}
  \label{long to readjusted fig}
\end{figure}
We view this lemma as a way to take $U$ and ``straighten'' the middle boundary matching component $\tilde{\eta}_c$, working the twists to either end.
\begin{proof}
  We denote
  \[
    U = \tilde{\eta}_b\cup\zeta_{bc}\cup\tilde{\eta}_c\cup\zeta_{cd}\cup\tilde{\eta}_d,\qquad U^{\conv} = \tilde{\eta}^{\conv}_b\cup\zeta_{bc}\cup\tilde{\eta}^{\conv}_c\cup\zeta_{cd}\cup\tilde{\eta}^{\conv}_d.
  \]
  Observe that $U$ occupies $\mathbf{G}_b\cup \mathbf{G}_{bc}\cup \mathbf{G}_c\cup \mathbf{G}_{cd} \cup \mathbf{G}_d$. Also occupying these facets is any $3$-cycle $K$  parametrized by $b\longline c\longline d\longline b$. In this spirit, we fix such a $3$-cycle $K = \zeta_{bc}\cup \tilde{\eta}_{c}\cup \zeta_{cd} \cup \tilde{\eta}'_{d}\cup \zeta_{bd}\cup \tilde{\eta}'_b$, where note the tubes $\zeta_{bc},\tilde{\eta}_{c},\zeta_{cd}\subset K$ are the same as in $U$. The tube $\tilde{\eta}'_{b}$ (resp. $\tilde{\eta}'_{d}$) is, however, a new tube that straddles faces $\mathbf{G}_{c}$, $\mathbf{G}_{d}$ (resp. $\mathbf{G}_{b}$, $\mathbf{G}_{c}$ ). We first split $K$ up into tubes $V_1 := \zeta_{bc}\cup \tilde{\eta}_{c}\cup \zeta_{cd}$, $V_2 := \tilde{\eta}'_{d}\cup \zeta_{bd}\cup \tilde{\eta}'_b$ and compute
  \begin{equation}\label{slice equation}
    V_1- V_1^{\conv} \cong V_2 - V_2^{\conv} + (bc+cd+bd+b+c+d+1)
  \end{equation}
  (see Figure \ref{slice comparison}).
  \begin{figure}
    \begingroup
      \renewcommand{\arraystretch}{6}:
    \begin{tabular}{ m{0.3em} m{5.0cm} m{0.3cm} m{5.2cm} m{3cm} }
      &\def\svgscale{0.48}
\begingroup%
  \makeatletter%
  \providecommand\color[2][]{%
    \errmessage{(Inkscape) Color is used for the text in Inkscape, but the package 'color.sty' is not loaded}%
    \renewcommand\color[2][]{}%
  }%
  \providecommand\transparent[1]{%
    \errmessage{(Inkscape) Transparency is used (non-zero) for the text in Inkscape, but the package 'transparent.sty' is not loaded}%
    \renewcommand\transparent[1]{}%
  }%
  \providecommand\rotatebox[2]{#2}%
  \newcommand*\fsize{\dimexpr\f@size pt\relax}%
  \newcommand*\lineheight[1]{\fontsize{\fsize}{#1\fsize}\selectfont}%
  \ifx\svgwidth\undefined%
    \setlength{\unitlength}{307.72780771bp}%
    \ifx\svgscale\undefined%
      \relax%
    \else%
      \setlength{\unitlength}{\unitlength * \real{\svgscale}}%
    \fi%
  \else%
    \setlength{\unitlength}{\svgwidth}%
  \fi%
  \global\let\svgwidth\undefined%
  \global\let\svgscale\undefined%
  \makeatother%
  \begin{picture}(1,1.14331327)%
    \lineheight{1}%
    \setlength\tabcolsep{0pt}%
    \put(0,0){\includegraphics[width=\unitlength,page=1]{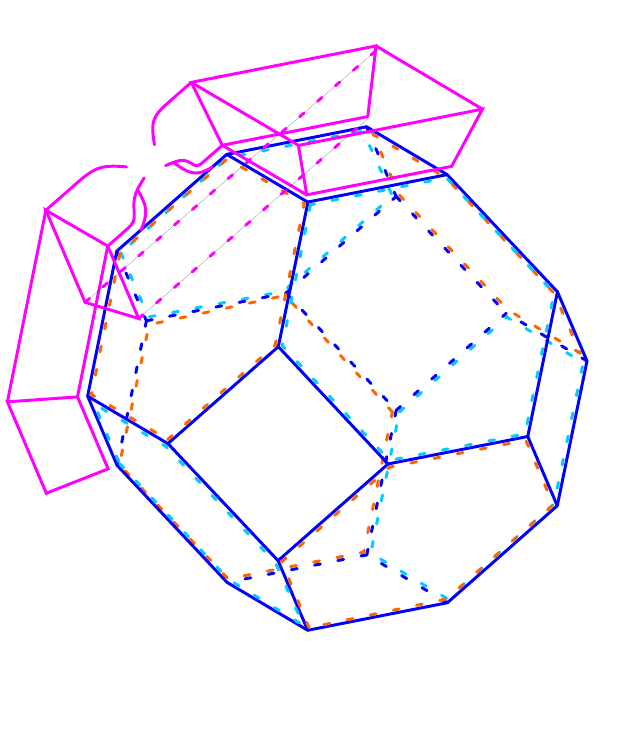}}%
    \put(0.14181735,0.27584869){\color[rgb]{0,0,0}\makebox(0,0)[lt]{\lineheight{1.25}\smash{\begin{tabular}[t]{l}$\mathbf{G}_b$\end{tabular}}}}%
    \put(0.07718831,0.95103174){\color[rgb]{0,0,0}\makebox(0,0)[lt]{\lineheight{1.25}\smash{\begin{tabular}[t]{l}$\mathbf{G}_c$\end{tabular}}}}%
    \put(0.78237585,0.83646209){\color[rgb]{0,0,0}\makebox(0,0)[lt]{\lineheight{1.25}\smash{\begin{tabular}[t]{l}$\mathbf{G}_d$\end{tabular}}}}%
    \put(0.45134226,1.12111511){\color[rgb]{0,0,0}\makebox(0,0)[lt]{\lineheight{1.25}\smash{\begin{tabular}[t]{l}$U$\end{tabular}}}}%
    \put(0,0){\includegraphics[width=\unitlength,page=2]{top_slice.pdf}}%
  \end{picture}%
\endgroup%
 & \raisebox{0.5cm}{$-$} & \def\svgscale{0.48}
\begingroup%
  \makeatletter%
  \providecommand\color[2][]{%
    \errmessage{(Inkscape) Color is used for the text in Inkscape, but the package 'color.sty' is not loaded}%
    \renewcommand\color[2][]{}%
  }%
  \providecommand\transparent[1]{%
    \errmessage{(Inkscape) Transparency is used (non-zero) for the text in Inkscape, but the package 'transparent.sty' is not loaded}%
    \renewcommand\transparent[1]{}%
  }%
  \providecommand\rotatebox[2]{#2}%
  \newcommand*\fsize{\dimexpr\f@size pt\relax}%
  \newcommand*\lineheight[1]{\fontsize{\fsize}{#1\fsize}\selectfont}%
  \ifx\svgwidth\undefined%
    \setlength{\unitlength}{309.55004594bp}%
    \ifx\svgscale\undefined%
      \relax%
    \else%
      \setlength{\unitlength}{\unitlength * \real{\svgscale}}%
    \fi%
  \else%
    \setlength{\unitlength}{\svgwidth}%
  \fi%
  \global\let\svgwidth\undefined%
  \global\let\svgscale\undefined%
  \makeatother%
  \begin{picture}(1,0.98146496)%
    \lineheight{1}%
    \setlength\tabcolsep{0pt}%
    \put(0,0){\includegraphics[width=\unitlength,page=1]{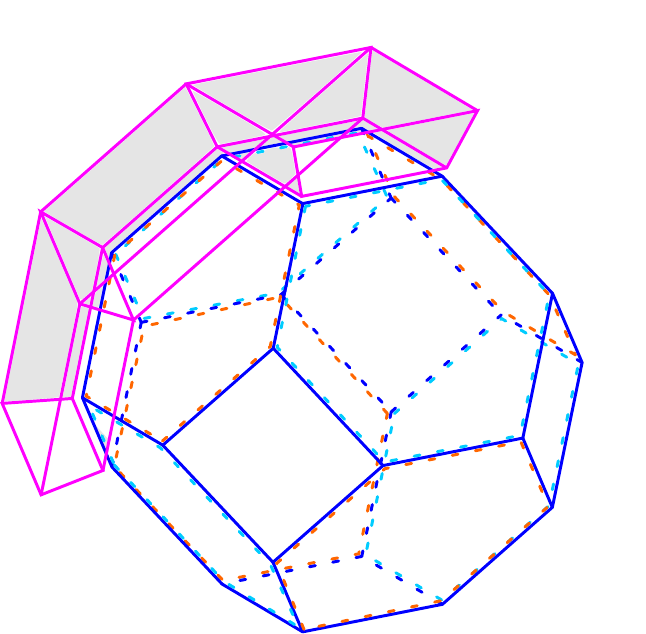}}%
    \put(0.14202279,0.11910658){\color[rgb]{0,0,0}\makebox(0,0)[lt]{\lineheight{1.25}\smash{\begin{tabular}[t]{l}$\mathbf{G}_b$\end{tabular}}}}%
    \put(0.07292937,0.78062395){\color[rgb]{0,0,0}\makebox(0,0)[lt]{\lineheight{1.25}\smash{\begin{tabular}[t]{l}$\mathbf{G}_c$\end{tabular}}}}%
    \put(0.78365683,0.69580293){\color[rgb]{0,0,0}\makebox(0,0)[lt]{\lineheight{1.25}\smash{\begin{tabular}[t]{l}$\mathbf{G}_d$\end{tabular}}}}%
    \put(0.48364734,0.95939748){\color[rgb]{0,0,0}\makebox(0,0)[lt]{\lineheight{1.25}\smash{\begin{tabular}[t]{l}$U^{\conv}$\end{tabular}}}}%
    \put(0,0){\includegraphics[width=\unitlength,page=2]{top_slice_convex.pdf}}%
  \end{picture}%
\endgroup%
& \\
      \raisebox{0.5cm}{$=$} &\def\svgscale{0.48}
\begingroup%
  \makeatletter%
  \providecommand\color[2][]{%
    \errmessage{(Inkscape) Color is used for the text in Inkscape, but the package 'color.sty' is not loaded}%
    \renewcommand\color[2][]{}%
  }%
  \providecommand\transparent[1]{%
    \errmessage{(Inkscape) Transparency is used (non-zero) for the text in Inkscape, but the package 'transparent.sty' is not loaded}%
    \renewcommand\transparent[1]{}%
  }%
  \providecommand\rotatebox[2]{#2}%
  \newcommand*\fsize{\dimexpr\f@size pt\relax}%
  \newcommand*\lineheight[1]{\fontsize{\fsize}{#1\fsize}\selectfont}%
  \ifx\svgwidth\undefined%
    \setlength{\unitlength}{326.97745496bp}%
    \ifx\svgscale\undefined%
      \relax%
    \else%
      \setlength{\unitlength}{\unitlength * \real{\svgscale}}%
    \fi%
  \else%
    \setlength{\unitlength}{\svgwidth}%
  \fi%
  \global\let\svgwidth\undefined%
  \global\let\svgscale\undefined%
  \makeatother%
  \begin{picture}(1,0.84057733)%
    \lineheight{1}%
    \setlength\tabcolsep{0pt}%
    \put(0,0){\includegraphics[width=\unitlength,page=1]{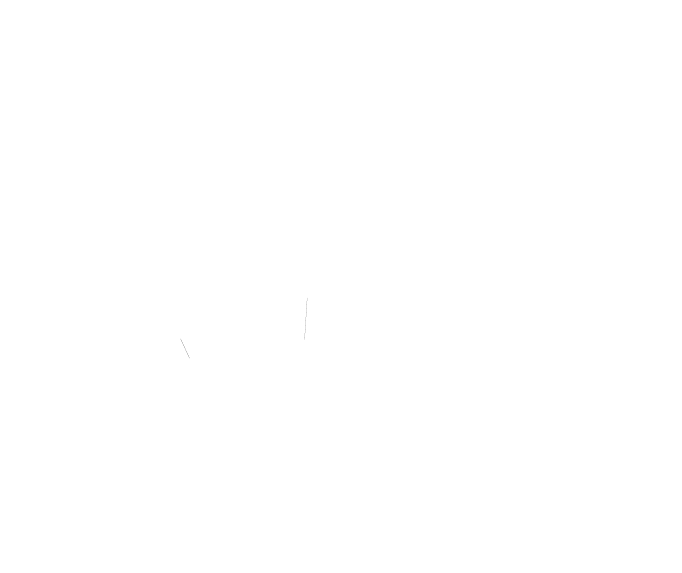}}%
    \put(0.27430711,0.81965852){\color[rgb]{0,0,0}\makebox(0,0)[lt]{\lineheight{1.25}\smash{\begin{tabular}[t]{l}$\tilde{\eta}_b\cup V_1 \cup \tilde{\eta}_d$\end{tabular}}}}%
    \put(0,0){\includegraphics[width=\unitlength,page=2]{bottom_slice.pdf}}%
    \put(0.02127368,0.04563727){\color[rgb]{0,0,0}\makebox(0,0)[lt]{\lineheight{1.25}\smash{\begin{tabular}[t]{l}$\mathbf{G}_b$\end{tabular}}}}%
    \put(0.11144139,0.63809659){\color[rgb]{0,0,0}\makebox(0,0)[lt]{\lineheight{1.25}\smash{\begin{tabular}[t]{l}$\mathbf{G}_c$\end{tabular}}}}%
    \put(0.79518767,0.66247301){\color[rgb]{0,0,0}\makebox(0,0)[lt]{\lineheight{1.25}\smash{\begin{tabular}[t]{l}$\mathbf{G}_d$\end{tabular}}}}%
    \put(0,0){\includegraphics[width=\unitlength,page=3]{bottom_slice.pdf}}%
  \end{picture}%
\endgroup%
 & \raisebox{0.5cm}{$-$} & \def\svgscale{0.48}
\begingroup%
  \makeatletter%
  \providecommand\color[2][]{%
    \errmessage{(Inkscape) Color is used for the text in Inkscape, but the package 'color.sty' is not loaded}%
    \renewcommand\color[2][]{}%
  }%
  \providecommand\transparent[1]{%
    \errmessage{(Inkscape) Transparency is used (non-zero) for the text in Inkscape, but the package 'transparent.sty' is not loaded}%
    \renewcommand\transparent[1]{}%
  }%
  \providecommand\rotatebox[2]{#2}%
  \newcommand*\fsize{\dimexpr\f@size pt\relax}%
  \newcommand*\lineheight[1]{\fontsize{\fsize}{#1\fsize}\selectfont}%
  \ifx\svgwidth\undefined%
    \setlength{\unitlength}{319.11184957bp}%
    \ifx\svgscale\undefined%
      \relax%
    \else%
      \setlength{\unitlength}{\unitlength * \real{\svgscale}}%
    \fi%
  \else%
    \setlength{\unitlength}{\svgwidth}%
  \fi%
  \global\let\svgwidth\undefined%
  \global\let\svgscale\undefined%
  \makeatother%
  \begin{picture}(1,0.86291272)%
    \lineheight{1}%
    \setlength\tabcolsep{0pt}%
    \put(0,0){\includegraphics[width=\unitlength,page=1]{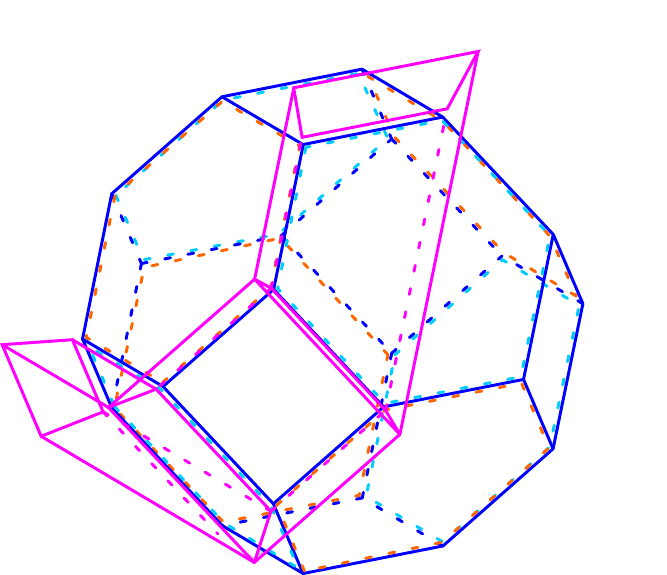}}%
    \put(0.04929368,0.04695192){\color[rgb]{0,0,0}\makebox(0,0)[lt]{\lineheight{1.25}\smash{\begin{tabular}[t]{l}$\mathbf{G}_b$\end{tabular}}}}%
    \put(0.13604205,0.68032258){\color[rgb]{0,0,0}\makebox(0,0)[lt]{\lineheight{1.25}\smash{\begin{tabular}[t]{l}$\mathbf{G}_c$\end{tabular}}}}%
    \put(0.79013908,0.66893022){\color[rgb]{0,0,0}\makebox(0,0)[lt]{\lineheight{1.25}\smash{\begin{tabular}[t]{l}$\mathbf{G}_d$\end{tabular}}}}%
    \put(0,0){\includegraphics[width=\unitlength,page=2]{bottom_slice_convex.pdf}}%
    \put(0.39806439,0.84147843){\color[rgb]{0,0,0}\makebox(0,0)[lt]{\lineheight{1.25}\smash{\begin{tabular}[t]{l}$T_b\cup V^{\conv}\cup T_d$\end{tabular}}}}%
    \put(0,0){\includegraphics[width=\unitlength,page=3]{bottom_slice_convex.pdf}}%
  \end{picture}%
\endgroup%
 & \raisebox{0.5cm}{$+\left(\begin{gathered}bc+cd+bd\\+b+c+d\\+1\end{gathered}\right)$} \\
    \end{tabular}
    \endgroup
    \caption{The tubes on the left combine to form a (contrived) $3$-cycle $K$, while the tubes on the right combine to form $K^{\conv}$. By Lemma \ref{3-cycle}, the difference of the top tubes is indeed the difference of the bottom tubes plus $(bc+cd+bd+b+c+d+1)$.}
    \label{slice comparison}
  \end{figure}
  To lighten notation, we define $T_r:=\tilde{\eta}_r^{\conv}$, $T'_r:=\tilde{\eta}_r'^{\conv}$ for general $r$. We use Equation \ref{slice equation} to compute $U-U^{\conv}$, with Figure \ref{differences visual aid} as a visual aid.
    \begin{figure}
      \begingroup
      \renewcommand{\arraystretch}{4}:
    \begin{tabular}{ p{0.1em} m{5.2cm} p{0.3cm} m{5.1cm} p{3cm}}
      &\def\svgscale{0.44}
\begingroup%
  \makeatletter%
  \providecommand\color[2][]{%
    \errmessage{(Inkscape) Color is used for the text in Inkscape, but the package 'color.sty' is not loaded}%
    \renewcommand\color[2][]{}%
  }%
  \providecommand\transparent[1]{%
    \errmessage{(Inkscape) Transparency is used (non-zero) for the text in Inkscape, but the package 'transparent.sty' is not loaded}%
    \renewcommand\transparent[1]{}%
  }%
  \providecommand\rotatebox[2]{#2}%
  \newcommand*\fsize{\dimexpr\f@size pt\relax}%
  \newcommand*\lineheight[1]{\fontsize{\fsize}{#1\fsize}\selectfont}%
  \ifx\svgwidth\undefined%
    \setlength{\unitlength}{342.92981966bp}%
    \ifx\svgscale\undefined%
      \relax%
    \else%
      \setlength{\unitlength}{\unitlength * \real{\svgscale}}%
    \fi%
  \else%
    \setlength{\unitlength}{\svgwidth}%
  \fi%
  \global\let\svgwidth\undefined%
  \global\let\svgscale\undefined%
  \makeatother%
  \begin{picture}(1,0.99829184)%
    \lineheight{1}%
    \setlength\tabcolsep{0pt}%
    \put(0,0){\includegraphics[width=\unitlength,page=1]{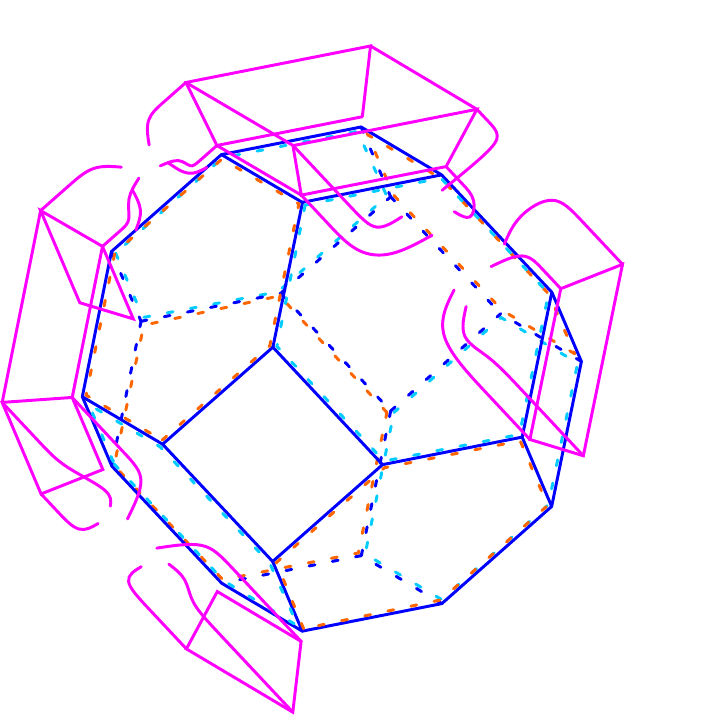}}%
    \put(0.04120855,0.16738459){\color[rgb]{0,0,0}\makebox(0,0)[lt]{\lineheight{1.25}\smash{\begin{tabular}[t]{l}$\mathbf{G}_b$\end{tabular}}}}%
    \put(0.03201579,0.80932579){\color[rgb]{0,0,0}\makebox(0,0)[lt]{\lineheight{1.25}\smash{\begin{tabular}[t]{l}$\mathbf{G}_c$\end{tabular}}}}%
    \put(0.8047152,0.72251512){\color[rgb]{0,0,0}\makebox(0,0)[lt]{\lineheight{1.25}\smash{\begin{tabular}[t]{l}$\mathbf{G}_d$\end{tabular}}}}%
    \put(0.3976939,0.97837234){\color[rgb]{0,0,0}\makebox(0,0)[lt]{\lineheight{1.25}\smash{\begin{tabular}[t]{l}$U$\end{tabular}}}}%
    \put(0,0){\includegraphics[width=\unitlength,page=2]{U_tube_unlabeled.pdf}}%
  \end{picture}%
\endgroup%
 & \raisebox{0.5cm}{$-$} & \def\svgscale{0.44}
\begingroup%
  \makeatletter%
  \providecommand\color[2][]{%
    \errmessage{(Inkscape) Color is used for the text in Inkscape, but the package 'color.sty' is not loaded}%
    \renewcommand\color[2][]{}%
  }%
  \providecommand\transparent[1]{%
    \errmessage{(Inkscape) Transparency is used (non-zero) for the text in Inkscape, but the package 'transparent.sty' is not loaded}%
    \renewcommand\transparent[1]{}%
  }%
  \providecommand\rotatebox[2]{#2}%
  \newcommand*\fsize{\dimexpr\f@size pt\relax}%
  \newcommand*\lineheight[1]{\fontsize{\fsize}{#1\fsize}\selectfont}%
  \ifx\svgwidth\undefined%
    \setlength{\unitlength}{318.55027086bp}%
    \ifx\svgscale\undefined%
      \relax%
    \else%
      \setlength{\unitlength}{\unitlength * \real{\svgscale}}%
    \fi%
  \else%
    \setlength{\unitlength}{\svgwidth}%
  \fi%
  \global\let\svgwidth\undefined%
  \global\let\svgscale\undefined%
  \makeatother%
  \begin{picture}(1,1.07698212)%
    \lineheight{1}%
    \setlength\tabcolsep{0pt}%
    \put(0,0){\includegraphics[width=\unitlength,page=1]{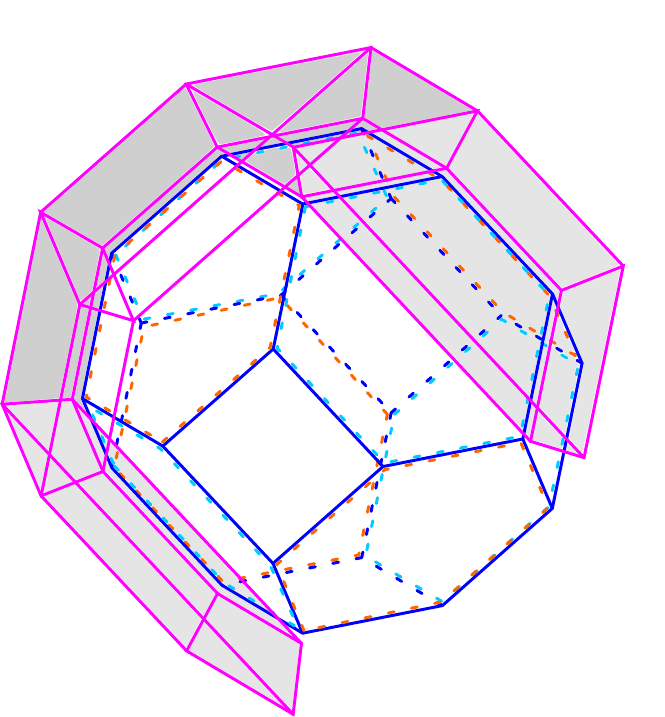}}%
    \put(0.05325267,0.18248297){\color[rgb]{0,0,0}\makebox(0,0)[lt]{\lineheight{1.25}\smash{\begin{tabular}[t]{l}$\mathbf{G}_b$\end{tabular}}}}%
    \put(0.08970486,0.87710679){\color[rgb]{0,0,0}\makebox(0,0)[lt]{\lineheight{1.25}\smash{\begin{tabular}[t]{l}$\mathbf{G}_c$\end{tabular}}}}%
    \put(0.78976945,0.84177086){\color[rgb]{0,0,0}\makebox(0,0)[lt]{\lineheight{1.25}\smash{\begin{tabular}[t]{l}$\mathbf{G}_d$\end{tabular}}}}%
    \put(0.46998264,1.05553819){\color[rgb]{0,0,0}\makebox(0,0)[lt]{\lineheight{1.25}\smash{\begin{tabular}[t]{l}$U^{\conv}$\end{tabular}}}}%
  \end{picture}%
\endgroup%
&\\
  \raisebox{0.5cm}{$\cong$}&\def\svgscale{0.44}
\begingroup%
  \makeatletter%
  \providecommand\color[2][]{%
    \errmessage{(Inkscape) Color is used for the text in Inkscape, but the package 'color.sty' is not loaded}%
    \renewcommand\color[2][]{}%
  }%
  \providecommand\transparent[1]{%
    \errmessage{(Inkscape) Transparency is used (non-zero) for the text in Inkscape, but the package 'transparent.sty' is not loaded}%
    \renewcommand\transparent[1]{}%
  }%
  \providecommand\rotatebox[2]{#2}%
  \newcommand*\fsize{\dimexpr\f@size pt\relax}%
  \newcommand*\lineheight[1]{\fontsize{\fsize}{#1\fsize}\selectfont}%
  \ifx\svgwidth\undefined%
    \setlength{\unitlength}{368.39948038bp}%
    \ifx\svgscale\undefined%
      \relax%
    \else%
      \setlength{\unitlength}{\unitlength * \real{\svgscale}}%
    \fi%
  \else%
    \setlength{\unitlength}{\svgwidth}%
  \fi%
  \global\let\svgwidth\undefined%
  \global\let\svgscale\undefined%
  \makeatother%
  \begin{picture}(1,0.86252518)%
    \lineheight{1}%
    \setlength\tabcolsep{0pt}%
    \put(0,0){\includegraphics[width=\unitlength,page=1]{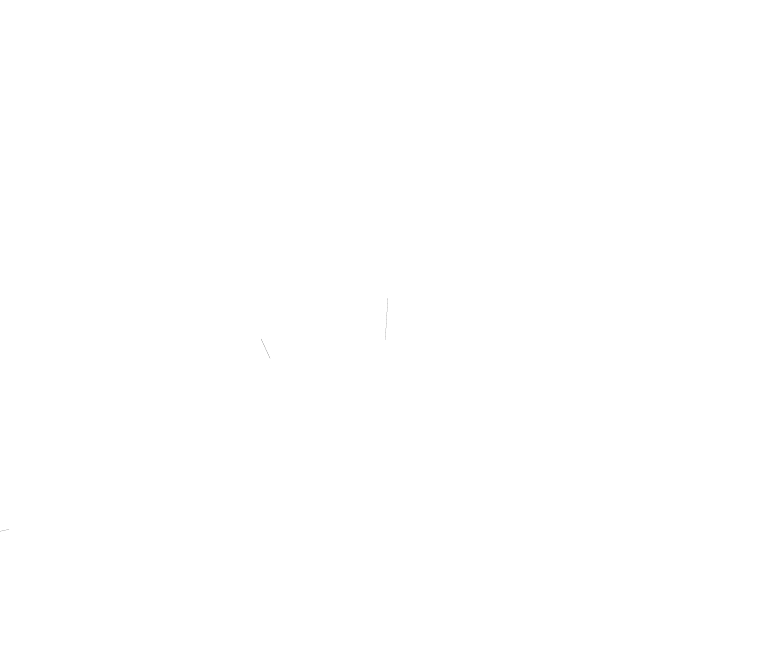}}%
    \put(0.34779387,0.84395837){\color[rgb]{0,0,0}\makebox(0,0)[lt]{\lineheight{1.25}\smash{\begin{tabular}[t]{l}$\tilde{\eta}_b\cup V_1 \cup \tilde{\eta}_d$\end{tabular}}}}%
    \put(0,0){\includegraphics[width=\unitlength,page=2]{U_tube_mirrored.pdf}}%
    \put(0.12317599,0.14707692){\color[rgb]{0,0,0}\makebox(0,0)[lt]{\lineheight{1.25}\smash{\begin{tabular}[t]{l}$\mathbf{G}_b$\end{tabular}}}}%
    \put(0.20320544,0.67292084){\color[rgb]{0,0,0}\makebox(0,0)[lt]{\lineheight{1.25}\smash{\begin{tabular}[t]{l}$\mathbf{G}_c$\end{tabular}}}}%
    \put(0.81821624,0.72712932){\color[rgb]{0,0,0}\makebox(0,0)[lt]{\lineheight{1.25}\smash{\begin{tabular}[t]{l}$\mathbf{G}_d$\end{tabular}}}}%
    \put(0,0){\includegraphics[width=\unitlength,page=3]{U_tube_mirrored.pdf}}%
  \end{picture}%
\endgroup%
& \raisebox{0.5cm}{$-$}& \def\svgscale{0.44}
\begingroup%
  \makeatletter%
  \providecommand\color[2][]{%
    \errmessage{(Inkscape) Color is used for the text in Inkscape, but the package 'color.sty' is not loaded}%
    \renewcommand\color[2][]{}%
  }%
  \providecommand\transparent[1]{%
    \errmessage{(Inkscape) Transparency is used (non-zero) for the text in Inkscape, but the package 'transparent.sty' is not loaded}%
    \renewcommand\transparent[1]{}%
  }%
  \providecommand\rotatebox[2]{#2}%
  \newcommand*\fsize{\dimexpr\f@size pt\relax}%
  \newcommand*\lineheight[1]{\fontsize{\fsize}{#1\fsize}\selectfont}%
  \ifx\svgwidth\undefined%
    \setlength{\unitlength}{325.11185533bp}%
    \ifx\svgscale\undefined%
      \relax%
    \else%
      \setlength{\unitlength}{\unitlength * \real{\svgscale}}%
    \fi%
  \else%
    \setlength{\unitlength}{\svgwidth}%
  \fi%
  \global\let\svgwidth\undefined%
  \global\let\svgscale\undefined%
  \makeatother%
  \begin{picture}(1,0.96633979)%
    \lineheight{1}%
    \setlength\tabcolsep{0pt}%
    \put(0,0){\includegraphics[width=\unitlength,page=1]{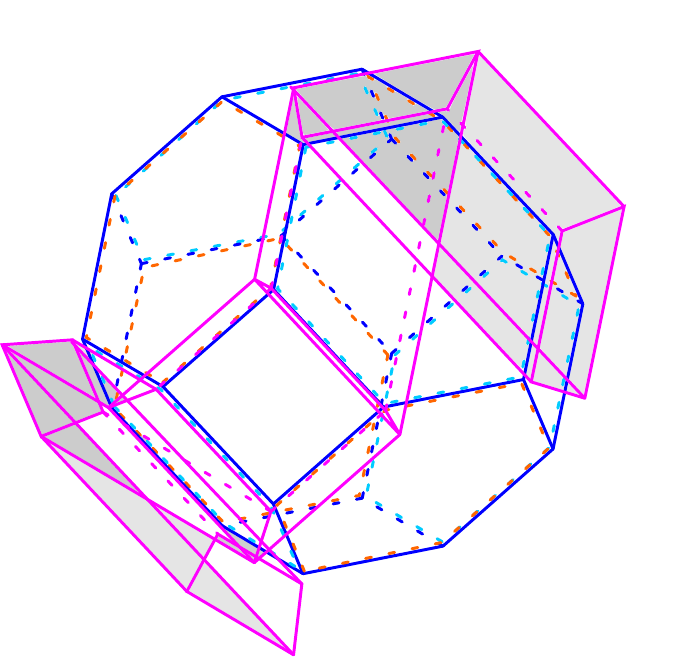}}%
    \put(0.04838399,0.16543736){\color[rgb]{0,0,0}\makebox(0,0)[lt]{\lineheight{1.25}\smash{\begin{tabular}[t]{l}$\mathbf{G}_b$\end{tabular}}}}%
    \put(0.1335314,0.78711904){\color[rgb]{0,0,0}\makebox(0,0)[lt]{\lineheight{1.25}\smash{\begin{tabular}[t]{l}$\mathbf{G}_c$\end{tabular}}}}%
    \put(0.79401214,0.82207427){\color[rgb]{0,0,0}\makebox(0,0)[lt]{\lineheight{1.25}\smash{\begin{tabular}[t]{l}$\mathbf{G}_d$\end{tabular}}}}%
    \put(0,0){\includegraphics[width=\unitlength,page=2]{mirrored_long_straight_tube.pdf}}%
    \put(0.39071807,0.94530073){\color[rgb]{0,0,0}\makebox(0,0)[lt]{\lineheight{1.25}\smash{\begin{tabular}[t]{l}$T_b\cup V^{\conv}\cup T_d$\end{tabular}}}}%
  \end{picture}%
\endgroup%
&
  \raisebox{0.5cm}{$+\left(\begin{gathered}bc+cd+bd\\+b+c+d\\+1\end{gathered}\right)$}\\
  \raisebox{0.5cm}{$\cong$} & \def\svgscale{0.44}
\begingroup%
  \makeatletter%
  \providecommand\color[2][]{%
    \errmessage{(Inkscape) Color is used for the text in Inkscape, but the package 'color.sty' is not loaded}%
    \renewcommand\color[2][]{}%
  }%
  \providecommand\transparent[1]{%
    \errmessage{(Inkscape) Transparency is used (non-zero) for the text in Inkscape, but the package 'transparent.sty' is not loaded}%
    \renewcommand\transparent[1]{}%
  }%
  \providecommand\rotatebox[2]{#2}%
  \newcommand*\fsize{\dimexpr\f@size pt\relax}%
  \newcommand*\lineheight[1]{\fontsize{\fsize}{#1\fsize}\selectfont}%
  \ifx\svgwidth\undefined%
    \setlength{\unitlength}{329.97743622bp}%
    \ifx\svgscale\undefined%
      \relax%
    \else%
      \setlength{\unitlength}{\unitlength * \real{\svgscale}}%
    \fi%
  \else%
    \setlength{\unitlength}{\svgwidth}%
  \fi%
  \global\let\svgwidth\undefined%
  \global\let\svgscale\undefined%
  \makeatother%
  \begin{picture}(1,0.96818702)%
    \lineheight{1}%
    \setlength\tabcolsep{0pt}%
    \put(0.03861463,0.94732199){\color[rgb]{0,0,0}\makebox(0,0)[lt]{\lineheight{1.25}\smash{\begin{tabular}[t]{l}$T_b \diamond (f_b,\alpha))\cup V^{\conv}\cup  (T_d\diamond(f_d,\beta))$\end{tabular}}}}%
    \put(0,0){\includegraphics[width=\unitlength,page=1]{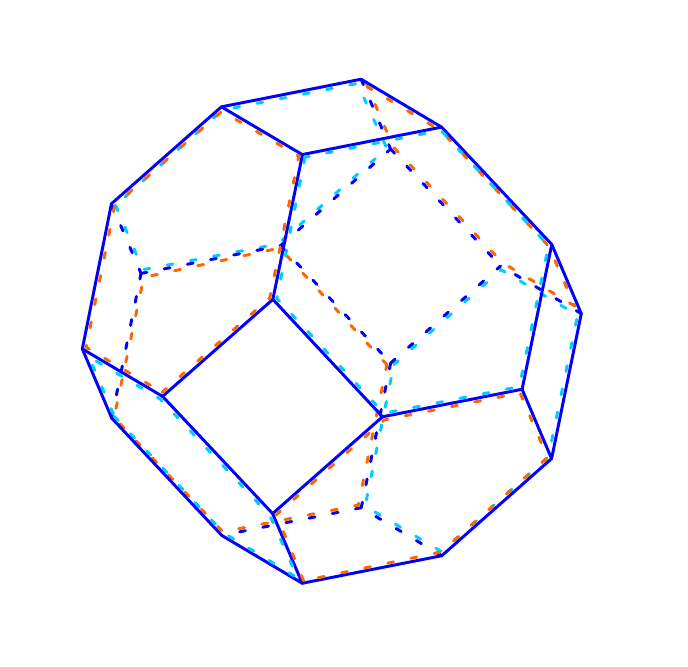}}%
    \put(0.0210804,0.16420198){\color[rgb]{0,0,0}\makebox(0,0)[lt]{\lineheight{1.25}\smash{\begin{tabular}[t]{l}$\mathbf{G}_b$\end{tabular}}}}%
    \put(0.04678848,0.86037437){\color[rgb]{0,0,0}\makebox(0,0)[lt]{\lineheight{1.25}\smash{\begin{tabular}[t]{l}$\mathbf{G}_c$\end{tabular}}}}%
    \put(0.79704989,0.81179475){\color[rgb]{0,0,0}\makebox(0,0)[lt]{\lineheight{1.25}\smash{\begin{tabular}[t]{l}$\mathbf{G}_d$\end{tabular}}}}%
    \put(0,0){\includegraphics[width=\unitlength,page=2]{mirrored_long_tube_compressed.pdf}}%
  \end{picture}%
\endgroup%
 & \raisebox{0.5cm}{$-$} & \def\svgscale{0.44} & \raisebox{0.5cm}{$+\left(\begin{gathered}bc+cd+bd\\+b+c+d\\+1\end{gathered}\right)$}\\
  \raisebox{0.5cm}{$\cong$}&\def\svgscale{0.44}
\begingroup%
  \makeatletter%
  \providecommand\color[2][]{%
    \errmessage{(Inkscape) Color is used for the text in Inkscape, but the package 'color.sty' is not loaded}%
    \renewcommand\color[2][]{}%
  }%
  \providecommand\transparent[1]{%
    \errmessage{(Inkscape) Transparency is used (non-zero) for the text in Inkscape, but the package 'transparent.sty' is not loaded}%
    \renewcommand\transparent[1]{}%
  }%
  \providecommand\rotatebox[2]{#2}%
  \newcommand*\fsize{\dimexpr\f@size pt\relax}%
  \newcommand*\lineheight[1]{\fontsize{\fsize}{#1\fsize}\selectfont}%
  \ifx\svgwidth\undefined%
    \setlength{\unitlength}{329.97743622bp}%
    \ifx\svgscale\undefined%
      \relax%
    \else%
      \setlength{\unitlength}{\unitlength * \real{\svgscale}}%
    \fi%
  \else%
    \setlength{\unitlength}{\svgwidth}%
  \fi%
  \global\let\svgwidth\undefined%
  \global\let\svgscale\undefined%
  \makeatother%
  \begin{picture}(1,1.03770297)%
    \lineheight{1}%
    \setlength\tabcolsep{0pt}%
    \put(0,0){\includegraphics[width=\unitlength,page=1]{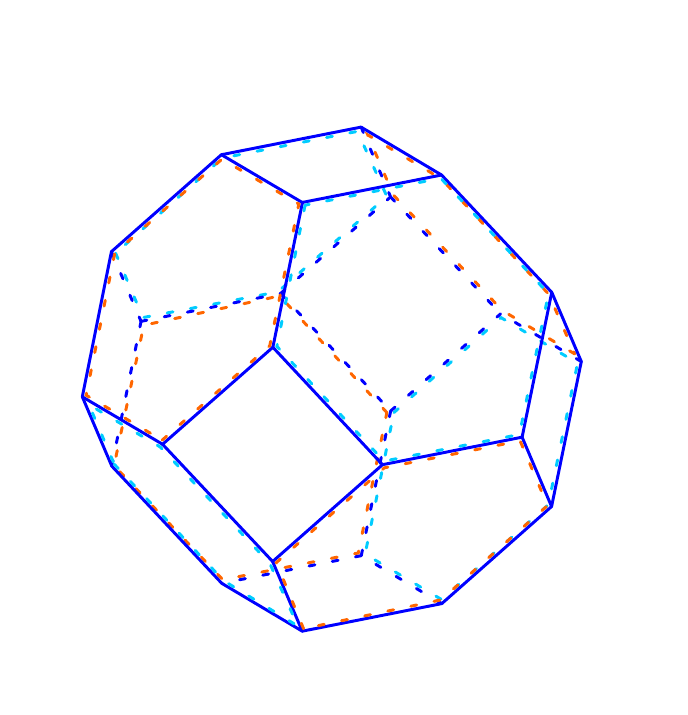}}%
    \put(0.02108034,0.16420251){\color[rgb]{0,0,0}\makebox(0,0)[lt]{\lineheight{1.25}\smash{\begin{tabular}[t]{l}$\mathbf{G}_b$\end{tabular}}}}%
    \put(0.04678843,0.8603749){\color[rgb]{0,0,0}\makebox(0,0)[lt]{\lineheight{1.25}\smash{\begin{tabular}[t]{l}$\mathbf{G}_c$\end{tabular}}}}%
    \put(0.79704971,0.81179528){\color[rgb]{0,0,0}\makebox(0,0)[lt]{\lineheight{1.25}\smash{\begin{tabular}[t]{l}$\mathbf{G}_d$\end{tabular}}}}%
    \put(0,0){\includegraphics[width=\unitlength,page=2]{U_conv_plus_twists_tube_unannotated.pdf}}%
    \put(0.07760508,1.01683762){\color[rgb]{0,0,0}\makebox(0,0)[lt]{\lineheight{1.25}\smash{\begin{tabular}[t]{l}$U^{\conv} \diamond (f_b, \alpha)\diamond (f_d, \beta)$\end{tabular}}}}%
    \put(0,0){\includegraphics[width=\unitlength,page=3]{U_conv_plus_twists_tube_unannotated.pdf}}%
  \end{picture}%
\endgroup%
 & \raisebox{0.5cm}{$-$} & \def\svgscale{0.44} & \raisebox{0.5cm}{$+\left(\begin{gathered}bc+cd+bd\\+b+c+d\\+1\end{gathered}\right)$}\\
    \end{tabular}
    \endgroup
      \caption{A step-by-step illustration of how we simplify $U-U^{\conv}$. In particular, we illustrate the equivalences (\ref{row 1}), (\ref{row 2}), (\ref{row 3}), (\ref{row 4}).}
      \label{differences visual aid}
    \end{figure}
    Our computation is as follows:
    \begingroup
    \allowdisplaybreaks
   \begin{align}
    U-U^{\conv}
    &= \tilde{\eta}_b\cup \zeta_{bc} \cup \tilde{\eta}_c\cup \zeta_{cd} \cup \tilde{\eta}_d-  T_b\cup \zeta_{bc} \cup T_c\cup \zeta_{cd} \cup T_d \label{row 1}\\[\jot]
    &= \tilde{\eta}_b\cup V_1 \cup \tilde{\eta}_d-  T_b\cup V_1^{\conv} \cup T_d\nonumber\\[\jot]
    &\cong \tilde{\eta}_b\cup V_2 \cup  \tilde{\eta}_d- T_b\cup V_2^{\conv}\cup T_d  + (bc+cd+bd+b+c+d+1)\label{row 2}\\[\jot]
    &= \tilde{\eta}_b\cup \tilde{\eta}'_b \cup \zeta_{bd} \cup \tilde{\eta}'_d \cup  \tilde{\eta}_d- T_b\cup T'_b \cup \zeta_{bd} \cup T'_d\cup T_d + (bc+cd+bd+b+c+d+1)\nonumber\\[\jot]
    &\cong (T_b \diamond (g_b,\alpha))\cup (T'_b\diamond (g'_b,\alpha')) \cup \zeta_{bd} \cup (T'_d\diamond (g'_d,\beta')) \cup  (T_d\diamond (g_d,\beta))\nonumber\\
    &\quad- T_b\cup T'_b \cup \zeta_{bd} \cup T'_d\cup T_d + (bc+cd+bd+b+c+d+1)\nonumber\\[\jot]
    &\cong (T_b \diamond (f_b,\alpha))\cup T'_b \cup \zeta_{bd} \cup T'_d \cup  (T_d\diamond(f_d,\beta))- T_b\cup T'_b \cup \zeta_{bd} \cup T'_d\cup T_d\nonumber\\
    &\quad + (bc+cd+bd+b+c+d+1)\nonumber\\[\jot]
    &\cong (T_b \diamond(f_b,\alpha))\cup V_2^{\conv}\cup  (T_d\diamond (f_d,\beta))- T_b\cup V_2^{\conv} \cup T_d\label{row 3}\\
    &\quad + (bc+cd+bd+b+c+d+1)\nonumber\\[\jot]
    &\cong (T_b \diamond (f_b,\alpha)\cup V_1^{\conv}\cup  (T_d\diamond (f_d,\beta))- T_b\cup V_1^{\conv} \cup T_d\label{row 4}\\
    &\quad + (bc+cd+bd+b+c+d+1)\nonumber\\[\jot]
    &\cong U^{\conv} \diamond (f_b,\alpha)\diamond (f_d,\beta) - U^{\conv} + (bc+cd+bd+b+c+d+1)\nonumber,
  \end{align}
  \endgroup
Here, $g_b$, $g_d$, $g'_b$, $g'_d$ are defined using Table \ref{pulling along cycle table}, and $f_b$, $f_d$ are defined using Table \ref{pulling along cycle table}. The equivalence is due to Proposition \ref{twist through concat}. Also, (\ref{row 1}), (\ref{row 2}), (\ref{row 3}), (\ref{row 4}) represent the $1^{\text{st}}$, $2^{\text{nd}}$, $3^{\text{rd}}$, and $4^{\text{th}}$ rows of Figure \ref{differences visual aid}.
\end{proof}
We have determined a new way of rewriting $U$, parametrized by $a\longline b\longline c\longline d\longline e$, as an addition of twists to $U^{\conv}$. The following proposition compares $U$ to $U'$, where $U'$ is a d.s.\ tube parametrized by $a\longline b\longline d\longline e$.
\begin{proposition}\label{chain move}
  Let $D\rightsquigarrow D'$ be an elementary move of facet chains, where $D$ denotes $a\longline b\longline c\longline d\longline e$ and $D'$ denotes $a\longline b\longline d\longline e$. Let $D$ (resp.\ $D'$) parametrize $U$ (resp.\ $U'$), where the ends of $U$ match with the corresponding ends of $U'$. We have the identity
  \[
    U - U'= (\omega_b+\omega_d + bc + cd + bd + b + c + d + 1),
  \]
  where we consult the fourth column of Table \ref{difference table} to determine the constants $\omega_b,\omega_d$.
\begin{table}
\begin{tabular}{ M{1.9cm} m{2cm} m{4.2cm} m{2.7cm} l}
  Cases & \makecell{Twist $\rho_{b}$\\ (resp. $\rho_d$)} & Twist $f_b$ (resp. $f_d$) & \makecell{Twist $f'_b$\\(resp. $f'_d$)} & \makecell{$\omega_b:= \rho_bf_bf'^{-1}_b$\\ (resp.\\ $\omega_d:= \rho_df_df'^{-1}_d$)}\\
\hexchange{$j$}{$k$}{$i$} & $\boldvarphi_{j, k}$ & $\boldvarphi_{k-1}\ldots\widehat{\boldvarphi_{j}}\ldots\boldvarphi_{i+1}$ & $\boldvarphi_{k-1}\ldots\boldvarphi_{i+1}$ & $k-j-1$\\ 
\hexchange{$j$}{$i$}{$k$} & $\boldvarphi_{j+1, i+1}$ & $\boldvarphi_{i+1}\ldots\widehat{\boldvarphi_{j}}\ldots\boldvarphi_{k-1}$ & $\boldvarphi_{i+1}\ldots \boldvarphi_{k-1}$ & $i-j$\\
\hexchange{$k$}{$j$}{$i$} & $\boldvarphi_{k, j}$ & $\boldvarphi_{j, k}\boldvarphi_{j-1}\ldots\boldvarphi_{i+1}$ & $\boldvarphi_{j-1}\ldots\boldvarphi_{i+1}$ & $0$\\
\hexchange{$k$}{$i$}{$j$} & $\boldvarphi_{k, i+1}$ & $\boldvarphi_{i+1}\ldots\boldvarphi_{j-1}\boldvarphi_{j, k}$ & $\boldvarphi_{i+1}\ldots\boldvarphi_{j-1}$ & $i-j-1$\\
\hexchange{$i$}{$k$}{$j$} & $\boldvarphi_{i+1, k}$ & $\boldvarphi_{k-1}\ldots\boldvarphi_{j+1}\boldvarphi_{i+1, j+1}$ & $\boldvarphi_{k-1}\ldots\boldvarphi_{j+1}$ & $k-j$\\ 
\hexchange{$i$}{$j$}{$k$} & $\boldvarphi_{i+1, j+1}$ & $\boldvarphi_{i+1, j+1}\boldvarphi_{j+1}\ldots\boldvarphi_{k-1}$ & $\boldvarphi_{j+1}\ldots\boldvarphi_{k-1}$ & $1$
\end{tabular}
\caption{The quantity $a_l$ denotes $a$ if $a<l$ and $a-1$ if $i>l$. Composition in the $3^{\text{rd}}$ and $4^{\text{th}}$ columns is the $\diamond$ composition, yet we omit the $\diamond$ symbol to lighten notation. $\omega_b$ (resp. $\omega_d$) denotes the difference between the twists $\rho_b\diamond f_b$ and $f'_b$ (resp. $\rho_d\diamond f_d$ and $f'_d$).}
\label{difference table}
\end{table}
\end{proposition}
\begin{proof}
  Our idea is to homotope $U$ into a tube that, just like $U'$, only occupies the facets $\mathbf{G}_b$, $\mathbf{G}_{bd}$, $\mathbf{G}_d$. It is after this homotopy that we can more easily compute $U-U'$. Following the illustration in Figure \ref{manipulation},
  \begin{figure}
    \begingroup
    \renewcommand{\arraystretch}{4}:
    \begin{tabular}{m{0.0cm} m{6.6cm} m{0.4cm} M{7.5cm}}
      & \def\svgscale{0.53}
\begingroup%
  \makeatletter%
  \providecommand\color[2][]{%
    \errmessage{(Inkscape) Color is used for the text in Inkscape, but the package 'color.sty' is not loaded}%
    \renewcommand\color[2][]{}%
  }%
  \providecommand\transparent[1]{%
    \errmessage{(Inkscape) Transparency is used (non-zero) for the text in Inkscape, but the package 'transparent.sty' is not loaded}%
    \renewcommand\transparent[1]{}%
  }%
  \providecommand\rotatebox[2]{#2}%
  \newcommand*\fsize{\dimexpr\f@size pt\relax}%
  \newcommand*\lineheight[1]{\fontsize{\fsize}{#1\fsize}\selectfont}%
  \ifx\svgwidth\undefined%
    \setlength{\unitlength}{372.56645623bp}%
    \ifx\svgscale\undefined%
      \relax%
    \else%
      \setlength{\unitlength}{\unitlength * \real{\svgscale}}%
    \fi%
  \else%
    \setlength{\unitlength}{\svgwidth}%
  \fi%
  \global\let\svgwidth\undefined%
  \global\let\svgscale\undefined%
  \makeatother%
  \begin{picture}(1,0.92083691)%
    \lineheight{1}%
    \setlength\tabcolsep{0pt}%
    \put(0,0){\includegraphics[width=\unitlength,page=1]{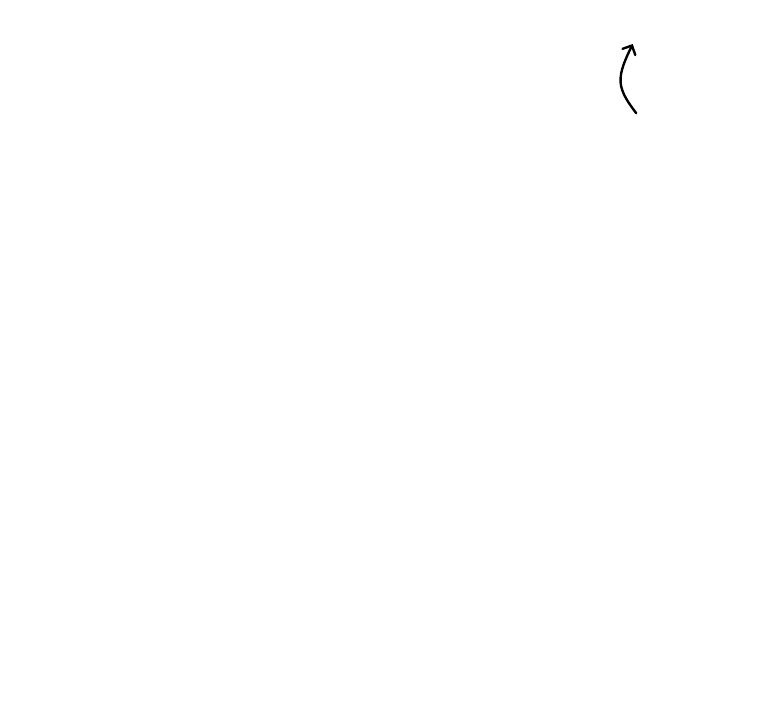}}%
    \put(0.83398897,0.81457697){\color[rgb]{0,0,0}\makebox(0,0)[lt]{\lineheight{1.25}\smash{\begin{tabular}[t]{l}lift\end{tabular}}}}%
    \put(0,0){\includegraphics[width=\unitlength,page=2]{shortening_part_1.pdf}}%
    \put(-0.00205331,0.43066288){\color[rgb]{0,0,0}\makebox(0,0)[lt]{\lineheight{1.25}\smash{\begin{tabular}[t]{l}lift\end{tabular}}}}%
    \put(0,0){\includegraphics[width=\unitlength,page=3]{shortening_part_1.pdf}}%
    \put(0.1905159,0.15602615){\color[rgb]{0,0,0}\makebox(0,0)[lt]{\lineheight{1.25}\smash{\begin{tabular}[t]{l}$\mathbf{G}_b$\end{tabular}}}}%
    \put(0.20557835,0.7539668){\color[rgb]{0,0,0}\makebox(0,0)[lt]{\lineheight{1.25}\smash{\begin{tabular}[t]{l}$\mathbf{G}_c$\end{tabular}}}}%
    \put(0.82024949,0.71972771){\color[rgb]{0,0,0}\makebox(0,0)[lt]{\lineheight{1.25}\smash{\begin{tabular}[t]{l}$\mathbf{G}_d$\end{tabular}}}}%
    \put(0.54682689,0.90250231){\color[rgb]{0,0,0}\makebox(0,0)[lt]{\lineheight{1.25}\smash{\begin{tabular}[t]{l}$U^{\conv}$\end{tabular}}}}%
  \end{picture}%
\endgroup%
& $\cong$ & \def\svgscale{0.53}
\begingroup%
  \makeatletter%
  \providecommand\color[2][]{%
    \errmessage{(Inkscape) Color is used for the text in Inkscape, but the package 'color.sty' is not loaded}%
    \renewcommand\color[2][]{}%
  }%
  \providecommand\transparent[1]{%
    \errmessage{(Inkscape) Transparency is used (non-zero) for the text in Inkscape, but the package 'transparent.sty' is not loaded}%
    \renewcommand\transparent[1]{}%
  }%
  \providecommand\rotatebox[2]{#2}%
  \newcommand*\fsize{\dimexpr\f@size pt\relax}%
  \newcommand*\lineheight[1]{\fontsize{\fsize}{#1\fsize}\selectfont}%
  \ifx\svgwidth\undefined%
    \setlength{\unitlength}{402.43707131bp}%
    \ifx\svgscale\undefined%
      \relax%
    \else%
      \setlength{\unitlength}{\unitlength * \real{\svgscale}}%
    \fi%
  \else%
    \setlength{\unitlength}{\svgwidth}%
  \fi%
  \global\let\svgwidth\undefined%
  \global\let\svgscale\undefined%
  \makeatother%
  \begin{picture}(1,0.89300152)%
    \lineheight{1}%
    \setlength\tabcolsep{0pt}%
    \put(0,0){\includegraphics[width=\unitlength,page=1]{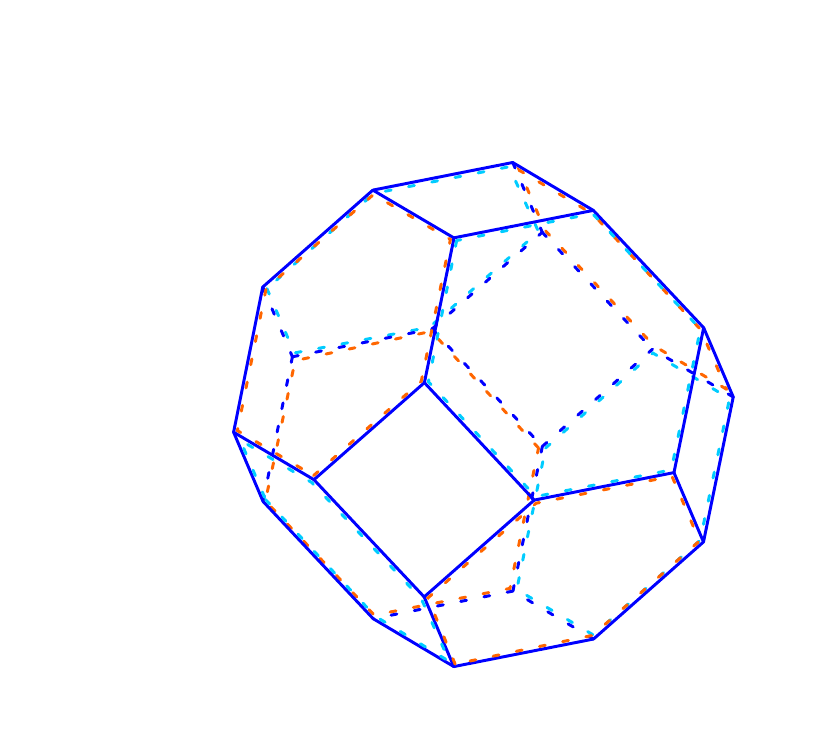}}%
    \put(0.23356771,0.13179023){\color[rgb]{0,0,0}\makebox(0,0)[lt]{\lineheight{1.25}\smash{\begin{tabular}[t]{l}$\mathbf{G}_b$\end{tabular}}}}%
    \put(0.14015787,0.78078087){\color[rgb]{0,0,0}\makebox(0,0)[lt]{\lineheight{1.25}\smash{\begin{tabular}[t]{l}$\mathbf{G}_c$\end{tabular}}}}%
    \put(0.83359144,0.6656301){\color[rgb]{0,0,0}\makebox(0,0)[lt]{\lineheight{1.25}\smash{\begin{tabular}[t]{l}$\mathbf{G}_d$\end{tabular}}}}%
    \put(0,0){\includegraphics[width=\unitlength,page=2]{shortening_part_2.pdf}}%
    \put(0.76720703,0.80549873){\color[rgb]{0,0,0}\makebox(0,0)[lt]{\lineheight{1.25}\smash{\begin{tabular}[t]{l}lay flat\end{tabular}}}}%
    \put(0,0){\includegraphics[width=\unitlength,page=3]{shortening_part_2.pdf}}%
    \put(-0.00190084,0.28301266){\color[rgb]{0,0,0}\makebox(0,0)[lt]{\lineheight{1.25}\smash{\begin{tabular}[t]{l}lay flat\end{tabular}}}}%
    \put(0,0){\includegraphics[width=\unitlength,page=4]{shortening_part_2.pdf}}%
  \end{picture}%
\endgroup%
\\
      $\cong$ &\def\svgscale{0.53}
\begingroup%
  \makeatletter%
  \providecommand\color[2][]{%
    \errmessage{(Inkscape) Color is used for the text in Inkscape, but the package 'color.sty' is not loaded}%
    \renewcommand\color[2][]{}%
  }%
  \providecommand\transparent[1]{%
    \errmessage{(Inkscape) Transparency is used (non-zero) for the text in Inkscape, but the package 'transparent.sty' is not loaded}%
    \renewcommand\transparent[1]{}%
  }%
  \providecommand\rotatebox[2]{#2}%
  \newcommand*\fsize{\dimexpr\f@size pt\relax}%
  \newcommand*\lineheight[1]{\fontsize{\fsize}{#1\fsize}\selectfont}%
  \ifx\svgwidth\undefined%
    \setlength{\unitlength}{361.13823117bp}%
    \ifx\svgscale\undefined%
      \relax%
    \else%
      \setlength{\unitlength}{\unitlength * \real{\svgscale}}%
    \fi%
  \else%
    \setlength{\unitlength}{\svgwidth}%
  \fi%
  \global\let\svgwidth\undefined%
  \global\let\svgscale\undefined%
  \makeatother%
  \begin{picture}(1,0.96825803)%
    \lineheight{1}%
    \setlength\tabcolsep{0pt}%
    \put(0.18364815,0.67912575){\color[rgb]{0,0,0}\makebox(0,0)[lt]{\lineheight{1.25}\smash{\begin{tabular}[t]{l}$\mathbf{G}_c$\end{tabular}}}}%
    \put(0.04909673,0.05216311){\color[rgb]{0,0,0}\makebox(0,0)[lt]{\lineheight{1.25}\smash{\begin{tabular}[t]{l}draw tight\end{tabular}}}}%
    \put(0,0){\includegraphics[width=\unitlength,page=1]{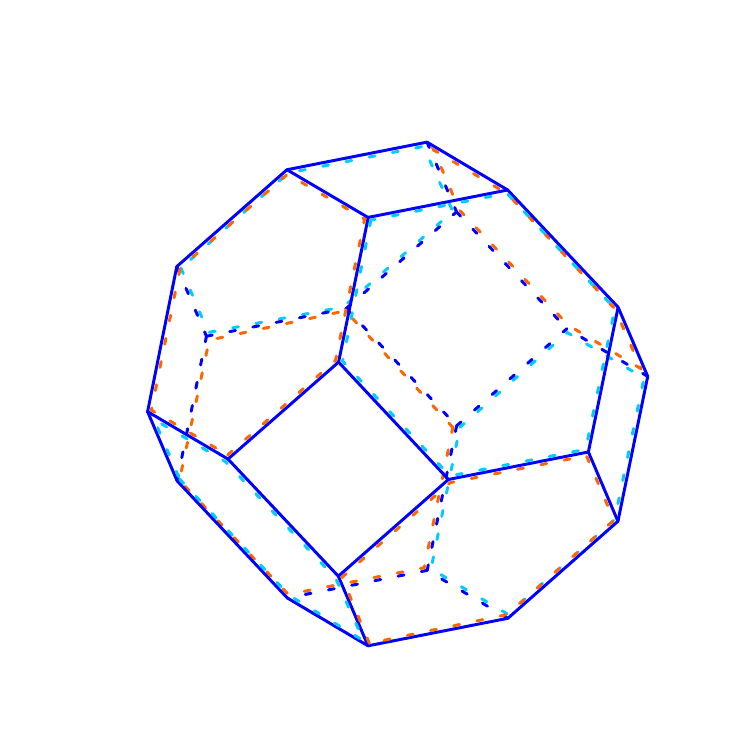}}%
    \put(0.01060444,0.12636736){\color[rgb]{0,0,0}\makebox(0,0)[lt]{\lineheight{1.25}\smash{\begin{tabular}[t]{l}$\mathbf{G}_b$\end{tabular}}}}%
    \put(0.81456138,0.74174975){\color[rgb]{0,0,0}\makebox(0,0)[lt]{\lineheight{1.25}\smash{\begin{tabular}[t]{l}$\mathbf{G}_d$\end{tabular}}}}%
    \put(0,0){\includegraphics[width=\unitlength,page=2]{shortening_part_3.pdf}}%
    \put(0.50013592,0.94931802){\color[rgb]{0,0,0}\makebox(0,0)[lt]{\lineheight{1.25}\smash{\begin{tabular}[t]{l}draw tight\end{tabular}}}}%
    \put(0.2344667,0.8535292){\color[rgb]{0,0,0}\makebox(0,0)[lt]{\lineheight{1.25}\smash{\begin{tabular}[t]{l}$W$\end{tabular}}}}%
  \end{picture}%
\endgroup%
 & $\cong$ & \def\svgscale{0.53}
\begingroup%
  \makeatletter%
  \providecommand\color[2][]{%
    \errmessage{(Inkscape) Color is used for the text in Inkscape, but the package 'color.sty' is not loaded}%
    \renewcommand\color[2][]{}%
  }%
  \providecommand\transparent[1]{%
    \errmessage{(Inkscape) Transparency is used (non-zero) for the text in Inkscape, but the package 'transparent.sty' is not loaded}%
    \renewcommand\transparent[1]{}%
  }%
  \providecommand\rotatebox[2]{#2}%
  \newcommand*\fsize{\dimexpr\f@size pt\relax}%
  \newcommand*\lineheight[1]{\fontsize{\fsize}{#1\fsize}\selectfont}%
  \ifx\svgwidth\undefined%
    \setlength{\unitlength}{281.41909862bp}%
    \ifx\svgscale\undefined%
      \relax%
    \else%
      \setlength{\unitlength}{\unitlength * \real{\svgscale}}%
    \fi%
  \else%
    \setlength{\unitlength}{\svgwidth}%
  \fi%
  \global\let\svgwidth\undefined%
  \global\let\svgscale\undefined%
  \makeatother%
  \begin{picture}(1,1.08942896)%
    \lineheight{1}%
    \setlength\tabcolsep{0pt}%
    \put(0.00110952,0.88182351){\color[rgb]{0,0,0}\makebox(0,0)[lt]{\lineheight{1.25}\smash{\begin{tabular}[t]{l}$\mathbf{G}_c$\end{tabular}}}}%
    \put(-0.0019827,0.22385666){\color[rgb]{0,0,0}\makebox(0,0)[lt]{\lineheight{1.25}\smash{\begin{tabular}[t]{l}$\mathbf{G}_b$\end{tabular}}}}%
    \put(0.74259946,0.8714625){\color[rgb]{0,0,0}\makebox(0,0)[lt]{\lineheight{1.25}\smash{\begin{tabular}[t]{l}$\mathbf{G}_d$\end{tabular}}}}%
    \put(0,0){\includegraphics[width=\unitlength,page=1]{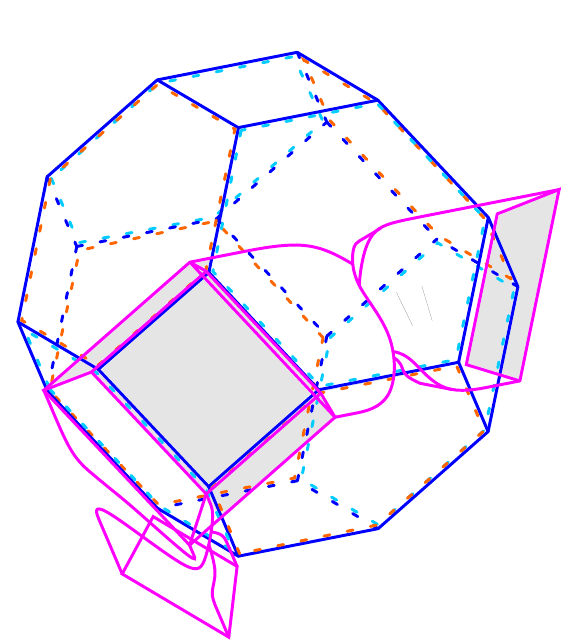}}%
    \put(0.07367739,1.06512405){\color[rgb]{0,0,0}\makebox(0,0)[lt]{\lineheight{1.25}\smash{\begin{tabular}[t]{l}$U'^{\conv}\diamond(\rho_b, \alpha)\diamond (\rho_d,\beta)$\end{tabular}}}}%
    \put(0,0){\includegraphics[width=\unitlength,page=2]{shortening_part_4.pdf}}%
  \end{picture}%
\endgroup%

    \end{tabular}
    \endgroup
    \caption{The homotopy of $U^{\conv}$ (top left) to $U'^{\conv} \diamond (\rho_b,\alpha) \diamond (\rho_d,\beta)$ (bottom right) described in the proof of Theorem \ref{chain move}. We first homotope $U^{\conv}$ to $W$ (bottom left) by flipping over $\mathbf{G}_{bcd}$, and then using Example \ref{gluing tubes homotopy}, we homotope $W$ to $U'^{\conv} \diamond (\rho_b,\alpha) \diamond (\rho_d,\beta)$.}
    \label{manipulation}
  \end{figure}
  we start with the tube $U^{\conv}$, lifting the far end $T_c\subset \mathbf{G}_c$ up in the $+J$-direction, and carrying it through the face $\mathbf{G}_{bcd}$, ``turning'' $T_c$ over and laying it back down in the facet $\mathbf{G}_{bd}$. The result $W$ is a concatenation of several tubes: On one end, we have a tube in $\mathbf{G}_b$ which moves directly from $\mathbf{G}_a$ to $\mathbf{G}_c$, a tube in $\mathbf{G}_{bc}$ which flips up and around back into $\mathbf{G}_b$, and a tube in $\mathbf{G}_b$ that moves directly from $\mathbf{G}_c$ to $\mathbf{G}_d$. These tubes look like they arise from Example \ref{gluing tubes homotopy}, and in fact, this is exactly what we use. We describe $W$ as
  \[
    W \cong( T'_b\cup \zeta_{bd}\cup T'_d) \diamond (\rho_b,\alpha) \diamond (\rho_d,\beta) = U'^{\conv} \diamond (\rho_b,\alpha) \diamond (\rho_d,\beta),
  \]
  where we consult the second column of Table \ref{difference table} to define $\rho_b$, $\rho_d$. Therefore,
  \begin{align*}
    U & \cong U^{\conv} \diamond (f_b,\alpha) \diamond (f_d,\beta) + (bc+cd+bd+b+c+d+1)\\
      &\cong W \diamond (f_b,\alpha) \diamond (f_d,\beta) + (bc+cd+bd+b+c+d+1)\\
      &\cong (U'^{\conv} \diamond (\rho_b,\alpha) \diamond (\rho_d,\beta)) \diamond (f_b,\alpha) \diamond (f_d,\beta) + (bc+cd+bd+b+c+d+1)\\
      &\cong U'^{\conv}\diamond (\rho_b\diamond f_b,\alpha) \diamond (\rho_d\diamond f_d,\beta) + (bc+cd+bd+b+c+d+1),
  \end{align*}
  where $f_b$, $f_d$ are determined in the column $3$ of Table \ref{difference table}. Furthermore, note $U' \cong U'^{\conv}\diamond (f'_b,\alpha) \diamond (f'_d,\beta)$, where $f'_b$, $f'_d$ defined in column $4$ of Table \ref{difference table}. Our computation of $U-U'$ proceeds as follows:
  \begin{align*}
    U - U'
    &= \left(U'^{\conv}\diamond (\rho_b\diamond f_b,\alpha) \diamond (\rho_d\diamond f_d,\beta) + (bc+cd+bd+b+c+d+1) \right)\\
    & \quad - \left(U'^{\conv}\diamond (f'_b,\alpha) \diamond (f'_d,\beta)\right)\\
    &= (\rho_b\diamond f_b \diamond f_b'^{-1},\alpha) \diamond (\rho_d\diamond f_d\diamond f_d'^{-1},\beta) + (bc+cd+bd+b+c+d+1).
  \end{align*}
  We can verify case-by-case that the twists $\rho_b\diamond f_b \diamond f_b'^{-1}$ and $\rho_d\diamond f_d\diamond f_d'^{-1}$ are always full twists, and in fact we can write $\rho_b\diamond f_b \diamond f_b'^{-1} = (\omega_b)$, $\rho_d\diamond f_d\diamond f_d'^{-1} = (\omega_d)$, where $\omega_b$, $\omega_d$ are defined in the last column of Table \ref{difference table}. We finally have $U-U' = (\omega_b)+(\omega_d) + (bc+cd+bd+b+c+d+1)$, as promised.
\end{proof}
  For an example of how we derive the last column of Table \ref{difference table}, we include our computation for the first row
  \paragraph{Proving the first row: $\rho_b\diamond f_b \diamond f_b'^{-1} = (k-j-1)$.}
  We compute the following:
  \begingroup
  \allowdisplaybreaks
\begin{align*}
  \rho_b\diamond f_b \diamond f_b'^{-1}
  &= \boldvarphi_{j, k}\boldvarphi_{k-1}\ldots\widehat{\boldvarphi_j}\ldots \boldvarphi_{i+1}(\boldvarphi_{k-1}\ldots\boldvarphi_{i+1})^{-1}\\
  &\cong \boldvarphi_{j, k}\boldvarphi_{k-1}\ldots\widehat{\boldvarphi_j}\ldots \boldvarphi_{i+1}(\boldvarphi_{j-1}\ldots\boldvarphi_{i+1})^{-1}(\boldvarphi_{k-1}\ldots\boldvarphi_{j})^{-1}\\
  &\cong \boldvarphi_{j, k}\boldvarphi_{k-1}\ldots\boldvarphi_{j+1}(\boldvarphi_{k-1}\ldots\boldvarphi_{j})^{-1}\\
  &\cong\boldvarphi_{k-1}\boldvarphi^{-1}_{j, k-1}\boldvarphi_{k-2}\ldots\boldvarphi_{j+1}(\boldvarphi_{k-1}\ldots\boldvarphi_{j})^{-1}\\
  &\pushright{\text{using Lemma \ref{commutators}}}\\
  &\cong\boldvarphi_{k-1}\boldvarphi_{k-2}\boldvarphi_{j, k-2}\boldvarphi_{k-3}\ldots\boldvarphi_{j+1}(\boldvarphi_{k-1}\ldots\boldvarphi_{j})^{-1}\\
  &\pushright{\text{again using Lemma \ref{commutators}}}\\
  &\cong\ldots
    \cong \begin{cases}
      \boldvarphi_{k-1}\ldots\boldvarphi_{j+1}\boldvarphi_{j, j+1}(\boldvarphi_{k-1}\ldots\boldvarphi_{j})^{-1} & \text{if $k-j$ is odd}\\
      \boldvarphi_{k-1}\ldots\boldvarphi_{j+1}\boldvarphi_{j, j+1}^{-1}(\boldvarphi_{k-1}\ldots\boldvarphi_{j})^{-1} & \text{if $k-j$ is even}
    \end{cases}\\
  &\cong \begin{cases}
    0 & \text{if $k-j$ is odd}\\
    1 & \text{if $k-j$ is even}
  \end{cases}
\end{align*}
\endgroup
  which tells us that if $\vcenter{\hbox{\hexchange{$c_b$}{$d_b$}{$a_b$}}}$ $\Bigl($ respectively, $\vcenter{\hbox{\hexchange{$c_d$}{$b_d$}{$e_d$}}}\Bigr)$ falls into the category the first row $\vcenter{\hbox{\hexchange{$j$}{$k$}{$i$}}}$, then $\rho_b\diamond f_b \diamond f_b^{-1}$ (resp. $\rho_d\diamond f_d \diamond f_d'^{-1}$) is equal to $(k-j-1)$.\par
  The rest of the rows are a similar exercise.
  \begin{corollary}\label{cycle move}
  Let $Z$, $Z'$ be facet cycles that differ by an elementary move
  \[
    (a\longline b\longline c\longline d\longline e)\rightsquigarrow (a\longline b\longline d\longline e).
  \]
  Suppose that $Z$, $Z'$ parametrize the cycles $K\subset\partial\mathcal{C}(r+2)$, $K'\subset\partial\mathcal{C}_{r'+2}$ respectively. We have the equality
  \begin{equation}\label{cycle difference}
    [K]-[K'] = \omega_b+\omega_d + bc + cd + bd + b + c + d + 1,
  \end{equation}
  where we again consult the fifth column of Table \ref{difference table} to determine the constants $\omega_b,\omega_d$.
\end{corollary}
\begin{proof}
  Note that by Lemma \ref{suspension agreement}, it suffices to prove Equation (\ref{cycle difference}) for just \textit{some} $K'$ parametrized by $Z'$. Now let $U\subset K$ be the tube parametrized by $D = (a\longline b\longline c\longline d\longline e)$. Replace $U\subset K$ with a tube $U'$ parametrized by $D' = (a\longline b\longline d\longline e)$ to yield a cycle $K'$. Since $U-U' = (\omega_b+\omega_d + bc + cd + bd + b + c + d + 1)$ by Lemma \ref{chain move}, we have the identity $[K]-[K'] = \omega_b+\omega_d + bc + cd + bd + b + c + d + 1$.
\end{proof}
\section{A formula $Q$ for all $[K]$.}
Let $K$ be parametrized by $Z$, where $Z$ is an (unsigned) facet cycle. We derive a formula for $[K]$ that generalizes our result for $3$-cycles. Starting with Corollary \ref{cycle move}, we have explicit formula for $[K]-[K']$, so long as their facet cycles $Z$, $Z'$ differ by an elementary move. We now have a strategy for computing $[K]$ for cycles $K$ with no turnarounds. Indeed, we denote $Z$ the facet cycle for $K$ and look for a sequence $Z = Z_1\rightsquigarrow Z_2\rightsquigarrow\ldots \rightsquigarrow Z_r$ of elementary moves terminating at a $3$-cycle, and repeatedly apply Corollary \ref{cycle move}. In the spirit of this strategy, we derive an explicit function $Q:\{\text{facet cycles } Z \}\to \mathbb{Z}/2$ such that $Q(Z)-Q(Z')$ measures the difference $[K]-[K']$ for any two facet cycles $Z$, $Z'$ differing by an elementary move. 
\begin{definition}\label{Q definition}
  Let $Z$ be a signed facet cycle $(a_1,\omega_1)\longline (a_2,\omega_2)\longline\ldots\longline (a_r,\omega_r)\longline (a_1,\omega_1)$. Choose a direction to orient $Z$, say $(a_1,\omega_1)\to (a_2,\omega_2)\to \ldots\to (a_r,\omega_r)\to (a_1,\omega_1)$. We define
  \begin{align}
    \label{other summands}
    Q(Z)&=\sum_{a \text{---} b} ab
    +\sum_{a} a
    +\sum_{a \text{---} b}\max(a,b)
    +\sum_{a \text{---} b\text{---} c}\max(a_b,c_b)\\
    \label{winding number}
    &\quad +1 + \#\{a\to \overrightarrow{b}\to c\ \vert\  a>b\} + \#\{a\to \overleftarrow{b}\to c\ \vert\  a<b\}\\
    &\quad \label{signed vertex} + \sum_{(a,\omega')\text{---} (b,\omega)} (a_b\omega + b_a\omega') + \sum_{(b,\omega)}\omega\mod 2.
  \end{align}
\end{definition}
Note that the summand (\ref{winding number}) does not depend on how we orient $Z_K$. Hence, $Q(Z)$ is well-defined. For $C\subset \Gamma(z,\mu)$ a cycle, we define $Q(C) := Q(Z(C))$, where $Z(C)$ is the corresponding facet cycle.
\begin{lemma}\label{Difference}
  Let $K$ be parametrized by a facet cycle $Z$ and let $K'$ be parametrized by a facet cycle $Z'$. Suppose that $Z$ and $Z'$ differ by an elementary move. We have the equality of differences $Q(Z)-Q(Z') = [K]-[K']$.
\end{lemma}
  \begin{proof}
    For concreteness, let $Z$ differ from $Z'$ by the elementary move
    \[
      (a\longline b\longline c\longline d\longline e)\rightsquigarrow (a\longline b\longline d\longline e).
    \]
    We compare $Q(Z)$ with $Q(Z')$, summand by summand:
  \subsubsection*{Difference of $\sum_{a \text{---} b} ab $:}
  The difference is $bc+bd+cd$.
  \subsubsection*{Difference of $\sum_{b} b$:}
  The difference is $c$.
  \subsubsection*{Difference of $\sum_{a \text{---} b}\max(a,b)$}
  The difference is $\max(b,c) + \max(c,d) + \max(b,d) = \Mid(b,c,d)$. Indeed, notice that between $Z_K$ and $Z_{K'}$, the list of consective pairs of vertices differs by only three pairs: namely $\{b,c\}$, $\{c,d\}$, and $\{b,d\}$.
  \subsubsection*{Difference of $\sum_{a\text{---} b\text{---} c}\max(a_b,c_b)$:}
  \begin{align*}
    &\max(a_b,c_b) + \max(b_c,d_c) + \max(c_d,e_d) + \max(a_b,d_b) + \max(b_c,e_c)\\
    &= (\max(a_b,c_b) + \max(a_b,d_b) + \max(c_b,d_b))  + (\max(c_b,d_b) + \max(b,d) + \max(b_d,c_d))\\
    &\quad + (\max(b_d,c_d) + \max(b_d,e_d) + \max(c_d,e_d))\\
    &= \Mid(a_b,c_b,d_b) + \Mid(b,c,d) + \Mid(b_d,c_d,e_d).
  \end{align*}
  Indeed, notice that between $Z_K$ and $Z_{K'}$, the list of consective triples of vertices differs by only $5$ triples: namely $\{a,b,c\}$, $\{b,c,d\}$, $\{c,d,e\}$, $\{a,c,d\}$, $\{b,c,e\}$.
    \subsubsection*{Difference of $\#\left\{a \text{---} b \text{---} c \enspace \middle\vert \text{ $b<a,c$}\right\}$:}
    $1|_{b<\Mid(a,c,d)}+1|_{d<\Mid(e,c,b)}+1$. We can verify this fact case-by-case.
  \subsubsection*{Difference of $\#\{a\to \overrightarrow{b}\to c\ \vert\  a>b\} + \#\{a\to \overleftarrow{b}\to c\ \vert\  a<b\}$:} We view the elementary move $Z\rightsquigarrow Z'$ as a pair $\vcenter{\hbox{\Dchange{$a$}{$b$}{$c$}{$d$}}}$, $\vcenter{\hbox{\Dchange{$e$}{$d$}{$c$}{$b$}}}$. Obtaining the values $\omega_b$, $\omega_d$ from Table \ref{difference facet cycle table}, we observe that the difference is $1+\omega_b+\omega_d$.
  \begin{table}
    \centering
    \begin{tabular}{ M{6.5cm} c m{5cm} }
      \makecell{Change\\ $\vcenter{\hbox{\Dchange{$a$}{$b$}{$c$}{$d$}}}$ $\Biggl(\text{resp.\ }\vcenter{\hbox{\Dchange{$e$}{$d$}{$c$}{$b$}}}\Biggr)$} & $\omega_{b}$ (resp.\ $\omega_{d}$) & $\omega'_{b}$ (resp.\ $\omega'_{d}$) \\
      \hline
      \Dchange{$i$}{$l$}{$j$}{$k$} & $1$ & $k_l+j_l+1$\\
      \Dchange{$k$}{$l$}{$j$}{$i$} & $0$ & $i_l+j_l$\\
      \Dchange{$i$}{$l$}{$k$}{$j$} & $0$ & $0$\\
      \Dchange{$j$}{$l$}{$k$}{$i$} & $1$ & $i_l+j_l+1$\\
      \Dchange{$j$}{$l$}{$i$}{$k$} & $0$ & $k_l+j_l$\\
      \Dchange{$k$}{$l$}{$i$}{$j$} & $1$ & $1$               
    \end{tabular}
    \caption{In the left column, $i$, $j$, $k$ respectively denote $\min(a,c,d)$, $\Mid(a,c,d)$, $\max(a,c,d))$ (similarly, $\min(e,c,b)$, $\Mid(e,c,b)$, $\max(e,c,b)$). So for example, the first row refers to the case $a<c<d$ (resp.\ $e<c<b$).}
    \label{difference facet cycle table}
  \end{table}
  \paragraph{Adding up the differences:}
  We get
  \begin{align*}
    Q(Z_K)-Q(Z_{K'})&=bc+bd+cd+c\\
    &\quad + \Mid(b,c,d)+(\Mid(a_b,c_b,d_b)+\Mid(b,c,d)+\Mid(b_d,c_d,e_d))\\
    &\quad + 1|_{b<\Mid(a,c,d)}+1|_{d<\Mid(e,c,b)}+1\\
    &\quad +(1+\omega_{b}+\omega_{d})\\
    &=bc+bd+cd+b+'c+d+1+\Mid(a,c,d)|_{b}+\Mid(b,c,e)|_{d}\\
    &\quad +(b+d+1) + \omega_{b}+\omega_{d}\\
    &= (1+\omega_{b}+\omega_{d})\\
    &=bc+bd+cd+b+c+d+1+\Mid(a,c,d)|_{b}+\Mid(b,c,e)|_{d}\\
    &\quad +(b_{d}+d_{b}) + \omega_{b}+\omega_{d}\\
    &= bc+bd+cd+b+c+d+1+\omega'_{b}+\omega'_{d}\mod 2,
  \end{align*}
  where we refer to the third column of Table \ref{difference facet cycle table} for $\omega'_{b},\omega'_{d}$. By Corollary \ref{cycle move}, $Q(Z_K)-Q(Z_{K'}) = [K]-[K']$.
\end{proof}
\begin{definition}
  By Lemma \ref{suspension agreement}, there exists a well-defined function $\mathcal{Q}: \{\text{signed facet cycles}\}\to \mathbb{F}_2$ such that $\mathcal{Q}(Z) = [K]$ for every $K$ parametrizing $Z$. Now let $\widetilde{Q}: B\to \mathbb{F}_2$ be a map, where $B\subset \{\text{signed facet cycles}\}$ is some subset. We say that $\widetilde{Q}$ is \textit{sincere on $B$} if $\widetilde{Q}|_{B} = \mathcal{Q}|_{B}$. If $\widetilde{Q} \equiv \mathcal{Q}$, we simply say $\widetilde{Q}$ is \textit{sincere}.
\end{definition}
\begin{proposition}
  The map $Q$ defined in Definition \ref{Q definition} is sincere on the set of unsigned facet cycles $Z$ without backtracks.
\end{proposition}
\begin{proof}
  $Q$ is sincere on the $3$-cycle $Z_0 = (0\longline 1\longline 2\longline 0)$ by Lemma \ref{3-cycle}. Now let $Z$ be an arbitrary unsigned facet cycle with no boundary matchings that match a face to itself. Following and a standard connectedness argument, we find a sequence of elementary moves $Z_0\rightsquigarrow Z_1\rightsquigarrow \ldots \rightsquigarrow Z_n = Z$, and apply Lemma \ref{Difference} at each step.
\end{proof}
\section{Proving $Q$ is sincere, and a general formula for $\text{Sq}^2$}
In the previous section, we only looked at instances where $[K]$ is parametrized by an unsigned facet cycle $Z$ with no turnarounds. In this section, we now examine all cases of $K$, so we include the possibility that $Z$ has both turnarounds and signs.
\begin{proposition}\label{General formula}
  $Q$ is sincere.
\end{proposition}
\begin{proof}
  The outline of our proof is  proving the above formula in successively general levels.
  \begin{itemize}
  \item Level 1: $Q$ is sincere on the set of unsigned facet cycles without turnarounds
  \item Level 2: $Q$ is sincere on the set of unsigned facet cycles $Z$ where all the turnarounds of $Z$ point in the same direction.
  \item Level 3: $Q$ is sincere on the set of unsigned facet cycles.
  \item Level 4: $Q$ is sincere.
  \end{itemize}
  We already proved Level 1 in Lemma \ref{Difference}. The following remark shows that it suffices to prove Level 2:
  \begin{remark}
    Level 3 follows from Level 2.
  \end{remark}
  \begin{proof}
    let $Z$ be a facet cycle with turnarounds. If we switch the orientation of a turnaround, we see that $[K]$ changes by adding $1$. But so does the quantity (\ref{winding number}) in $Q(Z)$. We can repeat the process of switching turnarounds in $Z$ to reduce ourselves to Level 2.
  \end{proof}
  
  \paragraph{Proof of Level 2 (all turnarounds in the same direction):}
  From any unsigned facet cycle $Z$, we can perform a sequence of simplifying moves to obtain a cycle $Z'$ that either 1. does not contain any direct turnarounds, or 2. is a $2$-cycle $\overrightarrow{a}\longline \overrightarrow{b}\longline \overrightarrow{a}$: $\vcenter{\hbox{
\begingroup%
  \makeatletter%
  \providecommand\color[2][]{%
    \errmessage{(Inkscape) Color is used for the text in Inkscape, but the package 'color.sty' is not loaded}%
    \renewcommand\color[2][]{}%
  }%
  \providecommand\transparent[1]{%
    \errmessage{(Inkscape) Transparency is used (non-zero) for the text in Inkscape, but the package 'transparent.sty' is not loaded}%
    \renewcommand\transparent[1]{}%
  }%
  \providecommand\rotatebox[2]{#2}%
  \newcommand*\fsize{\dimexpr\f@size pt\relax}%
  \newcommand*\lineheight[1]{\fontsize{\fsize}{#1\fsize}\selectfont}%
  \ifx\svgwidth\undefined%
    \setlength{\unitlength}{76.1923374bp}%
    \ifx\svgscale\undefined%
      \relax%
    \else%
      \setlength{\unitlength}{\unitlength * \real{\svgscale}}%
    \fi%
  \else%
    \setlength{\unitlength}{\svgwidth}%
  \fi%
  \global\let\svgwidth\undefined%
  \global\let\svgscale\undefined%
  \makeatother%
  \begin{picture}(1,0.2691116)%
    \lineheight{1}%
    \setlength\tabcolsep{0pt}%
    \put(0,0){\includegraphics[width=\unitlength,page=1]{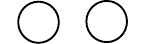}}%
    \put(-0.00632716,0.11056476){\color[rgb]{0,0,0}\makebox(0,0)[lt]{\lineheight{1.25}\smash{\begin{tabular}[t]{l}$a$\end{tabular}}}}%
    \put(0.82620245,0.09380272){\color[rgb]{0,0,0}\makebox(0,0)[lt]{\lineheight{1.25}\smash{\begin{tabular}[t]{l}$b$\end{tabular}}}}%
    \put(0,0){\includegraphics[width=\unitlength,page=2]{2-cycle.pdf}}%
  \end{picture}%
\endgroup%
}}$. Our simplifying moves involve taking a turnaround pulling through the turnaround vertex, and getting a shortened cycle with the turnaround vertex removed, in a move which looks like
  \[
    \def\svgwidth{3cm}
\begingroup%
  \makeatletter%
  \providecommand\color[2][]{%
    \errmessage{(Inkscape) Color is used for the text in Inkscape, but the package 'color.sty' is not loaded}%
    \renewcommand\color[2][]{}%
  }%
  \providecommand\transparent[1]{%
    \errmessage{(Inkscape) Transparency is used (non-zero) for the text in Inkscape, but the package 'transparent.sty' is not loaded}%
    \renewcommand\transparent[1]{}%
  }%
  \providecommand\rotatebox[2]{#2}%
  \newcommand*\fsize{\dimexpr\f@size pt\relax}%
  \newcommand*\lineheight[1]{\fontsize{\fsize}{#1\fsize}\selectfont}%
  \ifx\svgwidth\undefined%
    \setlength{\unitlength}{87.88092546bp}%
    \ifx\svgscale\undefined%
      \relax%
    \else%
      \setlength{\unitlength}{\unitlength * \real{\svgscale}}%
    \fi%
  \else%
    \setlength{\unitlength}{\svgwidth}%
  \fi%
  \global\let\svgwidth\undefined%
  \global\let\svgscale\undefined%
  \makeatother%
  \begin{picture}(1,0.92801205)%
    \lineheight{1}%
    \setlength\tabcolsep{0pt}%
    \put(0,0){\includegraphics[width=\unitlength,page=1]{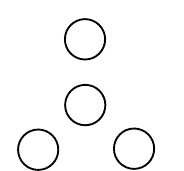}}%
    \put(0.39942512,0.86076762){\color[rgb]{0,0,0}\makebox(0,0)[lt]{\lineheight{1.25}\smash{\begin{tabular}[t]{l}$d$\end{tabular}}}}%
    \put(-0.00548558,0.0958589){\color[rgb]{0,0,0}\makebox(0,0)[lt]{\lineheight{1.25}\smash{\begin{tabular}[t]{l}$a$\end{tabular}}}}%
    \put(0.86126312,0.09392512){\color[rgb]{0,0,0}\makebox(0,0)[lt]{\lineheight{1.25}\smash{\begin{tabular}[t]{l}$c$\end{tabular}}}}%
    \put(0.31288949,0.45090645){\color[rgb]{0,0,0}\makebox(0,0)[lt]{\lineheight{1.25}\smash{\begin{tabular}[t]{l}$b$\end{tabular}}}}%
    \put(0,0){\includegraphics[width=\unitlength,page=2]{pruning_5_part_1.pdf}}%
  \end{picture}%
\endgroup%
\hspace{0.5cm}\raisebox{1.3cm}{$\rightsquigarrow$}\hspace{0.5cm}
    \def\svgwidth{3cm}
\begingroup%
  \makeatletter%
  \providecommand\color[2][]{%
    \errmessage{(Inkscape) Color is used for the text in Inkscape, but the package 'color.sty' is not loaded}%
    \renewcommand\color[2][]{}%
  }%
  \providecommand\transparent[1]{%
    \errmessage{(Inkscape) Transparency is used (non-zero) for the text in Inkscape, but the package 'transparent.sty' is not loaded}%
    \renewcommand\transparent[1]{}%
  }%
  \providecommand\rotatebox[2]{#2}%
  \newcommand*\fsize{\dimexpr\f@size pt\relax}%
  \newcommand*\lineheight[1]{\fontsize{\fsize}{#1\fsize}\selectfont}%
  \ifx\svgwidth\undefined%
    \setlength{\unitlength}{87.8809579bp}%
    \ifx\svgscale\undefined%
      \relax%
    \else%
      \setlength{\unitlength}{\unitlength * \real{\svgscale}}%
    \fi%
  \else%
    \setlength{\unitlength}{\svgwidth}%
  \fi%
  \global\let\svgwidth\undefined%
  \global\let\svgscale\undefined%
  \makeatother%
  \begin{picture}(1,0.92170663)%
    \lineheight{1}%
    \setlength\tabcolsep{0pt}%
    \put(0.40300531,0.85446158){\color[rgb]{0,0,0}\makebox(0,0)[lt]{\lineheight{1.25}\smash{\begin{tabular}[t]{l}$d$\end{tabular}}}}%
    \put(0,0){\includegraphics[width=\unitlength,page=1]{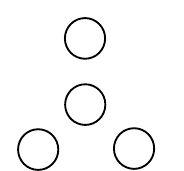}}%
    \put(-0.00548549,0.09585845){\color[rgb]{0,0,0}\makebox(0,0)[lt]{\lineheight{1.25}\smash{\begin{tabular}[t]{l}$a$\end{tabular}}}}%
    \put(0.86126314,0.09392468){\color[rgb]{0,0,0}\makebox(0,0)[lt]{\lineheight{1.25}\smash{\begin{tabular}[t]{l}$c$\end{tabular}}}}%
    \put(0.31288959,0.45090587){\color[rgb]{0,0,0}\makebox(0,0)[lt]{\lineheight{1.25}\smash{\begin{tabular}[t]{l}$b$\end{tabular}}}}%
    \put(0,0){\includegraphics[width=\unitlength,page=2]{pruning_5_part_2.pdf}}%
  \end{picture}%
\endgroup%

  \]
Of course, there are cases where for example $a\to b\to c$ and $a \to b\to d$ are themselves turnarounds, in which case, there is a similar shortening move. We call these moves \textit{pruning moves} as well. We need to measure how these pruning moves affect both the associated element $[K]$, and $Q(Z)$. But first, we need to define these elementary moves:
\begin{definition}
  Let $Z$ be a facet cycle $a_1\longline a_2\longline\ldots\longline a_n\longline a_1$, containing a facet chain $D = (a\longline b\longline \overrightarrow{c}\longline b\longline d)$. Define $Z'$ to be the facet cycle with the facet chain $D$ replaced by $D' = (a\longline b \longline d)$. If $a=b=d$, then the orientation of the vertex $b\in D'$ should agree with the orientation of the vertex $c\in D$. The move $D\rightsquigarrow D'$, and likewise, $Z\rightsquigarrow Z'$, is called a \textit{pruning move}.
\end{definition}
We exhibit these pruning moves in the first column of Table \ref{pull through figure}.
\begin{definition}
  Given (possibly repeating) integers $a,b,c$, we define $\sgn(a,b,c)$ by
  \[
    \sgn(a,b,c) := 1|_{a\leq b} + 1|_{b\leq c} + 1|_{a\leq c} + 1
  \]
  We can imagine $\sgn(a,b,c)$ to generically mean the permutation sign given by ordering $a,b,c$:
  \[
    (a,b,c)\mapsto (\text{min}\{a,b,c\},\Mid\{a,b,c\},\text{max}\{a,b,c\}),
  \]
  while imagining in the cases that $a=b$, $b=c$, and $a=c$, that respectively $b$ is a little bigger than $a$, $c$ is a little bigger than $b$, and $c$ is a little bigger than $a$.
\end{definition}
\begin{lemma}\label{pulling through}
  Let $Z$ be a facet cycle, and let $Z'$ be a pruning of $Z$, given by
  \[
    \overbrace{(a\longline b\longline c\longline b\longline d)}^{D}\rightsquigarrow \overbrace{(a\longline b\longline d)}^{D'}.
  \]
  Suppose $K$ is parametrized by $Z$ and $K'$ is parametrized by $Z'$. Then $[K]-[K'] = \Mid(a,c,d)_b+c_b+\sgn(a,c,d)$.
  \begin{table}
    \centering
    \begin{tabular}{|c| M{9.5cm} |}
     \hline
     Pruning move $Z\rightsquigarrow Z'$ & Difference $[K] - [K']$\\
      \hline
     \makecell{\def\svgscale{0.8}
\begingroup%
  \makeatletter%
  \providecommand\color[2][]{%
    \errmessage{(Inkscape) Color is used for the text in Inkscape, but the package 'color.sty' is not loaded}%
    \renewcommand\color[2][]{}%
  }%
  \providecommand\transparent[1]{%
    \errmessage{(Inkscape) Transparency is used (non-zero) for the text in Inkscape, but the package 'transparent.sty' is not loaded}%
    \renewcommand\transparent[1]{}%
  }%
  \providecommand\rotatebox[2]{#2}%
  \newcommand*\fsize{\dimexpr\f@size pt\relax}%
  \newcommand*\lineheight[1]{\fontsize{\fsize}{#1\fsize}\selectfont}%
  \ifx\svgwidth\undefined%
    \setlength{\unitlength}{87.88092546bp}%
    \ifx\svgscale\undefined%
      \relax%
    \else%
      \setlength{\unitlength}{\unitlength * \real{\svgscale}}%
    \fi%
  \else%
    \setlength{\unitlength}{\svgwidth}%
  \fi%
  \global\let\svgwidth\undefined%
  \global\let\svgscale\undefined%
  \makeatother%
  \begin{picture}(1,0.51815093)%
    \lineheight{1}%
    \setlength\tabcolsep{0pt}%
    \put(0,0){\includegraphics[width=\unitlength,page=1]{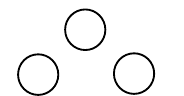}}%
    \put(-0.00548561,0.09585948){\color[rgb]{0,0,0}\makebox(0,0)[lt]{\lineheight{1.25}\smash{\begin{tabular}[t]{l}$a$\end{tabular}}}}%
    \put(0.8612631,0.09392571){\color[rgb]{0,0,0}\makebox(0,0)[lt]{\lineheight{1.25}\smash{\begin{tabular}[t]{l}$c$\end{tabular}}}}%
    \put(0.31288934,0.45090654){\color[rgb]{0,0,0}\makebox(0,0)[lt]{\lineheight{1.25}\smash{\begin{tabular}[t]{l}$b$\end{tabular}}}}%
    \put(0,0){\includegraphics[width=\unitlength,page=2]{pruning_1_part_1.pdf}}%
  \end{picture}%
\endgroup%
\raisebox{1.5em}{$\rightsquigarrow$} \def\svgwidth{2cm}
\begingroup%
  \makeatletter%
  \providecommand\color[2][]{%
    \errmessage{(Inkscape) Color is used for the text in Inkscape, but the package 'color.sty' is not loaded}%
    \renewcommand\color[2][]{}%
  }%
  \providecommand\transparent[1]{%
    \errmessage{(Inkscape) Transparency is used (non-zero) for the text in Inkscape, but the package 'transparent.sty' is not loaded}%
    \renewcommand\transparent[1]{}%
  }%
  \providecommand\rotatebox[2]{#2}%
  \newcommand*\fsize{\dimexpr\f@size pt\relax}%
  \newcommand*\lineheight[1]{\fontsize{\fsize}{#1\fsize}\selectfont}%
  \ifx\svgwidth\undefined%
    \setlength{\unitlength}{87.88095249bp}%
    \ifx\svgscale\undefined%
      \relax%
    \else%
      \setlength{\unitlength}{\unitlength * \real{\svgscale}}%
    \fi%
  \else%
    \setlength{\unitlength}{\svgwidth}%
  \fi%
  \global\let\svgwidth\undefined%
  \global\let\svgscale\undefined%
  \makeatother%
  \begin{picture}(1,0.51815028)%
    \lineheight{1}%
    \setlength\tabcolsep{0pt}%
    \put(0,0){\includegraphics[width=\unitlength,page=1]{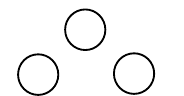}}%
    \put(-0.00548566,0.0958586){\color[rgb]{0,0,0}\makebox(0,0)[lt]{\lineheight{1.25}\smash{\begin{tabular}[t]{l}$a$\end{tabular}}}}%
    \put(0.86126303,0.09392483){\color[rgb]{0,0,0}\makebox(0,0)[lt]{\lineheight{1.25}\smash{\begin{tabular}[t]{l}$c$\end{tabular}}}}%
    \put(0.31288969,0.45090555){\color[rgb]{0,0,0}\makebox(0,0)[lt]{\lineheight{1.25}\smash{\begin{tabular}[t]{l}$b$\end{tabular}}}}%
    \put(0,0){\includegraphics[width=\unitlength,page=2]{pruning_1_part_2.pdf}}%
  \end{picture}%
\endgroup%
 \\ $w\neq y = z$} & \vspace{0.3cm}\def\svgscale{0.43}
\begingroup%
  \makeatletter%
  \providecommand\color[2][]{%
    \errmessage{(Inkscape) Color is used for the text in Inkscape, but the package 'color.sty' is not loaded}%
    \renewcommand\color[2][]{}%
  }%
  \providecommand\transparent[1]{%
    \errmessage{(Inkscape) Transparency is used (non-zero) for the text in Inkscape, but the package 'transparent.sty' is not loaded}%
    \renewcommand\transparent[1]{}%
  }%
  \providecommand\rotatebox[2]{#2}%
  \newcommand*\fsize{\dimexpr\f@size pt\relax}%
  \newcommand*\lineheight[1]{\fontsize{\fsize}{#1\fsize}\selectfont}%
  \ifx\svgwidth\undefined%
    \setlength{\unitlength}{261.82349017bp}%
    \ifx\svgscale\undefined%
      \relax%
    \else%
      \setlength{\unitlength}{\unitlength * \real{\svgscale}}%
    \fi%
  \else%
    \setlength{\unitlength}{\svgwidth}%
  \fi%
  \global\let\svgwidth\undefined%
  \global\let\svgscale\undefined%
  \makeatother%
  \begin{picture}(1,0.60452903)%
    \lineheight{1}%
    \setlength\tabcolsep{0pt}%
    \put(0,0){\includegraphics[width=\unitlength,page=1]{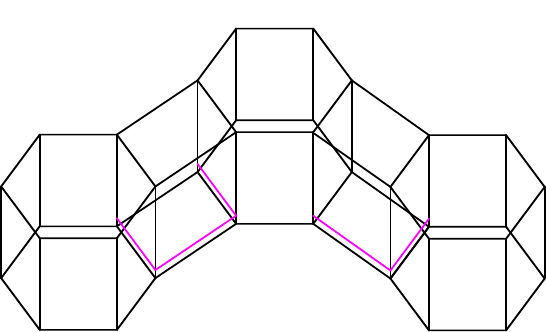}}%
    \put(0.10768171,0.38616691){\color[rgb]{0,0,0}\makebox(0,0)[lt]{\lineheight{1.25}\smash{\begin{tabular}[t]{l}$a$\end{tabular}}}}%
    \put(0.46872992,0.57840462){\color[rgb]{0,0,0}\makebox(0,0)[lt]{\lineheight{1.25}\smash{\begin{tabular}[t]{l}$b$\end{tabular}}}}%
    \put(0.82654692,0.39228683){\color[rgb]{0,0,0}\makebox(0,0)[lt]{\lineheight{1.25}\smash{\begin{tabular}[t]{l}$c$\end{tabular}}}}%
    \put(0,0){\includegraphics[width=\unitlength,page=2]{first_row_before.pdf}}%
  \end{picture}%
\endgroup%
 \raisebox{2.0em}{$ - $ } \def\svgscale{0.43}
\begingroup%
  \makeatletter%
  \providecommand\color[2][]{%
    \errmessage{(Inkscape) Color is used for the text in Inkscape, but the package 'color.sty' is not loaded}%
    \renewcommand\color[2][]{}%
  }%
  \providecommand\transparent[1]{%
    \errmessage{(Inkscape) Transparency is used (non-zero) for the text in Inkscape, but the package 'transparent.sty' is not loaded}%
    \renewcommand\transparent[1]{}%
  }%
  \providecommand\rotatebox[2]{#2}%
  \newcommand*\fsize{\dimexpr\f@size pt\relax}%
  \newcommand*\lineheight[1]{\fontsize{\fsize}{#1\fsize}\selectfont}%
  \ifx\svgwidth\undefined%
    \setlength{\unitlength}{261.82338204bp}%
    \ifx\svgscale\undefined%
      \relax%
    \else%
      \setlength{\unitlength}{\unitlength * \real{\svgscale}}%
    \fi%
  \else%
    \setlength{\unitlength}{\svgwidth}%
  \fi%
  \global\let\svgwidth\undefined%
  \global\let\svgscale\undefined%
  \makeatother%
  \begin{picture}(1,0.60101745)%
    \lineheight{1}%
    \setlength\tabcolsep{0pt}%
    \put(0,0){\includegraphics[width=\unitlength,page=1]{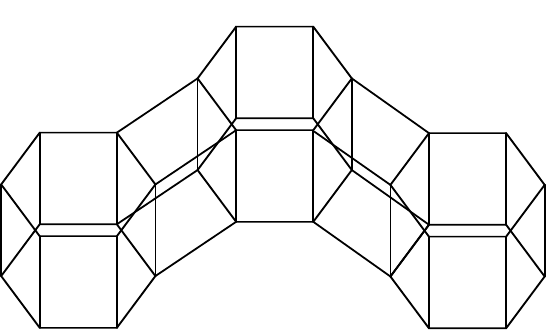}}%
    \put(0.47601099,0.57489314){\color[rgb]{0,0,0}\makebox(0,0)[lt]{\lineheight{1.25}\smash{\begin{tabular}[t]{l}$b$\end{tabular}}}}%
    \put(0,0){\includegraphics[width=\unitlength,page=2]{first_row_after.pdf}}%
    \put(0.10535779,0.39629246){\color[rgb]{0,0,0}\makebox(0,0)[lt]{\lineheight{1.25}\smash{\begin{tabular}[t]{l}$a$\end{tabular}}}}%
    \put(0.81930891,0.37569778){\color[rgb]{0,0,0}\makebox(0,0)[lt]{\lineheight{1.25}\smash{\begin{tabular}[t]{l}$c$\end{tabular}}}}%
  \end{picture}%
\endgroup%
\vspace{0.2cm}\\
      \hline
      \makecell{\def\svgscale{0.8}
\begingroup%
  \makeatletter%
  \providecommand\color[2][]{%
    \errmessage{(Inkscape) Color is used for the text in Inkscape, but the package 'color.sty' is not loaded}%
    \renewcommand\color[2][]{}%
  }%
  \providecommand\transparent[1]{%
    \errmessage{(Inkscape) Transparency is used (non-zero) for the text in Inkscape, but the package 'transparent.sty' is not loaded}%
    \renewcommand\transparent[1]{}%
  }%
  \providecommand\rotatebox[2]{#2}%
  \newcommand*\fsize{\dimexpr\f@size pt\relax}%
  \newcommand*\lineheight[1]{\fontsize{\fsize}{#1\fsize}\selectfont}%
  \ifx\svgwidth\undefined%
    \setlength{\unitlength}{87.88092005bp}%
    \ifx\svgscale\undefined%
      \relax%
    \else%
      \setlength{\unitlength}{\unitlength * \real{\svgscale}}%
    \fi%
  \else%
    \setlength{\unitlength}{\svgwidth}%
  \fi%
  \global\let\svgwidth\undefined%
  \global\let\svgscale\undefined%
  \makeatother%
  \begin{picture}(1,0.51815256)%
    \lineheight{1}%
    \setlength\tabcolsep{0pt}%
    \put(0,0){\includegraphics[width=\unitlength,page=1]{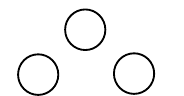}}%
    \put(-0.0054856,0.09585978){\color[rgb]{0,0,0}\makebox(0,0)[lt]{\lineheight{1.25}\smash{\begin{tabular}[t]{l}$a$\end{tabular}}}}%
    \put(0.86126313,0.09392551){\color[rgb]{0,0,0}\makebox(0,0)[lt]{\lineheight{1.25}\smash{\begin{tabular}[t]{l}$c$\end{tabular}}}}%
    \put(0.3128894,0.45090784){\color[rgb]{0,0,0}\makebox(0,0)[lt]{\lineheight{1.25}\smash{\begin{tabular}[t]{l}$b$\end{tabular}}}}%
    \put(0,0){\includegraphics[width=\unitlength,page=2]{pruning_2_part_1.pdf}}%
  \end{picture}%
\endgroup%
\raisebox{1.5em}{ $\rightsquigarrow$} \def\svgwidth{2cm}
\begingroup%
  \makeatletter%
  \providecommand\color[2][]{%
    \errmessage{(Inkscape) Color is used for the text in Inkscape, but the package 'color.sty' is not loaded}%
    \renewcommand\color[2][]{}%
  }%
  \providecommand\transparent[1]{%
    \errmessage{(Inkscape) Transparency is used (non-zero) for the text in Inkscape, but the package 'transparent.sty' is not loaded}%
    \renewcommand\transparent[1]{}%
  }%
  \providecommand\rotatebox[2]{#2}%
  \newcommand*\fsize{\dimexpr\f@size pt\relax}%
  \newcommand*\lineheight[1]{\fontsize{\fsize}{#1\fsize}\selectfont}%
  \ifx\svgwidth\undefined%
    \setlength{\unitlength}{87.8809579bp}%
    \ifx\svgscale\undefined%
      \relax%
    \else%
      \setlength{\unitlength}{\unitlength * \real{\svgscale}}%
    \fi%
  \else%
    \setlength{\unitlength}{\svgwidth}%
  \fi%
  \global\let\svgwidth\undefined%
  \global\let\svgscale\undefined%
  \makeatother%
  \begin{picture}(1,0.5181513)%
    \lineheight{1}%
    \setlength\tabcolsep{0pt}%
    \put(0,0){\includegraphics[width=\unitlength,page=1]{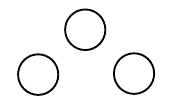}}%
    \put(-0.00548551,0.09585898){\color[rgb]{0,0,0}\makebox(0,0)[lt]{\lineheight{1.25}\smash{\begin{tabular}[t]{l}$a$\end{tabular}}}}%
    \put(0.86126336,0.0939257){\color[rgb]{0,0,0}\makebox(0,0)[lt]{\lineheight{1.25}\smash{\begin{tabular}[t]{l}$c$\end{tabular}}}}%
    \put(0.31289006,0.45090689){\color[rgb]{0,0,0}\makebox(0,0)[lt]{\lineheight{1.25}\smash{\begin{tabular}[t]{l}$b$\end{tabular}}}}%
    \put(0,0){\includegraphics[width=\unitlength,page=2]{pruning_2_part_2.pdf}}%
  \end{picture}%
\endgroup%
 \\ $w = y\neq z$} & \vspace{0.3cm}\def\svgscale{0.43}
\begingroup%
  \makeatletter%
  \providecommand\color[2][]{%
    \errmessage{(Inkscape) Color is used for the text in Inkscape, but the package 'color.sty' is not loaded}%
    \renewcommand\color[2][]{}%
  }%
  \providecommand\transparent[1]{%
    \errmessage{(Inkscape) Transparency is used (non-zero) for the text in Inkscape, but the package 'transparent.sty' is not loaded}%
    \renewcommand\transparent[1]{}%
  }%
  \providecommand\rotatebox[2]{#2}%
  \newcommand*\fsize{\dimexpr\f@size pt\relax}%
  \newcommand*\lineheight[1]{\fontsize{\fsize}{#1\fsize}\selectfont}%
  \ifx\svgwidth\undefined%
    \setlength{\unitlength}{261.82344692bp}%
    \ifx\svgscale\undefined%
      \relax%
    \else%
      \setlength{\unitlength}{\unitlength * \real{\svgscale}}%
    \fi%
  \else%
    \setlength{\unitlength}{\svgwidth}%
  \fi%
  \global\let\svgwidth\undefined%
  \global\let\svgscale\undefined%
  \makeatother%
  \begin{picture}(1,0.65451281)%
    \lineheight{1}%
    \setlength\tabcolsep{0pt}%
    \put(0,0){\includegraphics[width=\unitlength,page=1]{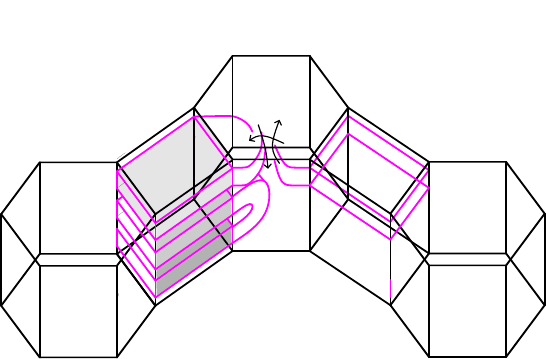}}%
    \put(0.4729287,0.5842439){\color[rgb]{0,0,0}\makebox(0,0)[lt]{\lineheight{1.25}\smash{\begin{tabular}[t]{l}$b$\end{tabular}}}}%
    \put(0.0987661,0.39020052){\color[rgb]{0,0,0}\makebox(0,0)[lt]{\lineheight{1.25}\smash{\begin{tabular}[t]{l}$a$\end{tabular}}}}%
    \put(0.83331177,0.38676799){\color[rgb]{0,0,0}\makebox(0,0)[lt]{\lineheight{1.25}\smash{\begin{tabular}[t]{l}$c$\end{tabular}}}}%
    \put(0,0){\includegraphics[width=\unitlength,page=2]{second_row_before.pdf}}%
  \end{picture}%
\endgroup%
 \raisebox{2.0em}{$ - $ } \def\svgscale{0.43}
\begingroup%
  \makeatletter%
  \providecommand\color[2][]{%
    \errmessage{(Inkscape) Color is used for the text in Inkscape, but the package 'color.sty' is not loaded}%
    \renewcommand\color[2][]{}%
  }%
  \providecommand\transparent[1]{%
    \errmessage{(Inkscape) Transparency is used (non-zero) for the text in Inkscape, but the package 'transparent.sty' is not loaded}%
    \renewcommand\transparent[1]{}%
  }%
  \providecommand\rotatebox[2]{#2}%
  \newcommand*\fsize{\dimexpr\f@size pt\relax}%
  \newcommand*\lineheight[1]{\fontsize{\fsize}{#1\fsize}\selectfont}%
  \ifx\svgwidth\undefined%
    \setlength{\unitlength}{261.82336041bp}%
    \ifx\svgscale\undefined%
      \relax%
    \else%
      \setlength{\unitlength}{\unitlength * \real{\svgscale}}%
    \fi%
  \else%
    \setlength{\unitlength}{\svgwidth}%
  \fi%
  \global\let\svgwidth\undefined%
  \global\let\svgscale\undefined%
  \makeatother%
  \begin{picture}(1,0.6010175)%
    \lineheight{1}%
    \setlength\tabcolsep{0pt}%
    \put(0,0){\includegraphics[width=\unitlength,page=1]{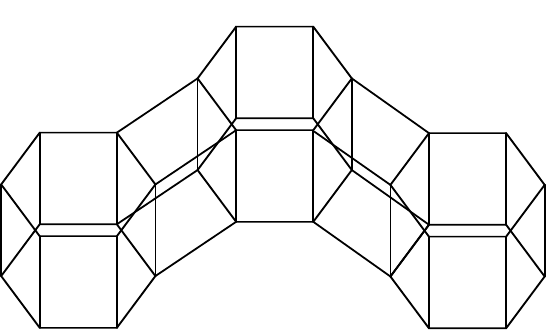}}%
    \put(0.47601097,0.57489297){\color[rgb]{0,0,0}\makebox(0,0)[lt]{\lineheight{1.25}\smash{\begin{tabular}[t]{l}$b$\end{tabular}}}}%
    \put(0,0){\includegraphics[width=\unitlength,page=2]{second_row_after.pdf}}%
    \put(0.10535784,0.39629243){\color[rgb]{0,0,0}\makebox(0,0)[lt]{\lineheight{1.25}\smash{\begin{tabular}[t]{l}$a$\end{tabular}}}}%
    \put(0.81930904,0.37569759){\color[rgb]{0,0,0}\makebox(0,0)[lt]{\lineheight{1.25}\smash{\begin{tabular}[t]{l}$c$\end{tabular}}}}%
  \end{picture}%
\endgroup%
\vspace{0.2cm}\\
     \hline
     \makecell{\def\svgscale{0.8}
\begingroup%
  \makeatletter%
  \providecommand\color[2][]{%
    \errmessage{(Inkscape) Color is used for the text in Inkscape, but the package 'color.sty' is not loaded}%
    \renewcommand\color[2][]{}%
  }%
  \providecommand\transparent[1]{%
    \errmessage{(Inkscape) Transparency is used (non-zero) for the text in Inkscape, but the package 'transparent.sty' is not loaded}%
    \renewcommand\transparent[1]{}%
  }%
  \providecommand\rotatebox[2]{#2}%
  \newcommand*\fsize{\dimexpr\f@size pt\relax}%
  \newcommand*\lineheight[1]{\fontsize{\fsize}{#1\fsize}\selectfont}%
  \ifx\svgwidth\undefined%
    \setlength{\unitlength}{50.94389007bp}%
    \ifx\svgscale\undefined%
      \relax%
    \else%
      \setlength{\unitlength}{\unitlength * \real{\svgscale}}%
    \fi%
  \else%
    \setlength{\unitlength}{\svgwidth}%
  \fi%
  \global\let\svgwidth\undefined%
  \global\let\svgscale\undefined%
  \makeatother%
  \begin{picture}(1,0.89383713)%
    \lineheight{1}%
    \setlength\tabcolsep{0pt}%
    \put(0,0){\includegraphics[width=\unitlength,page=1]{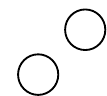}}%
    \put(-0.00946298,0.16536135){\color[rgb]{0,0,0}\makebox(0,0)[lt]{\lineheight{1.25}\smash{\begin{tabular}[t]{l}$a$\end{tabular}}}}%
    \put(0.53975131,0.77783644){\color[rgb]{0,0,0}\makebox(0,0)[lt]{\lineheight{1.25}\smash{\begin{tabular}[t]{l}$b$\end{tabular}}}}%
    \put(0,0){\includegraphics[width=\unitlength,page=2]{pruning_3_part_1.pdf}}%
  \end{picture}%
\endgroup%
\raisebox{1.5em}{$\rightsquigarrow$} \def\svgwidth{1.5cm}
\begingroup%
  \makeatletter%
  \providecommand\color[2][]{%
    \errmessage{(Inkscape) Color is used for the text in Inkscape, but the package 'color.sty' is not loaded}%
    \renewcommand\color[2][]{}%
  }%
  \providecommand\transparent[1]{%
    \errmessage{(Inkscape) Transparency is used (non-zero) for the text in Inkscape, but the package 'transparent.sty' is not loaded}%
    \renewcommand\transparent[1]{}%
  }%
  \providecommand\rotatebox[2]{#2}%
  \newcommand*\fsize{\dimexpr\f@size pt\relax}%
  \newcommand*\lineheight[1]{\fontsize{\fsize}{#1\fsize}\selectfont}%
  \ifx\svgwidth\undefined%
    \setlength{\unitlength}{50.94387926bp}%
    \ifx\svgscale\undefined%
      \relax%
    \else%
      \setlength{\unitlength}{\unitlength * \real{\svgscale}}%
    \fi%
  \else%
    \setlength{\unitlength}{\svgwidth}%
  \fi%
  \global\let\svgwidth\undefined%
  \global\let\svgscale\undefined%
  \makeatother%
  \begin{picture}(1,0.89383731)%
    \lineheight{1}%
    \setlength\tabcolsep{0pt}%
    \put(0,0){\includegraphics[width=\unitlength,page=1]{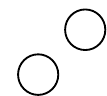}}%
    \put(-0.00946276,0.16536132){\color[rgb]{0,0,0}\makebox(0,0)[lt]{\lineheight{1.25}\smash{\begin{tabular}[t]{l}$a$\end{tabular}}}}%
    \put(0.53975153,0.77783655){\color[rgb]{0,0,0}\makebox(0,0)[lt]{\lineheight{1.25}\smash{\begin{tabular}[t]{l}$b$\end{tabular}}}}%
    \put(0,0){\includegraphics[width=\unitlength,page=2]{pruning_3_part_2.pdf}}%
  \end{picture}%
\endgroup%
 \\ $w=y=z$} & \vspace{0.3cm} \def\svgscale{0.45}
\begingroup%
  \makeatletter%
  \providecommand\color[2][]{%
    \errmessage{(Inkscape) Color is used for the text in Inkscape, but the package 'color.sty' is not loaded}%
    \renewcommand\color[2][]{}%
  }%
  \providecommand\transparent[1]{%
    \errmessage{(Inkscape) Transparency is used (non-zero) for the text in Inkscape, but the package 'transparent.sty' is not loaded}%
    \renewcommand\transparent[1]{}%
  }%
  \providecommand\rotatebox[2]{#2}%
  \newcommand*\fsize{\dimexpr\f@size pt\relax}%
  \newcommand*\lineheight[1]{\fontsize{\fsize}{#1\fsize}\selectfont}%
  \ifx\svgwidth\undefined%
    \setlength{\unitlength}{169.32503047bp}%
    \ifx\svgscale\undefined%
      \relax%
    \else%
      \setlength{\unitlength}{\unitlength * \real{\svgscale}}%
    \fi%
  \else%
    \setlength{\unitlength}{\svgwidth}%
  \fi%
  \global\let\svgwidth\undefined%
  \global\let\svgscale\undefined%
  \makeatother%
  \begin{picture}(1,0.855255)%
    \lineheight{1}%
    \setlength\tabcolsep{0pt}%
    \put(0,0){\includegraphics[width=\unitlength,page=1]{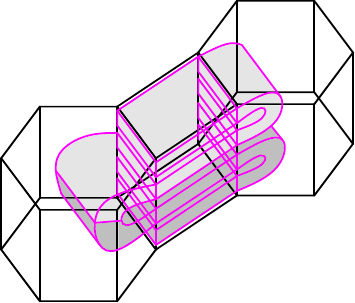}}%
  \end{picture}%
\endgroup%
 \raisebox{2.0em}{$ - $ } \def\svgscale{0.45}
\begingroup%
  \makeatletter%
  \providecommand\color[2][]{%
    \errmessage{(Inkscape) Color is used for the text in Inkscape, but the package 'color.sty' is not loaded}%
    \renewcommand\color[2][]{}%
  }%
  \providecommand\transparent[1]{%
    \errmessage{(Inkscape) Transparency is used (non-zero) for the text in Inkscape, but the package 'transparent.sty' is not loaded}%
    \renewcommand\transparent[1]{}%
  }%
  \providecommand\rotatebox[2]{#2}%
  \newcommand*\fsize{\dimexpr\f@size pt\relax}%
  \newcommand*\lineheight[1]{\fontsize{\fsize}{#1\fsize}\selectfont}%
  \ifx\svgwidth\undefined%
    \setlength{\unitlength}{169.32514942bp}%
    \ifx\svgscale\undefined%
      \relax%
    \else%
      \setlength{\unitlength}{\unitlength * \real{\svgscale}}%
    \fi%
  \else%
    \setlength{\unitlength}{\svgwidth}%
  \fi%
  \global\let\svgwidth\undefined%
  \global\let\svgscale\undefined%
  \makeatother%
  \begin{picture}(1,0.95013331)%
    \lineheight{1}%
    \setlength\tabcolsep{0pt}%
    \put(0,0){\includegraphics[width=\unitlength,page=1]{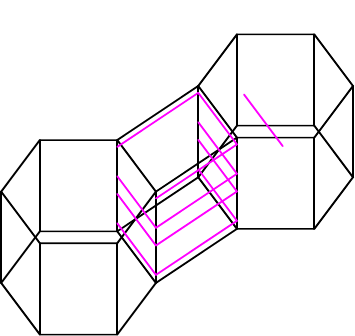}}%
    \put(0.16650595,0.59580739){\color[rgb]{0,0,0}\makebox(0,0)[lt]{\lineheight{1.25}\smash{\begin{tabular}[t]{l}$a$\end{tabular}}}}%
    \put(0.72718637,0.90973778){\color[rgb]{0,0,0}\makebox(0,0)[lt]{\lineheight{1.25}\smash{\begin{tabular}[t]{l}$b$\end{tabular}}}}%
    \put(0,0){\includegraphics[width=\unitlength,page=2]{third_row_after.pdf}}%
  \end{picture}%
\endgroup%
\vspace{0.2cm}\\
     \hline
     \makecell{\def\svgscale{0.8}
\begingroup%
  \makeatletter%
  \providecommand\color[2][]{%
    \errmessage{(Inkscape) Color is used for the text in Inkscape, but the package 'color.sty' is not loaded}%
    \renewcommand\color[2][]{}%
  }%
  \providecommand\transparent[1]{%
    \errmessage{(Inkscape) Transparency is used (non-zero) for the text in Inkscape, but the package 'transparent.sty' is not loaded}%
    \renewcommand\transparent[1]{}%
  }%
  \providecommand\rotatebox[2]{#2}%
  \newcommand*\fsize{\dimexpr\f@size pt\relax}%
  \newcommand*\lineheight[1]{\fontsize{\fsize}{#1\fsize}\selectfont}%
  \ifx\svgwidth\undefined%
    \setlength{\unitlength}{87.88094168bp}%
    \ifx\svgscale\undefined%
      \relax%
    \else%
      \setlength{\unitlength}{\unitlength * \real{\svgscale}}%
    \fi%
  \else%
    \setlength{\unitlength}{\svgwidth}%
  \fi%
  \global\let\svgwidth\undefined%
  \global\let\svgscale\undefined%
  \makeatother%
  \begin{picture}(1,0.51815034)%
    \lineheight{1}%
    \setlength\tabcolsep{0pt}%
    \put(0,0){\includegraphics[width=\unitlength,page=1]{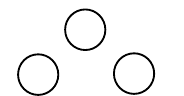}}%
    \put(-0.00548562,0.09585846){\color[rgb]{0,0,0}\makebox(0,0)[lt]{\lineheight{1.25}\smash{\begin{tabular}[t]{l}$a$\end{tabular}}}}%
    \put(0.86126304,0.09392468){\color[rgb]{0,0,0}\makebox(0,0)[lt]{\lineheight{1.25}\smash{\begin{tabular}[t]{l}$c$\end{tabular}}}}%
    \put(0.31288951,0.45090545){\color[rgb]{0,0,0}\makebox(0,0)[lt]{\lineheight{1.25}\smash{\begin{tabular}[t]{l}$b$\end{tabular}}}}%
    \put(0,0){\includegraphics[width=\unitlength,page=2]{pruning_4_part_1.pdf}}%
  \end{picture}%
\endgroup%
\raisebox{1.5em}{$\rightsquigarrow$} \def\svgwidth{2cm}
\begingroup%
  \makeatletter%
  \providecommand\color[2][]{%
    \errmessage{(Inkscape) Color is used for the text in Inkscape, but the package 'color.sty' is not loaded}%
    \renewcommand\color[2][]{}%
  }%
  \providecommand\transparent[1]{%
    \errmessage{(Inkscape) Transparency is used (non-zero) for the text in Inkscape, but the package 'transparent.sty' is not loaded}%
    \renewcommand\transparent[1]{}%
  }%
  \providecommand\rotatebox[2]{#2}%
  \newcommand*\fsize{\dimexpr\f@size pt\relax}%
  \newcommand*\lineheight[1]{\fontsize{\fsize}{#1\fsize}\selectfont}%
  \ifx\svgwidth\undefined%
    \setlength{\unitlength}{87.88094168bp}%
    \ifx\svgscale\undefined%
      \relax%
    \else%
      \setlength{\unitlength}{\unitlength * \real{\svgscale}}%
    \fi%
  \else%
    \setlength{\unitlength}{\svgwidth}%
  \fi%
  \global\let\svgwidth\undefined%
  \global\let\svgscale\undefined%
  \makeatother%
  \begin{picture}(1,0.51815084)%
    \lineheight{1}%
    \setlength\tabcolsep{0pt}%
    \put(0,0){\includegraphics[width=\unitlength,page=1]{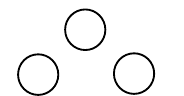}}%
    \put(-0.00548565,0.09585859){\color[rgb]{0,0,0}\makebox(0,0)[lt]{\lineheight{1.25}\smash{\begin{tabular}[t]{l}$a$\end{tabular}}}}%
    \put(0.86126314,0.09392482){\color[rgb]{0,0,0}\makebox(0,0)[lt]{\lineheight{1.25}\smash{\begin{tabular}[t]{l}$c$\end{tabular}}}}%
    \put(0.31288948,0.45090558){\color[rgb]{0,0,0}\makebox(0,0)[lt]{\lineheight{1.25}\smash{\begin{tabular}[t]{l}$b$\end{tabular}}}}%
    \put(0,0){\includegraphics[width=\unitlength,page=2]{pruning_4_part_2.pdf}}%
  \end{picture}%
\endgroup%
 \\ $w=z\neq y$} & \vspace{0.3cm} \def\svgscale{0.45}
\begingroup%
  \makeatletter%
  \providecommand\color[2][]{%
    \errmessage{(Inkscape) Color is used for the text in Inkscape, but the package 'color.sty' is not loaded}%
    \renewcommand\color[2][]{}%
  }%
  \providecommand\transparent[1]{%
    \errmessage{(Inkscape) Transparency is used (non-zero) for the text in Inkscape, but the package 'transparent.sty' is not loaded}%
    \renewcommand\transparent[1]{}%
  }%
  \providecommand\rotatebox[2]{#2}%
  \newcommand*\fsize{\dimexpr\f@size pt\relax}%
  \newcommand*\lineheight[1]{\fontsize{\fsize}{#1\fsize}\selectfont}%
  \ifx\svgwidth\undefined%
    \setlength{\unitlength}{261.82346855bp}%
    \ifx\svgscale\undefined%
      \relax%
    \else%
      \setlength{\unitlength}{\unitlength * \real{\svgscale}}%
    \fi%
  \else%
    \setlength{\unitlength}{\svgwidth}%
  \fi%
  \global\let\svgwidth\undefined%
  \global\let\svgscale\undefined%
  \makeatother%
  \begin{picture}(1,0.61531491)%
    \lineheight{1}%
    \setlength\tabcolsep{0pt}%
    \put(0,0){\includegraphics[width=\unitlength,page=1]{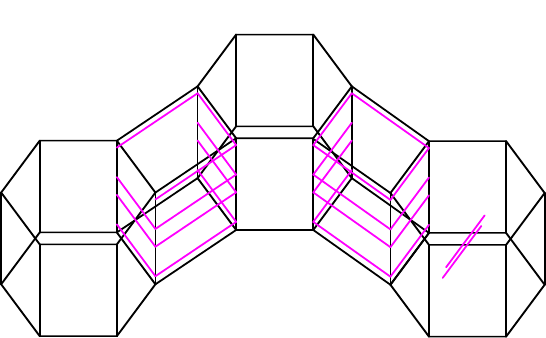}}%
    \put(0.10768178,0.38616696){\color[rgb]{0,0,0}\makebox(0,0)[lt]{\lineheight{1.25}\smash{\begin{tabular}[t]{l}$a$\end{tabular}}}}%
    \put(0.47028235,0.58919043){\color[rgb]{0,0,0}\makebox(0,0)[lt]{\lineheight{1.25}\smash{\begin{tabular}[t]{l}$b$\end{tabular}}}}%
    \put(0.82654708,0.39228679){\color[rgb]{0,0,0}\makebox(0,0)[lt]{\lineheight{1.25}\smash{\begin{tabular}[t]{l}$c$\end{tabular}}}}%
    \put(0,0){\includegraphics[width=\unitlength,page=2]{fourth_row_before.pdf}}%
  \end{picture}%
\endgroup%
 \raisebox{2.0em}{$ - $ } \def\svgscale{0.45}
\begingroup%
  \makeatletter%
  \providecommand\color[2][]{%
    \errmessage{(Inkscape) Color is used for the text in Inkscape, but the package 'color.sty' is not loaded}%
    \renewcommand\color[2][]{}%
  }%
  \providecommand\transparent[1]{%
    \errmessage{(Inkscape) Transparency is used (non-zero) for the text in Inkscape, but the package 'transparent.sty' is not loaded}%
    \renewcommand\transparent[1]{}%
  }%
  \providecommand\rotatebox[2]{#2}%
  \newcommand*\fsize{\dimexpr\f@size pt\relax}%
  \newcommand*\lineheight[1]{\fontsize{\fsize}{#1\fsize}\selectfont}%
  \ifx\svgwidth\undefined%
    \setlength{\unitlength}{261.8235118bp}%
    \ifx\svgscale\undefined%
      \relax%
    \else%
      \setlength{\unitlength}{\unitlength * \real{\svgscale}}%
    \fi%
  \else%
    \setlength{\unitlength}{\svgwidth}%
  \fi%
  \global\let\svgwidth\undefined%
  \global\let\svgscale\undefined%
  \makeatother%
  \begin{picture}(1,0.61531472)%
    \lineheight{1}%
    \setlength\tabcolsep{0pt}%
    \put(0,0){\includegraphics[width=\unitlength,page=1]{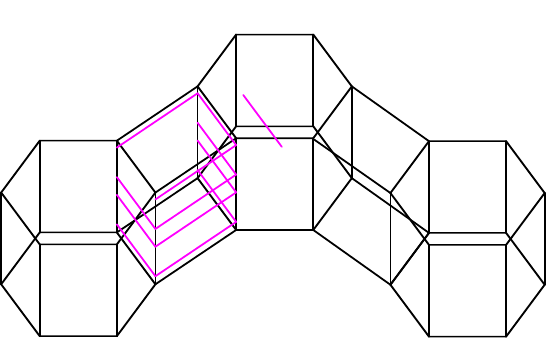}}%
    \put(0.10768184,0.38616695){\color[rgb]{0,0,0}\makebox(0,0)[lt]{\lineheight{1.25}\smash{\begin{tabular}[t]{l}$a$\end{tabular}}}}%
    \put(0.47028218,0.58919039){\color[rgb]{0,0,0}\makebox(0,0)[lt]{\lineheight{1.25}\smash{\begin{tabular}[t]{l}$b$\end{tabular}}}}%
    \put(0.82654689,0.39228695){\color[rgb]{0,0,0}\makebox(0,0)[lt]{\lineheight{1.25}\smash{\begin{tabular}[t]{l}$c$\end{tabular}}}}%
    \put(0,0){\includegraphics[width=\unitlength,page=2]{fourth_row_after.pdf}}%
  \end{picture}%
\endgroup%
\vspace{0.2cm}\\
     \hline
     \makecell{\def\svgscale{0.8}
      \raisebox{1.5em}{$\rightsquigarrow$} \def\svgwidth{2cm}
     \\ $w,y,z$ are all different} & \vspace{0.3cm} \def\svgscale{0.45}
\begingroup%
  \makeatletter%
  \providecommand\color[2][]{%
    \errmessage{(Inkscape) Color is used for the text in Inkscape, but the package 'color.sty' is not loaded}%
    \renewcommand\color[2][]{}%
  }%
  \providecommand\transparent[1]{%
    \errmessage{(Inkscape) Transparency is used (non-zero) for the text in Inkscape, but the package 'transparent.sty' is not loaded}%
    \renewcommand\transparent[1]{}%
  }%
  \providecommand\rotatebox[2]{#2}%
  \newcommand*\fsize{\dimexpr\f@size pt\relax}%
  \newcommand*\lineheight[1]{\fontsize{\fsize}{#1\fsize}\selectfont}%
  \ifx\svgwidth\undefined%
    \setlength{\unitlength}{261.82336041bp}%
    \ifx\svgscale\undefined%
      \relax%
    \else%
      \setlength{\unitlength}{\unitlength * \real{\svgscale}}%
    \fi%
  \else%
    \setlength{\unitlength}{\svgwidth}%
  \fi%
  \global\let\svgwidth\undefined%
  \global\let\svgscale\undefined%
  \makeatother%
  \begin{picture}(1,0.92638772)%
    \lineheight{1}%
    \setlength\tabcolsep{0pt}%
    \put(0,0){\includegraphics[width=\unitlength,page=1]{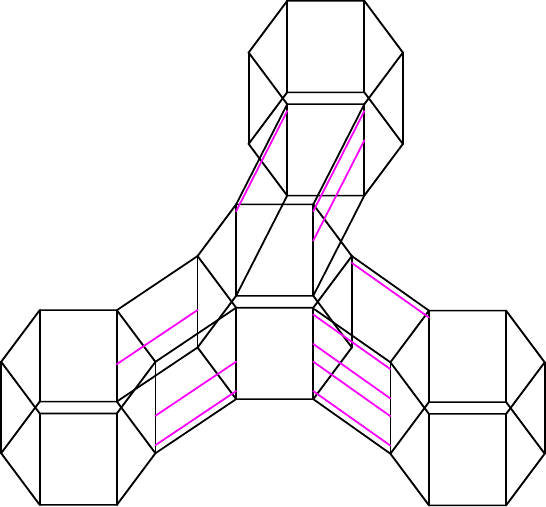}}%
    \put(0.10574439,0.37648141){\color[rgb]{0,0,0}\makebox(0,0)[lt]{\lineheight{1.25}\smash{\begin{tabular}[t]{l}$a$\end{tabular}}}}%
    \put(0.47185751,0.09338529){\color[rgb]{0,0,0}\makebox(0,0)[lt]{\lineheight{1.25}\smash{\begin{tabular}[t]{l}$b$\end{tabular}}}}%
    \put(0.83797079,0.37428904){\color[rgb]{0,0,0}\makebox(0,0)[lt]{\lineheight{1.25}\smash{\begin{tabular}[t]{l}$c$\end{tabular}}}}%
    \put(0.75042413,0.74259461){\color[rgb]{0,0,0}\makebox(0,0)[lt]{\lineheight{1.25}\smash{\begin{tabular}[t]{l}$d$\end{tabular}}}}%
    \put(0,0){\includegraphics[width=\unitlength,page=2]{fifth_row_turnaround_before.pdf}}%
  \end{picture}%
\endgroup%
 \raisebox{2.0em}{$ - $ }\def\svgscale{0.45}
\begingroup%
  \makeatletter%
  \providecommand\color[2][]{%
    \errmessage{(Inkscape) Color is used for the text in Inkscape, but the package 'color.sty' is not loaded}%
    \renewcommand\color[2][]{}%
  }%
  \providecommand\transparent[1]{%
    \errmessage{(Inkscape) Transparency is used (non-zero) for the text in Inkscape, but the package 'transparent.sty' is not loaded}%
    \renewcommand\transparent[1]{}%
  }%
  \providecommand\rotatebox[2]{#2}%
  \newcommand*\fsize{\dimexpr\f@size pt\relax}%
  \newcommand*\lineheight[1]{\fontsize{\fsize}{#1\fsize}\selectfont}%
  \ifx\svgwidth\undefined%
    \setlength{\unitlength}{261.82359831bp}%
    \ifx\svgscale\undefined%
      \relax%
    \else%
      \setlength{\unitlength}{\unitlength * \real{\svgscale}}%
    \fi%
  \else%
    \setlength{\unitlength}{\svgwidth}%
  \fi%
  \global\let\svgwidth\undefined%
  \global\let\svgscale\undefined%
  \makeatother%
  \begin{picture}(1,0.92638688)%
    \lineheight{1}%
    \setlength\tabcolsep{0pt}%
    \put(0,0){\includegraphics[width=\unitlength,page=1]{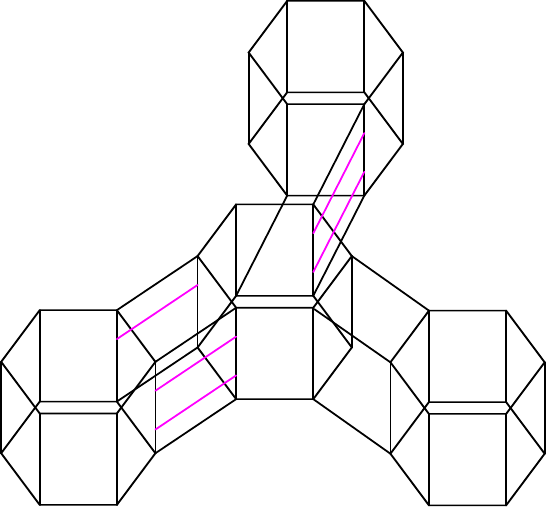}}%
    \put(0.10573674,0.37648113){\color[rgb]{0,0,0}\makebox(0,0)[lt]{\lineheight{1.25}\smash{\begin{tabular}[t]{l}$a$\end{tabular}}}}%
    \put(0.47184903,0.09911428){\color[rgb]{0,0,0}\makebox(0,0)[lt]{\lineheight{1.25}\smash{\begin{tabular}[t]{l}$b$\end{tabular}}}}%
    \put(0.83796181,0.37428876){\color[rgb]{0,0,0}\makebox(0,0)[lt]{\lineheight{1.25}\smash{\begin{tabular}[t]{l}$c$\end{tabular}}}}%
    \put(0.75041523,0.74259367){\color[rgb]{0,0,0}\makebox(0,0)[lt]{\lineheight{1.25}\smash{\begin{tabular}[t]{l}$d$\end{tabular}}}}%
    \put(0,0){\includegraphics[width=\unitlength,page=2]{fifth_row_after.pdf}}%
  \end{picture}%
\endgroup%
\vspace{0.2cm}\\
     \hline
   \end{tabular}
   \caption{The following table gives us how the type of pruning move $Z \rightsquigarrow Z'$ determines the difference $[K] - [K']$. The change in each case is $y_x+(\Mid\{w,y,z\})_x+\sgn(w,y,z)$.}
\label{pull through figure}
\end{table}
\end{lemma}
\begin{proof}
  We break up our proof into cases of pruning moves $Z\rightsquigarrow Z'$ described in the rows of Table \ref{pull through figure}. We may assume that $K'$ is defined by taking the tube $U\subset K$ parametrized by $D$, and replacing it with a tube $U'$ parametrized by $D'$. And just as in Corollary \ref{cycle move}, it suffices to compute $U-U' = ([K] - [K'])$. We first look at the second, third, and fourth rows, since they are the easiest.
  \paragraph{Rows 1 and 2} The difference $U-U'$ can be seen as the following:
  \[
    \def\svgscale{0.55}\raisebox{1cm}{ $ - $ }\def\svgscale{0.55}
  \]
  And as we can see, $U\cong U'$, since $U'$ is a homotopy of $U$ by ``pulling'' $U$ tight between faces $\mathbf{F}_y$ and $\mathbf{F}_x$. The third row can be visualized similarly, and in both cases, $U-U' = (0)$.
  \paragraph{Row 3:} The difference $U-U'$ can be visualized as the following:
  \[
    \def\svgscale{0.55}\raisebox{1cm}{ $ - $ }\def\svgscale{0.55},
  \]
  which can be seen as unfolding a crease in $U$ between faces $\mathbf{F}_a$ and $\mathbf{F}_b$. So $U\cong U'$, and therefore $U-U = (0)$.\par
  \paragraph{Row 4:} The difference $U-U'$ can be seen as the following:
  \[
    \def\svgscale{0.5}\raisebox{2.0em}{ $ - $ }\def\svgscale{0.5},
    \]
  which is the same as the difference $U - U^{\conv}$:
  \[
    \def\svgscale{0.5}\raisebox{2.0em}{ $ - $ } \def\svgscale{0.5}
\begingroup%
  \makeatletter%
  \providecommand\color[2][]{%
    \errmessage{(Inkscape) Color is used for the text in Inkscape, but the package 'color.sty' is not loaded}%
    \renewcommand\color[2][]{}%
  }%
  \providecommand\transparent[1]{%
    \errmessage{(Inkscape) Transparency is used (non-zero) for the text in Inkscape, but the package 'transparent.sty' is not loaded}%
    \renewcommand\transparent[1]{}%
  }%
  \providecommand\rotatebox[2]{#2}%
  \newcommand*\fsize{\dimexpr\f@size pt\relax}%
  \newcommand*\lineheight[1]{\fontsize{\fsize}{#1\fsize}\selectfont}%
  \ifx\svgwidth\undefined%
    \setlength{\unitlength}{261.8235118bp}%
    \ifx\svgscale\undefined%
      \relax%
    \else%
      \setlength{\unitlength}{\unitlength * \real{\svgscale}}%
    \fi%
  \else%
    \setlength{\unitlength}{\svgwidth}%
  \fi%
  \global\let\svgwidth\undefined%
  \global\let\svgscale\undefined%
  \makeatother%
  \begin{picture}(1,0.61531476)%
    \lineheight{1}%
    \setlength\tabcolsep{0pt}%
    \put(0,0){\includegraphics[width=\unitlength,page=1]{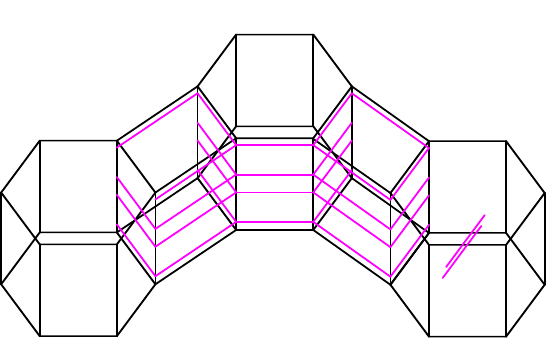}}%
    \put(0.10768184,0.38616691){\color[rgb]{0,0,0}\makebox(0,0)[lt]{\lineheight{1.25}\smash{\begin{tabular}[t]{l}$a$\end{tabular}}}}%
    \put(0.47028218,0.58919035){\color[rgb]{0,0,0}\makebox(0,0)[lt]{\lineheight{1.25}\smash{\begin{tabular}[t]{l}$b$\end{tabular}}}}%
    \put(0.82654689,0.39228687){\color[rgb]{0,0,0}\makebox(0,0)[lt]{\lineheight{1.25}\smash{\begin{tabular}[t]{l}$c$\end{tabular}}}}%
    \put(0,0){\includegraphics[width=\unitlength,page=2]{fourth_row_smoothed.pdf}}%
  \end{picture}%
\endgroup%
,
  \]
  For ease of notation, we will call $i=a_b$ and $j=c_b$:
  \[
    \def\svgscale{0.5}
\begingroup%
  \makeatletter%
  \providecommand\color[2][]{%
    \errmessage{(Inkscape) Color is used for the text in Inkscape, but the package 'color.sty' is not loaded}%
    \renewcommand\color[2][]{}%
  }%
  \providecommand\transparent[1]{%
    \errmessage{(Inkscape) Transparency is used (non-zero) for the text in Inkscape, but the package 'transparent.sty' is not loaded}%
    \renewcommand\transparent[1]{}%
  }%
  \providecommand\rotatebox[2]{#2}%
  \newcommand*\fsize{\dimexpr\f@size pt\relax}%
  \newcommand*\lineheight[1]{\fontsize{\fsize}{#1\fsize}\selectfont}%
  \ifx\svgwidth\undefined%
    \setlength{\unitlength}{302.86536269bp}%
    \ifx\svgscale\undefined%
      \relax%
    \else%
      \setlength{\unitlength}{\unitlength * \real{\svgscale}}%
    \fi%
  \else%
    \setlength{\unitlength}{\svgwidth}%
  \fi%
  \global\let\svgwidth\undefined%
  \global\let\svgscale\undefined%
  \makeatother%
  \begin{picture}(1,0.53889915)%
    \lineheight{1}%
    \setlength\tabcolsep{0pt}%
    \put(0,0){\includegraphics[width=\unitlength,page=1]{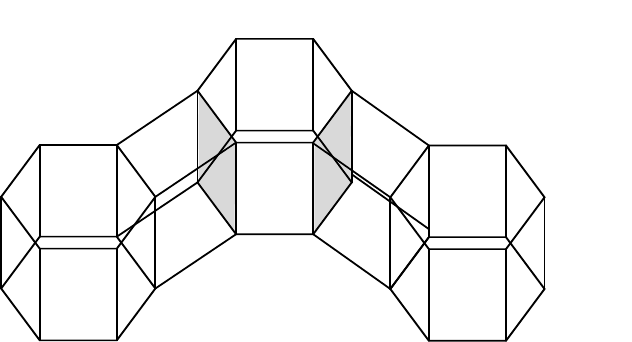}}%
    \put(0.09462072,0.33645116){\color[rgb]{0,0,0}\makebox(0,0)[lt]{\lineheight{1.25}\smash{\begin{tabular}[t]{l}$a$\end{tabular}}}}%
    \put(0.39672164,0.50799587){\color[rgb]{0,0,0}\makebox(0,0)[lt]{\lineheight{1.25}\smash{\begin{tabular}[t]{l}$b$\end{tabular}}}}%
    \put(0.72007581,0.33037873){\color[rgb]{0,0,0}\makebox(0,0)[lt]{\lineheight{1.25}\smash{\begin{tabular}[t]{l}$c$\end{tabular}}}}%
    \put(0,0){\includegraphics[width=\unitlength,page=2]{faces_of_F_b.pdf}}%
    \put(0.10372439,0.51616631){\color[rgb]{0,0,0}\makebox(0,0)[lt]{\lineheight{1.25}\smash{\begin{tabular}[t]{l}$\mathbf{F}_i(\mathbf{F}_b)$\end{tabular}}}}%
    \put(0.59410681,0.51604088){\color[rgb]{0,0,0}\makebox(0,0)[lt]{\lineheight{1.25}\smash{\begin{tabular}[t]{l}$\mathbf{F}_j(\mathbf{F}_b)$\end{tabular}}}}%
  \end{picture}%
\endgroup%
.
  \]
  So in other words, the $i^\text{th}$ face and $j^\text{th}$ face of the face $F_b$ are the faces that point towards to $F_a$ and $F_c$ respectively. For now we assume that $a < c$.\par
  By a repeated application of Proposition \ref{sliding}, we see
  \begingroup
  \allowdisplaybreaks
  \begin{align*}
    U &= U^{\conv} + (\boldvarphi_{j-1}\ldots\boldvarphi_{i+1},\alpha ) + (\boldvarphi_{j-1}\ldots\boldvarphi_{i+1},\beta)\\
      & = U^{\conv} + (\boldvarphi_{j-2}\ldots\boldvarphi_{i+1},\alpha ) + (\boldvarphi_{j-1}\boldvarphi_{j-1}\ldots\boldvarphi_{i+1},\beta)\\
      & = \ldots\\
      &= U^{\conv} + (\boldvarphi_{i+1},\alpha ) + (\boldvarphi_{i+2}\ldots\boldvarphi_{j-1}\boldvarphi_{j-1}\ldots\boldvarphi_{i+1},\beta)\\
      &= U^{\conv} + (\boldvarphi_{i+1}\ldots\boldvarphi_{j-1}\boldvarphi_{j-1}\ldots\boldvarphi_{i+1},\beta)\\
    &= U^{\conv} + (j-i-1),
  \end{align*}
  \endgroup
  where we realize the last equality by realizing each diamond composition of half-twists $\boldvarphi\diamond \boldvarphi$ becomes a full twist $(1)$.
  Therefore, $U- U^{\conv} = U-U' = i+j-1 = a_b + c_b - 1$.\par
  The case $a>c$ is similar to our worked case $a<c$.\par
  \paragraph{Row 5:}
  There are $6$ cases to work with, depending on the order of $a,b,d$, but the strategy for each case is the same. We lay out the general argument, keeping in mind the picture of $U-U'$ as
  \begin{equation}\label{column one before}
    \def\svgscale{0.5} \raisebox{5.0em}{$-$ } \def\svgscale{0.5}.
  \end{equation}
  We note $U - U'^{\conv} =  (U^{\conv} - U'^{\conv}) \diamond (U-U^{\conv})$, as depicted by
  \begingroup
  \begin{gather*}
    \allowdisplaybreaks
    \def\svgscale{0.38} \raisebox{4.0em}{$-$ } \def\svgscale{0.38}
\begingroup%
  \makeatletter%
  \providecommand\color[2][]{%
    \errmessage{(Inkscape) Color is used for the text in Inkscape, but the package 'color.sty' is not loaded}%
    \renewcommand\color[2][]{}%
  }%
  \providecommand\transparent[1]{%
    \errmessage{(Inkscape) Transparency is used (non-zero) for the text in Inkscape, but the package 'transparent.sty' is not loaded}%
    \renewcommand\transparent[1]{}%
  }%
  \providecommand\rotatebox[2]{#2}%
  \newcommand*\fsize{\dimexpr\f@size pt\relax}%
  \newcommand*\lineheight[1]{\fontsize{\fsize}{#1\fsize}\selectfont}%
  \ifx\svgwidth\undefined%
    \setlength{\unitlength}{261.82338204bp}%
    \ifx\svgscale\undefined%
      \relax%
    \else%
      \setlength{\unitlength}{\unitlength * \real{\svgscale}}%
    \fi%
  \else%
    \setlength{\unitlength}{\svgwidth}%
  \fi%
  \global\let\svgwidth\undefined%
  \global\let\svgscale\undefined%
  \makeatother%
  \begin{picture}(1,0.92638781)%
    \lineheight{1}%
    \setlength\tabcolsep{0pt}%
    \put(0,0){\includegraphics[width=\unitlength,page=1]{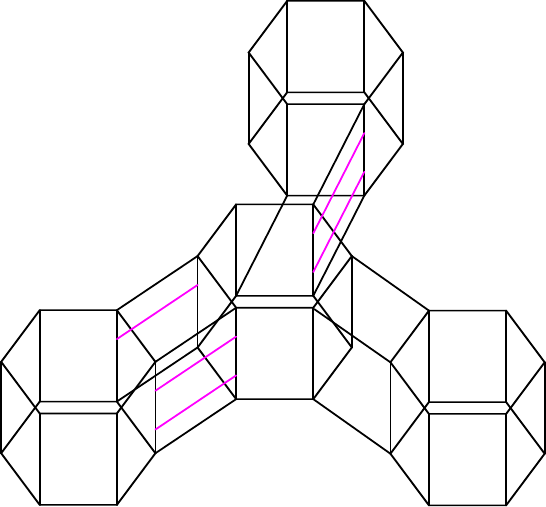}}%
    \put(0.10574435,0.37648155){\color[rgb]{0,0,0}\makebox(0,0)[lt]{\lineheight{1.25}\smash{\begin{tabular}[t]{l}$a$\end{tabular}}}}%
    \put(0.4718575,0.12203037){\color[rgb]{0,0,0}\makebox(0,0)[lt]{\lineheight{1.25}\smash{\begin{tabular}[t]{l}$b$\end{tabular}}}}%
    \put(0.83797059,0.37428918){\color[rgb]{0,0,0}\makebox(0,0)[lt]{\lineheight{1.25}\smash{\begin{tabular}[t]{l}$c$\end{tabular}}}}%
    \put(0.75042401,0.74259455){\color[rgb]{0,0,0}\makebox(0,0)[lt]{\lineheight{1.25}\smash{\begin{tabular}[t]{l}$d$\end{tabular}}}}%
    \put(0,0){\includegraphics[width=\unitlength,page=2]{fifth_row_untwisted.pdf}}%
  \end{picture}%
\endgroup%
\\    
    \raisebox{4.0em}{$=$} \def\svgscale{0.38}
\begingroup%
  \makeatletter%
  \providecommand\color[2][]{%
    \errmessage{(Inkscape) Color is used for the text in Inkscape, but the package 'color.sty' is not loaded}%
    \renewcommand\color[2][]{}%
  }%
  \providecommand\transparent[1]{%
    \errmessage{(Inkscape) Transparency is used (non-zero) for the text in Inkscape, but the package 'transparent.sty' is not loaded}%
    \renewcommand\transparent[1]{}%
  }%
  \providecommand\rotatebox[2]{#2}%
  \newcommand*\fsize{\dimexpr\f@size pt\relax}%
  \newcommand*\lineheight[1]{\fontsize{\fsize}{#1\fsize}\selectfont}%
  \ifx\svgwidth\undefined%
    \setlength{\unitlength}{261.82336041bp}%
    \ifx\svgscale\undefined%
      \relax%
    \else%
      \setlength{\unitlength}{\unitlength * \real{\svgscale}}%
    \fi%
  \else%
    \setlength{\unitlength}{\svgwidth}%
  \fi%
  \global\let\svgwidth\undefined%
  \global\let\svgscale\undefined%
  \makeatother%
  \begin{picture}(1,0.9263878)%
    \lineheight{1}%
    \setlength\tabcolsep{0pt}%
    \put(0,0){\includegraphics[width=\unitlength,page=1]{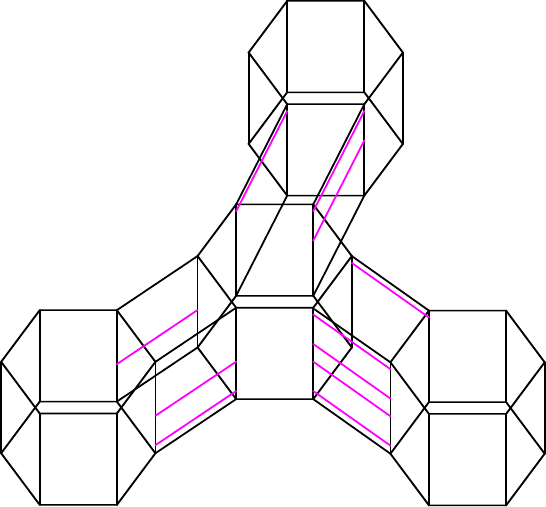}}%
    \put(0.10574437,0.37648151){\color[rgb]{0,0,0}\makebox(0,0)[lt]{\lineheight{1.25}\smash{\begin{tabular}[t]{l}$a$\end{tabular}}}}%
    \put(0.47185748,0.09911366){\color[rgb]{0,0,0}\makebox(0,0)[lt]{\lineheight{1.25}\smash{\begin{tabular}[t]{l}$b$\end{tabular}}}}%
    \put(0.83797068,0.37428914){\color[rgb]{0,0,0}\makebox(0,0)[lt]{\lineheight{1.25}\smash{\begin{tabular}[t]{l}$c$\end{tabular}}}}%
    \put(0.75042401,0.74259473){\color[rgb]{0,0,0}\makebox(0,0)[lt]{\lineheight{1.25}\smash{\begin{tabular}[t]{l}$d$\end{tabular}}}}%
    \put(0,0){\includegraphics[width=\unitlength,page=2]{fifth_row_smoothed.pdf}}%
  \end{picture}%
\endgroup%
 \raisebox{4.0em}{$-$ } \def\svgscale{0.38}
    \raisebox{4.0em}{$\circ$} \def\svgscale{0.38} \raisebox{4.0em}{$-$ } \def\svgscale{0.38}
  \end{gather*}
  \endgroup
  Denote $U^{\conv} = \{(\Theta^{\conv}_{ad},\alpha),(\Theta^{\conv}_{da},\beta)\}$, $U'^{\conv} = \{(\Theta'^{\conv}_{ad},\alpha),(\Theta'^{\conv}_{da},\beta)\}$, where $\Theta^{\conv}_{ad}$ (resp. $\Theta'^{\conv}_{ad}$) moves from $\mathbf{G}_a$ to $\mathbf{G}_b$ (purple to green). Note that $U^{\conv}$ is homotopic to a single half twist of $U'^{\conv}$, which we can see as the equivalence
  \begin{equation*}
    \def\svgscale{0.5} \raisebox{4.0em}{$\cong$}\def\svgscale{0.5}
\begingroup%
  \makeatletter%
  \providecommand\color[2][]{%
    \errmessage{(Inkscape) Color is used for the text in Inkscape, but the package 'color.sty' is not loaded}%
    \renewcommand\color[2][]{}%
  }%
  \providecommand\transparent[1]{%
    \errmessage{(Inkscape) Transparency is used (non-zero) for the text in Inkscape, but the package 'transparent.sty' is not loaded}%
    \renewcommand\transparent[1]{}%
  }%
  \providecommand\rotatebox[2]{#2}%
  \newcommand*\fsize{\dimexpr\f@size pt\relax}%
  \newcommand*\lineheight[1]{\fontsize{\fsize}{#1\fsize}\selectfont}%
  \ifx\svgwidth\undefined%
    \setlength{\unitlength}{261.82333879bp}%
    \ifx\svgscale\undefined%
      \relax%
    \else%
      \setlength{\unitlength}{\unitlength * \real{\svgscale}}%
    \fi%
  \else%
    \setlength{\unitlength}{\svgwidth}%
  \fi%
  \global\let\svgwidth\undefined%
  \global\let\svgscale\undefined%
  \makeatother%
  \begin{picture}(1,0.92638771)%
    \lineheight{1}%
    \setlength\tabcolsep{0pt}%
    \put(0,0){\includegraphics[width=\unitlength,page=1]{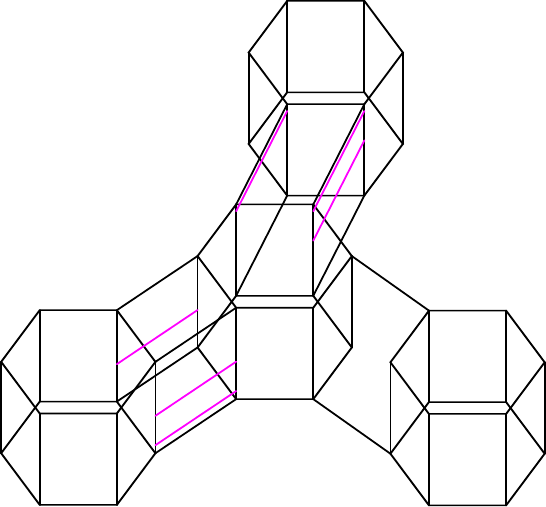}}%
    \put(0.10574419,0.3764814){\color[rgb]{0,0,0}\makebox(0,0)[lt]{\lineheight{1.25}\smash{\begin{tabular}[t]{l}$a$\end{tabular}}}}%
    \put(0.47185734,0.10484176){\color[rgb]{0,0,0}\makebox(0,0)[lt]{\lineheight{1.25}\smash{\begin{tabular}[t]{l}$b$\end{tabular}}}}%
    \put(0.83797066,0.37428886){\color[rgb]{0,0,0}\makebox(0,0)[lt]{\lineheight{1.25}\smash{\begin{tabular}[t]{l}$c$\end{tabular}}}}%
    \put(0.75042415,0.74259455){\color[rgb]{0,0,0}\makebox(0,0)[lt]{\lineheight{1.25}\smash{\begin{tabular}[t]{l}$d$\end{tabular}}}}%
    \put(0,0){\includegraphics[width=\unitlength,page=2]{fifth_row_one_twist.pdf}}%
  \end{picture}%
\endgroup%
,
  \end{equation*}
  $U^{\conv}\cong U'^{\conv}\diamond (\boldsymbol{\rho},\alpha)$, where $\boldsymbol{\rho} = [-\epsilon,\epsilon]^{A+B}\times \rho \times [0,1]$, and $\rho$ is the twist determined by Table \ref{flipback table} with input $\vcenter{\hbox{\hextwist{$c_b$}{$d_b$}{$a_b$}}}$. So in total, we observe $(U^{\conv} - U'^{\conv}) = (\boldsymbol{\rho},\Theta'^{\conv}_{ad})$ and thus compute
  \[
    U-U'^{\conv} = (\boldsymbol{\rho}, \alpha) \diamond (U-U^{\conv}),\qquad U-U' = (U'^{\conv}-U') \diamond (U-U'^{\conv}).
  \]
  \paragraph{Case 1: $a<d<c$.} We let $i = a_b$, $j = d_b$, $k = c_b$. We compute $U'^{\conv} - U' = (U' - U'^{\conv})^{-1} =  (\boldvarphi^{-1}_{i+1}\ldots\boldvarphi^{-1}_{j-1},\alpha)$:
  \begin{equation}\label{after to untwisted}
    \def\svgscale{0.5} \raisebox{5.0em}{$-$ } \def\svgscale{0.5},
  \end{equation}
  Now we compute $(U - U^{\conv}) = (\boldvarphi_{k-1}\ldots\boldvarphi_{i+1},\alpha) \diamond (\boldvarphi_{k-1}\ldots\boldvarphi_{j+1}, \beta)$:
  \begin{equation}\label{before to smoothed}
    \def\svgscale{0.5} \raisebox{5.0em}{$-$ } \def\svgscale{0.5}.
  \end{equation}
  Then we input $\vcenter{\hbox{\hextwist{$k$}{$j$}{$i$}}}$ into Table \ref{flipback table} to calculate the twist $\boldsymbol{\rho} = \boldvarphi_{k, j}$
  \begin{equation}\label{smoothed to untwisted}
    \def\svgscale{0.5}\raisebox{5.0em}{$-$ } \def\svgscale{0.5},
  \end{equation}
  Thus, we have
  \begin{align*}
    U - U'^{\conv}
    &= (\boldvarphi_{k, j}, \alpha) \diamond ((\boldvarphi_{k-1}\ldots\boldvarphi_{i+1},\alpha) \diamond (\boldvarphi_{k-1}\ldots\boldvarphi_{j+1}, \beta))\\
    &= (\boldvarphi_{k, j}\boldvarphi_{k-1}\ldots\boldvarphi_{i+1},\alpha) \diamond (\boldvarphi_{k-1}\ldots\boldvarphi_{j+1}, \beta)\\
    &= (\boldvarphi_{k-1}\boldvarphi_{k-1, j}\ldots\boldvarphi_{i+1},\alpha) \diamond (\boldvarphi_{k-1}\ldots\boldvarphi_{j+1}, \beta)\\
    &= \ldots\\
    &= (\boldvarphi_{k-1}\ldots\boldvarphi_{j+1}\boldvarphi_{j+1, j}^{(-1)^{k-j-1}}\boldvarphi_j\ldots\boldvarphi_{i+1},\alpha) \diamond (\boldvarphi_{k-1}\ldots\boldvarphi_{j+1}, \beta)\\
    &= (\boldvarphi_{k-1}\ldots\widehat{\boldvarphi_{j}}\ldots\boldvarphi_{i+1},\Theta'^{\conv}_{ad}) \diamond (\boldvarphi_{k-1}\ldots\boldvarphi_{j+1}, \Theta'^{\conv}_{da}) + (k-j-1)\\
    &= (\boldvarphi_{k-1}^{-1}\boldvarphi_{k-1}\ldots\widehat{\boldvarphi_{j}}\ldots\boldvarphi_{i+1},\alpha) \diamond (\boldvarphi_{k-2}\ldots\boldvarphi_{j+1}, \beta) + (k-j-1)\\
    &= \ldots\\
    &= (\boldvarphi_{j+1}^{-1}\ldots\boldvarphi_{k-1}^{-1}\boldvarphi_{k-1}\ldots\widehat{\boldvarphi_{j}}\ldots\boldvarphi_{i+1},\alpha) \diamond (c_{\Id}, \beta) + (k-j-1)\\
    &=  \boldvarphi_{j-1}\ldots\widehat{\boldvarphi_{j}}\ldots\boldvarphi_{i+1},\alpha) \diamond (c_{\Id}, \beta) + (k-j-1),
  \end{align*}
  which implies
  \begin{align*}
    U-U'
    &= (U'^{\conv}-U')\diamond (U-U'^{\conv})\\
    &= (\boldvarphi_{i+1}^{-1}\ldots\boldvarphi_{j-1}^{-1},\alpha) \diamond (\boldvarphi_{j-1}\ldots\boldvarphi_{i+1},\alpha) + (k-j-1)\\
    &= (k-j-1)\\
    &= (d_b+c_b+1) = (\Mid(a,c,d)_b+c_b+\sgn(a,c,d)).\qedhere
  \end{align*}
\end{proof}
\begin{lemma}\label{Q change}
  Let $Z$ be a facet cycle with all turnarounds oriented along the same direction and let $Z\rightsquigarrow Z'$ be a pruning move $(a\longline b\longline \overrightarrow{c}\longline b\longline d) \rightsquigarrow (a\longline b\longline d)$. If $K$ is parametrized by $Z$ and $K'$ is parametrized by $Z'$, then we have $Q(Z)-Q(Z') = \Mid(a,c,d)_b+c_b+\sgn(a,c,d)$.
\end{lemma}
\begin{proof}
  Orient $Z$ as, for example $a_1\to a_2\to\ldots a_n\to a_1$, so that the orientation agrees with the direction of the turnarounds. Also orient $Z'$ in the same way as $Z$ (so $Z$ agrees with $Z'$ outside of the pruning). Now we compute the difference:
  \begingroup
  \allowdisplaybreaks
\begin{align*}
  &Q(Z_K)-Q(Z_K') \\
  &=\Delta\left(\sum_{a \text{---}b} ab
    +\sum_{a} a
    +\sum_{a \text{---} b}\max(a,b)
    +\sum_{a \text{---} b\text{---} c}\max(a_b,c_b)\right)\\
    &\quad +\overbrace{\Delta(1)}^{=0} + \Delta(\#\{a\to \overrightarrow{b}\to c\ \vert\  a>b\} + \#\{a\to \overleftarrow{b}\to c\ \vert\  a<b\})\\
  &= \overbrace{(bc+cb)}^{=0} + (c+b) + (\overbrace{\max(b,c)+\max(c,b)}^{=0})\\
  &\quad + (\overbrace{\max(a_b,c_b) + \max(b_c,b_c) + \max(c_b,d_b)+ \max(a_b,d_b)}^{=\Mid(a,c,d)_b + b_c} \\
  &\quad + (1|_{a>b,a\leq c} + \overbrace{1|_{a<b,a>c}}^{1|_{a<b}(1+1|_{a\leq c})}) + 1|_{b>c}+ (1|_{c>b,c\leq d}+\overbrace{1|_{c<b,c>d}}^{1|_{c<b}(1+1|_{c\leq d})})+(1|_{a>b,a\leq d}+ \overbrace{1|_{a<b,a>d}}^{^{1|_{a<b}(1+1|_{a\leq d})}})\\
  &=c+b+(\Mid(a,c,d))_b+b_c+(1|_{a\leq c}+ 1|_{c\leq d} + 1_{a\leq d})\\
  &=c_b+(\Mid(a,c,d))_b+\sgn(a,c,d).\qedhere
\end{align*}
\endgroup
\end{proof}
\subsection{Continuation of proof of Level 2:}
Let $Z$ be a facet cycle where all the possible turnarounds are pointing in the same direction. If we perform any of the pruning moves $Z\rightsquigarrow Z'$ illustrated in Figure \ref{pull through figure}, Lemmas \ref{pulling through} and \ref{Q change} tell us that $Q(Z)-Q(Z')=[K]-[K']$, where $K$ (resp $K'$) is parametrized by $Z$ (resp.\ $Z'$). After a sequence of pruning moves $Z=Z_0\rightsquigarrow Z_1\rightsquigarrow \ldots \rightsquigarrow Z_n$ we can obtain either
\begin{itemize}
\item Case 1: A cycle $Z_n$ that has no turnarounds, or
\item Case 2: A $2$ vertex cycle $Z_n$ of the form $\overrightarrow{a}\longline\overrightarrow{b}\longline \overrightarrow{a}$.
  \begin{equation}
    
\end{equation}
\end{itemize}
Let $Z_n$ parametrize $K_n$. If $Z_n$ is in Case 1, then $Q(Z_n)=[K_n]$ by our proof of Level 1. If $Z_n$ is in Case 2, we can directly compute $Q(Z_n) = 1 = [K_n]$.\par
We see through a connectedness argument that the difference $Q(Z)-[K] = Q(Z_n) - [K_n] = 0$.
\end{proof}
\subsection{Proof of Level 4.}
We use another connectedness argument to prove that $Q$ is sincere for all facet cycles.\par
Suppose $Z  = (a_1,\omega_1)\longline (a_2,\omega_2)\ldots \longline (a_r,\omega_r)\longline (a_1,\omega_1)$, and $Z' = (a_1,\omega_1+1)\longline (a_2,\omega_2)\ldots \longline (a_r,\omega_r)\longline (a_1,\omega_1+1)$.
\begin{align*}
  Q(Z)-Q(Z')
  &= \sum_{\substack{(a',\omega')\longline (a,\omega)\\\text{in } Z}} (a_b\omega + b_a\omega') + \sum_{(a,\omega)\in Z}\omega - \sum_{\substack{(a',\nu')\longline (a,\nu)\\\text{in }Z'}} (a_b\nu + b_a\nu') - \sum_{(a,\nu)\in Z'}\omega\\
  &= (a_r)_{(a_1)} + (a_2)_{(a_1)} - 1.
\end{align*}
Now let $K$ be a tube parametrized by $Z$, and let $\tilde{\eta}$ be the boundary matching tube corresponding to $(a_1,\omega_1)$. If we ``reflect'' $\tilde{\eta}$ about the $J$-coordinate, we obtain a tube $\tilde{\eta}'$ which is parametrized by $D'$. Substituting in $\tilde{\eta}'$ for $\tilde{\eta}$ therefore gives us a tube $K'$ parametrized by $Z'$. Now observe that $\tilde{\eta}-\tilde{\eta}' = (a_r)_{(a_1)} + (a_2)_{(a_1)} - 1$. Indeed, if $\tilde{\eta} = \tilde{\eta}^{\conv} + (\boldvarphi_{j-1}\diamond \ldots \diamond \boldvarphi_{i+1},\alpha)$, then $\tilde{\eta}' = \tilde{\eta}^{\conv} \diamond (\boldvarphi^{-1}_{j-1}\diamond \ldots \diamond \boldvarphi^{-1}_{i+1},\alpha)$, so
\[
  \tilde{\eta}-\tilde{\eta}' \cong (\boldvarphi_{i+1}\diamond \ldots \diamond \boldvarphi_{j-1}\diamond \boldvarphi_{j-1}\diamond \ldots \diamond \boldvarphi_{i+1},\alpha) = (j-i-1) = (b_c + d_c + 1).
\]
Therefore, $[K] - [K'] = b_c + d_c - 1$.
In other words, if $\eta_c$ is $\eta^{\conv}$ with $(b_c + d_c - 1)$ half-twists added. The tube $\eta'_c$ is simply $\eta_c$ with each half-twist reversed, so the difference $\eta_c-\eta'_c$ is $(b_c + d_c - 1)$ full twists. (See Figure \ref{changing vertex figure} for an illustration of $\eta_c$ compared with $\eta'_c$.)
\begin{figure}
  \begin{tabular}{ c c c}
    \def\svgscale{0.65}
\begingroup%
  \makeatletter%
  \providecommand\color[2][]{%
    \errmessage{(Inkscape) Color is used for the text in Inkscape, but the package 'color.sty' is not loaded}%
    \renewcommand\color[2][]{}%
  }%
  \providecommand\transparent[1]{%
    \errmessage{(Inkscape) Transparency is used (non-zero) for the text in Inkscape, but the package 'transparent.sty' is not loaded}%
    \renewcommand\transparent[1]{}%
  }%
  \providecommand\rotatebox[2]{#2}%
  \newcommand*\fsize{\dimexpr\f@size pt\relax}%
  \newcommand*\lineheight[1]{\fontsize{\fsize}{#1\fsize}\selectfont}%
  \ifx\svgwidth\undefined%
    \setlength{\unitlength}{261.82336041bp}%
    \ifx\svgscale\undefined%
      \relax%
    \else%
      \setlength{\unitlength}{\unitlength * \real{\svgscale}}%
    \fi%
  \else%
    \setlength{\unitlength}{\svgwidth}%
  \fi%
  \global\let\svgwidth\undefined%
  \global\let\svgscale\undefined%
  \makeatother%
  \begin{picture}(1,0.60101746)%
    \lineheight{1}%
    \setlength\tabcolsep{0pt}%
    \put(0,0){\includegraphics[width=\unitlength,page=1]{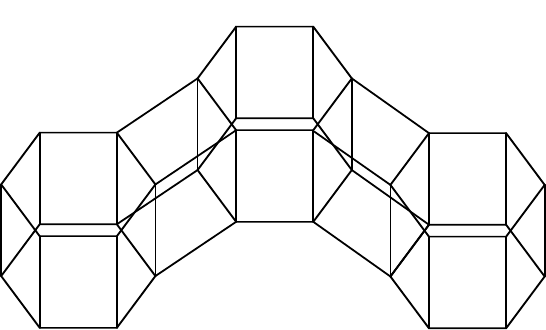}}%
    \put(0.47601107,0.57489306){\color[rgb]{0,0,0}\makebox(0,0)[lt]{\lineheight{1.25}\smash{\begin{tabular}[t]{l}$b$\end{tabular}}}}%
    \put(0,0){\includegraphics[width=\unitlength,page=2]{chain_unflipped.pdf}}%
    \put(0.10535781,0.39629236){\color[rgb]{0,0,0}\makebox(0,0)[lt]{\lineheight{1.25}\smash{\begin{tabular}[t]{l}$a$\end{tabular}}}}%
    \put(0.81930897,0.37569768){\color[rgb]{0,0,0}\makebox(0,0)[lt]{\lineheight{1.25}\smash{\begin{tabular}[t]{l}$c$\end{tabular}}}}%
    \put(0,0){\includegraphics[width=\unitlength,page=3]{chain_unflipped.pdf}}%
  \end{picture}%
\endgroup%
\vspace{0.0cm} &\raisebox{1.5cm}{\begin{tikzpicture}
      [decoration=snake,
      line around/.style={decoration={pre length=#1,post length=#1}}]
      \draw[->, decorate, line around = 1pt] (0,0)--(1.5,0);
    \end{tikzpicture}} &\def\svgscale{0.65}
\begingroup%
  \makeatletter%
  \providecommand\color[2][]{%
    \errmessage{(Inkscape) Color is used for the text in Inkscape, but the package 'color.sty' is not loaded}%
    \renewcommand\color[2][]{}%
  }%
  \providecommand\transparent[1]{%
    \errmessage{(Inkscape) Transparency is used (non-zero) for the text in Inkscape, but the package 'transparent.sty' is not loaded}%
    \renewcommand\transparent[1]{}%
  }%
  \providecommand\rotatebox[2]{#2}%
  \newcommand*\fsize{\dimexpr\f@size pt\relax}%
  \newcommand*\lineheight[1]{\fontsize{\fsize}{#1\fsize}\selectfont}%
  \ifx\svgwidth\undefined%
    \setlength{\unitlength}{261.82336041bp}%
    \ifx\svgscale\undefined%
      \relax%
    \else%
      \setlength{\unitlength}{\unitlength * \real{\svgscale}}%
    \fi%
  \else%
    \setlength{\unitlength}{\svgwidth}%
  \fi%
  \global\let\svgwidth\undefined%
  \global\let\svgscale\undefined%
  \makeatother%
  \begin{picture}(1,0.77696872)%
    \lineheight{1}%
    \setlength\tabcolsep{0pt}%
    \put(0,0){\includegraphics[width=\unitlength,page=1]{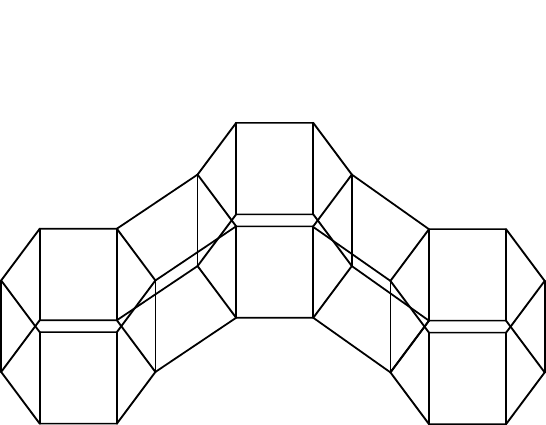}}%
    \put(0.47601108,0.57489312){\color[rgb]{0,0,0}\makebox(0,0)[lt]{\lineheight{1.25}\smash{\begin{tabular}[t]{l}$b$\end{tabular}}}}%
    \put(0,0){\includegraphics[width=\unitlength,page=2]{chain_flipped.pdf}}%
    \put(0.10535781,0.39629242){\color[rgb]{0,0,0}\makebox(0,0)[lt]{\lineheight{1.25}\smash{\begin{tabular}[t]{l}$a$\end{tabular}}}}%
    \put(0.81930907,0.37569774){\color[rgb]{0,0,0}\makebox(0,0)[lt]{\lineheight{1.25}\smash{\begin{tabular}[t]{l}$c$\end{tabular}}}}%
    \put(0,0){\includegraphics[width=\unitlength,page=3]{chain_flipped.pdf}}%
    \put(0.02176475,0.75084434){\color[rgb]{0,0,0}\makebox(0,0)[lt]{\lineheight{1.25}\smash{\begin{tabular}[t]{l}Flipping $J$-direction\\turns each $\mathbf{\varphi}_i$ to $\mathbf{\varphi}_i^{-1}$\end{tabular}}}}%
    \put(0,0){\includegraphics[width=\unitlength,page=4]{chain_flipped.pdf}}%
  \end{picture}%
\endgroup%

  \end{tabular}\vspace{0.0cm}
  \caption{Changing $\sigma(v)$ has the effect of flipping the J-factor in the boundary-matching tube of $K$ corresponding to $v$. Each successive flip $\boldvarphi_i$ effectively gets replaced with $\boldvarphi_i^{-1}$.}
  \label{changing vertex figure}
\end{figure}
So if $Z$ is an arbitrary signed facet cycle we find a sequence $Z_0,  Z_1, \ldots, Z_n = Z$ of facet cycles where $Z_0$ has vanishing signs and $Z_i$ differs from $Z_{i+1}$ by one vertex sign. A connectedness argument shows us that $Q(Z) = [K]$ for any $K$ parametrized by $Z$, thus proving that $Q$ is sincere.
\subsection{A computable Steenrod square formula}
\begin{definition}
  Let $\mathbf{c}$ be a cycle
  \begin{equation}\label{c written}
    \mathbf{c} = \sum_x \mu_x\cdot \mathcal{C}(x)\in C^l(\mathcal{X}_1(\mathscr{C});\mathbb{F}_2)
  \end{equation}
  Given a facewise boundary matching $\mathfrak{m}=(\mathfrak{b}_y, \mathfrak{s}_y)$ for $\mu:= \sum_x\mu_x\cdot x$, we define the cochain $\sq^2_{\mathfrak{m}}(\mathbf{c})\in C^{l+2}(\mathcal{X}_1(\mathscr{C});\mathbb{F}_2)$ by
  \begin{equation}\label{sq2 def}
    \langle\sq^2_\mathfrak{m}(\mathbf{c}),\mathcal{C}(z)\rangle = \sum_{\substack{\text{cycles }\\C\subset \Gamma(z,\mu)}} Q(C).
  \end{equation}
\end{definition}
\begin{theorem}\label{sq2 is Sq2}
  For a signed cubical flow category $\mathscr{C}$, suppose we have a cocycle $\mathbf{c}\in C^l(\mathcal{X}_1(\mathscr{C});\mathbb{F}_2)$ written as in (\ref{c written}), and a facewise boundary matching $\mathfrak{m}$ for $\mu$. We have the identity \normalfont $\left[\sq_\mathfrak{m}^2(\mathbf{c})\right] = \Sq^2([\mathbf{c}])$.
\end{theorem}
Our cocycle $\sq_{\mathfrak{m}}^2(\mathbf{c})$ is defined combinatorially, and by the above theorem, is a representative for $\Sq^2([\mathbf{c}])$. Therefore, we can define the operation $\Sq^2:H^l(\mathcal{X}_1(\mathscr{C});\mathbb{F}_2)\to H^{l+2}(\mathcal{X}_1(\mathscr{C});\mathbb{F}_2)$ combinatorially.
\begin{proof}
  Recall from Procedure \ref{computation idea} our method for computing $\Sq^2$. Take the corresponding cycle $\mathbf{c}' = \sum_x \mu_x\cdot \mathcal{C}'(x)\in C_{\cell}^m(Y';\mathbb{F}_2)$. We have $\Sq^2([\mathbf{c}']) = \Sq^2([c^*e^m]) = \mathfrak{c}^* \Sq^2 ([e^m]) = [\mathfrak{c}^* e^{m+2}]$, and for all $z$ with $\gr(z) = l + 2$,
  \[
    \langle \mathfrak{c}^*e^{m+2},\mathcal{C}'(z) \rangle = [\mathfrak{c}|_{\partial\mathcal{C}'(z)}] = \sum_{\substack{\text{cycles }\\K\subset \partial\mathcal{C}'(z)}} [\mathfrak{c}|_K] = \sum_{\substack{\text{cycles }\\K\subset \partial\mathcal{C}'(z)}} [K] = \sum_{C \subset \Gamma(z,\mu)} Q(C).
  \]
  Therefore, $\Sq^2([\mathbf{c}']) = \left[\sum_z \nu_z\cdot \mathcal{C}'(z)\right]$, where $\nu_z$ is defined in (\ref{sq2 def}). By the last step in Procedure \ref{computation idea}, we conclude $\Sq^2([\mathbf{c}]) = \left[\sum_z \nu_z\cdot \mathcal{C}(z)\right]$.
\end{proof}
\section{Defining $\text{Sq}^2$ on the family of signed cubical realizations $\mathcal{X}_k(\mathscr{C})$.}
Sarkar-Scaduto-Stoffregen \cite{MR4078823} have introduced a family of spaces $\mathcal{X}_k(L)$, for $k\geq 0$, noting $\mathcal{X}_0(L)$ is, by definition, the original (even) homotopy type $\mathcal{X}_{Kh}(L)$. They next call $\mathcal{X}_{o}(L):=\mathcal{X}_1(L)$ the \textit{odd Khovanov homotopy type}. Furthermore, the spaces $\mathcal{X}_{2k}(L)$ have cellular chain complexes isomorphic to $Kc_o(L)$, and the spaces $\mathcal{X}_{2k+1}(L)$ have cellular chain complexes isomorphic to $Kc(L)$. In \cite{MR4078823}, these spaces are defined as homotopy colimits of doubly, triply---and so on---signed refinements (doubly signed for $\mathcal{X}_2(L)$, triply signed for $\mathcal{X}_3(L)$). But in this paper, it is more convenient to view these spaces as \textit{$l$-signed cubical} realizations.
\begin{definition}
  Given a cubical neat embedding $\iota$ of a signed flow category $(\mathscr{C},\mathfrak{f},\sigma)$ relative to a tuple $\mathbf{d}$. We construct a CW complex $||\mathscr{C}||_k$, following the same exact construction of the signed cubical realization $||\mathscr{C}||$ (outlined in Definition \ref{cubical realization def}), but with some changes:
  \begin{enumerate}[label = (C-\arabic*), ref = (C-\arabic*)]
    \item The CW complex $||\mathscr{C}||_k$ has one cell $\mathcal{C}(x)$ for each $x\in \Ob(\mathscr{C})$. Letting $u$ denote $\mathfrak{f}(x)$, this cell is given by
	\[
	  \mathcal{C}(x)=\prod_{i=0}^{|u|-1}[-R,R]^{d_i}\times \prod_{i=|u|}^{n-1}[-\epsilon,\epsilon]^{d_i}\times J^k \times \widetilde{\mathcal{M}}_{\mathscr{C}_C(n)}(u,\overline{0}).
	\]
        Note the factor $J^k$ instead of $J$.
    \item For any $x,y\in \Ob(\mathscr{C})$ with $\mathfrak{f}(x)=u>\mathfrak{f}(y)=v$, the cubical neat embedding $\iota$ gives an embedding $\jmath_\gamma$ very similar to the signed cubical realization, but we highlight the difference:
      \[
        \jmath_\gamma: \mathcal{C}(y) \times \gamma \xhookrightarrow{\tau_k^{\sigma(\gamma)}\times\text{Id}}\mathcal{C}(y)\times \gamma\hookrightarrow \partial\mathcal{C}(x),
      \]
	where $\gamma \in \widehat{A}_{x,y}$, and $\tau_k: \mathcal{C}(y)\to \mathcal{C}(y)$ denotes the negation $(t,\ldots, t)\mapsto (-t, \ldots, -t)$ in the $J^k$-factor.
  \item The attaching maps $\partial \mathcal{C}(x)\to \mathcal{C}(y)$ are defined similarly to Definition \ref{cubical realization def}.
  \end{enumerate}
  The \textit{$k$-signed cubical realization} $\mathcal{X}_k(\mathscr{C})$ is defined to be the formal desuspension
  \[
    \mathcal{X}_k(\mathscr{C}):= \Sigma^{-(N+|\mathbf{d}|+k)}||\mathscr{C}||.
  \]
\end{definition}
Given a cocycle $\mathbf{c}\in C^{l}(\mathcal{X}_k(\mathscr{C});\mathbb{F}_2)$, we repeat a similar boundary matching argument with slightly different boundary matching tubes $\eta$, again creating similar cycles $K$ yielding classes $[K]$ with values we can compute. Now these cycles $K$ are parametrized by cycles $C$ in our special graph structure $\Gamma(z,\mu)$. The quantity $Q_k(C)$ defined below gives the value of $[K]$:``''
\begin{definition}
  Let $C\subset \Gamma(z,\mu)$ be a graph cycle $e_{a_1}\longline e_{a_2}\longline\ldots\longline e_{a_r}\longline e_{a_1}$. Choose a direction to orient $C$. We define
  \begingroup
  \allowdisplaybreaks
  \begin{align*}
    Q_{2k}(C)
    &=\sum_{a \text{---} b} ab
    +\sum_{a} a
    +\sum_{a \text{---} b}\max(a,b)
    +\sum_{a \text{---} b\text{---} c}\max(a_b,c_b)\\
    &\quad +1 + \#\{a\to \overrightarrow{b}\to c\ \vert\  a>b\} + \#\{a\to \overleftarrow{b}\to c\ \vert\  a<b\}\\
    &\quad+ k\left(\#\{{a\to \overrightarrow{b}\to c}\mid\Delta S(e_b) = 1\} - \#\{{a\to \overleftarrow{b}\to c}\mid\Delta S(e_b) = 1\}\right)/2,\\[\jot]
    Q_{2k+1}(C)
    &=\sum_{a \text{---} b} ab
    +\sum_{a} a
    +\sum_{a \text{---} b}\max(a,b)
    +\sum_{a \text{---} b\text{---} c}\max(a_b,c_b)\\
    &\quad +1 + \#\{a\to \overrightarrow{b}\to c\ \vert\  a>b\} + \#\{a\to \overleftarrow{b}\to c\ \vert\  a<b\}\\
    &\quad +\sum_{(a,\omega')\longline (b,\omega)} (a_b\omega + b_a\omega') + \sum_{(b,\omega)}\omega\\
    &\quad+ k\left(\#\{{a\to \overrightarrow{b}\to c}\mid\Delta S(e_b) = 1\} - \#\{{a\to \overleftarrow{b}\to c}\mid\Delta S(e_b) = 1\}\right)/2.
  \end{align*}
  \endgroup
  Now let $\mathbf{c}$ be a cocycle
  \begin{equation}
    \mathbf{c} = \sum_x \mu_x\cdot \mathcal{C}(x)\in C^l(\mathcal{X}_k(\mathscr{C});\mathbb{F}_2).
  \end{equation}
  For any facewise boundary matching $\mathfrak{m}=(\mathfrak{b}_y, \mathfrak{s}_y)$ for $\mu$, we define the cochain $\sq_{k,\mathfrak{m}}^2(\mathbf{c})\in C^{l+2}(\mathcal{X}_k(\mathscr{C});\mathbb{F}_2)$ exactly as before, that is,
  \begin{equation*}
    \langle\sq^2_\mathfrak{m}(\mathbf{c}),\mathcal{C}(z)\rangle = \sum_{\substack{\text{cycles }\\C\subset \Gamma(z,\mu)}} Q_k(C).
  \end{equation*}
  
\end{definition}
\begin{theorem}
  View the second Steenrod square $\Sq^2$ on $\mathcal{X}_k(\mathscr{C})$ as an operation $H^*(\mathscr{C},\mathbb{F}_2)\to H^{*+2}(\mathscr{C},\mathbb{F}_2)$. Then we have $\Sq^2([\mathbf{c}]) = [\sq_{k,\mathfrak{m}}^2(\mathbf{c})]$.
\end{theorem}
\begin{proof}
  The proof is analogous to the proof of Theorem \ref{sq2 is Sq2}.
\end{proof}

\begin{notation}
  We write $\Sq^2|_{\mathcal{X}_k(\mathscr{C})}:= \Sq^2|_{H^*(\mathcal{X}_k(\mathscr{C});\mathbb{F}_2)}$ as the second Steenrod square on the space $\mathcal{X}_k(\mathscr{C})$. We also view $\Sq^2|_{\mathcal{X}_k(\mathscr{C})}$ as an operation $H^*(\mathscr{C};\mathbb{F}_2)\to H^{*+2}(\mathscr{C};\mathbb{F}_2)$. In particular, for any link $L$ we view $\Sq^2|_{\mathcal{X}_k(L)}$ as an operation on $Kh(L;\mathbb{F}_2)$.
\end{notation}

\begin{proof}[Proof of Theorem \ref{sums to 0}]
  This is a direct consequence of our definition of $\sq_{k,\mathfrak{m}}^2(\mathbf{c})$.
\end{proof}

While $\Sq^2$ on $\mathcal{X}_{k}(L)$ does not depend on $k\ \mathrm{mod}\ 4$, we shall see by computations in Section \ref{computations} that $\Sq^2$ on $\mathcal{X}_{k}(L)$ does indeed depend on $k\ \mathrm{mod}\ 2$.
\section{A simplified formula for $\text{Sq}^2$}\label{simplified formula}
Fortunately, we can simplify our formulas for $\Sq^2$.
\begin{definition}
  For $C\subset \Gamma(z,\mu)$ a cycle, define
  \begin{align*}
    \overline{Q}_{2k}(C)
    &=\sum_{a \text{---} b} ab
    +\sum_{a} a +1 + \#\{a\to \overrightarrow{b}\to c\ \vert\  a>b\} + \#\{a\to \overleftarrow{b}\to c\ \vert\  a<b\}\\
    &\quad+ k\left(\#\{{a\to \overrightarrow{b}\to c}\mid\Delta S(e_b) = 1\} - \#\{{a\to \overleftarrow{b}\to c}\mid\Delta S(e_b) = 1\}\right)/2,\\[\jot]
    \overline{Q}_{2k+1}(C)
    &=\sum_{a \text{---} b} ab
    +\sum_{a} a +1 + \#\{a\to \overrightarrow{b}\to c\ \vert\  a>b\} + \#\{a\to \overleftarrow{b}\to c\ \vert\  a<b\}\\
    &\quad +\sum_{(a,\omega')\longline (b,\omega)} (a_b\omega + b_a\omega') + \sum_{(b,\omega)}\omega\\
    &\quad+ k\left(\#\{{a\to \overrightarrow{b}\to c}\mid\Delta S(e_b) = 1\} - \#\{{a\to \overleftarrow{b}\to c}\mid\Delta S(e_b) = 1\}\right)/2.
  \end{align*}
  We define the cochain $\overline{\sq}_{k,\mathfrak{m}}^2(\mu)$ by
  \[
    \langle \overline{\sq}_{k,\mathfrak{m}}^2(\mu),z\rangle = \sum_{C\subset \Gamma(z,\mu)}\overline{Q}_k(C)
  \]
\end{definition}
\begin{proposition}
  We have the identity $[\sq_{k,\mathfrak{m}}^2(\mu)] = [\overline{\sq}_{k,\mathfrak{m}}^2(\mu)]$.
\end{proposition}
\begin{proof}
  We measure the difference between the two cochains as
  \begin{align*}
    \langle\sq_{k,\mathfrak{m}}^2(\mu)-\overline{\sq}_{k,\mathfrak{m}}^2(\mu),z\rangle
    &= \sum_{C\subset \Gamma(z,\mu)}\Bigl(\sum_{a \text{---} b}\max(a,b)
      +\sum_{a \text{---} b\text{---} c}\max(a_b,c_b)\Bigr)\\
    &= \sum_{C\subset\Gamma(z,\mu)}\sum_{a \text{---} b}\max(a,b) + \sum_{C\subset\Gamma(z,\mu)}\sum_{a \text{---} b\text{---} c}\max(a_b,c_b)\\
    &=: \langle T_1,z\rangle + \langle T_2,z\rangle,
  \end{align*}
  and we argue that both $T_1$ and $T_2$ are coboundaries. To show that $T_1$ is a coboundary, we remark that we can view $T_1$ as $L(\mu)$, where $L$ is a linear map defined by
  \begin{align*}
  \langle L(x),z\rangle =  \sum_{\substack{\text{intervals}\\I\subset \mathcal{M}(z,x)}} m(I),
  \end{align*}
  where $m(I) := \max(a,b)$ for $\partial I = \{p\circ q, p'\circ q'\}$, $S_\mathbb{Z}(q) = a$, $S_\mathbb{Z}(q') = b$. Now $T_1$ is a coboundary because $L$ is nullhomotopic. Indeed, we can define a nullhomotopy $H :C^*_\mathcal{M}(\mathscr{C};\mathbb{F}_2) \to C^{*+1}_\mathcal{M}(\mathscr{C};\mathbb{F}_2)$ by
  \[
    \langle H(x),y\rangle = \sum_{p\in\mathcal{M}(y,x)}\binom{S_\mathbb{Z}(p)+1}{2}.
  \]
  We prove that $d H(x) + H d(x) = L(x)$ using the following diagrams:
  \[
  \begin{tabular}{m{3cm} m{1cm} m{3cm}}
    \begin{tikzpicture}[scale=1]
  \path[draw, <-, shorten <=\T,shorten >=\T, thick] (0,1) -- (1,2) node [midway, label={[label distance=-0.2cm]135:$q$}] {};
  \path[draw, <-, shorten <=\T,shorten >=\T, thick] (2,1) -- (1,2) node [midway, label={[label distance=-0.2cm]45:$q'$}] {};
  \path[draw, <-, shorten <=\T,shorten >=\T, thick] (1,0) -- (2,1) node [midway, label={[label distance=-0.35cm]-45:$p'$}] {};
  \path[draw, <-, shorten <=\T,shorten >=\T, thick] (1,0) -- (0,1) node [midway, label={[label distance=-0.3cm]225:$p$}] {};
     \node at (1, 2) {$z$};
   \node at (0, 1) {$y$};
   \node at (2, 1) {$y'$};
   \node at (1, 0) {$x$};
   \end{tikzpicture} & $\xrightarrow{\mathfrak{f}}$& \begin{tikzpicture}[scale=1]
  \path[draw, <-, shorten <=\T,shorten >=\T, thick] (0,1) -- (1,2) node [midway, label={[label distance=-0.2cm]135:$a$}] {};
  \path[draw, <-, shorten <=\T,shorten >=\T, thick] (2,1) -- (1,2) node [midway, label={[label distance=-0.2cm]45:$b$}] {};
  \path[draw, <-, shorten <=\T,shorten >=\T, thick] (1,0) -- (2,1) node [midway, label={[label distance=-0.3cm]-45:$a-1$}] {};
  \path[draw, <-, shorten <=\T,shorten >=\T, thick] (1,0) -- (0,1) node [midway, label={[label distance=-0.3cm]225:$b$}] {};
     \node at (1, 2) {$w$};
   \node at (0, 1) {$v$};
   \node at (2, 1) {$v'$};
   \node at (1, 0) {$u$};
   \end{tikzpicture}
  \end{tabular}
  \]
  where $a = S_{\mathbb{Z}}(q):= s_{\mathbb{Z}}(\mathcal{C}_{w,v})$, $b = S_{\mathbb{Z}}(q) = s_{\mathbb{Z}}(\mathcal{C}_{w,v'})$, with $a>b$. The contribution of an interval $I$ corresponding to the above diagram to $\langle dH(x) + Hd(x),z\rangle$ is
  \[
    \binom{a+1}{2} + \binom{a}{2} + \binom{b+1}{2} + \binom{b+1}{2} = a = m(I) \mod 2.
  \]
  Finally, we prove that $T_2$ is a coboundary. This is because
  \[
    \langle T_2,z\rangle = \sum_{q\in \mathcal{M}(z,y)}\sum_{(p,p')\in \mathfrak{s}_y}\max(S_{\mathbb{Z}}(p),S_{\mathbb{Z}}(p')),
  \]
  implying $T_2 = d T'$, where
  \[
    \langle T', y\rangle = \sum_{(p,p')\in \mathfrak{s}_y}\max(S_{\mathbb{Z}}(p),S_{\mathbb{Z}}(p')).\qedhere
  \]
\end{proof}

\section{A combinatorial proof that $\text{Sq}^2$ agrees with earlier formulas.}
Lipshitz-Sarkar \cite{MR3252965} derived a combinatorial formula for $\Sq^2 $ on the cubical realization of an unsigned flow category. Their formula allowed them to compute the stable homotopy types of the even Khovanov spectra $\mathcal{X}_e(L)$ of links $L$ up to 11 crossings.\par
Furthermore, Sch\"utz \cite{MR4970173} introduced combinatorially defined operations
\[
  \Sq_0^2,\ \Sq_1^2: H^*(\mathscr{C};\mathbb{F}_2)\to H^{*+2}(\mathscr{C};\mathbb{F}_2),
\]
which are both defined on signed flow categories (more precisely, Sch\"utz only requires that $\mathscr{C}$ is a \textit{signed $1$-flow category}, which is a looser requirement).\par
Sch\"utz's operations $\Sq_0^2$, $\Sq_1^2$ are generalizations of Sarkar-Lipshitz's formula for $\Sq^2$ in the sense that if $\mathscr{C}$ is a trivially signed cubical flow category (or analogously, $F:\underline{2}^n\to B_\sigma$ has all signs $+1$), $\Sq_0^2$, $\Sq_1^2$ agree with $\Sq^2$ on the level of cycles $\mu$. However, it has not been known how $\Sq_0^2$, $\Sq_1^2$ relate to $\Sq^2$ on the odd Khovanov spectrum. We give a combinatorial proof that $\Sq_0^2$, $\Sq_1^2$ do arise as honest Steenrod squares, with $\Sq_1^2$ being the Steenrod square on $\mathcal{X}_1(\mathscr{C})$ and $\Sq_1^2$. We first begin with a quick overview of $\Sq_0^2$, $\Sq_1^2$.
\begin{definition}
  The \textit{standard frame assignment} $f\in C^2(\mathcal{C}(n),\mathbb{F}_2)$ is the following $2$-cochain. If  $w = \{a_1,\ldots, a_{\kappa + 2}\}$ and $u = \{a_1,\ldots, \widehat{a_i},\ldots, \widehat{a_j},\ldots, a_{\kappa + 2}\}$, then
  \[
    f(\mathcal{C}_{w,u}) = (i-1)(j-i-1) = ij + i + j + 1\pmod 2\in\mathbb{F}_2.
  \]
  
\end{definition}
\begin{definition}
  Given a cycle $\mu\in C^l_\mathcal{M}(\mathscr{C};\mathbb{F}_2)$, we define a \textit{signwise boundary matching} $\widetilde{\mathfrak{m}}$ of $\mu$ as a collection of pairs $(\widetilde{\mathfrak{b}}_y,\widetilde{\mathfrak{s}}_y)$ where
  \begin{itemize}
  \item $\widetilde{\mathfrak{b}}_y$ is a fixed point free involution of $\mathcal{M}(y,c)$. We can also think of $\widetilde{\mathfrak{b}}_y$ as a partition of $\mathcal{M}(y,c)$ into unordered pairs of the form $\{p,p'\}$.
  \item  $\widetilde{\mathfrak{s}}_y$ is an ordering for all the pairs $\{p,p'\}$ where $S(p)=S(p')$. In other words, if the $p$ and $p'$-summands in (\ref{complex coeff}), Definition \ref{Z morse chain complex}, agree, then $\widetilde{\mathfrak{s}}_y$ orders $\{p,p'\}$
  \end{itemize}
\end{definition}

\begin{definition}
  Given a facewise boundary matching $\mathfrak{m} = (\mathfrak{b}_y,\mathfrak{s}_y)$, we define the \textit{corresponding signwise boundary matching} $\widetilde{\mathfrak{m}} = (\widetilde{\mathfrak{b}}_y,\widetilde{\mathfrak{s}}_y)$ as follows:
  \begin{itemize}
  \item The fixed-point free involutions are the same, that is $\widetilde{\mathfrak{b}}_y:=\mathfrak{b}_y$.
  \item If $\{p,p'\}\in \widetilde{\mathfrak{b}}_y$ is a matched pair with $S(p)=S(p')$, then $\widetilde{\mathfrak{s}}$ must order $\{p,p'\}$ as $(p,p')$ or $(p',p)$. We simply choose the ordering to be the ordering from $\mathfrak{s}_y$ (recall that every pair $\{p,p'\}\in \mathfrak{b}_y$ is ordered by $\mathfrak{s}_y$).
  \end{itemize}
  The ordering $\widetilde{\mathfrak{s}}_y$ does not order all pairs (compare with $\mathfrak{s}_y$). But for the pairs $\{p,p'\}$ that must be ordered, $\widetilde{\mathfrak{s}}_y$ uses the ordering of $\mathfrak{s}_y$.
\end{definition}
\begin{definition}[compare \cite{MR4970173}]\label{graph structure Schutz}
  Let $\mathscr{C}$ be a signed flow category, $\mu\in C^l(\mathcal{C};\mathbb{F}_2)$ a cocycle, and $\widetilde{\mathfrak{m}} = (\widetilde{\mathfrak{b}}_y,\widetilde{\mathfrak{s}}_y)$ a signwise boundary matching for $\mu$. Given $z\in \Ob{\mathscr{C}}$, we define another special graph structure $\widetilde{\Gamma}(z,\mu) := \widetilde{\Gamma}_{\widetilde{\mathfrak{m}}}(z,\mu)$ (see Definition \ref{special graph structure def}) as follows. (See Figure \ref{special graph structure drawing} for an illustration.) The vertex set $V$, edge set $E$, and function $S:V\to\mathbb{F}_2$ are defined the same as in Example \ref{R graph}:
  \begin{align*}
    V &:= \coprod_{\gr(y) = l+1} \mathcal{M}(y,\mu)\times \mathcal{M}(z,y)\\
    E' &:= \{e=\{(p_1,q_1), (p_2,q_2)\}\mid e=\partial I\text{ for some } x\in \mu\text{, } I\subset \mathcal{M}(z,x)\}\\
    E\backslash E' &= \{\{(p,q),(p',q)\}\mid \text{$p$ is boundary-matched with $p'$}\}\}\\
    S(p,q) &:= S(p) + S(q)\qquad
             \begin{gathered}\text{($S$ is the cubical sign assignment}\\
      \text{from Definition \ref{canonical sign assignment})}\end{gathered}
  \end{align*}
  The directed edges $E''$ are different from the directed edges in Example \ref{R graph}: we define $E'' = \{\{(p,q),(p',q)\}\in E\backslash E'\ \vert\ S(p,q)=S(p',q)\}$, and $e\in E''$ is directed from $(p,q)$ to $(p',q)$ if $(p,p')\in \widetilde{s}_y$.\par
  Additionally, we equip $\widetilde{\Gamma}(z,\mu)$ with a function $f:E'\to \mathbb{F}_2$, which we call a ``framing'' of $\mathscr{C}$. (The concept of a framing is discussed in \cite{MR4970173} and generalizes the frame assignment of \cite{MR3252965}.)\par
  Note that each graph component $C$ of $\Gamma(z,\mu)$ has an even number of directed edges. The proof is almost exactly the proof of Lemma \ref{even incoherent}.
  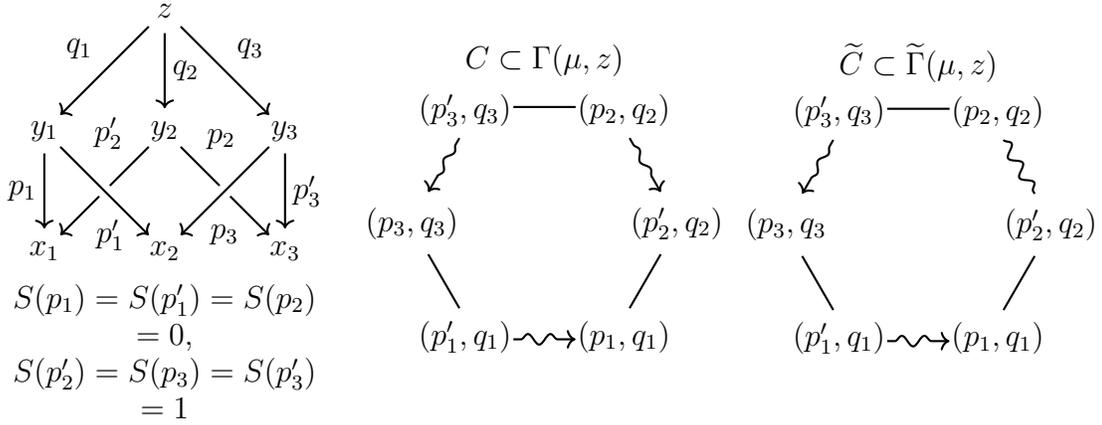
\begin{figure}
    \begin{tabular}{ m{4.4cm} m{4.7cm} m{4.5cm} }
      \makecell{\begin{tikzpicture}[scale=1.6]
  \path[draw, <-, shorten <=\T,shorten >=\T, thick] (0,1) -- (1,2) node [midway, label={[label distance=-0.2cm]135:$q_1$}] {};
  \path[draw, <-, shorten <=\T,shorten >=\T, thick] (1,1) -- (1,2) node [midway, label={[label distance=-0.2cm]0:$q_2$}] {};
  \path[draw, <-, shorten <=\T,shorten >=\T, thick] (2,1) -- (1,2) node [midway, label={[label distance=-0.2cm]45:$q_3$}] {};
  \path[draw, <-, shorten <=\T,shorten >=\T, thick] (0,0) -- (0,1) node [midway, label={[label distance=-0.2cm]180:$p_1$}] {};
  \path[, draw, <-, shorten <=\T,shorten >=\T, thick] (2,0) -- (2,1) node [midway, label={[label distance=-0.2cm]0:$p_3'$}] {};
  \path[draw, <-, shorten <=\T,shorten >=\T, thick] (1,0) -- (2,1) node [pos=0.3, label={[label distance=-0.25cm]-45:$p_3$}] {};
  \path[draw, <-, shorten <=\T,shorten >=\T, thick] (1,0) -- (0,1) node [pos=0.3, label={[label distance=-0.35cm]225:$p_1'$}] {};
  \path[draw, <-, shorten <=\T,shorten >=\P, thick] (0,0) -- (0.5,0.5);
  \path[draw, <-, shorten <=\T,shorten >=\P, thick] (2,0) -- (1.5,0.5);
  \draw[shorten <=\T,shorten >=\P, thick] (1,1) -- (0.5,0.5) node [midway, label={[label distance=-0.2cm]135:$p_2'$}] {};
  \draw[shorten <=\T,shorten >=\P, thick] (1,1) -- (1.5,0.5) node [midway, label={[label distance=-0.2cm]45:$p_2$}] {};
   \node at (1, 2) {$z$};
   \node at (0, 1) {$y_1$};
   \node at (1, 1) {$y_2$};
   \node at (2, 1) {$y_3$};
   \node at (0, 0) {$x_1$};
   \node at (1, 0) {$x_2$};
   \node at (2, 0) {$x_3$};
 \end{tikzpicture}\\ $S(p_1) = S(p_1')= S(p_2)$ \\$ = 0$,\\ $S(p'_2) = S(p_3) = S(p_3') $\\ $= 1$}&
 \begin{tikzpicture}[scale = 0.8]
  \path[thick, draw,->,decorate,
  decoration={snake,amplitude=.7mm,segment length=3mm,post length=1mm, pre length = 1mm}] (255:2)--(285:2);
   \path[thick, draw,->,decorate,
  decoration={snake,amplitude=.4mm,segment length=3mm,post length=1mm, pre length = 1mm}] (45:2) -- (15:2);
  \path[thick, draw,->,decorate,
  decoration={snake,amplitude=.4mm,segment length=3mm,post length=1mm, pre length = 1mm}] (135:2)--(165:2);
  \draw[thick] (75:2)--(105:2);
  \draw[thick] (195:2)--(225:2);
  \draw[thick] (315:2)--(345:2);
  \node at (90: 2.7) {$C\subset \Gamma(\mu,z)$};
  \node at (0:2.2) {$(p_2',q_2)$};
  \node at (55:2.3) {$(p_2,q_2)$};
  \node at (125:2.3) {$(p_3',q_3)$};
  \node at (180:2.2) {$(p_3,q_3)$};
  \node at (235:2.3) {$(p_1',q_1)$};
  \node at (305:2.3) {$(p_1,q_1)$};
\end{tikzpicture} &
\begin{tikzpicture}[scale = 0.8]
  \path[thick, draw,->,decorate,
  decoration={snake,amplitude=.7mm,segment length=3mm,post length=1mm, pre length = 1mm}] (255:2)--(285:2);
   \path[thick, draw, decorate,
  decoration={snake,amplitude=.7mm,segment length=3mm}](45:2) -- (15:2);
  \path[thick, draw,->,decorate,
  decoration={snake,amplitude=.4mm,segment length=3mm,post length=1mm, pre length = 1mm}] (135:2)--(165:2);
  \draw[thick] (75:2)--(105:2);
  \draw[thick] (195:2)--(225:2);
  \draw[thick] (315:2)--(345:2);
  \node at (90: 2.7) {$\widetilde{C}\subset \widetilde{\Gamma}(\mu,z)$};
  \node at (0:2.2) {$(p_2',q_2)$};
  \node at (55:2.3) {$(p_2,q_2)$};
  \node at (125:2.3) {$(p_3',q_3)$};
  \node at (180:2.2) {$(p_3,q_3$};
  \node at (235:2.3) {$(p_1',q_1)$};
  \node at (305:2.3) {$(p_1,q_1)$};
\end{tikzpicture}
    \end{tabular}
    \caption{Left: A subset of the chains $z\xrightarrow{q} y\xrightarrow{p} x$ that form the cycle $C\subset \Gamma_{\mathfrak{m}}(\mu, z)$ from Figure \ref{graph counting incoherent}. Middle, the cycle $C$. Right: The corresponding cycle $\widetilde{C}\subset \widetilde{\Gamma}_{\widetilde{\mathfrak{m}}}(\mu, z)$. Note that all the oriented edges in $\widetilde{C}$ must have orientation in the same direction as the corresponding edge in $C$. However, note that $\{(p_2,q_2),(p'_2,q_2)\}$ is unoriented, since $S(p_2)\neq S(p'_2)$.}
    \label{special graph structure drawing}
  \end{figure}
\end{definition}
\begin{remark}
  The cubical special graph structure $\widetilde{\Gamma}(z,\mu)$ has the exact same edge set $E$ as $\Gamma(z,\mu)$, but the difference arises with the directed edge set $E''$. In $\widetilde{\Gamma}(z,\mu)$, only the edges $e=\{v,v'\}$ where $S(v)=S(v')$ are oriented, but the orientation of these edges must agree with their orientation in $\Gamma(z,\mu)$.
\end{remark}

\begin{definition}[see \cite{MR4970173}]
  Let $\mathscr{C}$, $\mu$, and $(\widetilde{b}_y,\widetilde{s}_y)$ be as in Definition \ref{graph structure Schutz}, and let $C$ be a cycle in $\widetilde{\Gamma}(z,\mu)$. We define $F(C)\in \mathbb{F}_2$ to be the sum of the framing values $f(e')$ for $e'\in E'$ in $C$. Also let $D(C) \in \mathbb{F}_2$ denote the number of oriented edges in $C$ that point in a given direction. This quantity $D(C)$ is well-defined given the fact that there are an even number of edges $E''$ in $C$.
\end{definition}
\begin{definition}[see \cite{MR4970173}]
  \begin{align*}
    &f_\epsilon \left(\vcenter{\hbox{\begin{tikzpicture}[scale=0.7]
  \path[draw, <-, shorten <=\T,shorten >=\T] (0,1) -- (1,2) node [midway, label={[label distance=-0.2cm]135:$a'$}] {};
  \path[draw, <-, shorten <=\T,shorten >=\T] (2,1) -- (1,2) node [midway, label={[label distance=-0.2cm]45:$b'$}] {};
  \path[draw, <-, shorten <=\T,shorten >=\T] (1,0) -- (2,1) node [midway, label={[label distance=-0.3cm]-45:$d'$}] {};
  \path[draw, <-, shorten <=\T,shorten >=\T] (1,0) -- (0,1) node [midway, label={[label distance=-0.20cm]225:$c'$}] {};
     \node at (1, 2) {$z$};
   \node at (0, 1) {$y$};
   \node at (2, 1) {$y'$};
   \node at (1, 0) {$w$};
 \end{tikzpicture}}}\right)
      - f\left(\vcenter{\hbox{\begin{tikzpicture}[scale=0.7]
  \path[draw, <-, shorten <=\T,shorten >=\T] (0,1) -- (1,2) node [midway, label={[label distance=-0.2cm]135:$a$}] {};
  \path[draw, <-, shorten <=\T,shorten >=\T] (2,1) -- (1,2) node [midway, label={[label distance=-0.2cm]45:$b$}] {};
  \path[draw, <-, shorten <=\T,shorten >=\T] (1,0) -- (2,1) node [midway, label={[label distance=-0.3cm]-45:$d$}] {};
  \path[draw, <-, shorten <=\T,shorten >=\T] (1,0) -- (0,1) node [midway, label={[label distance=-0.20cm]225:$c$}] {};
     \node at (1, 2) {$w$};
   \node at (0, 1) {$v$};
   \node at (2, 1) {$v'$};
   \node at (1, 0) {$u$};
 \end{tikzpicture}}} \right)
    =
    \begin{cases*}
      0 & if $a=a'$, $b=b'$, $c=c'$, $d=d'$, \\
      1 & if $a= a'$, $b= b'$, $c\neq c'$, $d\neq d'$\\
      a+b & if $a\neq a'$, $b\neq b'$, $c= c'$, $d= d'$\\
      c & if $a\neq a'$, $b=b'$, $c\neq c'$, $d= d'$\\
      d & if $a=a'$, $b\neq b'$, $c\neq c'$, $d=d'$\\
      \epsilon + d & if $a= a'$, $b\neq b'$, $c\neq c'$, $d= d'$\\
      \epsilon + c & if $a\neq a'$, $b= b'$, $c=c'$, $d\neq d'$\\
      c + d & if $a\neq a'$, $b\neq b'$, $c\neq c'$, $d\neq d'$
    \end{cases*}
  \end{align*}
  Here, $a,b,c,d$ denote the sign assignment values of $s$, and $a',b',c',d'$ denote the sign assignment values of $S$. In terms of our sign map $\sigma$, we write
  \begin{align}\label{frame difference}
    &f_\epsilon \left(\vcenter{\hbox{\begin{tikzpicture}[scale=0.7]
  \path[draw, <-, shorten <=\T,shorten >=\T] (0,1) -- (1,2) node [midway, label={[label distance=-0.2cm]135:$q$}] {};
  \path[draw, <-, shorten <=\T,shorten >=\T] (2,1) -- (1,2) node [midway, label={[label distance=-0.2cm]45:$q'$}] {};
  \path[draw, <-, shorten <=\T,shorten >=\T] (1,0) -- (2,1) node [midway, label={[label distance=-0.3cm]-45:$p'$}] {};
  \path[draw, <-, shorten <=\T,shorten >=\T] (1,0) -- (0,1) node [midway, label={[label distance=-0.20cm]225:$p$}] {};
     \node at (1, 2) {$z$};
   \node at (0, 1) {$y$};
   \node at (2, 1) {$y'$};
   \node at (1, 0) {$w$};
 \end{tikzpicture}}}\right)
      - f\left(\vcenter{\hbox{\begin{tikzpicture}[scale=0.7]
  \path[draw, <-, shorten <=\T,shorten >=\T] (0,1) -- (1,2) node [midway, label={[label distance=-0.2cm]135:$a$}] {};
  \path[draw, <-, shorten <=\T,shorten >=\T] (2,1) -- (1,2) node [midway, label={[label distance=-0.2cm]45:$b$}] {};
  \path[draw, <-, shorten <=\T,shorten >=\T] (1,0) -- (2,1) node [midway, label={[label distance=-0.3cm]-45:$a_b$}] {};
  \path[draw, <-, shorten <=\T,shorten >=\T] (1,0) -- (0,1) node [midway, label={[label distance=-0.4cm]225:$b_a$}] {};
     \node at (1, 2) {$w$};
   \node at (0, 1) {$v$};
   \node at (2, 1) {$v'$};
   \node at (1, 0) {$u$};
 \end{tikzpicture}}} \right)
    =
    \begin{cases*}
      0 & if $\sigma(p)=\sigma(p')=\sigma(q)=\sigma(q')=0$, \\
      1 & if  $\sigma(q)=\sigma(q')=0$, $\sigma(p)=\sigma(p')=1$\\
      a+b & if $\sigma(p)=\sigma(p')=0$, $\sigma(q)=\sigma(q')=1$\\
      b_a & if $\sigma(p')=\sigma(q')=0$, $\sigma(p)=\sigma(q)=1$\\
      a_b & if $\sigma(p')=\sigma(q)=0$, $\sigma(p)=\sigma(q')=1$\\
      \epsilon + a_b & if $\sigma(p')=\sigma(q)=0$, $\sigma(p)=\sigma(q')=1$\\
      \epsilon + b_a & if $\sigma(p)=\sigma(q')=0$, $\sigma(p')=\sigma(q)=1$\\
      a + b & if $\sigma(p)=\sigma(p')=\sigma(q)=\sigma(q')=1$
    \end{cases*}
  \end{align}
\end{definition}
\begin{definition}[\cite{MR4970173}]\label{Sq2 epsilon def}
  Let $\mathscr{C}$ be a signed cubical flow category with standard sign assignment $S$ and frame assignment $f_\epsilon$. Given a cycle $\mu\in C^l(\mathscr{C};\mathbb{F}_2)$ and boundary matching $\widetilde{\mathfrak{m}}$, we define $\sq_{\epsilon,\widetilde{\mathfrak{m}}}^2(\mu)\in C^{l+2}(\mathscr{C};\mathbb{F}_2)$ by
  \[
    \langle\sq_{\epsilon,\widetilde{\mathfrak{m}}}^2(\mu), z\rangle = \sum_{C\in \Gamma(z,a)} 1 + F(C) + D(C).
  \]
  It is proved in \cite{MR4165986} that $\sq_{\epsilon,\widetilde{\mathfrak{m}}}^2(\mu)$ is a cocycle, and in fact, only differs by a coboundary if we change $\widetilde{\mathfrak{m}}$. Therefore, we have a well-defined operation $\Sq_\epsilon^2: H^l(\mathscr{C};\mathbb{F}_2)\to H^{l+2}(\mathscr{C};\mathbb{F}_2)$
\end{definition}
\begin{lemma}
  \begin{equation}\label{frame simplified}
    f_1 \left(\vcenter{\hbox{\begin{tikzpicture}[scale=0.7]
  \path[draw, <-, shorten <=\T,shorten >=\T] (0,1) -- (1,2) node [midway, label={[label distance=-0.2cm]135:$q$}] {};
  \path[draw, <-, shorten <=\T,shorten >=\T] (2,1) -- (1,2) node [midway, label={[label distance=-0.2cm]45:$q'$}] {};
  \path[draw, <-, shorten <=\T,shorten >=\T] (1,0) -- (2,1) node [midway, label={[label distance=-0.3cm]-45:$p'$}] {};
  \path[draw, <-, shorten <=\T,shorten >=\T] (1,0) -- (0,1) node [midway, label={[label distance=-0.20cm]225:$p$}] {};
     \node at (1, 2) {$z$};
   \node at (0, 1) {$y$};
   \node at (2, 1) {$y'$};
   \node at (1, 0) {$w$};
 \end{tikzpicture}}}\right) = \sigma(p\circ q) + b_a\sigma(q) + a_b\sigma(q') \mod 2,
  \end{equation}
  so in the case where we are using the face assignment $f_1$,
  \begin{equation}\label{F(C) formula}
    F(C) = \sum_{e_a\overset{e}{\text{---}} e_b}\left( ab + \sigma(e) + b_a\sigma_2(e_a) + a_b\sigma_2(e_b)\right).
  \end{equation}
  
\end{lemma}
\begin{proof}
  When $\epsilon = 1$, the difference of frame assignments (\ref{frame difference}) is equal to
  \[
    \sigma(p\circ q) + b_a\sigma_2(e_a)+a_b\sigma_2(e_b).
  \]
  Combining with the identity
  \begin{equation*}
    f\left(\vcenter{\hbox{\begin{tikzpicture}[scale=0.7]
  \path[draw, <-, shorten <=\T,shorten >=\T] (0,1) -- (1,2) node [midway, label={[label distance=-0.2cm]135:$a$}] {};
  \path[draw, <-, shorten <=\T,shorten >=\T] (2,1) -- (1,2) node [midway, label={[label distance=-0.2cm]45:$b$}] {};
  \path[draw, <-, shorten <=\T,shorten >=\T] (1,0) -- (2,1) node [midway, label={[label distance=-0.3cm]-45:$d$}] {};
  \path[draw, <-, shorten <=\T,shorten >=\T] (1,0) -- (0,1) node [midway, label={[label distance=-0.20cm]225:$c$}] {};
     \node at (1, 2) {$w$};
   \node at (0, 1) {$v$};
   \node at (2, 1) {$v'$};
   \node at (1, 0) {$u$};
 \end{tikzpicture}}} \right) =
    \begin{cases*}
      a(b-a-1) & if $a<b$\\
      b(a-b-1) & if $b<a$
    \end{cases*} = ab \mod 2,
  \end{equation*}
  and summing over all edges $e_a\xline{e}e_b$ we recover the equation from our lemma.
\end{proof}
\begin{lemma}
  \begin{equation}\label{D(C) formula}
  \begin{aligned}
    D(\widetilde{C}) &= \sum_{a\to \overrightarrow{b}\to c}(a_b+c_b + 1 + \Delta\sigma(e_b))
  \end{aligned}
  \end{equation}
  
\end{lemma}
\begin{proof}
  Observe the following equalities
  \begin{align*}
    D(\widetilde{C}) &= \#\left\{a\to \overrightarrow{b}\to c\ \middle| \ \Delta S(e_b) = 1 \right\}\\
    &= \sum_{a\to \overrightarrow{b}\to c}(\Delta S(e_b)+1)= \sum_{a\to \overrightarrow{b}\to c}\left(a_b+c_b+1+\Delta\sigma(e_b)\right).\qedhere
  \end{align*}
\end{proof}

\begin{theorem}\label{Sq2 equal}
  Fix a signed flow category $\mathscr{C}$ and fix $l\geq 0$, $l\equiv 1 \mod 4$. If we view $\Sq^2|_{\mathcal{X}_l(\mathscr{C})}$ as an operation $H^*(\mathscr{C};\mathbb{F}_2) \to H^{*+2}(\mathscr{C};\mathbb{F}_2)$, then we have $\Sq^2 = \Sq_1^2$, where $\Sq_1^2$ is defined in \cite{MR4970173}. In particular, $\Sq_1^2$ agrees with the Steenrod square on the odd Khovanov homotopy types $\mathcal{X}_l(L)$, $l \equiv 1 \mod 4$.
 \end{theorem}
 We remark that if $\mathscr{C}$ is an \textit{unsigned} cubical flow category, then $\Sq^2_0=\Sq^2_1$, and both formulas agree on the nose with the Lipshitz-Sarkar formula for $\Sq^2$. Therefore, the above theorem provides a combinatorial proof that our even second Steenrod square $\Sq^2: H^*(\mathcal{X}_e(L);\mathbb{F}_2)\to H^{*+2}(\mathcal{X}_e(L);\mathbb{F}_2)$ agrees with Lipshitz-Sarkar's $\Sq^2$ formula on even $Kh$.

\begin{proof}
  Let $\mu\in C^l_\mathcal{M}(\mathscr{C};\mathbb{F}_2)$ be a cocycle, let $\mathfrak{m}$ be a facewise boundary matching for $\mu$, and let $\widetilde{\mathfrak{m}}$ be the corresponding signwise boundary matching. For a fixed $z$ of grading $\gr(z) = l+2$, we recall the $z$-coefficients
  \[
    \langle \overline{\sq}_{1,\mathfrak{m}}^2(\mu),z\rangle = \sum_{C\in \Gamma(z,\mu)} \overline{Q}_1(C)\qquad \langle \sq_{1,\widetilde{\mathfrak{m}}}^2(\mu),z\rangle = \sum_{C\in \widetilde{\Gamma}(z,\mu)} \widetilde{Q}(C)
  \]
  where $\widetilde{Q}(C) = 1+F(C)+D(C)$. Fix a cycle $C\in \Gamma(z,\mu)$ and its corresponding cycle $\widetilde{C}\in \widetilde{\Gamma}(z,\mu)$; then subtract $\widetilde{Q}(\widetilde{C})$ from $\overline{Q}(C)$ $\pmod 2$ to obtain
\begingroup
\allowdisplaybreaks
\begin{align*}
  &\overline{Q}_1(C)-\widetilde{Q}(\widetilde{C})\\
  &=\left(\begin{aligned}
    &&&\sum_{a \text{---} b} ab
    +\sum_{a} a
    +\sum_{a \text{---} b\text{---} c}\max(a_b,c_b)\\
    &&&+1+\#\{a\to \overrightarrow{b}\to c\ \vert\  a>b\} + \#\{a\to \overleftarrow{b}\to c\ \vert\  a<b\}\\
    &&& + \sum_{a \to b} (a_b\sigma_2(e_b) + b_a\sigma_2(e_a)) + \sum_{b}\sigma_2(e_b)
    \end{aligned}\right)\\
  &-\left(1 + \sum_{a \xrightarrow{e} b}\left( ab + \sigma(e) + b_a\sigma_2(e_a) + a_b\sigma_2(e_b)\right)+ \sum_{a\to \overrightarrow{b}\to c}(a_b+c_b + 1 + \Delta\sigma(e_b))\right)\\
  &= \sum_{a \text{---} b} ab
    +\sum_{a} a
    +\sum_{a \text{---} b\text{---} c}\max(a_b,c_b)+\#\{a\to \overrightarrow{b}\to c\ \vert\  a>b\} + \#\{a\to \overleftarrow{b}\to c\ \vert\  a<b\}\\
  & \quad+ \sum_{a \to b} (a_b\sigma_2(e_b) + b_a\sigma_2(e_a)) + \sum_{b}\sigma_2(e_b)\\
  &\quad+\sum_{a \xrightarrow{e} b}\left( ab + \sigma(e) + b_a\sigma_2(e_a) + a_b\sigma_2(e_b)\right)+ \sum_{a\to \overrightarrow{b}\to c}(a_b+c_b + 1 + \Delta\sigma(e_b))\\
  &=\sum_{a} a+\sum_{a \text{---} b\text{---} c}a_b+\#\{a\to \overrightarrow{b}\to c\ \vert\  a>b\} + \#\{a\to \overleftarrow{b}\to c\ \vert\  a<b\}+ \sum_{b}\sigma_2(e_b)\\
  &\quad+\sum_{a \overset{e}{\text{---}}  b}\sigma(e)+ \sum_{a\to \overrightarrow{b}\to c}(1 + \Delta\sigma(e_b))\\
  &=\Bigl(\sum_{a} a+\sum_{a\rightarrow b} a_b+\#\{a\to b\ \vert\ a<b\}\Bigr) + \Bigl(\sum_{b}\sigma_2(e_b) + \sum_{a \overset{e}{\text{---}} b}\sigma(e)+ \sum_{a\to \overrightarrow{b}\to c}(\Delta\sigma(e_b))\Bigr)\\
  &= \sum_a 1 + \sum_a \sigma(t(e_a)),
\end{align*}
\endgroup
where $t(e_a)$ denotes $p_1$ if $e_a$ is the edge $(p_0,q)\longline(p_1,q)$, and $\mathfrak{m}$ orders $p_0<p_1$.
\[
  \langle\overline{\sq}_{1,\mathfrak{m}}^2-\sq_{1,\widetilde{\mathfrak{m}}}^2(\mu)(\mu),z\rangle =  \Bigl(\sum_{\substack{\text{intervals}\\I\subset \mathcal{M}(z,\mu)}} 1\Bigr) + \Bigl(\sum_{q\in \mathcal{M}(z,y)}\sum_{(p_0,p_1)\in \mathfrak{s}_y}\sigma(p_1)\Bigr).
\]
Now the first term can be seen as a coboundary from a similar nullhomotopy argument as in Section \ref{simplified formula}. The second term can also be seen as a coboundary, by looking at the interior sum.
\end{proof}
Recall the operation $\Sq_0^2$ outlined in Definition \ref{Sq2 epsilon def}. It is a similar combinatorial exercise to prove the following:
\begin{theorem}\label{Sq0 agrees}
  Fix $l\geq 0$, $l\equiv 3 \mod 4$. If we view $\Sq^2|_{\mathcal{X}_l(L)}$ as an operation on $Kh_o(L;\mathbb{F}_2)$, then we have $\Sq^2|_{\mathcal{X}_l(L)} = \Sq_0^2$.
\end{theorem}
\begin{proof}
  It might be easier to use the fact proved in Theorem \ref{Sq2 equal} that $\Sq^2|_{\mathcal{X}_1(L)}  = \Sq_1^2$. Indeed, we are left with comparing the cancelled out terms in $\Sq^2|_{\mathcal{X}_1(L)} - \Sq^2|_{\mathcal{X}_3(L)}$ and $\Sq_0^2-\Sq_1^2$.
\end{proof}
\begin{proof}[Proof of Theorem \ref{agrees with Schutz}]
  The theorem follows from Theorems \ref{Sq2 equal}, \ref{Sq0 agrees}.
\end{proof}

\section{Khovanov homotopy types for width three knots}\label{computations}

Fix $l\geq 0$ and let $\bullet= l\ \mathrm{mod}\ 2$ (we interpret $e = 0\ \mathrm{mod}\ 2$, $o = 1\ \mathrm{mod}\ 2$).
\begin{definition}[see \cite{MR3252965}]
  For any link $L$, we define a function $St_l= St_l(L) :\mathbb{Z}^2\to \mathbb{N}^4$ is defined as follows: Fix gradings $(i,j)\in \mathbb{Z}^2$, let $k\in \{i,i+1\}$, and let $\Sq_{(k)}^1$ denote the map $\Sq^1: Kh_\bullet^{k,j}(L)\to Kh_\bullet^{k+1,j}(L)$. Now let $\Sq^2$ denote the map $\Sq^2|_{\mathcal{X}_l(L)}: Kh_\bullet^{*,j}(L)\to Kh_\bullet^{*+2,j}(L)$.\par
  Let $r_1$ be the rank of $\Sq_l^2: Kh_\bullet^{i,j}(L)\to Kh_\bullet^{i+2,j}(L)$, and let $r_2 = \rank(\Sq^2|_{\ker \Sq^1_{(i)}})$, $r_3 = \dim \left( \im \Sq^1_{(i+1)}\cap \im \Sq^2\right)$, and $r_4 = \dim \left(\im \Sq^1_{(i+1)}\cup\im \bigl(\Sq^2|_{\ker \Sq^1_{(i)}}\bigr)\right)$. We define 
  \[
    St_l(i,j):= (r_2-r_4,r_1-r_2-r_3+r_4,r_4,r_3-r_4).
  \]
\end{definition}
\begin{proposition}[compare \cite{MR3252965}]\label{classification}
  Suppose the Khovanov homology $Kh_\bullet(L)$ of a link $L$ satisfies the following properties:
  \begin{enumerate}[label = (\arabic*), ref = (\arabic*)]
  \item $Kh^{i,j}_\bullet(L)$ lies on three adjacent diagonals, say, $2i-j=\sigma,\sigma+2,\sigma+4$.\label{3 diag}
  \item $Kh^{i,j}_\bullet(L)$ is a product of copies of $\mathbb{Z}$, $\mathbb{Z}/3$, and $\mathbb{Z}/2$.\label{torsion}
  \item There is no torsion on the diagonal $2i-j=\sigma$\label{lowest diag}.
  \end{enumerate}
  Then the stable homotopy types of the spectra $\mathcal{X}^j_l(L)$ are determined by $Kh_\bullet(L)$ and $St_l(L)$ as follows: Fix a $q$-grading $j\in \mathbb{Z}$, let $i=(\sigma+j)/2$, and let $St_l(i,j)=(x_1,x_2,x_3,x_4)$; the Khovanov spectrum $\mathcal{X}_l^j(L)$ is stably homotopy equivalent to
  \[
    \Bigl(\bigvee^{x_1}\Sigma^{i-2}\mathbb{C}P^2\Bigr)\vee \Bigl(\bigvee^{x_2}\Sigma^{i-3}\mathbb{R}P^5/\mathbb{R}P^2\Bigr)\vee \Bigl(\bigvee^{x_3}\Sigma^{i-2}\mathbb{R}P^4/\mathbb{R}P^1\Bigr)\vee \Bigl(\bigvee^{x_4}\Sigma^{i-2}\mathbb{R}P^2\wedge\mathbb{R}P^2\Bigr)
  \]
  and a wedge of Moore spaces. Furthermore, such a wedge decomposition into these factors is unique. In particular, $\mathcal{X}_l^j$ is a wedge sum of Moore spaces if and only if $x_1=x_2=x_3=x_4=0$. 
\end{proposition}
In the absence of Criterion \ref{torsion}, we have a slight weakening of Proposition \ref{classification}, which uses \cite[Theorems 11.2, 11.7]{MR1361886} to decompose $X_l(L)$ into elementary Chang complexes.
\begin{proposition}[\cite{MR1361886}]\label{more torsion}
  Suppose the Khovanov homology $Kh_\bullet(L)$ of a link $L$ satisfies Criteria \ref{3 diag} and \ref{lowest diag}. Then we can decompose $\mathcal{X}^j_l(L)$ into a wedge of a collection of Moore spaces and a collection of Chang complexes of the form $X(\eta)$, $X(\eta q)$, $X(_p\eta)$, $X(_p\eta q)$, where $p$ and $q$ are powers of $2$. Furthermore, this decomposition is unique.
\end{proposition}
\section{Computations}
Let $\mathcal{X}_l(L)$ be a Khovanov spectrum, where $l\geq 0$ and $L$ a link, and let $\bullet = l\ \mathrm{mod}\ 2$. It can be checked from the databases \cite{knotinfo} that, with the exception of the case where both $l$ is odd and $K \in \{K11n19, m(K11n19)\}$, $\mathcal{X}_l(K)$ satisfies the conditions of Proposition \ref{classification} for all knots $K$ of 11 crossings or fewer. Therefore, the homotopy types of the rest of the spaces $\mathcal{X}_l(L)$ are determined by the (integral) Khovanov homology $Kh_\bullet$ and the function $St_l$. 
\begin{remark}
  Indeed, the Khovanov homologies $Kh_o(K_0)$ and $Kh_o(m(K_0))$ each have a copy of $\mathbb{Z}/4$, violating Condition \ref{torsion}. However we can use Proposition \ref{more torsion} to conclude that $\mathcal{X}_1(K_0)\cong \mathcal{X}_3(m(K_0))$ is a wedge of Moore spaces together with $\Sigma^{-3}(\mathbb{R}P^5/\mathbb{R}P^2)$, and $\mathcal{X}(m(K_0))$ is a wedge of Moore spaces together with an elementary Chang complex $_4\eta 2$.
\end{remark}
 We present the $St_1$, $St_3$ values for prime knots $K$ of 11 or fewer crossings in Table \ref{St table} (for the $St_0$ values, see \cite{MR3252965}). Interestingly, $St_2(L)$ is trivial for all prime knots and links up to 11 crossings, so we do not list the $St_2$ values. To save space, we only include the knots $K$ for which $St_l(K)$ is not identically $(0,0,0,0)$ for all $l$, and for these knots, we only list the tuples $(i,j)$ for which $St_l(i,j)\neq (0,0,0,0)$. We refer to \cite{knotinfo} to check the even and odd Khovanov homologies $Kh_e(K)$, $Kh_o(K)$.\par
 We collect the data for the MorseLink presentations of \cite{knotatlas} and use several Python programs to carry out the computations. All the programs and computations are available at \url{https://github.com/charuvinda/KhovanovSteenrod}.
 We summarize some results obtained by our computations of $St_l$.
 \begin{proof}[Proof of Theorem \ref{depends more than mod 2}]
   We find that $St_2(T_{3,4})$ is identically $(0,0,0,0)$, and so by Proposition \ref{classification}, we see that $\mathcal{X}_{2}(T_{3,4})$ is a wedge sum of Moore spaces. However, $\mathcal{X}_0(T_{3,4})$ is not a wedge sum of Moore spaces by \cite[Theorem 1]{MR3252965}. For the second statement, we note that $St_1(T_{3,4})$ is identically $(0,0,0,0)$, while $St_3(T_{3,4})(2,11) = (1,0,0,0)$.
 \end{proof}
\begin{proof}[Proof of Theorem \ref{not dual}]
  $\mathcal{X}_1(T_{3,-4})^\vee$ has nontrivial $\Sq^2$ (since $\mathcal{X}_1(T_{3,-4})$ does), but $\mathcal{X}_1(m(T_{3,-4})) = \mathcal{X}_1(8_{19})$ has trivial $\Sq^2$ (see Table \ref{St table}). Furthermore, $\mathcal{X}_3(T_{3,4})^\vee$ has nontrivial $\Sq^2$ (since $\mathcal{X}_3(T_{3,4})$ does), but $\mathcal{X}_3(m(T_{3,-4})) = \mathcal{X}_1(8_{19})$ has trivial $\Sq^2$.
\end{proof}
\begin{proof}[Proof of Theorem \ref{not smash product}]
  Fix an $l$ as in the theorem. Observe that for degree reasons and homological reasons ($Kh_o(T_{2,3})$ has no torsion and is homologically thin), $\mathcal{X}_l(T_{2,3})$ is a wedge of spheres $S^n$, implying $\mathcal{X}_l(T_{2,3})\wedge \mathcal{X}_l(T_{2,3})$ is also a wedge of sphers. However, $St_l(T_{2,3}\sqcup T_{2,3})$  maps $(4,14) \mapsto (1,0,0,0)$ for all $l$, meaning $\mathcal{X}_l(T_{2,3}\sqcup T_{2,3})$ is never a wedge of spheres.
\end{proof}
\begin{proof}[Proof of Theorem \ref{relation with sum of reduced}]
  $\mathcal{X}_1^{-11}(T_{3,-4})$ has nontrivial $\Sq^2$, but for degree reasons, neither $\widetilde{\mathcal{X}}^{-12}_1(T_{3,-4})$ nor $\widetilde{\mathcal{X}}^{-10}_1(T_{3,-4})$ has nontrivial $\Sq^2$.
\end{proof}
\begin{proof}[Proof of Theorem \ref{mutant}]
  For all odd $k$, $\mathcal{X}_k(T_{2,3}\sqcup T_{2,3})$ has a nontrivial second Steenrod square, while $\mathcal{X}_k(T_{2,3}\# T_{2,3}\sqcup U)$ has no nontrivial second Steenrod squares.
\end{proof}
We include some a question that arose from our $St_l$ computations:
\begin{question}
  Does there exist a prime knot or link $L$ for which $\mathcal{X}_l^j(L)$ contains $\Sigma^m\mathbb{C}P^2$ in some wedge sum decomposition, for some $l,j,m$?
\end{question}
A point related to this question is that for $L = T_{2,3}\sqcup T_{2,3}$, we have that $\mathcal{X}_1^j(L)$ contains a copy of $\Sigma^{2}\mathbb{C}P^2$ in its wedge sum decomposition, so the question has already been answered in the affirmative for arbitrary links.
\begin{scriptsize}
\begin{longtable}[l]{l L{7.0cm}L{7.0cm}}
  \caption{}\label{St table}\\
  \toprule
  \large{$L$} & \large{$St_1(L)$} & \large{$St_3(L)$}\\
  \midrule
  \endfirsthead
  $8_{19}$ &   & $(2, 11)\mapsto(0, 0, 1, 0)$\\
  $9_{42}$ &   & $(-2, -1)\mapsto(0, 0, 1, 0)$\\
  $10_{124}$ &  $(5, 19)\mapsto(0, 1, 0, 0)$ & $(2, 13)\mapsto(0, 0, 1, 0)$\\
  $10_{128}$ &   & $(2, 11)\mapsto(0, 0, 1, 0)$\\
  $10_{132}$ &  $(-4, -7)\mapsto(0, 0, 0, 1)$ & $(-4, -7)\mapsto(0, 0, 1, 0),\ (-5, -9)\mapsto(0, 0, 1, 0),\ (-2, -3)\mapsto(0, 0, 1, 0)$\\
  $10_{136}$ &   & $(-2, -1)\mapsto(0, 0, 1, 0)$\\
  $10_{139}$ &  $(5, 19)\mapsto(0, 1, 0, 0)$ & $(2, 13)\mapsto(0, 0, 1, 0)$\\
  $10_{145}$ &  $(-6, -13)\mapsto(0, 0, 0, 1),\ (-4, -9)\mapsto(0, 1, 0, 0)$ & $(-6, -13)\mapsto(0, 0, 1, 0),\ (-7, -15)\mapsto(0, 0, 1, 0)$\\
  $10_{152}$ &  $(-4, -13)\mapsto(0, 1, 0, 0)$ & $(-7, -19)\mapsto(0, 0, 1, 0)$\\
  $10_{153}$ &  $(0, 1)\mapsto(0, 1, 0, 0),\ (-2, -3)\mapsto(0, 0, 0, 1),\ (1, 3)\mapsto(0, 1, 0, 0)$ & $(0, 1)\mapsto(0, 0, 0, 1),\ (-2, -3)\mapsto(0, 0, 1, 0),\ (-3, -5)\mapsto(0, 0, 1, 0)$\\
  $10_{154}$ &  $(5, 17)\mapsto(0, 1, 0, 0)$ & $(2, 11)\mapsto(0, 0, 1, 0)$\\
  $10_{161}$ &  $(-4, -11)\mapsto(0, 1, 0, 0)$ & $(-7, -17)\mapsto(0, 0, 1, 0)$\\
  K$11n6$ &  $(0, 1)\mapsto(0, 1, 0, 0),\ (-1, -1)\mapsto(0, 1, 0, 0),\ (-3, -5)\mapsto(0, 0, 0, 1)$ & $(-1, -1)\mapsto(0, 0, 0, 1),\ (-4, -7)\mapsto(0, 0, 1, 0),\ (-3, -5)\mapsto(0, 0, 1, 0)$\\
  K$11n9$ &  $(5, 17)\mapsto(0, 1, 0, 0),\ (4, 15)\mapsto(0, 1, 0, 0),\ (3, 13)\mapsto(0, 1, 0, 0),\ (1, 9)\mapsto(0, 0, 0, 1)$ & $(0, 7)\mapsto(0, 0, 1, 0),\ (3, 13)\mapsto(0, 0, 0, 1),\ (2, 11)\mapsto(0, 0, 1, 0),\ (1, 9)\mapsto(0, 0, 1, 0)$\\
  K$11n12$ &  $(2, 7)\mapsto(0, 1, 0, 0),\ (0, 3)\mapsto(0, 1, 0, 0),\ (3, 9)\mapsto(0, 1, 0, 0)$ & $(2, 7)\mapsto(0, 0, 0, 1)$\\
  K$11n19$ &  $(0, -1)\mapsto(0, 1, 0, 0)$ & $(-3, -7)\mapsto(0, 0, 1, 0)$\\
  K$11n20$ &  $(0, 1)\mapsto(0, 1, 0, 0)$ & \\
  K$11n24$ &   & $(-2, -1)\mapsto(0, 0, 1, 0)$\\
  K$11n27$ &   & $(2, 11)\mapsto(0, 0, 1, 0)$\\
  K$11n31$ &  $(5, 15)\mapsto(0, 1, 0, 0),\ (4, 13)\mapsto(0, 2, 0, 0),\ (1, 7)\mapsto(0, 0, 0, 1),\ (3, 11)\mapsto(0, 1, 0, 0)$ & $(4, 13)\mapsto(0, 0, 0, 1),\ (2, 9)\mapsto(0, 0, 1, 0),\ (1, 7)\mapsto(0, 0, 1, 0),\ (0, 5)\mapsto(0, 0, 1, 0),\ (3, 11)\mapsto(0, 0, 0, 1)$\\
  K$11n34$ &  $(0, 1)\mapsto(0, 2, 0, 0),\ (-1, -1)\mapsto(0, 1, 0, 0),\ (-2, -3)\mapsto(0, 0, 0, 1),\ (1, 3)\mapsto(0, 1, 0, 0),\ (-3, -5)\mapsto(0, 0, 0, 1)$ & $(0, 1)\mapsto(0, 0, 0, 1),\ (-1, -1)\mapsto(0, 0, 0, 1),\ (-2, -3)\mapsto(0, 0, 1, 0),\ (-4, -7)\mapsto(0, 0, 1, 0),\ (-3, -5)\mapsto(0, 0, 2, 0)$\\
  K$11n38$ &  $(-3, -3)\mapsto(0, 0, 0, 1),\ (0, 3)\mapsto(0, 1, 0, 0)$ & $(-3, -3)\mapsto(0, 0, 1, 0),\ (-1, 1)\mapsto(0, 0, 0, 1),\ (-4, -5)\mapsto(0, 0, 1, 0)$\\
  K$11n39$ &  $(-2, -1)\mapsto(0, 0, 0, 1),\ (1, 5)\mapsto(0, 1, 0, 0),\ (0, 3)\mapsto(0, 1, 0, 0)$ & $(-2, -1)\mapsto(0, 0, 1, 0),\ (-3, -3)\mapsto(0, 0, 1, 0),\ (0, 3)\mapsto(0, 0, 0, 1)$\\
  K$11n42$ &  $(0, 1)\mapsto(0, 2, 0, 0),\ (-1, -1)\mapsto(0, 1, 0, 0),\ (-2, -3)\mapsto(0, 0, 0, 1),\ (1, 3)\mapsto(0, 1, 0, 0),\ (-3, -5)\mapsto(0, 0, 0, 1)$ & $(0, 1)\mapsto(0, 0, 0, 1),\ (-1, -1)\mapsto(0, 0, 0, 1),\ (-2, -3)\mapsto(0, 0, 1, 0),\ (-4, -7)\mapsto(0, 0, 1, 0),\ (-3, -5)\mapsto(0, 0, 2, 0)$\\
  K$11n45$ &  $(-2, -1)\mapsto(0, 0, 0, 1),\ (1, 5)\mapsto(0, 1, 0, 0),\ (0, 3)\mapsto(0, 1, 0, 0)$ & $(-2, -1)\mapsto(0, 0, 1, 0),\ (-3, -3)\mapsto(0, 0, 1, 0),\ (0, 3)\mapsto(0, 0, 0, 1)$\\
  K$11n49$ &  $(-3, -3)\mapsto(0, 0, 0, 1),\ (-1, 1)\mapsto(0, 1, 0, 0),\ (0, 3)\mapsto(0, 1, 0, 0)$ & $(-3, -3)\mapsto(0, 0, 1, 0),\ (-1, 1)\mapsto(0, 0, 0, 1),\ (-4, -5)\mapsto(0, 0, 1, 0)$\\
  K$11n57$ &  $(4, 15)\mapsto(0, 0, 0, 1),\ (1, 9)\mapsto(0, 0, 0, 1)$ & $(0, 7)\mapsto(0, 0, 1, 0),\ (3, 13)\mapsto(0, 0, 1, 0),\ (2, 11)\mapsto(0, 0, 1, 0),\ (1, 9)\mapsto(0, 0, 1, 0)$\\
  K$11n61$ &  $(0, 5)\mapsto(0, 0, 0, 1)$ & $(2, 9)\mapsto(0, 0, 1, 0),\ (-1, 3)\mapsto(0, 0, 1, 0),\ (0, 5)\mapsto(0, 0, 1, 0)$\\
  K$11n67$ &  $(2, 7)\mapsto(0, 1, 0, 0),\ (1, 5)\mapsto(0, 1, 0, 0),\ (-1, 1)\mapsto(0, 0, 0, 1)$ & $(-2, -1)\mapsto(0, 0, 1, 0),\ (1, 5)\mapsto(0, 0, 0, 1),\ (-1, 1)\mapsto(0, 0, 1, 0)$\\
  K$11n70$ &  $(1, 7)\mapsto(0, 0, 0, 1)$ & $(-2, 1)\mapsto(0, 0, 1, 0),\ (1, 7)\mapsto(0, 0, 1, 0),\ (0, 5)\mapsto(0, 0, 1, 0)$\\
  K$11n73$ &  $(-2, -1)\mapsto(0, 0, 0, 1),\ (1, 5)\mapsto(0, 1, 0, 0),\ (0, 3)\mapsto(0, 1, 0, 0)$ & $(-2, -1)\mapsto(0, 0, 1, 0),\ (-3, -3)\mapsto(0, 0, 1, 0),\ (0, 3)\mapsto(0, 0, 0, 1)$\\
  K$11n74$ &  $(-2, -1)\mapsto(0, 0, 0, 1),\ (1, 5)\mapsto(0, 1, 0, 0),\ (0, 3)\mapsto(0, 1, 0, 0)$ & $(-2, -1)\mapsto(0, 0, 1, 0),\ (-3, -3)\mapsto(0, 0, 1, 0),\ (0, 3)\mapsto(0, 0, 0, 1)$\\
  K$11n77$ &  $(5, 19)\mapsto(0, 1, 0, 0)$ & $(2, 13)\mapsto(0, 0, 1, 0)$\\
  K$11n79$ &   & $(-2, -1)\mapsto(0, 0, 1, 0)$\\
  K$11n80$ &  $(-1, -3)\mapsto(0, 1, 0, 0),\ (-3, -7)\mapsto(0, 0, 0, 1),\ (0, -1)\mapsto(0, 1, 0, 0)$ & $(-4, -9)\mapsto(0, 0, 1, 0),\ (-1, -3)\mapsto(0, 0, 0, 1),\ (-3, -7)\mapsto(0, 0, 1, 0)$\\
  K$11n81$ &   & $(2, 11)\mapsto(0, 0, 1, 0)$\\
  K$11n88$ &   & $(2, 11)\mapsto(0, 0, 1, 0)$\\
  K$11n92$ &  $(0, 1)\mapsto(0, 1, 0, 0)$ & \\
  K$11n96$ &  $(-2, -1)\mapsto(0, 0, 0, 1),\ (1, 5)\mapsto(0, 1, 0, 0),\ (0, 3)\mapsto(0, 1, 0, 0)$ & $(-2, -1)\mapsto(0, 0, 2, 0),\ (-3, -3)\mapsto(0, 0, 1, 0),\ (0, 3)\mapsto(0, 0, 0, 1)$\\
  K$11n97$ &  $(-1, -1)\mapsto(0, 0, 0, 1),\ (2, 5)\mapsto(0, 1, 0, 0),\ (1, 3)\mapsto(0, 1, 0, 0)$ & $(-1, -1)\mapsto(0, 0, 1, 0),\ (-2, -3)\mapsto(0, 0, 1, 0),\ (1, 3)\mapsto(0, 0, 0, 1)$\\
  K$11n102$ &  $(-5, -9)\mapsto(0, 0, 0, 1),\ (-2, -3)\mapsto(0, 1, 0, 0),\ (-3, -5)\mapsto(0, 1, 0, 0)$ & $(-5, -9)\mapsto(0, 0, 1, 0),\ (-6, -11)\mapsto(0, 0, 1, 0),\ (-3, -5)\mapsto(0, 0, 0, 1)$\\
  K$11n104$ &  $(4, 15)\mapsto(0, 1, 0, 0),\ (3, 13)\mapsto(0, 1, 0, 0),\ (1, 9)\mapsto(0, 0, 0, 1)$ & $(0, 7)\mapsto(0, 0, 1, 0),\ (3, 13)\mapsto(0, 0, 0, 1),\ (2, 11)\mapsto(0, 0, 1, 0),\ (1, 9)\mapsto(0, 0, 1, 0)$\\
  K$11n111$ &  $(2, 9)\mapsto(0, 1, 0, 0),\ (1, 7)\mapsto(0, 1, 0, 0),\ (-1, 3)\mapsto(0, 0, 0, 1)$ & $(-2, 1)\mapsto(0, 0, 1, 0),\ (1, 7)\mapsto(0, 0, 0, 1),\ (-1, 3)\mapsto(0, 0, 1, 0)$\\
  K$11n116$ &  $(0, 1)\mapsto(0, 1, 0, 0),\ (-1, -1)\mapsto(0, 1, 0, 0),\ (-3, -5)\mapsto(0, 0, 0, 1)$ & $(-1, -1)\mapsto(0, 0, 0, 1),\ (-4, -7)\mapsto(0, 0, 1, 0),\ (-3, -5)\mapsto(0, 0, 1, 0)$\\
  K$11n126$ &   & $(2, 11)\mapsto(0, 0, 1, 0)$\\
  K$11n133$ &  $(0, 5)\mapsto(0, 0, 0, 1)$ & $(2, 9)\mapsto(0, 0, 1, 0),\ (-1, 3)\mapsto(0, 0, 1, 0),\ (0, 5)\mapsto(0, 0, 1, 0)$\\
  K$11n135$ &  $(4, 13)\mapsto(0, 1, 0, 0),\ (1, 7)\mapsto(0, 0, 0, 1),\ (3, 11)\mapsto(0, 1, 0, 0)$ & $(1, 7)\mapsto(0, 0, 1, 0),\ (0, 5)\mapsto(0, 0, 1, 0),\ (3, 11)\mapsto(0, 0, 0, 1)$\\
  K$11n138$ &   & $(-2, -1)\mapsto(0, 0, 1, 0)$\\
  K$11n143$ &  $(2, 7)\mapsto(0, 1, 0, 0),\ (1, 5)\mapsto(0, 1, 0, 0),\ (-1, 1)\mapsto(0, 0, 0, 1)$ & $(-2, -1)\mapsto(0, 0, 1, 0),\ (1, 5)\mapsto(0, 0, 0, 1),\ (-1, 1)\mapsto(0, 0, 1, 0)$\\
  K$11n145$ &  $(-2, -1)\mapsto(0, 0, 0, 1),\ (1, 5)\mapsto(0, 1, 0, 0),\ (0, 3)\mapsto(0, 1, 0, 0)$ & $(-2, -1)\mapsto(0, 0, 1, 0),\ (-3, -3)\mapsto(0, 0, 1, 0),\ (0, 3)\mapsto(0, 0, 0, 1)$\\
  K$11n151$ &  $(2, 9)\mapsto(0, 1, 0, 0),\ (1, 7)\mapsto(0, 1, 0, 0),\ (-1, 3)\mapsto(0, 0, 0, 1)$ & $(-2, 1)\mapsto(0, 0, 1, 0),\ (1, 7)\mapsto(0, 0, 0, 1),\ (-1, 3)\mapsto(0, 0, 1, 0)$\\
  K$11n152$ &  $(2, 9)\mapsto(0, 1, 0, 0),\ (1, 7)\mapsto(0, 1, 0, 0),\ (-1, 3)\mapsto(0, 0, 0, 1)$ & $(-2, 1)\mapsto(0, 0, 1, 0),\ (1, 7)\mapsto(0, 0, 0, 1),\ (-1, 3)\mapsto(0, 0, 1, 0)$\\
  K$11n183$ &  $(5, 17)\mapsto(0, 1, 0, 0)$ & $(2, 11)\mapsto(0, 0, 1, 0)$\\
 \bottomrule
\end{longtable}
\end{scriptsize}
\bibliography{Bibliography} 
\bibliographystyle{alpha}
\end{document}